\documentclass[leqno,11pt]{amsart}
\usepackage{amsmath,amssymb,amsfonts,amscd}
\usepackage[mathscr]{eucal}
\usepackage{tikz}
\usepackage{verbatim}
\usepackage[colorlinks,pagebackref]{hyperref}
\usepackage{rotating}
\usepackage{xcolor}
\definecolor{antique}{cmyk}{0.0, 0.055, 0.08, 0.}

\usepackage{color}
\theoremstyle{plain}
\newtheorem{thm}{Theorem}[section]

\newtheorem*{thm*}{Theorem}

\newtheorem{sublem}[equation]{Lemma}

\newtheorem{subcor}[equation]{Corollary}
\newtheorem{subprop}[equation]{Proposition}

\newtheorem{subthm}[equation]{Theorem}

\theoremstyle{definition}

\newtheorem{cosa}[thm]{}
\newtheorem{subcosa}[equation]{}

\newtheorem{subex}[equation]{Example}
\newtheorem{subexs}[equation]{Examples}

\theoremstyle{remark}

\newtheorem{subrem}[equation]{Remark}
\newtheorem{subrems}[equation]{Remarks}
\newtheorem{subrems*}{Remarks.}

\newtheorem{subxrz}[equation]{Exercise}

\numberwithin{equation}{thm}

\numberwithin{equation}{thm}

\newcommand{\D}{\boldsymbol{\mathsf{D}}}

\newcommand{\K}{\boldsymbol{\mathsf{K}}}

\newcommand{\Hr}{\textup{H}}
\newcommand{\Hp}{\textup{H}_\varphi}
\newcommand{\LL}{\mathsf L}
\newcommand{\R}{\mathsf R}

\newcommand{\tr}{\textup{tr}}

\newcommand{\SA}{\mathsf{A}}
\newcommand{\SB}{\mathsf{B}}

\newcommand{\SEE}{\mathsf{E}}
\newcommand{\SF}{\mathsf{F}}
\newcommand{\SG}{\mathsf{G}}

\newcommand{\SL}{\mathsf{L}}

\newcommand{\sA}{\mathscr{A}}

\newcommand{\sE}{\mathscr{E}}

\newcommand{\1}{\mathbf{1}}

\newcommand{\op}{{\mathsf o\mathsf p}}

\newcommand{\ZZ}{\mathbb Z}

 \newcommand{\Rf}{\R f^{}_{\<\<*}}

\newcommand{\sExt}{\mathscr E\<{\mathit{x\mkern-.75mu t}}}

\newcommand{\fst}{{f^{}_{\<\<*}}}

\newcommand{\ush}[1]{{#1^{\textup{\texttt\#}}}}
\newcommand{\Ush}[2]{{#1_#2^{\textup{\texttt\#}}}}

\newcommand{\pt}{\phi^{\flat}}

\newcommand{\oO}{\>{\overline{\<\CO\<}\>}}
\newcommand{\oQ}{\overline{Q}}

\newcommand{\oY}{\overline{Y\<}\>}

\newcommand{\qc}{\mathsf{qc}}

\newcommand{\Dqc}{\D_{\mathsf{qc}}}

\newcommand\Dqcpl{\D_\qc^{\lift.95,\text{\cmt\char'053},}}

\newcommand\Dpl{\D^{\lift.95,\text{\cmt\char'053},}}
\def\Kb#1;#2;{K^\bullet_{\!#1}(#2)}

\font\cmt=cmtex10

\newcommand{\CH}{\mathcal H}
\newcommand{\CI}{\mathcal I}
\newcommand{\CJ}{\mathcal J}
\newcommand{\CL}{\mathcal L}
\newcommand{\CM}{\mathcal M}
\newcommand{\CN}{\mathcal N}
\newcommand{\CO}{\mathcal O}

\newcommand{\CS}{\mathcal S}
\newcommand{\CT}{\mathcal T}

\newcommand{\bpic}{\begin{tikzpicture}}
\newcommand{\epic}{\end{tikzpicture}}

\newcommand{\Otimes}[1]{\otimes^\LL_{#1}}

\newcommand{\sHom}{\CH om}

\newcommand{\set}{\!:=}

\newcommand{\sX}{{\<\<X}}

\newcommand{\sst}{\scriptstyle}
\newcommand{\sss}{\scriptscriptstyle}

\newcommand{\smallcirc}{{\>\>\lift1,\sst{\circ},\,}}

\newcommand{\<}{\mkern-1mu}
\renewcommand{\>}{\mkern1mu}
\newcommand{\va}[1]{\vspace{#1pt}}

\def\cit#1;#2;{\cite[#1]{#2}}

\newcommand{\kf}{\kern.5pt}

\def\lift#1,#2,{\vbox to 0pt{\vskip-#1 ex\hbox{$\scriptstyle #2$}\vss}}

\newcommand{\OX}{\mathcal O_{\<\<X}}
\newcommand{\OY}{\mathcal O_Y}
\newcommand{\OZ}{\mathcal O_{\<Z}}

\newcommand{\OW}{\mathcal O_W}
\newcommand{\OV}{\mathcal O_V}
\newcommand{\OU}{\mathcal O_U}

\newcommand{\fundamentalclassa}[1]{{\boldsymbol{\mathsf{a}}}_{#1}}

\newcommand{\fundamentalclassb}[1]{{\boldsymbol{\mathsf{b}}}_{#1}}

\newcommand{\circled}[1]{\textcircled{\raisebox{-.25pt}{\scriptsize{#1}}}}

\newcommand{\lto}{\longrightarrow}
\newcommand{\xto}{\xrightarrow}

\newcommand{\lot}{\longleftarrow}

\newcommand\iso{{\mkern8mu\longrightarrow \mkern-25.5mu{}^\sim\mkern17mu}}
\newcommand\osi{{\mkern8mu\longleftarrow \mkern-25.5mu{}^{\>\sim}\mkern17mu}}

\DeclareMathOperator{\spec}{Spec}
\DeclareMathOperator{\Hom}{Hom}
\DeclareMathOperator{\tor}{Tor}
\DeclareMathOperator{\ext}{Ext}
\DeclareMathOperator{\stor}{\CT\!\<\<\mathit{or}}
\DeclareMathOperator{\id}{id}

\DeclareMathOperator{\via}{{\textup{via}}}

\newcommand{\Iso}{\vbox to 0pt{\vss\hbox{$\widetilde{\phantom{nn}}$}\vskip-7pt}}

\newcommand{\ov}{\overline}
\newcommand{\brf}{\bar{f\:\<}^{\!\<*}}
\newcommand{\SH}[1]{\mathfrak{s}_{\<#1}\>\>}
\newcommand{\upcheck}[1]{{#1}^{\sss\boldsymbol \vee}}

\def\cA #1; #2;{\cite[p.\,#1, #2]{A}}
\def\cT #1; #2;{\cite[p.\,#1, #2]{T}}

\def\lift#1,#2,{\vbox to 0pt{\vskip-#1 ex\hbox{$\scriptstyle #2$}\vss}}

\def\drlm#1{\underset{\vtop{\vskip-4.2pt\hbox to 14pt{\rightarrowfill} \vskip-10pt\hbox{$\scriptstyle \ #1$}}}\to\lim\,}

\def\dirlm#1{\lim\hskip-1.65em\lower1.37ex
       \hbox{\smash[b]{$
                   \underset{\lift 1.37,
                                         {\hbox to 0pt{\hss$\scriptscriptstyle#1$\hss}},
                                  }
                     {\colon \hbox to 1.37em {\rightarrowfill}}
               $} }                      
     \!\<}
\newcommand{\bsr}{\bar{\SH{}}_{\<\<\!R}}
\newcommand{\bst}{\bar{\SH{}}_{\<\<T}}
\newcommand{\cf}{c^{}_{\mkern-1.5mu f}}

\begin{document}
\author[J. Lipman]{Joseph Lipman}
\address{142 Ranch Ln\\
Santa Barbara CA 93111, USA}
\email{jlipman@purdue.edu}
\urladdr{http://www.math.purdue.edu/\~{}lipman/}


\title{Grothendieck Duality theories---abstract and concrete}

\begin{abstract}
Grothendieck Duality---the theory of the twisted inverse image pseudofunctor $(-)^!$ 
over a suitable category of scheme\kf-maps---can be developed \emph{concretely},  with emphasis on explicit constructions, 
or, in greater generality, \emph{abstractly},  with emphasis on category-theoretic considerations. 
We aim to connect these approaches, a nontrivial matter involving some alluring relations,  \mbox{for instance} among differential forms, residues and duality.  In particular, it emerges that the culminating Ideal Theorem  in  Hartshorne's ``Residues and Duality" holds for arbitrary essentially-finite-type maps of noetherian schemes and bounded-below complexes with quasi-coherent cohomology.

What appears here mostly concerns pseudo\kf-coherent finite maps. The rest is being prepared.

\end{abstract} 

\maketitle 

\tableofcontents
 
\section{Introduction}\label{intro}
\begin{cosa}\label{NoTerm}
(Notation and terminology.) 
A \emph{ringed space} is a topological space $X$ furnished with a sheaf $\OX$ of commutative rings. 
A map $(f,\theta)\colon X\to Y\<$  of ringed spaces consists of a continuous map $f\colon X\to Y$ 
and a homomorphism of sheaves of rings $\theta\colon\OY\!\to\!\fst\OX$. 
(The appropriate $\theta$, though under\-stood to be present,  is often left out of the notation.)
Such spaces and  maps form a category $\mathbf S$:
the composition of $(g,\psi):Y\to Z$ with $(f,\theta)$ 
is $(g\smallcirc\! f, \>\>g_*\theta\smallcirc\< \psi)$, and the identity $\id_X$ of~$X$ is the map for which $Y=X$ and both $f$ and $\theta$ are identity maps.
Schemes and their maps constitute a full subcategory. 

A diagram depicting $\mathbf S$-maps  is \emph{natural} if each unlabeled arrow in it represents\- a map whose description, while omitted, is presumed evident. Arrows decorated with a
``\raisebox{2pt}{$\Iso$}" or  ``$\>\>\simeq\>\>$" represent isomorphisms.

The abelian category of $\OX$-modules on a ringed space $X$  is denoted~$\sA(X)$; 
$\sA_{\qc}(X)\subset \sA(X)$ is the
full subcategory spanned by the \emph{quasi-coherent} $\OX$-modules. 
The derived category of $\sA(X)$ is denoted $\D(X)$; $\Dqc(X)\subset \D(X)$ is the full subcategory   spanned by the $\OX$-complexes with quasi-coherent cohomology.
$\Dpl(X)\subset\D(X)$ is the full subcategory   
spanned by the \emph{locally} cohomologically bounded-below complexes (those $C\in\D(X)$ 
for which there is an open cover $(X_\alpha)_{\alpha\in A}$ of $X$ and  for each $\alpha$ 
an integer $n_\alpha$ such that the restriction 
$(H^iC)|_{X_\alpha}\<\<$ vanishes for all $i<n_\alpha\>$);
and $\Dqcpl(X)\set \Dqc(X)\cap\Dpl(X)$.\looseness=-1\va1

To reduce clutter, for any monoidal category $(\mathbf C,\otimes)$ and $A,\>B,\>C\in\mathbf C$ we will
identify---harmlessly, via the natural isomorphism---$(A\otimes B)\otimes C$ with $A\otimes (B\otimes C)$,
and denote either of these objects as $A\otimes B\otimes C$.

Assigning to each ringed-space map $f\colon X\to Y$ the derived direct image $\R\fst\colon\D(X)\to\D(Y)$
and the derived inverse image $\LL f^*\colon\D(Y)\to\D(X)$ leads to a pair of \emph{adjoint monoidal
pseudofunctors} on $\mathbf S$, see \cite[3.6.7, 3.6.10]{li}.\looseness=-1

The abbreviation ``qcqs" denotes ``quasi-compact and quasi-separated."%
\footnote{In the oft to be referred-to exposition \cite{li}, this condition is called ``concentrated." The frequent subsequent references to \cite{li} (of which
these notes may be viewed as a continuation) are due much more to its approach and convenience than to any originality.}
A basic fact of Grothendieck Duality theory is that \emph{for any map \mbox{$f\colon X\to Y$} of qcqs schemes, 
the restriction to\/ $\Dqc(X)$ of\/  $\R\fst$ has a right adjoint,} i.e., 
there exist a functor $f^{\<\times}\colon\D(Y)\to\Dqc(X)$ and 
a functorial map
\[
\tau^\times_{\!f}\colon\R\fst f^{\<\times} G\to G\qquad (G\in\D(Y)) 
\]
such that for any complex $F\in\Dqc(X)$, the natural composite map\va{-4}
\[
\Hom_{\D(X)}(F, f^{\<\times} G\>)\to\Hom_{\D(Y)}(\R\fst F\<, \R\fst f^{\<\times} G\>)
\xto{\via\tau_{\!f}^{\<\times}\<}\Hom_{\D(Y)}(\R\fst F\<, G\>)
\]
is an isomorphism. (See, e.g., \cite[Corollary 4.1.2]{li} and the notes following its proof.)

\end{cosa}

\begin{cosa}\label{Ideal intro} 
One of our goals, ultimately,  is to prove the ``Ideal Theorem," 
called in \cite[p.\,6]{RD}  the \emph{primum mobile} of that book, 
and which---with restrictions we won't need
(existence of dualizing complexes, coherence of cohomology of 
sheaf-complexes)---is one of its main results (\emph{ibid.,}~p.\,383, Corollary 3.4).

Paraphrased, the Ideal Theorem asserts, \emph{first of all,} the existence of a 
\emph{duality pseudofunctor,} by which is meant a $\Dqcpl$-valued pseudofunctor $(-)^!$ on the category~$\sE$ of finite\kf-type separated maps of noetherian schemes, and for each proper $\sE$-map~$f$ a functorial map 
$\tau^{}_{\!f}\colon\Rf f^!\to\id$, satisfying the following properties, of which (i), (ii) and~ (iv) jointly determine these data up to unique isomorphism:

\smallskip
(i) For any \'etale $\sE$-map~~$f$,  $f^!$ is the usual restriction functor~$f^*\<$.\va1

(ii) (Duality). For any proper $\sE$-map~$f$, $\tau^{}_{\!f}$ makes $f^!$  right-adjoint
to $\Rf$ (i.e., in \S\ref{NoTerm}, one can take $(f^{\<\times}\<\<,\tau^\times_{\!f})\set(f^!\<,\tau^{}_{\!f})$).\va1

(iii) (Flat base change). For any fiber square in $\sE$
\begin{equation}\label{fiber}
\CD
X'@>v>> X\\
@VgV\mkern 20mu V @VVfV \\
Z' @>\lift4.8,\displaystyle\clubsuit,>u> Z
\endCD
\end{equation}
with $f$ (hence $g$) proper and $u$ (hence $v$) flat,
and $F\in\Dqcpl(Z)$, the map
\begin{equation}\label{beta}
\beta_{\clubsuit}(F\>)\colon:v^*\!f^!\<\<F\to g^!\<u^*\<\<F
\end{equation}
adjoint (via (ii)) to the natural composition
$$
\R g_* v^*\! f^!\<F{\iso}u^*\Rf f^!\<F
\xto[\lift1.35,\!u^*\<\tau^{}_{\!f}\!,]{} u^{\<*}\<\<F\\[1.5pt]
$$
is an \emph{isomorphism}.\va1

(iv) (For gluing (i) and (ii)). In (iii), if $u$ (hence $v$) is an open immersion, then $\beta_{\clubsuit}(F\>)$ is equal to 
the natural composite isomorphism (which exists for \emph{any} commutative square~\eqref{fiber} with $u$ and $v$ \'etale)
$$
v^*\<\<f^!\<F= v^!\<f^!\<F
\iso (fv)^!\<F=(ug)^!\<F\iso
g^!\<u^!\<F  =g^!\<u^*F.
$$

This much of the Ideal Theorem  is contained in Theorems 4.8.1
and~4.8.3 of the notes \cite{li}, and was extended to essentially-finite\kf-type maps 
by Nayak in \cite[\S5.2]{Nk09}.   The methods of proof are largely category-theoretic,
in~ line with the ``abstract" development of Grothendieck Duality initiated by Verdier and Deligne (see Deligne's Appendix in \cite{RD}, and also  \cite{Nm96}).\va4

\begin{small}
The pseudo\-functor~$(-)^!$ extends from $\Dqcpl$ to $\Dqc$ if one restricts to proper maps or to \mbox{$\sE$\kf-maps} of finite tor-dimension \cite[\S\S5.7--5.9]{AJL11}---and even without such restrictions if one relaxes ``pseudofunctor"\va1 to ``oplax functor,"\va{1} i.e., one allows that for an $\sE$\kf-diagram  
$W\overset{\lift .7,g\>,}\lto X\overset{\lift 1.15,f\>,}\lto Z $\va{.5} the associated map $(fg)^!\to g^!f^!$ need not be an isomorphism, see \cite{Nm23}. (For maps of finite 
tor-dimension the agreement of the oplax $(-)^!$  with the preceding pseudofunctor results from \cite[Prop.~13.11]{Nm23}.) In~fact, Neeman's results apply to a
broad class of noetherian stacks, including those of Deligne\kf-Mumford. 

Nayak has also established extensions to composites of pseudoproper maps and \'etale maps of formal schemes 
\cit Theorem 7.1.6;Nk05;, and to composites of proper flat pseudo\kf-coherent maps and \'etale maps of qcqs schemes \cit Theorem~7.3.2;Nk05;. In another direction, extensions are emerging in derived algebraic geometry, see, for example, \cit \S3;Ga;, \cite{LZ}, \cite{Sch}.

\end{small}

\medbreak

\emph{Secondly}---and this will be the focus of our attention---the Ideal Theorem gives \emph{concrete\- realizations} of the pseudofunctor $(-)^!$ over the subcategories of~$\sE$ spanned, respectively, by its finite maps and its smooth maps, and 
concrete descriptions of abstractly specified pseudo\kf\-functorial maps associated to some combinations of these two types of map.  

A concrete realization of $(-)^!$ is, informally stated, a \emph{concretely describable} duality pseudofunctor $\ush{(-)}\<$,   pref\-erably, though not necessarily,  canonical.%
\footnote{ \emph{Concrete} and \emph{canonical} are somewhat flexible concepts.    In what sense, for instance, is the number 1 canonical?  As Humpty Dumpty
said, ``a~word means\:\dots just what I choose it to mean."}

A concrete realization of a functorial map built up from concrete elementary maps involving the identity functor id, the derived tensor product and derived direct image, by means of categorical operations like adjunction, composition and  successive application of previously defined functors, is a concrete description of such a map.  (This somewhat vague characterization will be clarified by a number of examples, starting in section~\ref{bc}.)

\enlargethispage{-7pt}
In particular, for  proper $\sE$-maps $f$ one wants a \emph{concrete right adjoint}~$\ush f$  \emph{of\/~$\R\fst$}, and a \emph{concrete counit map} $\>\R\fst\ush f\<\to\id$, varying pseudofunctorially. Two such pairs are necessarily canonically isomorphic.\va3

The first part of this exposition explores constructs associated to certain \emph{finite} maps $f\colon X\to Y\<$, 
for which, as in \cite[pp.\:164--175]{RD}---though 
with weaker  hypotheses, see~\S\ref{2.4}---one defines ``quasi-concretely" a functor
\[
f^\flat\colon\Dqcpl(Y)\to\Dqcpl(X)\\[2pt]
\] 
plus a functorial $\fst\OX$-isomorphism 
\[
\bar t^{}_{\!f}\colon\R\fst f^\flat\iso\R\>\sHom_Y(\fst\OX,-)\>,
\]
such that with $t^{}_{\!f}$ the natural composite map
\[
\R\fst f^\flat 
\,\underset{\lift1.25,\bar t^{}_{\!f},}
\iso\R\>\sHom_Y(\fst\OX,-)
\lto \R\>\sHom_Y(\OY\<,-)
\iso\id,
\]
$(f^\flat\<, t^{}_{\!f})$ is a right adjoint for $\R\fst\>$. Thus\va{.6}  $(f^\flat\<, t^{}_{\!f})$ is a realization of~$\,(f^!\<,\tau^{}_{\!f})$
(\emph{pseudofunctorially,}  see \S\ref{pf adj}).\va{.6} The definition involves  
$\R\>\sHom(\fst\OX,-)$ and the left adjoint $\bar{f\:\<}^{\!\<*}\<$ of $\bar\fst$ where
$\bar f\colon(X,\OX)\to (Y,\fst\OX):=\oY$ is the natural ringed-space map (see Example~\ref{locally}); these functors, being
characterized by universal properties, are well-defined only up to canonical isomorphism,
and thus limit the extent to which $(f^\flat\<, t^{}_{\!f})$ can be considered  to be concrete or canonical.
 \va1

Sometimes, simpler realizations exist.
For example, restricting to finite maps $f\colon X\to Y$ that are locally finitely presentable and flat (that is,  
$\fst\OX$~is locally\- free of finite rank over~$\OY$),  
and to quasi-coherent $\OY$-complexes~$F$, there is a right adjoint $(f^\flat\<, \bar t^{}_{\!f})$\va{-1}
for the $\D(\sA_\qc)$-valued pseudo\-functor~$(-)_*$ with
$f^\flat\<\< F\set\bar{f\:\<}^{\!\<*}\sHom_Y(\fst\OX, F\>)\in\sA_{\qc}(X)$,  
and $\bar t^{}_{\!f}(F\>)$ the isomorphism
\[
\fst f^\flat \<\<F=\fst\bar{f\:\<}^{\!\<*}\sHom_Y(\fst\OX, F\>)
\iso\sHom_Y(\fst\OX,F\>)
\]
arising from $\bar{f\:\<}^{\!\<*}\<\colon\sA_\qc(\oY)\to\sA_\qc(X)$ being  
quasi-inverse to~$\bar\fst$. This $(f^\flat\<, \bar t^{}_{\!f})$ is as concrete or canonical as $\bar{f\:\<}^{\!\<*}$ is. If $Y$ is separated and quasi-compact, the natural functor $\D(\sA_\qc(Y))\to\Dqc(Y)$ is an equivalence, so every complex $G\in\Dqc(Y)$ is isomorphic (functorially) to a quasi-coherent complex $QG\>$; thus, over such schemes one gets  a realization of $(f^!\<,\tau^{}_{\!f})$ that\va{-1} is concrete or canonical to the extent that 
the functors $\bar f^*$ and $Q$ are.

Suppose, moreover, $f$ is  \emph{\'etale}. The usual trace map $\fst\OX\to\OY$
lifts to an $\fst\OX$-isomorphism $\fst\OX\iso\sHom_Y(\fst\OX, \OY)$, giving, for $G\in\D(Y)$, the first of the functorial $\fst\OX$-isomorphisms (the second being natural)
\begin{equation}\label{etaliso}
\fst\OX\otimes_Y G\iso \sHom_Y(\fst\OX,\OY)\otimes_YG\iso\sHom_Y(\fst\OX,G).
\end{equation}

\enlargethispage*{10pt}
If $\phi\colon \oY\to Y$ is the ringed-space map\va{.5} corresponding to the 
natural map $\OY\to\fst\OX$ (so that $f=\phi\bar f\>$), then the functor \va1
$\fst\OX\otimes_Y(-)\colon\D(Y)\to\D(\oY)$ is left-adjoint to $\phi_*$, and so may
be identified with $\phi^*$. \va{.5} Then by applying~$\bar{f\:\<}^{\!\<*}$ to~\eqref{etaliso}, one gets a functorial $\OX$-isomorphism\va{-1}
\[
c^{\>\flat}_{\<f}\colon f^*\<G\iso \bar{f\:\<}^{\!\<*}\sHom_Y(\fst\OX, G)=f^\flat G,
\]
\vskip-1pt\noindent
whence the concrete realization $(f^*\<,\textup{tr}^{}_{\!f})$
of~$\,(f^!\<,\tau^{}_{\!f})$, 
where for $G\in\Dqcpl(Y)$, $\textup{tr}^{}_{\!f}(G)$~is the natural composite map\va{-5}
\[
\R\fst f^*\<G\xto{\R\fst c^{\>\flat}_{\<f} }\R\fst f^\flat G\lto G,
\]
which can be shown (using e.g., \cite[(3.7.1)]{li} with $f\set\bar f$\va{.6} and $g\set\phi$ plus \cite[(3.4.7)(ii)]{li} with $A\set\phi^*G$ and $f\set\bar f\>$) to be the natural composite map\looseness=-1
\[
\R\fst f^*\<G\iso \fst\OX\<\otimes_Y G\xto{\!\textup{trace}\>\otimes\>\id\,} \OY\<\otimes_Y G \iso G.
\]
Again, this realization is canonical insofar as the left adjoint $f^*$ of $\fst$ is.

For an extension to ``almost \'etale" $\<f$, see Proposition~\ref{complementary}.\va3

When $f$ is a map of affine schemes, this all has a well-known commutative\kf-algebra translation. Indeed, for any commutative ring $R$, sheafification and the derived global-section functor induce inverse \emph{equivalences} between duality theory over 
the derived category of $R$-modules and that over $\Dqc(\spec R)$,  allowing us to realize functors and functorial maps in the latter through concrete commutative\kf-algebra constructions in the former. 

For example,
when $S$ is a finite\- $R$-algebra, with corresponding scheme\kf-map $f\colon\spec S\to\spec R$,
then one gets a concrete right adjoint for $\R\fst$ by sheafifying the 
fact that the restriction-of-scalars functor
from $\D(S)$ to~$\D(R)$ has as right adjoint the~functor $\R\<\<\Hom_{R}(S,-)$ together with the natural functorial $\D(R)$-map \va{-1} 
\[
\R\<\<\Hom_{R}(S,M)\to\R\<\<\Hom_{R}(R,M)=M \qquad (M\in\D(R)).
\]
\vskip-1pt

Another example is the commutative\kf-algebra map corresponding to the projection map 
((iii) in \S1.3 below), as described in Lemma~\ref{dualproj}.

There is much more along these lines in \S3.\va2

As another example, a detailed account of duality  provides,
for a \emph{smooth} $\sE$-map 
$f\colon X\to Y$ with fibers of pure dimension~$d$ and  $\Omega^{\>d}_{\<f}$ the sheaf of relative $d$\kf-forms,  a \emph{canonical functorial isomorphism}
\begin{equation}\label{v_f}
\ush f\<F\set\Omega^{\>d}_{\<f}[d\>]\otimes_\sX \<f^*\<\<F\underset{c^{}_{\<\<f}(F\>)}\iso f^!\<F\qquad(F\in\Dqc(Y))
\end{equation}
\vskip-1pt\noindent(where  ``$\>[d\>]\>$'' denotes $d$\kf-fold translation in $\D(X)$),  having pseudofunctorial variance~
induced by canonical isomorphisms of the form
\pagebreak[3]
\[
\Omega^{\>e}_{\<f}[e]\otimes_X \<f^*\Omega^{\>d}_{\<g}[d\>]\iso
\Omega^{\>d+e}_{gf}[d+e]
\] 
for smooth $\sE$-maps $g\colon Y\to Z$ with fibers of pure dimension~$e\>$;
and further, 
when~$f$ is also \emph{proper,} an explicit elucidation of the composite map
\[
\textup{Tr}^{}_{\!f}\colon\R\fst \ush f 
\underset{\R\fst \cf} \iso 
\R\fst f^!\xto[\lift.6,\:\tau^{}_{\!f},]{} \id\\[-3pt]
\]
via the theory of residues, as sketched in \cite[pp.\,398--400]{V} and developed in \cite{HS} or \cite[Proposition 4.2.2]{LS92}.%
\footnote{In \cite[p.\,383]{RD} this is required only for $X\set \mathbb P^{\>d}(Y)$ and $f\colon X\to Y$ the natural map, in which case $\textup{Tr}^{}_{\!f}$ can be described 
explicitly via \v Cech complexes, see e.g., \cite[\S5]{HK}.} 

Assuming the first part of the Ideal Theorem, together with a canonical representation of $g^!$ when $g$ is a regular immersion\va{.6} \cite[p.\,180, Corollary 7.3]{RD} 
(reproduced, in essence, in Proposition~\ref{KosReg} below), 
Verdier constructed such a $\cf$\va{-.6} in \cite[Proof of Theorem 3]{V}. (A much expanded treatment of this ``fundamental class" is given in \cite{LN17}.)
But proving pseudofunctoriality of~$\cf$  is~far from straightforward.
That, and much more about $\cf$ and its relation to traces and residues, is addressed in~\cite{NS}, in the context of formal schemes.\va2

The finite \'etale situation is the overlap between the finite one and the  smooth one. The ``concrete" proof in \cite{RD} of the Ideal Theorem depends on pseudofunctorially gluing\va{-1} 
$(-)^\flat$ and $\Omega_{(-)}^*[*]$ along that overlap. 
In contrast, the basic idea here will be to show\va{.5} (with weaker hypotheses) that the above pseudo\-functor $(-)^!$, \emph{abstractly constructed} via gluing of $(-)^*$ over \'etale maps and of a right adjoint of~$\R(-)_*$ over proper maps, is isomorphic over finite (resp.~smooth) maps to $(-)^\flat$ (resp.~$\Omega_{(-)}^*[*]$).

\end{cosa}

\begin{cosa}\label{expressions}
The foregoing instantiates a more general theme, as follows.

For our relatively unsophisticated purposes, the \emph{abstract  yoga of duality} takes a dualizing structure 
on~a category~ $\mathbf C$ to be an adjoint pair $(\>{}^{\boldsymbol *},{}_{\boldsymbol *}\>)$ of monoidal pseudofunctors (\cite[\S3.6]{li} or ~\cite[Lecture~3]{lil}),
taking values in closed categories \mbox{$\D_\sX\ (X\in\mathbf C)$} with product~$\otimes_\sX$ and 
unit~$\OX$, \emph{plus} for each $\mathbf C$-map $\psi\colon X\to Y$ a right adjoint $\psi^\times$ of the functor $\psi_*\colon\D_\sX\to\D_Y$, with specified counit map $\psi_*\psi^\times\to\id$,
\emph{plus} a class~$\mathsf S$ of oriented commutative $\mathbf C$-squares\va{-6}%
\footnote{An oriented commutative square is a quadruple of maps $(u,f,v,g)$ such that $ug=fv$, \emph{plus} an ordering  of the pair $(u,f)$.  A diagram such as the following one always represents a commutative square that 
is oriented by putting the bottom arrow $u$ ahead of $f$.}
\[
\CD
X'@>v>> X\\[-4pt]
@VgV\mkern 20mu V @VVfV \\[-3pt]
Z' @>\lift4,\displaystyle\spadesuit,>u> Z,
\endCD
\] 
with $\mathsf S$ closed under vertical and horizontal juxta\-position,
such that for all $\psi\colon X\to Y,$ the following maps are \emph{isomorphisms}:\va2

(i) the map $\psi^*\OY\to \OX$ adjoint to the natural map $\OY\to\psi_*\OX$;

(ii) for all $F,\>G\in\D_Y,$ the map $\psi^*(F\otimes_Y G)\to \psi^*\<F\otimes_X \psi^*\< G$ adjoint to the natural composite map
\[
F\otimes_Y G \to \psi_*\psi^* \<F\otimes_Y \psi_*\psi^*\<G \to \psi_*(\psi^*\< F\otimes_\sX \psi^*\<G);
\]

(iii) for all $E\in\D_X$ and $F\in\D_Y,$ the  map 
$\psi_*E\otimes_Y\< \<F\to\psi_*(E\otimes_\sX \psi^*\<F\>) $
(``\kf projection map") adjoint to the natural composite map
\[
\psi^*(\psi_*E\otimes_Y \<\<F\>)\to \psi^*\psi_*E\otimes_\sX\psi^*\< F\to E\otimes_X \psi^*\<F;
\]

(iv) for each $\spadesuit\in\mathsf S$ and $G\in\D_{\sX}$,
 the map $u^*\!\fst G\to g_*v^* G$ adjoint to the natural composite map
\(
g^*u^*\!\fst G\iso v^*\<\<f^*\mkern-2.5mu\fst G \lto v^*\<G\>;
\)

\enlargethispage*{5pt}
(v) for each $\spadesuit\in\mathsf S$
 and $F\in\D_{\<\<Z}$, the  map
\(
 v^*\<\<f^\times\<\<F\to g^\times\<u^*\<\<F,
\)
adjoint to the natural composite map
\(
g_* v^*\mkern-2mu f^\times\<\<F{\iso}
u^*\!\fst f^\times\<\<F
\lto u^{\<*}\<\<F.
\)

On these axiomatic foundations, one constructs, using only categorical operations (adjunction, composition, composite functors\,\dots), 
a super\-structure of pseudofunctorial maps, and compatibilities among them expressed by commutative diagrams. 
(For more in this vein, see e.g., \cite[\S3.5.4]{li} or \cite[Lectures 4 and 6]{lil}. Even more generally, see \cite{Ho} or \cite[Part I]{CD}.)\looseness=-1

This abstract theory is modeled by a variety of specific situations.
 For example (somewhat oversimplified), $\mathbf C$ could be some category of commutative rings, $\D_\sX$ the category of $X$-modules (or its derived category) with the usual closed structure, and $(\>{}^{\boldsymbol *},{}_{\boldsymbol *}\>)$ the usual (or derived) extension- and restriction-of-scalars pseudofunctors. (For elaboration, see \S\ref{affex}.) Or, $\mathbf C$ could be some category of ringed spaces, $\D_\sX$ the category of quasi-coherent $\OX$-modules (or $\Dqcpl(X)$, see \S\ref{Ideal intro}), and $(\>{}^{\boldsymbol *},{}_{\boldsymbol *}\>)$ the usual (or derived) inverse\kf- and direct-image pseudofunctors. Or, $\mathbf C$ could be the category of compactifiable maps of qcqs schemes, and $\D_\sX$ the derived category of torsion sheaves on $X$ with the \'etale topology \cite{De73}. Other categories that support duality theories are those whose objects are certain finite diagrams of noetherian schemes, with flat arrows \cite{Ha}, or certain finite ringed spaces \cite{SS}, or certain\- algebraic stacks \cite{Nm23}. With a few extras,  one can also consider categories of topological rings (local\- duality) or noetherian formal schemes, each $\D_\sX$ being a suitable ordinary or derived category \cite{AJL99}. There are other examples, for instance those mentioned in \S\ref{Ideal intro}. Undoubtedly, more will emerge in the future. \va2

In specific situations, to  enliven things and enhance applicability  one needs \emph{concrete interpretations} of the functors and maps in the preceding conditions (i)--(v), as well~as of useful pseudofunctorial maps that can be categorically derived. 

For instance, in the context of \S\ref{Ideal intro}, concrete interpretations of~$f^{\times}$ were indicated for smooth or finite $\sE$-maps. 
As for noteworthy derived maps, in that situation there are concrete descriptions, at least via flat or injective resolutions, in various places in \cite{RD}. However, the question of whether the maps so described are the same as the corresponding categorically-defined ones is not often addressed.
(See the footnote in the proof of~\ref{tensor} below.) So~one cannot say without further ado that, for example, the  multi\-tudinous diagrams in \cite{li} that are abstractly shown to commute remain commutative when their maps are interpreted as in \cite{RD}.

\va3

\emph{The overall goal here is to embed the concrete duality theory in \cite{RD}  
$(\<$amended in \cite{Co}$)$ into the abstract one in \cite{li}$, $ by showing how important maps in \cite{RD} 
can be described in category-theoretic terms.\va3}%
\footnote{For orientation, consider the analogous theme in the elementary duality theory of ordinary 
restriction- and extension-of-scalars pseudofunctors between modules over  commutative rings.}

Section~\ref{pc finite} below is devoted to doing this for certain finite maps, section~\ref{affex} to translation into commutative\kf-algebra terms.
Beyond  (i)--(v), there being an endless number of maps that can be categorically deduced,  only a few salient examples will be examined closely, such as the maps
\begin{align*}
\LL f^*\<\<E\Otimes{\sX} f^\times\<\< F&\lto f^\times\<(E\Otimes{Y}F\>) &&(E\<, \:F\in\Dqc(Y)),\\
\R\>\sHom_\sX(\LL f^*E\<, f^\times\<\< F\>)&\lto f^\times\R\>\sHom_Y(E\<,F\>) &&(E\<, \:F\in\Dqc(Y)),
\end{align*}
associated with a finite map $f\colon X\to Y\<$.
(See~\ref{flat and tensor} and~\ref{flat and hom}.) 

It is intended that  smooth $\sE$-maps will be treated in a subsequent part of this exposition. 
For such maps, basic results---even for the context of formal schemes---can be found in \cite{NS}.

\end{cosa}

\section{Pseudo-coherent finite maps}
\label{pc finite}

Recall that a scheme\kf-map $f\colon X\to Y$ is \emph{affine} (resp.~\emph{finite}) if for
each affine open subscheme $U\subset  Y\<$---or equivalently, for every member of 
some affine open cover of $Y$---the scheme $f^{-1}U$ is affine (resp.~affine and
such that the natural map makes $\Gamma(f^{-1}U,\OX)$ into a finite $\Gamma(U,\OY)$-module), see \cite[1.3.2, 6.1.4]{EGA2}. Any affine $f$ is \emph{separated,} and when for each $U$ the $\Gamma(U,\OY)$-module $\Gamma(f^{-1}U,\OX)$ is locally finitely presentable, \emph{proper}.\va1


This section~\ref{pc finite}  is concerned with a\va{-1.5} 
concrete $\Dqcpl$-valued pseudofunctor~$(-)^\flat$ together with  functorial maps $\R(-)_*(-)^\flat\to \id$,
over the category~$\mathbf\Phi$ of  finite maps $f\colon X\to Y$ that are \emph{pseudo-coherent}, 
meaning $\fst\OX$ is locally resolvable by a complex of finite\kf-rank  locally free $\OY$-modules (see \ref{quasi}).
(For example,  finite~maps of  locally noetherian schemes, finite~locally free maps, and regular immersions all are pseudo\kf-coherent.) 
These data constitute a pseudofunctorial
right adjoint for $\R(-)_{*}\>$.  
Restricting to qcqs schemes in~$\mathbf\Phi$, one has then
a  \emph{concrete realization of the pseudofunctorial pair} $((-)^\times, \tau_-^\times)$.\va1
 
The locally noetherian case is treated in  \cite[pp.\,164--175]{RD}, 
where it is indicated\- that the ``usual reductions" cut things down to the elementary context of modules over commutative rings. (Cf.~section \ref{affex} below.) 
The present approach  is more general and technical, and also more explicit, than that classical one, but basically similar, 
as follows.\va2

Fix a scheme $Y\<$. The direct-image functor gives an equivalence between (i): a category whose objects are pairs $(X,F\>)$ with $X$ a scheme affine over~$Y$ and $F$ a quasi-coherent $\OX$-module 
and, (ii): the oppo\kf\-site of a category whose objects are pairs 
$(\SL,\SF\>)$ with $\SL$ a quasi-coherent $\OY$-algebra and $\SF$ a quasi-coherent $\SL$-module, see \cite[\S9.2]{EGA1}. 
In greater generality, with
$F$ and~$\SF$ replaced by objects in  $\Dqc(X)$ and $\Dqc(\SL)$ respectively, the derived direct-image functor induces an equivalence, see~Proposition~\ref{^* equivalence}.  There is an explicit $\Dqc(-)$-valued duality pseudofunctor over quasi-coherent $\OY$-algebras, globalizing the well-known pseudofunctor over  commutative rings. (For the latter, see~\ref{affine duality}--\ref{pfush}). To transfer this pseudo\kf\-functor over to  
$Y\<$-schemes, via the equivalence, one needs to remain in a quasi-coherent context.
This can be done, for instance, 
using right adjoints $\R\oQ_L$ for the inclusions $\D_\qc(\SL)\hookrightarrow \D(\SL)$; but discussion along these lines appears only briefly, in~\S\ref{general affine},
because $\R\oQ$ is awkward to explicate for non-affine schemes, and it doesn't commute with open immersions.\looseness=-1 

Rather, we'll just restrict to $\mathbf\Phi$, where $\R\oQ$ is not needed because $\mathbf\Phi$-maps 
\mbox{$f\colon X\to Y$}
have the following key property (Lemma \ref{qcHom}):
\[
\R\>\sHom_Y(\fst\OX, \Dqcpl(Y))\subset\Dqcpl(Y).
\] 
Consequently, the equivalence $\R\fst\colon\Dqc(X)\overset{\lift.4,\approx\>,}\lto \Dqc(\fst\OX)$ 
in ~\ref{^* equivalence} reduces finding  a concrete right adjoint $(f^\flat\<, t^{}_{\!f}\>)$ for~$\R\fst\colon\Dqcpl(X)\to\Dqcpl(Y)$
to~finding one for the restriction-of-scalars functor $\phi_*\colon\Dqc(\fst\OX)\to\Dqc(Y)$.
But \emph{derived adjoint associativity,} as enhanced in Proposition~\ref{adjass0}, implies that the functor~ 
$\phi_*\colon\D(\fst\OX)\to\D(Y)$ has the right adjoint 
\[
\phi^\flat(-)\set\R\>\sHom_Y(\fst\OX, -),\\[2pt]
\] 
with counit  
the natural $\D(Y)$-map (``\kern.75pt evaluation at 1")
\[
\phi_*\phi^\flat G= \R\>\sHom_Y(\fst\OX, G\>)\lto \R\>\sHom_Y(\OY\<, G\>)= G\qquad (G\in\D(Y)),\\[2pt]
\] 
giving the desired construction (see Proposition~~\ref{represent}).

\pagebreak[3]
In fact Theorem~\ref{qp duality} says more: for any 
$\mathbf\Phi$-map $f\colon X\to Y\<$, $F\in\Dqc(X)$ and $G\in\Dqcpl(Y)$,
 there is  a \emph{sheafified duality isomorphism} \va{-3}
\[
\R\fst\R\>\sHom_\sX\<(F, f^\flat  G\>)
\iso
\R\>\sHom_Y(\R\fst F,G\>),
\]
which turns out to be a concrete realization of the standard abstract one,
namely the natural composite
\[
\R\fst\R\>\sHom_\sX\<(F, f^!  G\>)\to\R\>\sHom_Y\<(\R\fst F, \R\fst f^!  G\>)\to
\R\>\sHom_Y(\R\fst F,G\>),
\]
see Proposition~\ref{qc duality2}. \va3

Arguing as one does for $(-)^{\<\times}$, one gets, for pseudo\kf-coherent finite maps, basic properties of $(-)^\flat$  such as  \emph{pseudofunctoriality} (Proposition~\ref{pseudofunc}) and \emph{tor-independent base change}  (Theorem~\ref{indt base change}).

In further illustration  of \S\ref{expressions},  concrete interpretations are given, for suitable $f\colon X\to Y$ and
$F,G\in\Dqc(Y)$,
of some basic abstractly defined maps: the pseudofunctoriality isomorphism (Proposition~\ref{pfush}),
and the \mbox{base\kf-change} map  of \eqref{beta}---with $f^\flat$ in place of~$f^!$ (Proposition~\ref{concbc});
and~the maps 
\begin{align*}
f^\flat\<\<F\Otimes{X}\LL f^*\<G&\to f^\flat(F\Otimes{Y}G),\\
\R\>\sHom^{}_\sX(\LL f^*\<\>F\<, f^\flat G)&\to f^\flat\R\>\sHom_Y(F,G),
\end{align*}
(\ref {flat and tensor} and~\ref{flat and  hom}, respectively).
\va1

Section~\ref{excompl} deals with the role played in concrete duality for a perfect affine map $f\colon X\to Y$ by the functorial \emph{trace~map} 
\[
\tr_{\<\<f}(G)\colon\R\fst\LL f^*\<G\lto G\qquad(G\in\Dqc(Y))
\]
from \cite[p.\,154, 8.1]{Il}, and by its dual, the \emph{fundamental class} map  
\[
C_{\<\<f}(G)\colon \LL f^*\<G\lto f^\flat G\qquad(G\in\Dqc(Y)),
\]
which is an isomorphism whenever $f$ is \'etale.  

 Section~\ref{regimm} discusses duality for Koszul-regular immersions---a class of perfect 
closed immersions of schemes which includes all regular immersions.
On this class, there is a well-known concrete realization of~$(-)^\flat$, involving normal bundles 
(see Proposition~\ref{KosReg}). The surprisingly difficult task of showing that this realization is \emph{pseudo\-functorial} (see Theorem~\ref{Kozpf1}) gets reduced to the case of affine schemes, which is translated into commutative\kf-algebra terms in section~\ref{affex}, and then disposed of at the end of that section,
along the lines of the proof in \cite[Appendix C.6]{NS}. (\cite[\S\S2.5--2.6]{Co} contains the only other proof I know of.)

\begin{cosa}\label{2.1}
Associate to an affine scheme\kf-map $f\colon X\to Y$ the map of ringed spaces
$$
(X,\OX)\xto{\bar f\:\set(f\<,\,\id)\>}(Y,\fst\OX)=:\oY.
$$ 
This $\bar f$ is \emph{flat,} i.e., for all $x\in X$, the stalk $\CO_{\sX\<\<,\>x}$ is flat 
over $\CO_{\>\oY\<\<,\>\bar f(x)}$:
to check this one can assume that
$X=\spec S$, $Y=\spec R$, with $f$ corresponding to a commutative\kf-ring homomorphism  
$\varphi\colon R\to S$, then note that if $p$ is a prime $S$-ideal, the stalk $S_p$ of $\OX$ at $p$ is a
localization of---thus flat over---the stalk $S\otimes_R R_{\varphi^{-1}p}\>$ of $\fst\OX$ at~$f(p)$.
So the functor $\bar{f\:\<}^{\!\<*}\colon\sA(\oY)\to\sA(X)$ is exact. \looseness=-1

The restriction of $\bar\fst$ to $\sA_\qc(X)$ is an equivalence (necessarily exact) 
from $\sA_{\qc}(X)$ to $\sA_{\qc}(\oY)$ \cite[p.\,361, (9.2.1)]{EGA1}%
\footnote{whose proof states incorrectly, if harmlessly, that $\oY$ is a \emph{locally} ringed space.
This proof has other features that make it harder to read than the original \cite[(1.4.1), (1.4.3)]{EGA2}.\looseness=-1}.
The left adjoint $\bar{f\:\<}^{\!\<*}\<\<$ of $\bar\fst$ takes $\sA_\qc(\oY)$ to  $\sA_\qc(X)$
\cite[p.\,108, (5.1.4)]{EGA1}, and so provides a quasi-inverse equivalence.

These quasi-inverse equivalences,
being exact, \va{.6}
extend to quasi-inverse equivalences of derived categories\va{-2}
 \[
 \D\big(\sA_{\qc}(X)\big)\rightleftarrows\D\big(\sA_{\qc}(\oY)\big).
 \] 

\begin{subex}\label{locally} (Cf.~\cite[top of p.\,362]{EGA1}) When $f=\spec(\varphi)$ with $\varphi\colon R\to S$ as~above, here is what's happening,\va{.5} in concrete terms. 

Denote  by $\varphi_*$ the functor  ``restriction-of-scalars via $\varphi$" from $S$-modules to $R$-modules. Let $M$ be an $S$-module,  and let
$\widetilde M_S$ (resp.~$\widetilde M_R\set(\varphi_*M)\,\>\>\widetilde{}_{\lift.9,\!\!\<R,}\>$) be the corresponding quasi-coherent 
\mbox{$\OX$-(resp.~$\fst\OX$-)}module.
 Any~quasi-coherent $\fst\<\OX$-module~$\CM$ is naturally isomorphic to such
an~$\widetilde M_R\>$:\va1
take $M\set \Gamma(Y,\>\mathcal M\>)$ \cite[p.\,207, (1.4.5)]{EGA1}; and
 \cite[p.\,214, (1.7.7(ii))]{EGA1}\va{.6} gives a natural $\fst\OX$-isomorphism 
 $\widetilde M_R\cong\bar \fst \widetilde M_S$. Hence, \va1
since $\bar{f\:\<}^{\!\<*}$ is quasi-inverse to $\bar \fst $, there is a natural isomorphism 
\(
\bar{f\:\<}^{\!\<*}\widetilde M_R\cong\widetilde M_S\>.
\)\va2

\end{subex}

\begin{subcosa}\label{2.2}
If $F_1\to F_2 \to F\to F_3\to F_4$ is an exact sequence of $\fst\OX$-modules with $F_1$, 
$F_2$, $F_3$ and~$F_4$ quasi-coherent, then $F$ is quasi-coherent (as follows 
via~\cite[p.\,218, (2.2.4)]{EGA1} from the similar property of $\OY$-modules
\mbox{\cite[p.\,217, (2.2.2)(iii)]{EGA1}}). Hence $\Dqc(\oY)\subset\D(\>\oY)$ is a triangulated subcategory (identifiable with the derived 
category of the category of $\fst\OX$-complexes with quasi-coherent homology, see \cite[(1.9.1)]{li}).

Since $\bar{f\:\<}^{\!\<*}$ is exact and, as above,  preserves quasi-coherence, therefore\va{-1} 
\begin{equation}\label{f and qc1}
\bar{f\:\<}^{\!\<*}\Dqc(\oY)\subset\Dqc(X) \quad
\textup{and}\quad
\bar{f\:\<}^{\!\<*}\Dqcpl(\oY)\subset\Dqcpl(X). 
\end{equation}
\vspace{-13pt}

\pagebreak[3]

Also,\va{-2}
\begin{equation}\label{f and qc2}
\R\bar\fst (\Dqc(X))\subset \Dqc(\oY) \quad
\textup{and}\quad
\R\bar\fst (\Dqcpl(X))\subset \Dqcpl(\oY). 
\end{equation}

\noindent For, with $\psi=\psi^{}_{\<\<f}\colon \OY\to \fst \OX$ 
the homomorphism associated to $f$ and 
\begin{equation}\label{phi}
\phi=\phi_{\<\<f}\set(\id,\psi)\colon \oY\to Y
\end{equation}
the resulting ringed-space map (so that $f=\phi\bar f\>$), there are, for all $i\in\ZZ$ and~$E\in\Dqc(X)$, 
natural isomorphisms 
\begin{equation*}
\phi_*H^i\R\bar\fst E\iso H^i\phi_*\R\bar\fst E\iso H^i\R(\phi \bar f\>)^{}_{\<*} \>E
=H^i\R\fst E\>;
\end{equation*}
hence---since\va{.4} an $\fst\OX$-module $G$ is quasi-coherent if $\phi_*G$ is \cite[p.\,218, (2.2.4)]{EGA1}, and, clearly,\va{.6} vanishes if  $\phi_*G$ does---it's enough to prove \eqref{f and qc2}  with 
\mbox{$\bar f\colon X\to \oY$} replaced by~$f\colon X\to Y\<$,
which is done in \cite[3.9.2]{li}. (Or, reduce to where $Y$ and~$X$ are affine, say $X=\spec R$, and apply \cite[p.\,225, 5.1]{BN}.)\va2

\end{subcosa}

\begin{subprop}\label{^* equivalence} 
The functor $\bar{f\:\<}^{\!\<*}\colon\D(\>\oY)\to\D(X)$ induces an equivalence from\/ $\Dqc(\oY)$ to\/ $\Dqc(X)$
$($resp.~$\Dqcpl(\oY)$ to\/ $\Dqcpl(X)),$ with quasi-inverse induced by\/
$\R\bar\fst \colon\D(X)\to\D(\>\oY)$.
\end{subprop}

\begin{proof}
The functor $\R\bar\fst$ is right adjoint to $\bar{f\:\<}^{\!\<*}\colon\D(\>\oY)\to\D(X)$ \cite[(3.2.1)]{li}; so
in view of \eqref{f and qc1} and \eqref{f and qc2}, it suffices that the counit  and unit maps
\[
\epsilon^{}_{\<\<E}\colon\bar{f\:\<}^{\!\<*}\R\bar\fst E\to E\ \ (E\in\Dqc(X))
\quad\>\textup{and}\>\quad 
\eta^{}_F\colon F\to \R\bar\fst \>\bar{f\:\<}^{\!\<*}\<F\ \ (F \in \Dqc(\oY)) \]
both be isomorphisms (see \cite[p.\,93, Theorem 1]{M}). 

To show that $\epsilon^{}_{\<\<E}$ is an isomorphism  one can assume\va{.6} that $E$ is K-injective 
($\:\set$~q-injective \cite[\S2.3]{li}), 
so that $\R\bar\fst E\cong\bar\fst E$ and $\epsilon^{}_{\<\<E}$ can be identified with the counit map 
associated to the adjunction  between $\bar{f\:\<}^{\!\<*}$ and $\bar \fst$ (see \cite[(3.2.1.3)]{li}), an isomorphism because these functors are, as above, quasi-inverse equivalences.

\pagebreak

As for $\eta^{}_F$, since $\epsilon^{}_{\<\smash{\bar{f\:\<}^{\!\<*}}\!\<F^{}\>}$ is an isomorphism and the composite map
\[
\bar{f\:\<}^{\!\<*}\!F\xto{\<\bar{f\:\<}^{\!\<*}\<\<\<\eta^{}_F\>} \bar{f\:\<}^{\!\<*}\R\bar\fst \>\bar{f\:\<}^{\!\<*}\!F
\xto{\<\epsilon^{}_{\<\smash{\bar{f\:\<}^{\!\<*}}\!\<F^{}\>}}\bar{f\:\<}^{\!\<*}\<\<F
\]
is the identity map, therefore $\bar{\<f\>}^{\<\<*}\<\<\eta^{}_F$ is an isomorphism. 
Since $\bar f$ is~flat 
one~has, denoting  cohomology functors by $H^n\<$,   that   for all $n\in\ZZ$, $\bar{f\:\<}^{\!\<*}\!H^n\eta^{}_F$ is isomorphic to $H^n\<\bar{\<f\>}^{\<\<*}\<\<\eta^{}_F$\va{.6} and so is an isomorphism. Since $\bar{f\:\<}^{\!\<*}|_{\sA_\qc(\oY)}$ is an equivalence, therefore every $H^n\eta^{}_F$ is an isomorphism, so
$\eta^{}_F$ is an isomorphism.
\end{proof}

For any ringed-space map $h\colon V\to W\<$ and $E,$ $E'\in \D(V)$, one has the natural 
bifunctorial composite\va3
\begin{equation}\label{nu}
\begin{aligned}
\nu(E,E'\>)\colon\R^{\mathstrut}h_*\R\>\sHom_V(E\<,\>E'\>)&\lto \R h_*\R\>\sHom_V(\LL h^*\R h_* E\<,\>E'\>)\\
&\!\iso\!
\R\>\sHom_W(\R h_* E\<,\>\R h_* E'\>)
\end{aligned}
\end{equation}
with the isomorphism as in \cite[3.2.3(ii)]{li}. Using  \cite[3.2.5(f)]{li} and the last assertion in \cite[3.2.3(i)]{li}, or otherwise, one verifies that application of the functor $H^0\R\Gamma(W,-)$ to this
composite map produces the obvious map 
\[
\Hom_{\D(V)}(E,E'\>)\lto\Hom_{\D(W)}(\R h_*E,\R h_*E'\>).
\]

 If the natural map
$\LL h^*\R h_* E\to E$ is an isomorphism---for example if $h\colon V\to W$ is 
\mbox{$\bar f\colon X\to\oY$}
(as above) and $E\in\Dqc(X)$---then so is the composite map~\eqref{nu}.\va1
Hence:

\va2

\begin{subcor}[Sheafified duality for $\bar f\,$]\label{duality for barf}
For\/ $E\in\Dqc(X),$ $F\in\Dqc(\oY),$ one has the composite bifunctorial isomorphism
\[
\R\bar\fst\R\>\sHom^{}_\sX(E, \bar{f\:\<}^{\!\<*}\!F\>)
\underset{\eqref{nu}}\iso
\R\>\sHom_{\>\oY}(\R\bar\fst E, \R\bar\fst\bar{f\:\<}^{\!\<*}\!F\>)
\underset{\textup{\ref{^* equivalence}}}\iso \R\>\sHom_{\>\oY}(\R\bar\fst E, \>F\>),
\]
to which application of $H^0\R\Gamma(Y,-)$ gives the adjunction isomorphism
\[
\Hom_{\Dqc(X)}(E,\bar{f\:\<}^{\!\<*}\!F\>)\iso\Hom_{\Dqc(\oY)}(\R \bar\fst E,F\>).
\makebox[0pt]{\kern135pt\qed}
\]

\end{subcor}

The equivalence \ref{^* equivalence} between $\Dqc(\oY)$ and~$\Dqc(X)$ is compatible with standard additional structures. For example, it respects the derived tensor product and 
sheafified hom functors, in the following sense:
\begin{subcor}\label{tensor} 
For\/ $E, F\in\Dqc(X),$ the natural maps are isomorphisms
\begin{align*}
\kappa\colon\R\bar\fst E\Otimes{\mkern.5mu\oY} \R\bar\fst F&\iso  \R\bar\fst(E\Otimes{\sX} F\>),\\
\nu\colon \R\bar\fst\R\>\sHom^{}_\sX(E\<,F\>)&\underset{\!\textup{\ref{nu}}}\iso\R\>\sHom_{\>\oY}(\R\bar\fst E\<,\R\bar\fst F\>);
\end{align*}
and if\/ $\R\>\sHom^{}_\sX(E,F\>)\in\Dqc(X),$ then the natural composite\looseness=-1
\[
\bar{f\:\<}^{\!\<*}\R\>\sHom_{\>\oY}(\R\bar\fst E\<,\R\bar\fst F\>)\xto[\ \lift1,\rho,\;]{}
\R\>\sHom^{}_\sX(\bar{f\:\<}^{\!\<*}\R\bar\fst E\<,\bar{f\:\<}^{\!\<*}\R\bar\fst F\>)
\underset{\!\textup{\ref{^* equivalence}}}{\<\iso\<\<}
\R\>\sHom^{}_\sX(E\<,F\>)\\[-1pt]
\]
is an isomorphism too.
\end{subcor}

\pagebreak[3]
\begin{proof}
By definition (\cite[3.2.4(ii)]{li}), $\kappa$ is the unique map such that the following
diagram, with $\epsilon$ as in the proof of~\ref{^* equivalence}, and the isomorphism on the left as in \cite[3.2.4(i)]{li}, commutes:
\[
\def\1{$\bar{f\:\<}^{\!\<*}(\R\bar\fst E\Otimes{\mkern.5mu\oY} \R\bar\fst F\>)$}
\def\2{$\bar{f\:\<}^{\!\<*}\R\bar\fst(E\Otimes{\sX} F\>)$}
\def\3{$\bar{f\:\<}^{\!\<*}\R\bar\fst E\Otimes{\mkern.5mu\oY} \bar{f\:\<}^{\!\<*}\R\bar\fst F$}
\def\4{$E\Otimes{\sX} F$}
  \bpic[xscale=4.5, yscale=1.5]

   \node(11) at (1,-1){\1};   
   \node(12) at (2,-1){\2}; 
   
   \node(21) at (1,-2){\3};  
   \node(22) at (2,-2){\4};
  
    \draw[->] (11)--(12) node[above, midway, scale=.75] {$\bar{f\:\<}^{\!\<*}\!\<\kappa$} ;  
    
    \draw[->] (21)--(22) node[below, midway, scale=.75] {$\epsilon^{}_E\Otimes{\sX}\epsilon^{}_F$} ; 
    
    \draw[->] (11)--(21) node[left=1pt, midway, scale=.75] {$\simeq$} ;
    \draw[->] (12)--(22) node[right=1pt, midway, scale=.75] {$\epsilon_{E\Otimes{\sX} F}$} ;
   \epic
\]
By \ref{^* equivalence}, the counit maps $\epsilon^{}_{\<E},$  $\epsilon^{}_F$ and 
$\epsilon_{E\Otimes{\sX} F}$
are isomorphisms, and\va{-3} therefore
so is $\bar{f\:\<}^{\!\<*}\!\kappa$, whence, by~\ref{^* equivalence} again, so is $\kappa$.\va1

That $\nu$ is an isomorphism was noted above (just before~\ref{duality for barf}).\va1
 
As for the last assertion, it holds by assumption that 
\[
\R\>\sHom_{\oY}(\R\bar\fst E,\R\bar\fst F\>)\underset{\nu}\cong \bar\fst \R\>\sHom^{}_\sX(E\<,F\>)\in\Dqc(\oY)\\[-5pt]
\]
and \va{-2}
\[
\R\>\sHom^{}_\sX(\bar{f\:\<}^{\!\<*}\R\bar\fst E\<,\bar{f\:\<}^{\!\<*}\R\bar\fst F\>)
\underset{\textup{\ref{^* equivalence}}}\cong
\R\>\sHom^{}_\sX(E\<,F\>)\in\Dqc(X).
\]
But it follows easily from \ref{^* equivalence} that for $A\in\Dqc(\oY)$\va{.4}
and  $B\in\Dqc(X)$, a~$\D(X)$-map $\bar f^{*}\<\<A\to B$ is an isomorphism\va{.2} $\Leftrightarrow$ so is
its adjoint $A\to\bar\fst B$.   So the map $\rho$  is an isomorphism, since
by its definition \cite[(3.5.4.5)]{li}%
\footnote{Apropos of \S\ref{expressions}, it is left to the interested reader to show that for ringed spaces
this~$\rho$ is the~same as the explicit map in \cite[p.\,109, 5.8]{RD}. For this, \cite[3.5.6(g)]{li}, with 
$\alpha\colon [D,E]\otimes D\to E$ the natural map, might be useful.}
it is adjoint to the natural composite $\D(\oY)$-isomorphism
\begin{align*}
\R\>\sHom_{\>\oY}(\R\bar\fst E\<,\R\bar\fst F\>)
&\iso
\R\>\sHom_{\>\oY}(\R\bar\fst E\<,\R\bar\fst \bar{f\:\<}^{\!\<*}\<\<(\R\bar\fst F\>)\>)\\
&\iso
\R\bar\fst\R\>\sHom_{X}(\bar{f\:\<}^{\!\<*}\R\bar\fst E\<, \bar{f\:\<}^{\!\<*}\R\bar\fst F\>).
\end{align*}
\end{proof}

\begin{subcor}\label{project}
For\/ $E\in\Dqc(X)$ and\/ $F\in\Dqc(\oY),$ the projection maps
\begin{align*}
p^{}_{1\<,\>\bar f}\colon \R\bar\fst E\Otimes{\mkern.5mu\oY}F
&\lto
\R\bar\fst(E\Otimes{\sX}\bar{f\:\<}^{\!\<*}\!F\>) \\
p^{}_{2,\>\bar f}\colon F\Otimes{\mkern.5mu\oY} \R\bar\fst E
&\lto \R\bar\fst(\bar{f\:\<}^{\!\<*}\<\<F\Otimes{\sX}E\>)
\end{align*}
are isomorphisms.
\end{subcor}

\noindent\emph{Proof.}
As in \cite[3.4.6.2]{li}, $p^{}_{1\<,\>\bar f}$ is the natural composite isomorphism
\begin{align*}
\R\bar\fst E\Otimes{\mkern.5mu\oY} F
&\underset{\textup{\ref{^* equivalence}}}\iso \R\bar\fst\bar{f\:\<}^{\!\<*}(\R\bar\fst E\Otimes{\mkern.5mu\oY} F\>)\\
&\iso\R\bar\fst(\bar{f\:\<}^{\!\<*}\R\bar\fst E\Otimes{\sX}\bar{f\:\<}^{\!\<*}\!F\>)
\underset{\textup{\ref{^* equivalence}}}\iso \R\bar\fst(E\Otimes{\sX}\bar{f\:\<}^{\!\<*}\!F\>); 
\end{align*}
and similarly for $p^{}_{2,\>\bar f}\>$.\hfill$\square$

\end{cosa}

\begin{cosa}\label{2.3}
Let $Y$ be a ringed space, and $\psi\colon\OY\to\CS$ an $\OY$-algebra. Let $\oY$ be the 
ringed space~$(Y\<\<,\>\CS)$, and $\phi\colon\oY\to Y$ the ringed-space map $(\id_Y,\psi)$.\va2

\begin{small}
(This subsection is independent of the preceding one. Subsequently, only the case where $\psi\colon \OY\to \CS\set \fst \OX$ is as before---just after \eqref{f and qc2}---will be needed.)\par
\end{small}

\vskip2pt
Note: $\CO_{\>\oY}=\CS$, $\sA(\oY)$ is the category of $\CS$-modules, 
\mbox{$\phi_*\colon \sA(\oY)\to\sA(Y)$}
is just restriction of scalars, $\otimes_{\>\oY}=\otimes^{}_{\CS}$ and $\sHom_{\oY}=\sHom^{}_{\CS}\>$.\va2

 The functor \va{-2}
\[\sHom_\psi\colon \sA(\oY)^\op\times\sA(Y)\to\sA(\oY)
\]
(where ``$\,^\op\>\>$" denotes ``opposite category") is given by
\[
 \sHom_\psi(F,G\>)\set\sHom_Y(\phi_*F, \>G\>)\in\sA(\oY)\qquad \bigl(F\in\sA(\oY), \,G\in\sA(Y)\bigr),
\]
with scalar multiplication given by the following natural composite map,
where~$m_{\phi_*\<F}\colon \CS\otimes_Y\phi_*F\to\> \phi_*F\>$ is scalar multiplication,
\begin{align*}
\CS\otimes^{}_Y\sHom_Y(\phi_*&F, \>G\>)=\sHom_Y(\phi_*F, \>G\>) \otimes^{}_Y\CS\\[-2pt]
&\mkern-8mu\xto{\!\via\>m_{\phi_*\<F\<}}
 \sHom_Y(\CS \otimes^{}_Y \phi_*F, \>G\>)\otimes^{}_Y\CS\\[3pt]
&\mkern-8mu\iso
\sHom_Y(\CS,\sHom_Y(\phi_*F, \>G\>)\>)\otimes^{}_Y\CS
\lto \sHom_Y(\phi_*F, \>G\>),
\end{align*}
and with the the obvious\va2 action on maps in $\sA(\oY)^\op\times\sA(Y)$.

\pagebreak[3]
\begin{small}
This scalar multiplication is adjoint to the natural composite
\[
\sHom_Y(\phi_*F, \>G\>)\otimes_Y \CS\otimes_Y\phi_*F
\xto{\!\via m_{\phi_{\<*}\<\<F\<}}\,
\sHom_Y(\phi_*F, \>G\>)\otimes_Y\phi_*F\lto\,G\>,
\]
that is, the border of the following natural diagram with $[-,-]\set\sHom_Y(-,-)$,  $\otimes\set\otimes_Y$
and  $m\set m_{\phi_*\<F}$, is commutative.\va{-2}
\[\mkern-4mu
\def\1{$[\phi_*F,G\>]\<\otimes\<\big(\CS\<\otimes\< \phi_*F\big)$}
\def\2{$[\phi_*F,G\>]\<\otimes\< \phi_*F$}
\def\3{$G$}
\def\4{$[\CS\<\otimes\<\phi_*F,G\>]\<\otimes\<\big(\CS\<\otimes\< \phi_*F\big)$}
\def\5{$\big([\CS,[\phi_*F,G\>]]\<\otimes\<\CS\big)\<\otimes\< \phi_*F$}
\def\6{$[\phi_*F,\<G\>]\<\otimes\< \phi_*F$}
  \bpic[xscale=4.6, yscale=2]
   \node(11) at (1,-1){\1};   
   \node(12) at (2.1,-1){\2}; 
   \node(13) at (3,-1){\3};
   
   \node(21) at (1,-2){\4};  
   \node(22) at (2.1,-2){\5};
   \node(23) at (3,-2){\6};
   
    \draw[->] (11)--(12) node[above, midway, scale=.75] {$\via\>m$} ;  
    \draw[->] (12)--(13) ;
    
    \draw[->] (21)--(22) node[above, midway, scale=.75] {$\lift -.1,\Iso,$} ; 
    \draw[->] (22)--(23) ;
    
    \draw[->] (11)--(21) node[left=1pt, midway, scale=.75] {$\via\>m$} ;
    \draw[<-] (13)--(23)  ;
  
   \draw[->] (1.48,-1.859)--(2.95,-1.12) ; 
   
   \node at (1.5,-1.5) [scale=.9]{\circled1};
   \node at (2.63,-1.62) [scale=.9]{\circled2};
   
  \epic
\]
Indeed, subdiagram \circled1 commutes by \cite[3.5.3(h)]{li} (with $B=\CS\otimes \phi_*F\<$, $A=\phi_*F$ 
and $C=G$), and \circled2 by, e.g.,  the Kelly-Mac\,Lane coherence theorem \cite[p.\,107, Thm.~2.4]{KM} applied to~\circled2 after replacement of $\CS$ and $\phi_*F$, respectively, by arbitrary objects $D$ and $E$
in $\sA(Y)$.

It follows that scalar multiplication factors naturally as\va{-2}
\[
[\phi_*F,G\>]\<\otimes\CS\lto [\phi_*F,G\>]\<\otimes[\phi_*F,\phi_*F\>]\lto [\phi_*F,G\>]\\[-2pt]
\]
where the second map is ``internal composition'' \cite[3.5.3(h)]{li}, that is, the map adjoint to the natural composite\va{-2}
\[
[\phi_*F,G\>]\<\otimes([\phi_*F,\phi_*F\>]\otimes\phi_*F\>)\lto [\phi_*F,G\>]\<\otimes\phi_*F\lto G.
\]

\par\va2
\end{small}
\enlargethispage*{5pt}

Let  
\(
\R\>\sHom_\psi\colon\D(\>\oY)^\op\times\>\D(Y)\<\to\D(\>\oY)
\) 
be a right-derived functor of $\sHom_\psi\>$, specified on objects by choosing for each $\OY$-complex $G$ 
a K-injective resolution $G\to I_G$ and then setting, for any $\CS$-complex $F\<$, 
\[
\R\>\sHom_\psi(F\<, G\>)\set\sHom_\psi(F\<,I_G\>),
\]
and specified on maps in the standard way.

\pagebreak[3]
The next proposition  is a slightly upgraded version~of 
\emph{derived adjoint\- associativity}---which it becomes when $\psi$~is an isomorphism.
(This termin\-ology is clarified in Remark~\ref{conjugate}.)

\begin{subprop}\label{adjass0}
There is a unique trifunctorial\/ $\D(\oY)$-isomorphism
\begin{multline*}
\alpha(E,F,G\>)\colon\R\>\sHom_\psi(E\Otimes{\>\oY}\<\<F, \>G\>)\iso\R\>\sHom_{\>\oY}(E, \R\>\sHom_\psi(F,G\>))\\   
\big(E,F\in\D(\>\oY),\;G\in\D(Y)\big)  
\end{multline*} 
such that the following natural diagram, with $\CH\set\sHom$ and\/ $\alpha^{}_0(E,F,G\>)$ the standard isomorphism of\/ $\CO_{\>\oY}$-complexes, commutes\/$.$ 
\[\mkern-4mu
\def\1{$\CH_\psi(E\otimes_{\>\oY}\<\<F, \>G\>)$}
\def\2{$\R\CH_\psi(E\otimes_{\>\oY}\<\<F, \>G\>)$}
\def\3{$\R\CH_\psi(E\Otimes{\>\oY}\<\<F, \>G\>)$}
\def\4{$\CH_{\>\oY}\<\big(\<E, \CH_{\psi}(F,G\>)\<\big)$}
\def\5{$\R\CH_{\>\oY}\<\big(\<E, \CH_{\psi}(F,G\>)\<\big)$}
\def\6{$\R\CH_{\>\oY}\<\big(\<E, \R\CH_{\psi}(F,G\>)\<\big)$}
 \bpic[xscale=4.25, yscale=1.6]

   \node(11) at (1,-1){\1};   
   \node(12) at (1.972,-1){\2}; 
   \node(13) at (3,-1){\3};
   
   \node(21) at (1,-2){\4};  
   \node(22) at (1.972,-2){\5};
   \node(23) at (3,-2){\6}; 
   
    \draw[->] (11)--(12) ;  
    \draw[->] (12)--(13) ;
 
    \draw[->] (21)--(22) ;  
    \draw[->] (22)--(23) ;

    \draw[->] (11)--(21) node[right=1pt, midway, scale=.75] {$\simeq$}
                                   node[left=1pt, midway, scale=.75] {$\alpha^{}_0(E,F,G\>)$} ;
    \draw[->] (13)--(23) node[left=1pt, midway, scale=.75] {$\simeq$}
                                   node[right=1pt, midway, scale=.75] {$\alpha(E,F,G\>)$} ; 
                                       
    \epic
\]
  \end{subprop}
 
In fact, $\alpha(E,F,G\>)=\alpha^{}_0(P_E,F,I_G\>)$ 
where $P_E\to E$ (resp.~$G\to I_G$) is a quasi-isomorphism with $P_E$ a K-flat 
$\CO_{\>\oY}\>$-complex
(resp.~$I_G$ a K-injective $\OY$-complex). 
The \emph{proof} is the same as that of \cite[Proposition (2.6.1)$^*$]{li}, except that in the remarks
preceding \emph{loc.\,cit.}~one sets\va3
\begin{align*}   
&\mathbf L_{\text{\bf 1}}' \set\mathbf K(\oY) \textup{ (the homotopy category of $\CO_{\>\oY}$-complexes)}, 
\\[3pt]
&\mathbf L_{\text{\bf 2}}'\set \textup{full subcategory of $\mathbf K(\oY)$ spanned by 
all K-flat $\CO_{\>\oY}\>$-complexes,}\\[3pt]
&\mathbf L_{\text{\bf 3}}'\set \textup{full subcategory of $\mathbf K(Y)$ spanned by 
all K-injective $\OY$-complexes,}
\end{align*}
adjusts $\mathbf L_{\text{\bf i}}''$ and $\mathbf D_{\text{\bf i}}''$ accordingly ($\mathbf i=1,2,3$),\va1
and then observes that for \mbox{$(F,G\>)\in \mathbf L_{\text{\bf 2}}'\times \mathbf L_{\text{\bf 3}}'$,}
if the $\CO_{\>\oY}\>$-complex~$E$ is exact then so is 
\[
\CH_{\>\oY}\<\big(\<E, \CH_{\psi}(F,G\>))\cong \CH_\psi(E\Otimes{\>\oY}\<\<F, \>G\>),
\]
that is, $\CH_{\psi}(F,G\>)$ is K-injective.
\va1
\hfill{\qed}
\vskip2pt

Upon replacing $\CS$ by $\OY$ in the above definition of scalar multiplication for~$\sHom_\psi(F,G\>)$, 
one checks the equality of $\OY$-complexes
\begin{equation}\label{phiRHompsi0}
\phi_*\sHom_\psi(F,G\>)=\sHom_Y(\phi_*F,G\>);
\end{equation}
replacing $G$ by $I_G$ gives the natural isomorphism (in $\D(Y)$) 
\begin{equation}\label{phiRHompsi}
\phi_*\R\>\sHom_\psi(F,G\>)\cong\R\>\sHom_Y(\phi_*F,G\>).
\end{equation}

As is readily verified,  the isomorphism~\eqref{phiRHompsi}  is naturally identifiable with 
the inverse of the isomorphism $\alpha_\phi(\CO_{\>\oY},F,G\>)$ given by the next Corollary.

\begin{subcor}\label{adjass}
There is a unique trifunctorial\/ $\D(Y)$-isomorphism
\begin{multline*}
\alpha_\phi(E,F,G\>)\colon\R\>\sHom_Y\<\<\big(\phi_*(E\Otimes{\>\oY}F\>), \>G\big)\<\iso\<\phi_*\R\>\sHom_{\>\oY}(E, \R\>\sHom_\psi(F,G\>)\big)\\   
\big(E,F\in\D(\>\oY),\;G\in\D(Y)\big)  
\end{multline*} 
such that the following natural\/ $\D(Y)$-diagram, with $\CH\set\sHom$, commutes\/$.$ 
\[\mkern-5mu
\def\1{$\CH_Y\<\big(\phi_*(E\otimes_{\>\oY}\<\<F\>), \>G\big)$}
\def\2{$\R\CH_Y\<\big(\phi_*(E\otimes_{\>\oY}\<\<F\>), \>G\big)$}
\def\3{$\R\CH_Y\<\big(\phi_*(E\Otimes{\>\oY}\<\<F\>), \>G\big)$}
\def\4{$\phi_*\CH_{\>\oY}\<\big(\<E, \CH_{\psi}(F,G\>)\<\big)$}
\def\5{$\phi_*\R\CH_{\>\oY}\<\big(\<E, \CH_{\psi}(F,G\>)\<\big)$}
\def\6{$\phi_*\R\CH_{\>\oY}\<\big(\<E, \R\CH_{\psi}(F,G\>)\<\big)$}
 \bpic[xscale=4.37, yscale=1.7]

   \node(11) at (1,-1){\1};   
   \node(12) at (1.98,-1){\2}; 
   \node(13) at (3,-1){\3};
   
   \node(21) at (1,-2){\4};  
   \node(22) at (1.98,-2){\5};
   \node(23) at (3,-2){\6}; 
   
    \draw[->] (11)--(12) ;  
    \draw[->] (12)--(13) ;
 
    \draw[->] (21)--(22) ;  
    \draw[->] (22)--(23) ;

    \draw[->] (11)--(21) node[left=1pt, midway, scale=.75] {$\simeq$};
    \draw[->] (13)--(23) node[left=1pt, midway, scale=.75] {$\simeq$}
                                   node[right=1pt, midway, scale=.75] {$\alpha_\phi(E,F,G\>)$} ; 
                                       
    \epic
\]
  \end{subcor}
  
 \begin{proof} Applying $\phi_{*}$ to the diagram in~\ref{adjass0}, one sees that the diagram 
 in \ref{adjass}
 commutes when $\alpha_\phi(E,F,G\>)$ is the isomorphism $\phi_*\alpha(E,F,G\>)$.
  
 Uniqueness need only be checked when $E$ is K-flat and $G$ is K-injective, in which case it holds because the maps in the top row of the latter diagram are isomorphisms. 
 \end{proof}

\pagebreak[3]
\begin{subcor}[Sheafified duality for $\phi$]\label{duality for phi}
Setting\/ 
\begin{equation*}\label{phiflat}
\pt(-)\set \R\>\sHom_\psi(\CO_{\>\oY},-)\colon\D(Y)\to\D(\>\oY),
\end{equation*}
one has, for $E\in\D(\>\oY)$ and $G\in\D(Y),$ the  bifunctorial isomorphism
\[
\alpha_\phi(E,\CO_{\>\oY},G\>)\colon\R\>\sHom_Y(\phi_*E, \>G\>)\iso\phi_*\R\>\sHom_{\>\oY}(E, \pt G\>).
\]
\end{subcor}

\noindent\emph{Remark.} The inverse of $\alpha_\phi(E,\CO_{\>\oY},G\>)$ is described in Proposition~\ref{duality map}.
 
\begin{subcor}[Global duality for $\phi$]\label{adjunction0}
For\/ $E,F\in\D(\>\oY)$ and\/ $G\in\D(Y)$ one has,  with $\alpha_\phi\set\alpha_\phi(E,F,G\>),$
the functorial isomorphism\va3
\begin{equation*}
\textup{H}^0\R^{\mathstrut}\Gamma\big(Y,\alpha_\phi\big)
\colon\!\Hom_{\D(Y)}\!\<\big(\phi_*(E\Otimes{\>\oY}\<\<F\>), \>G\big)
\iso\<\Hom_{\D(\oY)}\<\big(E, \R\>\sHom_\psi(F,G\>)\big).
\end{equation*}

In particular, one has the adjunction\/ $\phi_*\!\dashv \pt$
given by the functorial isomorphism, with $\alpha'_\phi\set\alpha_\phi(E,\CO_{\>\oY},G\>),$ \va3
\begin{equation*}\label{adjass1}
\textup{H}^0\R^{\mathstrut}\Gamma\big(Y,\alpha'_\phi\big)\colon\Hom_{\D(Y)}\!\<\big(\phi_*E, \>G\big)\iso
\Hom_{\D(\oY)}\<(E, \>\pt G\big).
\end{equation*}
\end{subcor}

\begin{subrem}\label{conjugate} For fixed $E$ and $F$ in $\D(\oY)$, 
the functorial isomorphism $\alpha(E\<,F\<,G)$ in~\ref{adjass0} is \emph{right-conjugate} (see for instance \cite[3.3.5, 3.3.7]{li}), via natural adjunctions, to the
standard associativity isomorphism
\[
\phi_*\big(D\Otimes{\>\oY}(E\Otimes{\>\oY} F\>)\big) \osi 
\phi_*\big((D\Otimes{\>\oY} E)\Otimes{\>\oY} F\big)\qquad (D\in\D(\oY)).
\]
\end{subrem}

\begin{subcosa}\label{counit unit}
As always, one can explicate an adjunction through the associated
counit and unit maps. One does so for $\phi_*\!\dashv \pt$
by means of the well-known counit and unit maps 
for the standard adjunction $\phi_*\!\dashv\sHom_\psi(\CO_{\>\oY},-)$
(of~functors between $\sA(\oY)$ and $\sA(Y)$), obtaining that 
the corresponding  counit map at $G\in\D(Y)$ is the natural composite
\begin{equation*}
 \begin{aligned}\label{counit}
\phi_*\pt G=\phi_*\R\>\sHom_\psi(\CO_{\>\oY},G\>)&\underset{\eqref{phiRHompsi}}\iso\R\>\sHom_Y(\phi_*\CO_{\>\oY}, G\>)\\
&\,\underset{\via\psi}\lto\,\R\>\sHom_Y(\OY\<, G\>)\iso G\>;
 \end{aligned}\tag*{(\ref{counit unit}.1)}
\end{equation*}

\pagebreak
\noindent and the corresponding unit map at $E\in\D(\>\oY)$ is the natural composite\va3
\[\label{unit}
E\<\<\iso\<\<\sHom_{\oY}(\CO_{\>\oY},E)\overset{\,\,j\,}\hookrightarrow \sHom_\psi(\CO_{\>\oY}, \>\phi_*E\>)
\to \R\>\sHom_\psi(\CO_{\>\oY}, \>\phi_*E\>)=\pt\phi_*E.
\tag*{(\ref{counit unit}.2)}
\]
(The inclusion map $j$ is easily seen to be $\CO_{\>\oY}$-linear.)
\end{subcosa}

\begin{subcosa}\label{pphi}
The functor $\otimes_\psi\colon \sA(\oY)\times\sA(Y)\to\sA(\oY)$ is given by
\[
F\otimes_\psi G\>\set\phi_* F\otimes_Y G \in\sA(\oY)\qquad \bigl(F\in\sA(\oY), \,G\in\sA(Y)\bigr),
\] 
where the scalar multiplication map is the natural composite
\[
\CO_{\>\oY} \otimes_Y \<(\phi_*F\otimes_Y G\>)\iso (\CO_{\>\oY} \otimes_Y \phi_*F\>)\otimes_Y G  \lto  
\phi_*F\otimes_Y G,
\]
and the action of $\otimes_\psi$ on maps in $\sA(\oY)\times\sA(Y)$ is the obvious\va2 one.


One has for $E\in\sA(\oY), \;G\in\sA(Y)$, the natural isomorphism
\[
\Hom_{\oY}(\CO_{\>\oY}\otimes_\psi \<G,\> E\>)\iso \Hom^{}_Y(G, \phi_*E\>).
\]
Thus the functor $\CO_{\>\oY}\otimes_\psi\< -$ is left-adjoint to $\phi_*$, and
so can be identified with~$\phi^*\<$, after which the counit $\sA(\oY)$-map $\phi^*\phi_*\to\id$ becomes the scalar multiplication
\[
\CO_{\>\oY}\otimes_\psi \phi_*F\lto F\qquad(F\in\sA(\oY)),
\]
and the unit map $\id\to\phi_*\phi^*$ becomes the natural 
$\D(Y)$-map\looseness=-1
\[
G\iso\OY\otimes_Y G\lto \phi_*\CO_{\>\oY}\otimes_Y G = \phi_*(\CO_{\>\oY}\otimes_\psi\< G\>)
\qquad(G\in\sA(Y)),
\]

There is an obvious functorial isomorphism
\begin{equation*}\label{pphi1}
F\otimes_\psi G\iso F\otimes_{\>\oY}(\CO_{\>\oY}\>\otimes_\psi G)=F\otimes_{\>\oY}\phi^*\<\<G.
\tag{\ref{pphi}.1}
\end{equation*}
One has then the natural isomorphism
\begin{equation*}\label{pphi2}
\ \phi_*(F\otimes_{\>\oY}\>\phi^*G\>)=\phi_*(F\otimes_{\>\oY}(\CO_{\>\oY}\otimes_\psi G\>)\>)\cong\phi_*(F\otimes_\psi G\>)=\phi_*F\otimes_YG,
\tag{\ref{pphi}.2}
\end{equation*}
whose inverse is easily seen to be the \emph{projection map}
\[
p'_{1\<,\>\phi}\colon\phi_*F\otimes_{Y} G \lto \phi_*(F\otimes_{\>\oY}\phi^*\<\<G\>),
\]
cf.~\cite[3.4.6]{li},
so that $p'_{1\>,\>\phi}$ \emph{is an isomorphism.} 

\pagebreak[3]
Similarly, one has the isomorphism
\[
p'_{2,\>\phi}\colon G\otimes_{Y} \phi_*F \iso \phi_*(\phi^*\<G\otimes_{\>\oY}F\>);
\]
and, as in \cite[3.4.6.1]{li},  the following natural diagram commutes:
 \[
\def\1{$\phi_*F\otimes_{Y} G$}
\def\2{$\phi_*(F\otimes_{\>\oY}\phi^*\<\<G\>)$}
\def\3{$G\otimes_{Y} \phi_*F$}
\def\4{$\phi_*(\phi^*\<G\otimes_{\>\oY}F\>)$}
 \CD
 \bpic[xscale=3.4, yscale=1.5]

   \node(11) at (1,-1){\1} ;   
   \node(12) at (2,-1){\2} ; 
  
   \node(21) at (1,-2){\3} ;  
   \node(22) at (2,-2){\4} ;
   
    \draw[->] (11)--(21) node[left=1pt, midway, scale=.75] {$\simeq$} ;

    \draw[->] (12)--(22) node[right=1pt, midway, scale=.75] {$\simeq$} ;
    
    \draw[->] (11)--(12) node[below=1pt, midway, scale=.75]{$p'_{1\<,\>\phi}$}
                                   node[above,midway, scale=.75]{$\Iso$}  ;
    \draw[->] (21)--(22) node[below=1pt, midway, scale=.75]{$p'_{2,\>\phi}$}
                                   node[above,midway, scale=.75]{$\Iso$}  ;       
 \epic
  \endCD
\]

This all applies more generally, \emph{mutatis mutandis,} to $\CO_{\>\oY}$-complexes~$F$ and 
$\OY$-complexes~$G$.
In that case, upon 
replacing $G$ by a quasi-isomorphic K-flat complex one sees that the
\emph{derived} projection map is an isomorphism
\begin{equation*}\label{projphi}
p^{}_{1\<,\>\phi}\colon\phi_* F\Otimes{Y}G\iso \phi_*(F\Otimes{\>\oY}\LL \phi^* G\>)\qquad(F\in\D(\oY),\,G\in\D(Y)).
\tag{\ref{pphi}.3}
\end{equation*}
Similarly (and by \cite[3.4.6.1]{li}, equivalently), one has the isomorphism
\begin{equation*}\label{projphi2}
p^{}_{2\<,\>\phi}\colon G\Otimes{Y}\phi_* F\iso \phi_*(\LL \phi^* G\Otimes{\>\oY}F\>)\qquad(F\in\D(\oY),\,G\in\D(Y)).
\tag{\ref{pphi}.4}
\end{equation*}
\begin{small}\noindent
One checks that \eqref{projphi2} is left-conjugate to the inverse of the duality isomorphism
$\R\>\sHom_Y(\phi_*F, \>E\>)\iso\phi_*\R\>\sHom_{\>\oY}(F, \pt E\>)$  
given by Corollary~\ref{duality for phi}.\par
\end{small}
\end{subcosa}
\end {cosa}

\begin{cosa}\label{2.4}
Let $f\colon X\to Y\<$ be an affine scheme\kf-map, $\bar f$ as in \S\ref{2.1}, and $\phi$ and~$\psi$ as in the lines following \eqref{f and qc2}. So $f=\phi\bar f$; and one gets properties of $f$ by combining  the corresponding ones of $\bar f$ and~$\phi$.\va1

For example, from \eqref{projphi}, \eqref{projphi2} and \ref{project}
one gets, using transitivity of projection maps \cite[3.7.1]{li}, a simple proof of the well-known fact that 
\emph{for all\/ $E\in\Dqc(X)$ and\/ $G\in\Dqc(Y)$ the projection\va2 maps are isomorphisms}\looseness=-1
\begin{equation}\label{projf}
\R\fst E\Otimes{Y}G\!\iso\! \R\fst(E\Otimes{\<X}\LL f^*\<G\>),\ \ \ 
G\Otimes{Y}\R\fst E\!\iso\!\R\fst(\LL f^*\<G\Otimes{\<X}E\>).
\end{equation}
(For the general case of arbitrary qcqs maps, see e.g., \cite[3.9.4]{li}.)

\pagebreak[3]
Here we will emphasize duality results, a basic one being Theorem~\ref{qc duality} (sheafified affine duality).\va2

First, with $\pt\set\R\>\sHom_\psi(\fst\OX,-)$ as in~\eqref{duality for phi}, set 
\begin{equation}\label{fflat}
f^\flat \set \bar{f\:\<}^{\!\<*}\<\<\pt.
\end{equation}

\begin{sublem}\label{fst fflat}
If\/ $G\in\D(Y)$ is such that\/ $\R\>\sHom_Y(\fst\OX, G\>)\in\Dqc(Y)$ then
$\pt G\in\Dqc(\oY)$ $($whence $f^\flat G\in\Dqc(X)),$ and so one has the composite isomorphism
\[
\bar t_G\colon\R\fst f^\flat G
= \phi_*\R\bar\fst \bar{f\:\<}^{\!\<*}\<\pt G
\underset{\textup{\ref{^* equivalence}}}
\iso \phi_*\pt G
\,\underset{\textup{\eqref{phiRHompsi}}}
\iso\,\R\>\sHom_Y(\fst\OX,G\>).
\]

\end{sublem}

\begin{proof} 
The functor $\phi_*$ is exact, and $\phi_*\pt G\cong\R\>\sHom_Y(\fst\OX, G\>)\in\Dqc(Y)$, 
so by \cite[p.\,218, (2.2.4)]{EGA1}, $\pt G\in\Dqc(\oY)$. \end{proof}

Consequently: 
\begin{subthm}[Sheafified affine duality]\label{qc duality} 
Let\/ $G\in\D(Y)$ be such that\/ $\R\>\sHom_Y(\fst\OX, G\>)\in\Dqc(Y)$.
For all\/ $F\in\Dqc(X)$ one has the natural composite bifunctorial \emph{duality isomorphism}
\begin{align*}
\R\fst\R\>\sHom_\sX\<(F, f^\flat\<  G\>)&\iso
\phi_*\R\bar\fst\R\>\sHom_\sX\<(F, \bar{f\:\<}^{\!\<*}\!\pt G\>)\\
&\underset{\ref{duality for barf}}\iso
\phi_*\R\>\sHom_{\>\oY}(\R\bar\fst F, \pt G\>)\\
&\underset{\ref{duality for phi}}\iso
\R\>\sHom_Y(\phi_*\R\bar\fst F,G\>) \<\iso\< 
\R\>\sHom_Y(\R\fst F,G\>).\makebox[10pt]{\hss$\qquad\>\> \square$}
\end{align*}
\end{subthm}

\smallskip

Proposition~\ref{qc duality2} below gives another factorization of the duality isomorphism in~\ref{qc duality}.
 The next Proposition, in essence, globalizes that.
\vskip3pt

\begin{subprop}\label{represent}

Let\/ $G\in\D(Y)$ satisfy\/ $\R\>\sHom_Y(\fst\OX, G\>)\in\Dqc(Y),$ and let\/ $t^{}_{\<G}$  be the composite map
\[
\R\fst f^\flat G
= \phi_*\R\bar\fst \bar{f\:\<}^{\!\<*}\<\pt G
\underset{\textup{\ref{^* equivalence}}}
\iso \phi_*\pt G
\,\underset{\textup{\ref{counit}}}
\lto\, G
\]
given by the counit maps associated to the adjunctions\/ $\bar{f\:\<}^{\!\<*}\!\dashv \R\fst$ and $\phi_*\!\dashv \pt$ in\/ \textup{\ref{^* equivalence}} and\/~\textup{\ref{adjunction0}} respectively. Then\/$:$\va2

\textup{(i)} $\big(f^\flat G, \>t^{}_{\<G}\big)$\va{.6} represents the  contravariant functor\/ 
$\Hom_{\D(Y)}(\R\fst-, \:G\>)$ from\/ $\Dqc(X)$ to the category of\/ $\Gamma(Y,\OY)$-modules.\va1

\textup{(ii)} $t^{}_{\<G}$ is the  natural composite map
\[
\R\fst f^\flat G
\,\underset{\textup{\ref{fst fflat}}}
\iso\,\R\>\sHom_Y(\fst\OX,G\>)
\lto \R\>\sHom_Y(\OY\<,G\>)
\iso G.
\]
\end{subprop}

\begin{proof}
(i) The assertion is that 
for all $F\in\Dqc(X)$, the map gotten by going clockwise around the following natural diagram from  top left
to bottom left  is an isomorphism---which holds because, clearly, the diagram commutes. 
\[
\def\1{$\Hom_{\D(X)}\<(F, \bar{f\:\<}^{\!\<*}\!\pt G\>)$}
\def\2{$\Hom_{\D(\oY)}(\R\bar\fst F, \R\bar\fst\bar{f\:\<}^{\!\<*}\!\pt G\>)$}
\def\3{$\Hom_{\D(\oY)}(\R\bar\fst F, \pt G\>)$}
\def\4{$\Hom_{\D(Y)}(\phi_*\R\bar\fst F, G\>)$}
\def\5{$\Hom_{\D(Y)}(\R\fst F, \R\fst\bar{f\:\<}^{\!\<*}\!\pt G\>)$}
\def\6{$\Hom_{\D(Y)}(\phi_*\R\bar\fst F, \phi_*\R\bar\fst\bar{f\:\<}^{\!\<*}\!\pt G\>)$}
\def\7{$\Hom_{\D(Y)}(\phi_*\R\bar\fst F, \phi_*\pt G\>)$}
\def\8{$\Hom_{\D(Y)}(\R\fst F, G\>)$}
 \bpic[xscale=7, yscale=1.65]

   \node(11) at (1,-1){\1};   
   \node(12) at (2,-1){\5}; 
   
   \node(21) at (1,-2){\2};  
   \node(22) at (2,-2){\6};
  
   \node(31) at (1,-3){\3};  
   \node(32) at (2,-3){\7};
   
   \node(42) at (2,-4){\4};  
   \node(41) at (1,-4){\8};
   
    \draw[->] (11)--(12) node[above=1pt, midway, scale=.75] {$$};  
    
    \draw[->] (21)--(22) node[above=1pt, midway, scale=.75] {$$}; 
    
    \draw[->] (31)--(32) node[above=1pt, midway, scale=.75] {$$}; 
    
    \draw[double distance=2pt] (41)--(42) ;
    
    \draw[->] (11)--(21)  node[right=1pt, midway, scale=.75] {$\simeq$}
                                   node[left=1pt, midway, scale=.75] {\ref{^* equivalence}};
    \draw[->] (21)--(31) node[right=1pt, midway, scale=.75] {$\simeq$}
                                   node[left=1pt, midway, scale=.75] {\ref{^* equivalence}}; 
    \draw[->] (31)--(41) node[right=1pt, midway, scale=.75] {$\simeq$}
                                  node[left=1pt, midway, scale=.75] {\ref{adjunction0}};
   
    \draw[double distance=2pt] (12)--(22) ;
    \draw[->] (22)--(32)  node[left=1pt, midway, scale=.75] {$\simeq$}
                                    node[right=1pt, midway, scale=.75] {\ref{^* equivalence}}; 
                                    
    \draw[->] (32)--(42) node[right=1pt, midway, scale=.75] {\ref{counit}} ;
    \epic
\]

(ii) Left to the reader.
\end{proof}


\begin{subcor}\label{flat=times}
If\/ $Y$ is qcqs, then for\/ $G,$ $t^{}_{\<G}$ as in~Proposition~\textup{\ref{represent}} and
\/ $\tau^{}_{\<G}\set\tau^{}_{\<\<f}(G\>)\colon\R\fst f^{\<\times}\<G\to G$ the canonical map,
there is a unique $\D(X)$-map \va{-5}
\[
\xi^{}_{\<f}\colon f^\flat\<G\to f^{\<\times} G
\]
such that\/ $t^{}_{\<G}=\tau^{}_{\<G}\smallcirc \R\fst\xi^{}_{\<f}\>;$
and this\/~$\xi^{}_{\<f}$ is an isomorphism.
\end{subcor}

\begin{proof}
Both $(f^\flat G, t^{}_{\<G}\big)$ and $(f^{\<\times}\< G,\tau^{}_{\<G}\>)$
represent $\Hom_{\D(Y)}(\R\fst-, \:G\>)$. 
\end{proof}
 
\medskip
\begin{subcosa}\label{quasi}
We review some conditions under which \mbox{$\R\>\sHom_Y^\bullet(\fst\OX,\,G\>)\in\Dqc(Y)$}
for \emph{all} $G\in\Dqcpl(Y)$ (resp.~$\Dqc(Y)$)---see Lemma~\ref{qcHom} below. \va2

For this we need the notion of \emph{pseudo-coherence,}
discussed in detail in~the primary source~\cite{Il}, or, more accessibly, 
in \cite[pp.\;283{\it ff,} \S2]{TT}, or in \cite[tag 08E4]{Stacks}. (A brief summary appears in 
\cite[\S4.3]{li}.) The simplest characterization is that an
$\OY$-complex~$E$ is pseudo\kf-coherent if 
over each affine open $U\subset Y$, the
restriction~$E|_U$ is $\D(U)$-isomorphic to a bounded-above complex~$F$ of finite\kf-rank locally free
$\OX$-modules. (When $E$ is an \mbox{$\OY$-module,} this means that $E|_U$ is 
resolvable by such an $F$.) It suffices for pseudo\kf-coherence of~$E$ that this condition hold 
over each member of a covering of $Y$ by affine open subschemes.\va1

Such an $F$ being K-flat, it holds for any scheme\kf-map $h\colon Z\to Y$ that, with $h_U\colon h^{-1}U\to U$ the
induced map, the natural map is an isomorphism
$\LL h_U^*\<F\iso h_U^*\< F$, and hence that if $E$ is pseudo\kf-coherent then so is $\LL h^*\<\< E$.\va1

A finite scheme\kf-map $f\colon X\to Y$ is pseudo\kf-coherent if for any pseudo\kf-coherent $\OX$-complex~$E$, $\R\fst E$ is a pseudo\kf-coherent $\OY$-complex (see \cite[remark just before 4.7.3.4]{li}).
 For this condition to hold, it suffices that $\fst\OX$ be pseudo\kf-coherent.  
(This assertion, being local, need only be shown when
$X$ and~$Y$ are affine, for which case see \cite[4.3.2, (i)$\Leftrightarrow$(iii)]{LN07}.) 
If $Y$~is locally noetherian, then every finite $f$ is pseudo\kf-coherent.\va1

If, in addition, $\fst\OX$ has finite tor-dimension locally (and hence globally if $Y$ is quasi-compact), i.e., 
$\fst\OX$ is a perfect complex \cite[p.\,135, 5.8.1]{Il}, then  $f$ is \emph{quasi-perfect}, i.e., for any perfect $\OX$-complex~$E$, $\R\fst E$ is
a perfect $\OY$-complex. (Use \cite[Proposition 2.1]{LN07} locally on $Y\<$.) 
It is equivalent that  $f$ be \emph{perfect}, as defined in~
\cite[p.\,250, D\'efinition 4.1]{Il}---see, for instance, \cite[Example 4.7.3(d)]{li}.\va1

For example, perfection holds for any affine map $f$ that is flat and locally finitely presentable, so that 
$\fst\OX$ is a flat and locally finitely presentable---i.e., locally free of finite rank---$\OY$-module, 
see \cite[p.\,357, (9.1.15)(i)]{EGA1}.\va1

Perfection also holds whenever $f$ is a regular immersion, so that $\fst\OX$ is locally quasi-isomorphic to a Koszul complex.\va1

Clearly, a composition of two pseudo\kf-coherent (resp.~perfect) finite maps is again pseudo\kf-coherent (resp.~perfect). 
\end{subcosa}

\medskip
\begin{sublem}\label{qcHom}
If $F\in\D(Y)$ is pseudo-coherent\/ $($resp.~perfect$\>\>)$ then for  \mbox{all\/ $G\in \Dqcpl(Y)$} 
$($resp.~$\Dqc(Y)\>),$ 
$\R\>\sHom_Y(F\<,\>\>G\>)\in \Dqcpl(Y)$  $($resp.~$\Dqc(Y))$. 
 \end{sublem}

\begin{proof}
The assertion is essentially local, so one can assume $Y$ affine and then proceed as in the proof of \cite[Lemma 4.3.5]{li}. Alternatively, use \cite[p.\,73--74, Proposition 7.3 (ii) and (iii)]{RD}, or see 
\cite[tag~0A6H]{Stacks}.
\end{proof}

\medskip

\begin{subcor}[Sheafified finite duality]\label{qp duality} 
Let\/ $f\colon X\to Y$ be a pseudo\kf-coherent  $($resp.\ perfect$\>)$ finite
scheme\kf-map, and\/ $G\in\Dqcpl(Y)$ $($resp.\ $\Dqc(Y))$.
For all\/ $F\in\Dqc(X)$ one has the composite bifunctorial \emph{duality isomorphism}
\begin{align*}
\R\fst\R\>\sHom_\sX\<(F, f^\flat  G\>)&\,=\!=\ \>
\phi_*\R\bar\fst\R\>\sHom_\sX\<(F, \bar{f\:\<}^{\!\<*}\!\pt G\>)\\
&\underset{\eqref{duality for barf}}\iso
\phi_*\R\>\sHom_{\>\oY}(\R\bar\fst F, \pt G\>)\\
&\underset{\eqref{duality for phi}}\iso
\R\>\sHom_Y(\phi_*\R\bar\fst F,G\>) =\!= 
\R\>\sHom_Y(\R\fst F,G\>)
\end{align*}
\end{subcor} 

\begin{proof}
This follows immediately  from \ref{qc duality}  and ~\ref{qcHom}.
\end{proof}

\smallskip
Application of the functor $\Hr^0\R\Gamma(Y,-)$ yields an adjunction $\R\fst\!\dashv f^\flat$ 
composed of those given by~\ref{^* equivalence} and \ref{adjunction0}:
\begin{subcor}\label{right adjoint} 
For any\/ $f$ as in\/~\textup{\ref{qp duality}} the functor $f^\flat\colon\Dqcpl(Y)\to\Dqcpl(X)$ is right-adjoint to\/ $\R\fst\colon\Dqcpl(X)\to\Dqcpl(Y),$  with unit map
the natural functorial composite map
\[
u_F\colon F\underset{\ref{^* equivalence}}\iso \bar{f\:\<}^{\!\<*}\R\bar\fst F \underset{\ref{unit}}\lto 
\bar{f\:\<}^{\!\<*}\<\<\pt\phi_*\R\bar\fst F=\!= f^\flat\R\fst F\>,
\]
and counit  the functorial map\/ $t_G$ $($see\textup{~\ref{represent}),}
which identifies naturally with the canonical map\/ \textup{(``\kf evaluation at 1")}
\[
\R\>\sHom_Y(\fst\OX,\>G\>)\lto \R\>\sHom_Y(\OY\<,\>G\>)=G.
\]

When\/ $f$ is perfect,  the superscript ``\;$^\textup{\cmt\char'053}$" can be omitted.\hfill$\square$
\end{subcor}

\begin{subcosa} (Cf.~\cite[p.172, 6.8]{RD}.) Similarly, if $f$ is \emph{any} closed immersion then 
$f^\flat$~is right-adjoint to\/ $\fst=\R\fst\colon\D(X)\to \D(Y)$. That's because
for such ~$f,$ Lemma~\ref{^* equivalence} and hence Proposition~\ref{qc duality}
hold for all $G\in\D(Y)$.  
\end{subcosa}

\begin{subex}\label{f*unit}
\emph{Let\/ $G\in\D(Y)$ be such that\/ $\R\>\sHom_Y(\fst\OX, G\>)\in\Dqc(Y),$ 
$u$~the unit map in\/ \textup{\ref{right adjoint},} and $\mu$ the natural composite}
\[
\fst \OX\Otimes{Y}\fst\OX\lto\fst \OX\otimes_Y\fst\OX\lto\fst\OX\>.
\]
\emph{The following natural diagram commutes.}\va3
\[
\def\1{$\R\fst f^\flat G$}
\def\2{$\R\fst f^\flat\R\fst f^\flat G$}
\def\3{$\R\>\sHom_Y(\fst \OX, \R\>\sHom_Y(\fst\OX, \>G\>)\>)$}
\def\4{$\R\>\sHom_Y(\fst \OX, G\>)$}
\def\5{$\R\>\sHom_Y(\fst \OX\Otimes{Y}\<\fst\OX, \>G\>)$}
 \bpic[xscale=5, yscale=1.9]

   \node(11) at (1,-1){\1};   
   \node(12) at (2,-1){\2}; 
   
   \node(22) at (2,-2){\3};  
    
   \node(31) at (1,-3){\4};  
   \node(32) at (2,-3){\5};

    \draw[->] (11)--(12) node[above, midway, scale=.75] {$\R\fst u_{\<\<f^{\mkern-.5mu\flat} \<G}$};  
    
    \draw[->] (31)--(32) node[below=1pt, midway, scale=.75] {$\via\mu$}; 
    
    \draw[->] (11)--(31)  node[right=1pt, midway, scale=.75] {$\simeq$}
                                   node[left=1pt, midway, scale=.75] {\textup{\ref{fst fflat}}};
                                   
    \draw[->] (12)--(22) node[left=1pt, midway, scale=.75] {$\simeq$}
                                   node[right=1pt, midway, scale=.75] {\textup{\ref{fst fflat}}}; 
    \draw[->] (22)--(32) node[left=1pt, midway, scale=.75] {$\simeq$};
   
       \epic
\]
This assertion will not be used, so the (not entirely trivial) proof is omitted.
\end{subex}
\end{cosa}

\pagebreak[3]

\begin{cosa}\label{alt dual}
The following Proposition~\ref{duality map} (resp.~\ref{qc duality2}) shows that the duality isomorphism in Corollary~\ref{duality for phi} (resp.~Theorem~\ref{qc duality})
is concordant with the abstract duality map given by \cite[4.2.1]{li}. Proposition~\ref{via duality} 
characterizes the isomorphism $\xi_f\colon f^\flat G\iso f^\times G$ in \ref{flat=times} by means of that abstract map.\va1

\begin{sublem}\label{commnu}
Let\/ $h\colon V\to W$ be a ringed-space map, and let\/ $E,\; F$ be\/ $\OV$-complexes.
With\/ $\nu^{}_0=\nu^{}_0(E,F\>)$ the natural map\/ $($cf.~\textup{\cite[(3.1.4)]{li}),} $\nu=\nu(E,F\>)$ as in~\eqref{nu}$,$ and\/ $\CH\set\CH om,$ the following natural diagram commutes.

\[
\def\1{$h_*\CH_V(E\<,\>F\>)$}
\def\2{$\R h_*\CH_V(E\<,\>F\>)$}
\def\3{$\CH_W(h_*E, h_*F\>)$}
\def\5{$\R\CH_W(h_*E, \R h_*F\>)$}
\def\6{$\R\CH_W(\R h_*E, \R h_*F\>)$}
\def\7{$\R\CH_W(h_*E,  h_*F\>)$}
\def\8{$\R h_*\R\CH_V(E\<,\>F\>)$}
 \bpic[xscale=4.5, yscale=1.6]
   
   \node(11) at (1,-1){\1} ;
   \node(12) at (2,-1){\2} ;   
   \node(13) at (3,-1){\8} ;

   \node(23) at (3,-2){\6} ; 
   
   \node(31) at (1,-3){\3} ;  
   \node(32) at (2,-3){\7} ;
   \node(33) at (3,-3){\5} ;
   
    \draw[->] (11)--(12) ;  
    \draw[->] (12)--(13) ;
    
    \draw[->] (31)--(32) ;
    \draw[->] (32)--(33) ;
    
    \draw[->] (11)--(31) node[left=1pt, midway, scale=.75] {$\nu^{}_0$} ;
    
    \draw[->] (13)--(23) node[right=1pt, midway, scale=.75] {$\nu$} ; 
    \draw[->] (23)--(33)  ; 
                                      
    \epic
\]
\end{sublem}

\pagebreak[3]
\begin{proof}
The diagram expands naturally as follows, with maps labeled $\epsilon^{\bullet}$ (resp.~$\eta$) induced by the counit  map
$h^*h_*E\to E$ (resp.~the unit map $E\to h_{*}h^{*}E$): 
\[\mkern-4mu
\def\1{$h_*\CH_V(E\<,\>F\>)$}
\def\2{$\R h_*\R\CH_V(E\<,\>F\>)$}
\def\3{$\R h_*\R\CH_V(h^{\<*}\<h_*E\<,\>F\>))$}
\def\4{$h_*\CH_V(h^{\<*}\<h_*E\<,\>F\>)$}
\def\5{$\R h_*\R\CH_V(\LL h^{\<*}\<h_*E\<,F\>))$}
\def\6{$\CH_W(h_*E,  h_*F\>)$}
\def\7{$\R h_*\R\CH_V(\LL h^*\R h_*E\<,F\>)\ \:$}
\def\8{$\R\CH_W(h_*E, h_*F\>)$}
\def\9{$\R\CH_W(h_*E, \R h_*F\>)$}
\def\ten{$\R\CH_W(\R h_*E, \R h_*F\>)$}
\def\lvn{$\CH_W(h_*h^{\<*}\<h_*E, h_*F\>)$}
\def\twv{$\R h_*\CH_V(E\<,\>F\>)$}
 \bpic[xscale=3.75, yscale=1.3]

   \node(11) at (1,-.5){\1} ;   
   \node(12) at (2.22,-.5){\twv} ; 
   \node(14) at (3.5,-.5){\2} ; 
  
   \node(22) at (1.57,-1.5){\4} ;  
   \node(23) at (2.84,-1.5){\3} ; 
     
   \node(42) at (1.57,-3.5){\lvn} ;
   \node(44) at (3.5,-2.5){\7} ; 
  
   \node(63) at (2.2,-2.5){\5} ;
   \node(64) at (3.5,-3.5){\ten} ;
   
   \node(71) at (1,-4.5){\6} ;  
   \node(72) at (2.22,-4.5){\8} ;  
   \node(74) at (3.5,-4.5){\9} ;   
 
    \draw[->] (11)--(12) ;
    \draw[->] (12)--(14) ;

    \draw[->] (22)--(23) ;
    
    \draw[->] (71)--(72) ;
    \draw[->] (72)--(74) ;
    
    \draw[->] (11)--(71) node[left, midway, scale=.75] {$\nu^{}_0(F\<,F')$} ;
    
    \draw[->] (22)--(42) node[left=-1pt, midway, scale=.75] {$\nu^{}_0(h^*h_*F\<\<,F')$} ;
    
    \draw[->] (14)--(44) ;  
    \draw[->] (44)--(64) node[right=1, midway, scale=.75] {$\simeq$} ;  
    \draw[->] (64)--(74) ;  
    
    \draw[->] (11)--(22) node[above=-2pt, midway, scale=.75] {$\mkern18mu \epsilon$} ;
    \draw[<-] (42)--(71) node[above=-2pt, midway, scale=.75] {$\epsilon'\mkern15mu$} ;
    \draw[<-](1.2,-4.25)--(1.48,-3.76)  node[below=-2pt, midway, scale=.75] {$\mkern18mu \eta$} ;
    \draw[->] (14)--(23) node[above=-2pt, midway, scale=.75] {$\mkern-20mu \epsilon$} ;
    \draw[->] (23)--(63) ;
    \draw[->] (63)--(74) ;
    \draw[->] (44)--(63) ;
    
    \node at (1.3,-2)[scale=.9]{\circled1} ;
    \node at (2.2,-3.5)[scale=.9]{\circled2} ;
    \node at (3.02,-2)[scale=.9]{\circled3} ;
    
  \epic
\]

Commutativity of subdiagram \circled1 results from the obvious equalities
\[
\nu_0(E,F\>) =\eta\epsilon'\<\nu_0(E,F\>) = \eta \nu^{}_0(h^*h_*E\<\<,F\>)\>\epsilon.
\]

Commutativity of \circled2 results from \cite[3.2.3(ii)]{li}\va{,6} (modulo replacement in \cite[3.1.5, 3.1.6]{li} of $(f, A,B)$ by $(h, h_*E,F\>)$\,\ldots)\va{.4}

Commutativity of  \circled3 results from that of \cite[3.2.1.2]{li}.  

Commutativity of the unlabeled subdiagrams is clear.
\end{proof}

\begin{sublem} \label{tensor and _*}
Let\/ $h\colon V\to W$ be a ringed-space map, and let\/ $E,\; F$ be\/ $\OV$-complexes.
The following natural diagram commutes.
\[
\def\1{$\R h_*E\Otimes{W}\R h_*F$}
\def\2{$\R h_*(E\Otimes{V}F\>)$}
\def\3{$h_*E\Otimes{W}h_*F$}
\def\4{$\R h_*(E\otimes_{V}\<F\>)$}
\def\5{$h_*E\otimes_{W}h_*F$}
\def\6{$h_*(E\otimes_{V}\<F\>)$}
 \bpic[xscale=3, yscale=1.25]

   \node(11) at (1,-1){\1} ;   
   \node(13) at (3,-1){\2} ; 
   
   \node(21) at (1,-2){\3} ;
   \node(23) at (3,-2){\4} ;
   
   \node(31) at (1,-3){\5} ;  
   \node(33) at (3,-3){\6} ;
   
    \draw[->] (11)--(13) node[above, midway, scale=.75] {$\gamma$}   
                                   node[below, midway, scale=.75] {\cite[3.2.4(ii))]{li}} ;
    \draw[->] (31)--(33) node[below, midway, scale=.75] {$\gamma^{}_0$} ;

    \draw[<-] (11)--(21) node[left=1pt, midway, scale=.75] {$$} ;
    \draw[->] (21)--(31) node[left=1pt, midway, scale=.75] {$$} ; 
    
    \draw[->] (13)--(23) node[left=1pt, midway, scale=.75] {$$} ;
    \draw[<-] (23)--(33) node[left=1pt, midway, scale=.75] {$$} ; 

 \epic
\]

\end{sublem}

\begin{proof} By definition, (or, if a definition other than that in \cite[3.2.4]{li} is preferred, one shows that)  $\gamma$ is adjoint to the natural composite map
\[
\LL h^*(\R h_*E\Otimes{W} \R h_*F\>) \lto \LL h^*\R h_*E\Otimes{V} \LL h^*\R h_*F \lto E\Otimes{V} F.
\]

By \cite[3.1.9]{li}, and application of the ``duality principle'' of \cite[\S3.4.5]{li} to the argument at the beginning of that subsection%
\footnote{\ In connection with \cite[\S3.4.5]{li}, in the proof of \cite[3.2.4(ii)]{li}, the erroneous phrase 
``the adjoint (3.5.4.1) of (3.4.2.1)'' should be replaced by ``(3.4.5.1)''.}
(or otherwise), $\gamma^{}_0$ is adjoint to the natural composite map
\[
h^{*}(h_{*}E\otimes_W\< h_{*}F\>) \lto h^{*}h_{*}E\otimes_V\< h^{*}h_{*}F \lto E\otimes_V\< F.
\]

It suffices therefore to show
commutativity of the natural diagram
\begin{small}
\[\mkern-3mu
\def\1{$\R h_*E\Otimes{W}\R h_*F$}
\def\2{$\R h_*(E\Otimes{V}\<F\>)$}
\def\3{$h_*E\Otimes{W}h_*F$}
\def\4{$\R h_*(E\otimes_{V}\<F\>)$}
\def\5{$h_*E\otimes_W\<h_*F$}
\def\6{$h_*(E\otimes_{V}\<F\>)$}
\def\7{$\R h_*\LL h^*(\R h_*E\Otimes{W}\R h_*F\>)$}
\def\8{$\R h_*(\LL h^*\R h_*E\Otimes{V}\LL h^*\R h_*F\>)$}
\def\9{$\R h_*\LL h^*(h_*E\Otimes{W} h_*F\>)$}
\def\ten{$\R h_*\LL h^*(h_*E\otimes_W\< h_*F\>)$}
\def\lvn{$\R h_* h^*(h_*E\otimes_W\< h_*F\>)$}
\def\twv{$h_* h^*(h_*E\otimes_W\< h_*F\>)$}
\def\thn{$h_*(h^*h_*E\otimes_V\<  h^*h_*F\>)$}
\def\frn{$\R h_*( h^*h_*E\otimes_V\<  h^*h_*F\>)$}
\def\ffn{$\R h_*(\LL h^*h_*E\Otimes{V} \LL h^*h_*F\>)$}
\def\sxn{$\R h_*(h^*h_*E\Otimes{V} h^*h_*F\>)$}
 \bpic[xscale=3.47, yscale=1.75]

   \node(11) at (1,-.5){\1} ;   
   \node(12) at (1.8,-1.5){\7} ;
   \node(13) at (2.55,-.5){\8} ; 
   \node(14) at (3.95,-.51){\2} ; 
   
   \node(21) at (1,-3){\3} ;
   \node(22) at (1.97,-3){\9} ;
   \node(225) at (2.55,-2.25){\ffn} ;
   \node(235) at (3.35,-1.5){\sxn} ;
   \node(23) at (3.35,-3){\frn} ;
   \node(24) at (3.95,-4){\4} ;
   
   \node(32) at (1.6,-4){\ten} ;
   \node(33) at (2.75,-4){\lvn} ;
   
   \node(41) at (1,-5){\5} ;  
   \node(42) at (1.9,-5){\twv} ; 
   \node(43) at (3,-5){\thn} ; 
   \node(44) at (3.95,-5){\6} ;
   
    \draw[->] (12)--(13) ;  
   
    \draw[->] (21)--(22) ;  
   
    \draw[->] (32)--(33) ;

    \draw[->] (41)--(42) ;  
    \draw[->] (42)--(43) ;
    \draw[->] (43)--(44) ;
    
    \draw[<-] (11)--(21) node[left=1pt, midway, scale=.75] {$$} ;
    \draw[->] (21)--(41) ; 
    
    \draw[<-] (12)--(22)  ;
    \draw[<-] (32)--(22)  ;
    \draw[->] (225)--(13) ;
    \draw[->] (14)--(24)  ;
    \draw[<-] (24)--(44)  ; 
      
    \draw[->] (11)--(12) ;
    \draw[->] (13)--(14) ; 
    \draw[->] (33)--(23) ;
    \draw[->] (23)--(24) ;
    \draw[->] (41)--(32) ;
    \draw[->] (42)--(33) ; 
    \draw[->] (22)--(225) ;
    \draw[->] (225)--(235) ;
    \draw[->] (235)--(14) ;  
    \draw[->] (235)--(23) ;
    
   \node at (3.05,-1.03) [scale=.9]{\circled1} ;
   \node at (2.65,-3.02) [scale=.9]{\circled2} ;
   \node at (1.8,-4.5) [scale=.9]{\circled3} ;
      
 \epic
\]
\end{small}

Commutativity of the unlabeled subdiagrams is easily checked. 

Commutativity of \circled1 (respectively~\circled3) is given by that of \cite[(3.2.1.2)]{li}, (respectively \cite[(3.2.1.3)]{li}).

Commutativity of \circled2 is given by  \cite[(3.2.4(i)]{li}.
\end{proof}

As in \S\ref{2.3}, let $Y$ be any ringed space, $\psi\colon\OY\to\CS$ an $\OY$-algebra, $\oY$ the 
ringed space~$(Y\<\<,\>\CS)$, and $\phi\colon\oY\to Y$ the ringed-space map $(\id_Y,\psi)$. Then
$\CO_{\>\oY}=\CS$, $\sA(\oY)$ is the category of $\CS$-modules, 
\mbox{$\phi_*\colon \sA(\oY)\to\sA(Y)$}
is restriction of scalars, $\otimes_{\>\oY}=\otimes^{}_{\<\CS}$ and $\sHom_{\oY}=\sHom^{}_{\CS}\>$.

\begin{sublem}\label{gamma-nu} For\/ $E,F\in\D(\>\oY),$ $G\in\D(Y),$ the following natural diagram, 
with $\CH\set\sHom$  and with\/  
$\gamma\colon\phi_*E\<\Otimes{Y}\<\phi_*F\to  \phi_*(E\Otimes{\mkern.5mu\oY}F\>)$ 
as in\/ \cite[3.2.4(ii))]{li}$,$ commutes.\va{2}
\[
\def\1{$\phi_*\R\CH_{\oY}(E, \R\CH_\psi(F\<, G\>))$}
\def\2{$\R\CH_Y(\phi_*E, \phi_*\R\CH_\psi(F\<, G\>))$}
\def\3{$\R\CH_Y(\phi_*(E\Otimes{\mkern.5mu\oY}F\>), G\>)$}
\def\4{$\R\CH_Y(\phi_*E\<\Otimes{Y}\<\phi_*F\<, \>G\>)$}
\def\5{$\R\CH_Y(\phi_*E, \R\CH_Y(\phi_*F\<, G\>))$}
 \bpic[xscale=3, yscale=1.6]

   \node(11) at (3,-1){\1};   
   \node(13) at (3,-3){\2}; 
   
   \node(21) at (1,-1){\3};  
   \node(22) at (1,-2.03){\4};
   \node(23) at (1,-3){\5};  
    
    \draw[->] (11)--(21)  node[above=1pt, midway, scale=.75] {$\Iso$}
                                    node[below, midway, scale=.75] {\textup{\ref{adjass}}} ;
   
    \draw[->] (13)--(23) node[above=.75pt, midway, scale=.75] {$\Iso$}
                                   node[below=1.5pt, midway, scale=.75] {\eqref{phiRHompsi}} ;
                                   
    \draw[->] (11)--(13) node[right=1pt, midway, scale=.75] {\eqref{nu}}  ;  
    
    \draw[<-] (23)--(22) node[left=1pt, midway, scale=.75] {$\simeq$} ; 
    \draw[<-] (22)--(21) node[left=1pt, midway, scale=.75] {$\via\gamma$} ; 

  \epic
\]
\end{sublem}

\pagebreak[3]
\begin{proof} 
It may be assumed that $F$ is a K-flat $\CS$-complex and that $G$~is a \mbox{K-injective} $\OY$-complex; 
and then
(as in the proof of ~\ref{adjass0}) 
$\CH_\psi(\CS,G\>)$ is a K-injective $\CS$-complex.
The diagram in question expands naturally as  

\begin{small}
 \[\mkern-4mu
\def\1{$\R\CH_Y(\phi_*(E\Otimes{\mkern.5mu\oY}F\>), G\>)$}
\def\2{$\phi_{\<*}\R\CH_{\oY}\<\<\big(\<E, \R\CH_\psi(F,G\>)\big)$}
\def\3{$\CH_Y(\phi_{\<*}(E\otimes_{\>\oY}\<\<F\>),G\>)$}
\def\4{$\phi_{\<*}\CH_{\oY}\<\<\big(\<E, \CH_\psi(F,G\>)\big)$}
\def\5{$\phi_{\<*}\R\CH_{\oY}\<\<\big(\<E, \CH_\psi(F,G\>)\big)$}
\def\6{$\CH_{Y}\<\<\big(\phi_{\<*}E, \phi_{\<*}\CH_\psi(F,G\>)\big)$}
\def\7{$\CH_{Y}\<\<\big(\phi_{\<*}E, \CH_Y(\phi_*F,G\>)\big)$}
\def\8{$\R\CH_{Y}\<\<\big(\phi_{\<*}E, \phi_{\<*}\CH_\psi(F,G\>)\big)$}
\def\9{$\R\CH_{Y}\<\<\big(\phi_{\<*}E, \CH_Y(\phi_*F,G\>)\big)$}
\def\ten{$\R\CH_Y(\phi_*E\<\Otimes{Y}\<\phi_*F\<, \>G\>)$}
\def\lvn{$\R\CH_Y\<\<\big(\phi_{\<*}E,\R\CH_Y(\phi_{\<*}F,G\>)\big)$}
\def\twv{$\R\CH_{Y}\<\<\big(\phi_{\<*}E, \phi_{\<*}\R\CH_\psi(F,G\>)\big)$}
\def\sxn{$\CH_Y(\phi_*E\otimes_Y\<\phi_*F,G\>)$}
 \bpic[xscale=2.1, yscale=2.65]

   \node(11) at (.75,-.5){\1};   
   \node(15) at (4.95,-.5){\2}; 
      
   \node(22) at (2, -1.5){\3} ;    
   \node(21) at (2,-2){\sxn};
   \node(23) at (3.15,-.9){\4};
   \node(24) at (4.3,-1.5){\5};
   
   \node(33) at (3.15,-2.5){\6};
   
   \node(42) at (2,-3){\7};
   \node(43) at (4.3,-3){\8};
   \node(44) at (3.15,-3.55){\9};
   
   \node(51) at (.75,-2.5){\ten};  
   \node(53) at (1,-4){\lvn};
   \node(55) at (4.95,-4){\twv};

    \draw[->] (15)--(11) node[below=1pt, midway, scale=.75] {\textup{\ref{adjass}}} 
                                  node[above=1pt, midway, scale=.75] {$\Iso$}  ;  
    
    \draw[->] (23)--(22) node[above=-1.2pt, midway, scale=.75] {\rotatebox{28}{$\Iso$}} ;
    \draw[->] (24)--(23) node[above=-1.2pt, midway, scale=.75] {\rotatebox{-28}{$\Iso$}};   
                                
    \draw[double distance=2pt] (33)--(42) ;
    
    \draw[->] (42)--(44) ;                                  

    \draw[->] (33)--(43) ;   
    
    \draw[->] (43)--(55) ;
    \draw[->] (55)--(53) node[above=1pt, midway, scale=.75] {$\Iso$} 
                                      node[below=1pt, midway, scale=.75] {\eqref{phiRHompsi}};
    \draw[->] (15)--(24) node[above=-.5pt, midway, scale=.75] {\rotatebox{58}{$\Iso$}} ;
    
     \draw[<-] (.75,-3.87)--(.75, -2.635) node[left=1pt, midway, scale=.75] {$\simeq$} ;
     \draw[<-]  (.75, -2.365)--(11) node[left=1pt, midway, scale=.75] {$\via\gamma$} ;
     \draw[->] (42)--(53) ;
     \draw[<-] (5.4, -3.87)--(5.4,-.63) node[right=1pt, midway, scale=.75] {$\eqref{nu}$} ;
     \draw[->] (24)--(43) node[right=1pt, midway, scale=.75] {$\eqref{nu}$} ; 
     \draw[->] (23)--(33)  ;
     \draw[<-] (21)--(22) ;
     \draw[<-] (42)--(21) node[right, midway, scale=.75] {$\simeq$} ;
     \draw[double distance=2pt] (43)--(44) ;
  
    \draw[->] (22)--(11) node[above=-1.75pt, midway, scale=.75] {\rotatebox{-40}{$\ \Iso$}} ; 
    \draw[->] (44)--(53) ;
    \draw[->] (21)--(51)   ;
  
  \node at (2,-.91)[scale=.9]{\circled1} ;
  \node at (1.42,-1.75)[scale=.9]{\circled2} ;
  \node at (2.6,-1.75)[scale=.9]{\circled3} ;
  \node at (3.725, -2)[scale=.9]{\circled4} ;
  \node at (1.42, -2.77)[scale=.9]{\circled5} ;
  \epic
  \]
\end{small}
It suffices therefore to show commutativity of all the subdiagrams.\va2

Checking commutativity of the unlabeled subdiagrams is straightforward.\va2 

Commutativity of subdiagram \circled1 is given by that of the\va1 diagram in~\ref{adjass}.\va2

Commutativity of \circled2 follows from Lemma~\ref{tensor and _*}; of \circled4 from Lemma~\ref{commnu}; and of \circled5 from Corollary~{\textup{\ref{adjass}}} (with $\CS=\OY$ and $\psi=\id$).\va2

Commutativity of \circled3 is easily verified: just do so after applying global sections
over an arbitrary open $U\subset Y\<$.
\end{proof}

 \pagebreak
 \begin{subprop}\label{duality map} 
The isomorphism\/~$\alpha_\phi(E,\CO_{\>\oY},G\>)^{-1}$ in\/ \textup{~\ref{duality for phi}}
factors as
\[
\phi_*\R\>\sHom_{\>\oY}(E, \pt G\>)\xto[\eqref{nu}\>]{}
\R\>\sHom_{Y}(\phi_*E, \phi_*\pt G\>)\xto[\!\<\textup{\ref{counit}}]{\via}
\R\>\sHom_Y(\phi_*E, \>G\>).
\]
\end{subprop}

\begin{proof}
In Lemma~\ref{gamma-nu}, set $F\set\CO_{\>\oY}$ to get that the border of the following natural diagram commutes:
\[\mkern-3mu
\def\1{$\R\CH_Y(\phi_*(E\Otimes{\mkern.5mu\oY}\CO_{\>\oY}\>), G\>)$}
\def\2{$\R\CH_Y(\phi_*E\<\Otimes{Y}\<\phi_*\CO_{\>\oY}\<, \>G\>)$}
\def\3{$\R\CH_Y(\phi_*E, \R\CH_Y(\phi_*\CO_{\>\oY}\<, G\>))$}
\def\4{$\phi_*\R\CH_{\oY}(E\<, \phi^\flat G\>)$}
\def\5{$\R\CH_Y(\phi_*E\<, \phi_*\phi^\flat G\>))$}
\def\6{$\R\CH_Y(\phi_*E, G\>)$}
\def\7{$\R\CH_Y(\phi_*E\<\Otimes{Y}\<\OY\<, \>G\>)$}
\def\8{$\R\CH_Y(\phi_*E, \R\CH_Y(\OY\<, G\>))$}
 \bpic[xscale=4.35, yscale=2]
     
   \node(11) at (2.98,-1){\4};
   \node(12) at (2,-1){\6};  
   \node(21) at (1,-1){\1};  
   
   \node(22) at (1,-2.03){\2};
   \node(24) at (1.6,-1.51){\7};
   \node(25) at (2.4,-2.51){\8};
   \node(26) at (2.98,-2.03){\5};
     
   \node(13) at (2.98,-3){\5}; 
   \node(23) at (1,-3){\3};  
    
    \draw[->] (11)--(12)  node[above, midway, scale=.75] {$\Iso$}
                                    node[below, midway, scale=.75] {\textup{\ref{duality for phi}}} ;
    \draw[->] (12)--(21)  node[above, midway, scale=.75] {$\Iso$} ;
                                    
    \draw[->] (13)--(23) node[above=.75pt, midway, scale=.75] {$\Iso$}
                                   node[below=1.5pt, midway, scale=.75] {\eqref{phiRHompsi}} ;
                                   
    \draw[->] (11)--(26) node[right=1pt, midway, scale=.75] {\eqref{nu}}  ;  
    \draw[double distance =2] (26)--(13) ;
    \draw[->] (23)--(22) node[left=1pt, midway, scale=.75] {$\simeq$} ; 
    \draw[<-] (.98,-1.83)--(.98,-1.19) node[left, midway, scale=.75] {$\via\gamma$} ; 
    \draw[->] (1.02,-1.83)--(1.02,-1.19) node[right, midway, scale=.75] {$\bar\gamma$} ; 
    
    \draw[->] (22)--(24) ;
    \draw[<-] (12)--(24) ;
    \draw[->] (24)--(25) ;
    \draw[->] (25)--(2, -1.16) node[below, midway, scale=.75] {} ;
    \draw[->] (23)--(25) ;
    \draw[->] (26)--(12) node[above=-5pt, midway, scale=.75] {\rotatebox{-26}{\kern20pt$\via$}} 
                                   node[below=-7pt, midway, scale=.75] {\rotatebox{-26}{\kern5pt\textup{\ref{counit}}}} ;                               

   \node at (1.4,-1.28) [scale=.9]{\circled1} ;
   \node at (1.57,-2.4) [scale=.9]{\circled2} ;
   \node at (2.06,-1.85) [scale=.9]{\circled3} ;
   \node at (2.38,-1.85) [scale=.9]{\circled4} ;
  \epic
\]

Here, the map $\bar\gamma$ is the left inverse of ``via $\gamma$" induced by 
the right inverse of~$\gamma$ that is
given by the  natural commutative diagram
 \[
\def\1{$\phi_*E\Otimes{Y}\phi_*\CO_{\>\oY}$}
\def\2{$\phi_*(E\Otimes{\>\oY}\CO_{\>\oY})$}
\def\3{$\phi_*E\Otimes{Y}\OY$}
\def\4{$\phi_*E$}
 \bpic[xscale=4.1, yscale=1.5]

   \node(11) at (1,-2){\1} ;   
   \node(12) at (1,-1){\2} ; 
  
   \node(21) at (2,-2){\3} ;  
   \node(22) at (2,-1){\4} ;
   
    \draw[->] (11)--(12) node[left=1pt, midway, scale=.75] {$\gamma$} ;

    \draw[->] (21)--(22) node[right=1pt, midway, scale=.75] {$\simeq$} ;
    
    \draw[<-] (11)--(21) ;
    \draw[<-] (12)--(22) node[above, midway, scale=.75] {$\Iso$};       
 \epic
\]
 (cf.~subdiagram~\circled2 in the proof of \cite[3.4.7(iii)]{li}). Thus subdiagram \circled1 commutes; and
careful diagram-chasing shows it enough to prove
that subdiagrams \circled2, \circled3 and \circled4 commute.

Commutativity of \circled2 and of \circled4 is obvious.

Finally, one finds, using the definition of the maps involved 
(see \cite[3.5.3(e), 3.5.6(e)]{li}),  that the left-conjugate of \circled3 is the natural diagram
\[
\def\1{$F\Otimes{Y}\phi_*E$}
\def\2{$F\Otimes{Y}(\phi_*E\Otimes{Y}\<\OY)$}
\def\3{$(F\Otimes{Y}\phi_*E)\Otimes{Y}\<\OY\<,$}
 \bpic[xscale=4, yscale=1]

   \node(13) at (3,-1){\1} ; 
   
   \node(22) at (2.2,-2){\2} ; 
   
   \node(33) at (3,-3){\3} ;adjoint associativity
     
     \draw[<-] (33)--(13) ; 
  
    \draw[->] (22)--(13) node[above=-1pt, midway, scale=.75] {$\simeq\mkern15mu$};
    \draw[<-] (22)--(33) ; 
 \epic
\]
whose commutativity is easily shown when $F$ is K-flat, so that it holds 
for all $F\in\D(Y)$, whence commutativity holds for \circled3.
\end{proof}

Now specialize to where $Y$ is a scheme and, with $f\colon X\to Y$ an affine scheme\kf-map,
$\psi\colon \OY\to\CS\set \fst\OX$ is the associated map.

\begin{subprop}\label{qc duality2}
The duality isomorphism in Theorem~\textup{\ref{qc duality}} factors as 
\begin{align*}
\R\fst\R\>\sHom_\sX\<(F,\< f^\flat G\>)\<\xto{\!\eqref{nu}}\<\R\>\sHom_Y\<(\R\fst F, \R\fst f^\flat G\>)
\<\xto[\textup{\ref{represent}}]{\!\via\> t^{}_{\<G}\>}\<\R\>\sHom_Y(\R\fst F,G\>).
\end{align*}
\end{subprop}
\begin{proof} 
As in Lemma~\ref{fst fflat}, $\pt G\in\Dqc(\oY)$.
The Proposition amounts to commutativity of the border of the diagram
\begin{equation*}\label{qcdualdiag}
\def\1{$\R\fst\R\>\sHom_\sX\<(F, \bar{f\:\<}^{\!\<*}\!\pt G\>)$}
\def\2{$\phi_*\R\>\sHom_{\>\oY}(\R\bar\fst F, \R\bar\fst\bar{f\:\<}^{\!\<*}\!\pt G\>)$}
\def\3{$\phi_*\R\>\sHom_{\>\oY}(\R\bar\fst F, \pt G\>)$}
\def\4{$\R\>\sHom_{Y}(\phi_*\R\bar\fst F, G\>)$}
\def\5{$\R\>\sHom_Y(\R\fst F, \R\fst\bar{f\:\<}^{\!\<*}\!\pt G\>)$}
\def\6{$\R\>\sHom_Y(\phi_*\R\bar\fst F, \phi_*\R\bar\fst\bar{f\:\<}^{\!\<*}\!\pt G\>)$}
\def\7{$\R\>\sHom_Y(\phi_*\R\bar\fst F, \phi_*\pt G\>)$}
\def\8{$\R\>\sHom_Y(\R\fst F, G\>)$}
 \bpic[xscale=6, yscale=1.45]

   \node(11) at (1,-1){\1};   
   \node(12) at (2,-1){\5}; 
   
   \node(21) at (1,-2){\2};  
   \node(22) at (2,-2){\6};
  
   \node(31) at (1,-3){\3};  
   \node(32) at (2,-3){\7};
   
   \node(44) at (2,-4){\4};  
   \node(42) at (1,-4){\8};
   
    \draw[->] (11)--(12) node[above, midway, scale=.75] {\textup{\eqref{nu}}};  
    
    \draw[->] (21)--(22) node[above, midway, scale=.75] {\textup{\eqref{nu}}}; 
    
    \draw[->] (31)--(32) node[above, midway, scale=.75] {\textup{\eqref{nu}}}; 
    
    \draw[double distance=2pt] (42)--(44) ;
    
    \draw[->] (11)--(21)  node[right=1pt, midway, scale=.75] {$\simeq$}
                                   node[left=1pt, midway, scale=.75] {\textup{\ref{tensor}}};
    \draw[->] (21)--(31) node[right=1pt, midway, scale=.75] {$\simeq$}
                                   node[left=1pt, midway, scale=.75] {\textup{\ref{^* equivalence}}}; 
    \draw[->] (31)--(42) node[right=1pt, midway, scale=.75] {$\simeq$}
                                  node[left=1pt, midway, scale=.75] {\textup{\ref{duality for phi}}};
   
    \draw[double distance=2pt] (12)--(22) ;
    \draw[->] (22)--(32)  node[left=1pt, midway, scale=.75] {$\simeq$}
                                    node[right=1pt, midway, scale=.75] {\textup{\ref{^* equivalence}}}; 
                                    
    \draw[->] (32)--(44) node[right=1, midway, scale=.75] {\textup{\ref{counit}}} ;
    
   \node at (1.48,-1.43) [scale=.9]{\circled1} ;
   \node at (1.48,-2.43) [scale=.9]{\circled2} ;
   \node at (1.48,-3.53) [scale=.9]{\circled3} ;
  \epic
\end{equation*}

Commutativity of \circled1 results from that of the second diagram in \cite[3.7.1.1]{li}---where $``(gf)"$ 
should be $``(f\<g)"$---with $(f,g,F')\set(\phi,\bar f,f^\flat\< G)$; that of \circled2 is clear; and that
of \circled3 is given by \ref{duality map} with $E\set\R\bar\fst F$.
\end{proof}
\end{cosa}

For a map \mbox{$f\colon X\to Y$} of qcqs schemes, and $G\in\D(Y),$
the \emph{abstract duality map} $\delta=\delta(f,G\>)$ is the natural composite 
(with $\tau^{}_{\<G}\set\tau^{}_{\<\!f}(G)$)\va{-3}
\begin{align*}
\R\fst f^{\<\times}\<G=\R \fst\R\>\sHom^{}_\sX(\OX,f^{\<\times}\<G\>)\xto{\!\eqref{nu}}
&\;\R\>\sHom_Y(\R\fst\OX,\R\fst f^{\<\times}\<G\>\>)\\[-2pt]
\xto{\<\via\> \tau^{}_{\<G}}&\;\R\>\sHom_Y(\R\fst\OX,G\>).
\end{align*}
\emph{If\/ $\>\R\>\sHom_Y(\R\fst\OX,G\>)\in\Dqc(Y)$ then\/ $\delta$ is an isomorphism.}
To see this via Yoneda's Lemma, one checks (cf.~\cite[p.\,404]{V}) that for $E\in\Dqc(Y)$, the map\looseness=-1
\[
\Hom_{\D(Y)}(E, \R\fst f^{\<\times}\<G\>)\to\Hom_{\D(Y)}(E,\R\>\sHom_Y(\R\fst\OX,G\>))
\]
induced by $\delta$
is naturally isomorphic to the map 
\[
\Hom_{\D(Y)}(\R\fst\LL f^*E,G\>)\to \Hom_{\D(Y)}(E\Otimes{Y}\R\fst\OX,G\>)
\]
induced by the projection \emph{isomorphism}
\[
E\Otimes{Y}\R\fst\OX\iso\R\fst(\LL f^*\<\<E\Otimes{\sX}\OX) = \R\fst\LL f^*\<\<E.
\]

\begin{subprop}\label{via duality} 
Let\/ $f\colon X\to Y$ be an affine map of qcqs schemes, let\/ $G\in\D(Y)$ be
such that\/ $\R\>\sHom_Y(\fst\OX, G\>)\in\Dqc(Y),$ and let\/ $\delta=\delta(f,G\>)$ be the above duality 
isomorphism. 
The isomorphism\/ $\xi^{}_{\<f}$ of Corollary\/ \textup{\ref{flat=times}}\va{-2.7} 
is the unique\/ $\D(X)$-map 
\(
\xi\colon f^\flat\<G\to f^{\<\times} \<G
\)
such that the composite isomorphism\/ 
\[
\bar t_G\colon\R\fst f^\flat G
= \phi_*\R\bar\fst \bar{f\:\<}^{\!\<*}\<\pt G
\underset{\textup{\ref{^* equivalence}}}
\iso \phi_*\pt G
\,\underset{\textup{\eqref{phiRHompsi}}}
\iso\,\R\>\sHom_Y(\fst\OX,G\>).
\]
\vskip-3pt
\noindent in\/ \textup{\ref{fst fflat}} factors  as\va{-3}
\[
\R\fst f^\flat G \xto{\!\R\fst\xi\>}\R\fst f^{\<\times}\<G\xto{\ \delta\ }\R\>\sHom_Y(\fst\OX,G\>)\>.
\]
\end{subprop}

\begin{proof} If  $\>\>\bar t_G=\delta\smallcirc\R\fst\xi$ 
then the following natural diagram commutes, i.e., $\tau^{}_{\<G}\smallcirc\R\fst\xi\>$
is the map  $t^{}_{\<G}$ in \ref{represent}, and so by~\ref{flat=times}, $\xi$ is 
the isomorphism~$\xi^{}_{\<f}\>$.\va{-3} 
\[\mkern-3mu
\def\1{$\R\fst f^\flat G$}
\def\2{$\R\fst f^{\<\times}\<G$}
\def\3{$\phi_*\R\bar \fst\bar{f\;\<}^{\!\<*}\!\pt G$}
\def\4{$\phi_*\pt G$}
\def\5{$\R\>\sHom_Y(\fst\OX,G\>)$}
\def\6{$\R\>\sHom_Y(\OY\<,\R\fst f^{\<\times}\<G\>)$}
\def\7{$\R\>\sHom_Y(\OY\<,G\>)= G$}
\def\8{$\R\fst\sHom^{}_\sX(\OX,f^{\<\times}\<G\>)$}
\def\9{$\qquad\R\>\sHom_Y(\fst\OX,\R\fst f^{\<\times}\<G\>)$}
 \bpic[xscale=4.85, yscale=1.3]

   \node(11) at (1,-1){\1};   
   \node(12) at (2,-1){\2}; 
       
   \node(22) at (2,-2){\8};
   \node(23) at (3,-1){\6};
   
   \node(31) at (1,-3){\3};
   \node(32) at (2,-3){\9};
   
   \node(41) at (1,-4){\4};  
   \node(42) at (2,-4){\5};
   \node(43) at (3,-4){\7};

    \draw[->] (11)--(12) node[above=.5pt, midway, scale=.75] {$\R\fst\xi$} ;   
    \draw[double distance=2pt] (12)--(23) ;

    \draw[->] (41)--(42)  node[above=1pt, midway, scale=.75]{$\Iso$}
                                 node[below=1pt,midway,scale=.75]{\eqref{phiRHompsi}} ;

    \draw[->] (42)--(43);
      
  \draw[double distance=2pt] (11)--(31) ;
  \draw[->] (31)--(41) node[left=1pt, midway, scale=.75]{$\simeq$}
                                 node[right=1pt,midway,scale=.75]{\textup{\ref{^* equivalence}}} ;
  
  \draw[double distance=2pt] (12)--(22) ;   
  \draw[->] (22)--(32) node[left, midway, scale=.75] {\eqref{nu}} ;  
  \draw[->] (32)--(42) node[left=.5pt, midway, scale=.75] {$\via\tau^{}_{\<G}$} ;                            
      
  \draw[->] (23)--(43) node[right=1pt, midway, scale=.75] {$\via\tau^{}_{\<G}$};
  
  \draw[->] (32)--(23) ; 
  
 \node at (2.4,-1.35)[scale=.75]{\textup{commutes, by}} ;
 \node at (2.4,-1.57)[scale=.75]{\textup{\cite[3.5.6(e)]{li}}} ;
 \node at (1.485,-2.5)[scale=.75]{$\bar t_G=\delta\smallcirc\R\fst\xi$} ;
  \epic
  \]
\vskip-3pt
It remains to be shown that $\>\>\bar t_G=\delta\smallcirc\R\fst\xi^{}_{\<f}$, that is, the outer border of the following natural diagram commutes.\va3
\begin{small}
\[\mkern-3mu
\def\1{$\R\fst f^\flat G$}
\def\2{$\R\fst f^{\<\times}\<G$}
\def\3{$\phi_*\R\bar\fst\bar{f\;\<}^{\!\<*}\!\pt G$}
\def\4{$\R\fst\R\>\sHom^{}_\sX(\OX,f^\flat G\>)$}
\def\5{$\R\fst\R\>\sHom^{}_\sX(\OX,f^{\<\times} G\>)$}
\def\6{$\R\>\sHom_Y(\fst\OX,\R\fst f^\flat\<G\>)$}
\def\7{$\R\>\sHom_Y(\fst\OX,\R\fst f^{\<\times}\<G\>)$}
\def\8{$\phi_*\pt G$}
\def\9{$\R\>\sHom_Y(\fst\OX,G\>)$}
\def\ten{$\phi_*\R\>\sHom_{\>\oY}(\bar\fst\OX,\R\bar\fst\bar{f\;\<}^{\!\<*}\!\pt G\>)$}
\def\lvn{$\phi_*\R\>\sHom_{\>\oY}(\bar\fst\OX,\pt G\>)$}
 \bpic[xscale=4.87, yscale=2]

   \node(11) at (1,-1){\1};   
   \node(13) at (3,-1){\2}; 
   
   \node(21) at (1,-1.75){\3};    
   \node(22) at (2.05,-1.75){\4};
   \node(23) at (3,-1.75){\5};
   
   \node(251) at (1.525,-2.5){\ten};
   
   \node(32) at (2.05,-3){\6};
   \node(33) at (3,-3){\7};
   
   \node(351) at (1.525,-3.5){\lvn};

   \node(42) at (1,-4.25){\8};  
   \node(43) at (3,-4.25){\9};   

    \draw[->] (11)--(13) node[above, midway, scale=.75] {$\R\fst\xi^{}_{\<f}$} ;   
  
    \draw[->] (22)--(23) node[above, midway, scale=.75] {$\via\>\xi^{}_{\<f}$} ;
    
    \draw[->] (32)--(33) node[above, midway, scale=.75] {$\<\via\>\xi^{}_{\<f}\>\>$} ;

    \draw[->] (42)--(43) node[above=1pt, midway, scale=.75]{$\Iso$}
                                   node[below=1pt,midway,scale=.75]{\eqref{phiRHompsi}} ;
      
   \draw[double distance=2pt] (11)--(21) ;  
   \draw[->] (21)--(42) node[left=1pt, midway, scale=.75]{$\simeq$}
                                 node[right=.5pt,midway,scale=.75]{\textup{\ref{^* equivalence}}} ; 
                                 
   \draw[->] (251)--(351) node[right=1pt, midway, scale=.75]{$\simeq$}
                                 node[left=.5pt,midway,scale=.75]{\textup{\ref{^* equivalence}}} ;
   
   \draw[->] (22)--(32) node[right=.5pt, midway, scale=.75] {\eqref{nu}} ;  
   
   \draw[double distance=2pt] (13)--(23) ;
   \draw[->] (23)--(33) node[right=.5pt, midway, scale=.75] {\eqref{nu}} ;    
   \draw[->] (33)--(43) node[right=1pt, midway, scale=.75] {$\via\tau^{}_{\<\<G}$};
  
   \draw[double distance=2pt] (11)--(22) ;
   \draw[->] (22)--(251) node[above=-2.6pt, midway, scale=.75] {\eqref{nu}\kern35pt} ;
   \draw[->] (32)--(43) node[left=1pt, midway, scale=.75] {$\via\> t_G\mkern13mu$} ;                            
   \draw[->] (351)--(43) node[below=-2,midway,scale=.75]{\ref{duality for phi}\kern45pt} ;
   \draw[->] (21)--(251) ;
   \draw[->] (42)--(351) ;
   
  \node at (1.45,-1.77)[scale=.75]{\circled1} ;
  \node at (1.67,-3.92)[scale=.75]{\circled2} ;
  \node at (1.91,-3.28)[scale=.75]{\circled3} ;
  \node at (2.75,-3.5)[scale=.75]{\circled4} ;
  
 \epic
\]
\end{small}
\vskip-12pt 
 
Commutativity of the unlabeled subdiagrams is obvious.  Commutativity of \circled1 is a simple consequence of that of the first diagram in  \cite[3.5.6(e)]{li}.
That of \circled2 is what is asserted immediately after \eqref{phiRHompsi}, with $F\set\bar\fst\OX$.  That  of~\circled3 is given by Proposition~\ref{qc duality2},  and of \circled4 by Corollary~\ref{flat=times}. 
\end{proof}

\begin{cosa}\label{pf adj} (Pseudofunctoriality)
One verifies formally that over the category~$\mathbf F$  of  pseudo\kf-coherent (resp.~perfect) finite scheme\kf-maps\va{.75} there is a unique 
$\Dqcpl$(resp.~$\Dqc$)-valued contravariant pseudofunctor $(-)^\flat$\va1   such that\va{-4} for any  map $f\in\mathbf F$, Corollary~\ref{right adjoint} holds,\va1 and 
for any $W\xto{\,g\,}X\xto{\,f\,}Y$ in $\mathbf F$,\va{.5} the associated isomorphism $g^\flat\<\< f^\flat\iso(fg)^\flat$ is  the natural\va{3} composite (see \S\ref{counit unit})
\begin{equation}\label{pf flat}
g^\flat \<\<f^{\flat^{\mathstrut}} \to  (fg)^\flat\R(fg)_{\<*}\>g^\flat \<\<f^\flat\iso 
(fg)^\flat\R\fst\R g^{}_*\> g^\flat \<\<f^\flat
\to (fg)^\flat\R f_{\<*} f^\flat \to (fg)^\flat,
\end{equation}
\vskip2pt\noindent
i.e., it is right-conjugate to the natural isomorphism $\R\fst\R g^{}_* \osi \R(fg)^{}_*$, 
see \cite[3.3.5]{li} with $(\fst, g_{\<*}, f^*\<\<,\>g^*)$ replaced by 
$(g^\flat\<\< f^\flat,(fg)^\flat\<,\R\fst \R g^{}_*,\R(fg)^{}_*)$, and cf.~\cite[3.6.8.1]{li}.
 
 Restricting to qcqs schemes, one has the same statement with $(-)^{\<\times}$ in place of $(-)^\flat$. 
Consequently:

\begin{subprop}\label{pseudofunc} 
Over the category of\va1 pseudo-coherent\/ $(\<$resp.~perfect$\>)$ finite maps
of qcqs schemes, the\va{-1} map\/ $\xi^{}_{\<f}$ in~\textup{\ref{flat=times}} is the\/ $f$-component of an \emph{isomorphism\- of pseudofunctors} $(-)^\flat\iso(-)^{\<\times}$.\hfill{$\square$}
\end{subprop}

\begin{subrem} The $\mathbf F$-maps $f$, $g$ and $fg$ entail maps of sheaves of rings 
\mbox{$\psi\colon\OY\to \fst\OX$,} $\xi\colon\OX\to g_*\OW$ and $\zeta\set (\fst\xi)\smallcirc \psi$,
and ringed-space maps 
$\bar f\colon X\to(Y,\fst\OX)$,
 $\bar g\colon W\to (X,g_*\OW)$, $\ov{fg}\colon W\to (Y, \fst g_*\OW)$ (see \S\ref{2.1}); and \eqref{pf flat} is a functorial isomorphism
\[
\bar g^*\R\>\sHom_\xi(g_*\OW, \brf\R\>\sHom_\psi(\fst\OX,-))\iso 
\ov{fg}^{\>\>*}\R\>\sHom_{\>\zeta}(\fst g_*\OW, -).
\]

The description of such an isomorphism in \cite[p.\,167\kf]{RD} needs noetherian hypotheses not assumed here.
In any case, locally, where one deals, in essence, with a composition $R\to S\to T$ of ring homomorphisms,
one realizes \eqref{pf flat} as the sheafification of the natural $\D(T)$-isomorphism
\[
\R\<\<\Hom_S(T, \R\<\<\Hom_R(S,\SG))\iso \R\<\<\Hom_R(T,\SG)\qquad(\SG\in\D(R)),
\]
see the paragraph preceding Proposition~\ref{pfush}.
\end{subrem}

\smallskip
The duality isomorphism in Corollary~\ref{qp duality} has the following transitivity property, making it compatible with \eqref{pf flat}. \va2

\begin{subprop}\label{transdual} Let\/ $W\xto{\,g\,}X\xto{\,f\,}Y$ in\/ $\mathbf F$ be as above.
For\/ \mbox{$F\in\D(W)$} and\/ $G\in\D(Y),$ the following natural diagram commutes.
 \[
\def\1{$\R(fg)_*\R\>\sHom_W(F,\> g^\flat \<\<f^\flat G)$}
\def\3{$\R\fst\R g_*\R\>\sHom_W(F,\> g^\flat \<\<f^\flat G)$}
\def\4{$\R(fg)_*\R\>\sHom_W(F, (fg)^\flat G)$}
\def\5{$\R\>\sHom_Y(\R(fg)_*F, G)$}
\def\7{$\R\>\sHom_Y(\R\fst\R g_*F, G)$}
\def\8{$\R\fst\R\>\sHom^{}_\sX(\R g_*F,\> f^\flat G)$}
 \bpic[xscale=3, yscale=1.5]

   \node(11) at (1,-1){\1};   
   \node(13) at (3,-1){\4}; 
   
   \node(21) at (1,-2){\3};  
   \node(23) at (3,-2){\5};
   
   \node(31) at (1,-3){\8};  
   \node(33) at (3,-3){\7};

    \draw[->] (11)--(13) node[above=1pt,midway,scale=.75]{\eqref{pf flat}} ;   

    \draw[->] (31)--(33) node[below=1pt, midway, scale=.75]{\textup{\ref{qp duality}}} ;

   \draw[->] (11)--(21) node[left=1pt, midway, scale=.75]{$\simeq$} ;  
   \draw[->] (21)--(31) node[left=1pt, midway, scale=.75]{\textup{\ref{qp duality}}} ; 
   
   \draw[->] (23)--(33) node[right=1pt, midway, scale=.75]{$\simeq$} ;
   \draw[->] (13)--(23) node[right=1pt, midway, scale=.75]{\textup{\ref{qp duality}}} ;  
  
 \epic
\]
\end{subprop}

\begin{proof} Using Proposition~\ref{qc duality2}, and with $\CH\set\sHom$,
expand the diagram naturally as follows:

\begin{small}
\[
\def\1{$\R(fg)_{\<*}\>\R\CH_W(F,\>\> g^\flat\<\< f^\flat G)$}
\def\2{$\R\CH_Y(\R(fg)_{\<*}\>F,\>\> \R(fg)_{\<*}\>g^\flat\<\< f^\flat G)$}
\def\3{$\R\fst\R g_*\R\CH_W(F,\>\> g^\flat\<\< f^\flat G)$}
\def\4{$\R(fg)_{\<*}\>\R\CH_W(F,\> (fg)^\flat G)$}
\def\5{$\R\CH_Y(\R(fg)_{\<*}\>F,\> G)$}
\def\6{$\R\CH_Y(\R(fg)_{\<*}\>F,\> \R(fg)_{\<*}\>(fg)^\flat G)$}
\def\7{$\R\CH_Y(\R\fst\R g_*F,\> G)$}
\def\8{$\R\fst\R\CH_X(\R g_*F,\>\> \R g_*g^\flat\<\< f^\flat G)$}
\def\9{$\R\CH_Y(\R\fst\R g_*F,\>\>\R\fst\R g_*g^\flat\<\< f^\flat G)$}
\def\ten{$\R\fst\R\CH_X(\R g_*F,\>\> f^\flat G)$}
\def\lvn{$\R\CH_Y(\R\fst\R g_*F,\>\>\R\fst f^\flat G)$}
 \bpic[xscale=4, yscale=1.2]

   \node(11) at (1,-1){\1};   
   \node(13) at (2.98,-1){\4}; 
   
   \node(21) at (1,-2.5){\3};  
   \node(22) at (1.95,-1.75){\2};  
   \node(23) at (2.98,-2.5){\6};

   \node(31) at (1,-4){\8};  
   \node(32) at (1.95,-3.25){\9};  
   \node(33) at (2.98,-4){\5};
   
   \node(41) at (1,-5.5){\ten}; 
   \node(42) at (1.95,-4.75){\lvn};   
   \node(43) at (2.98,-5.5){\7};
   
    \draw[->] (11)--(13) node[above=1pt,midway,scale=.75]{\eqref{pf flat}} ;   

    \draw[->] (23)--(33)  ;
   
    \draw[->] (42)--(43)  ;

   \draw[->] (11)--(21) node[left=1pt, midway, scale=.75]{$\simeq$} ;  
   \draw[->] (21)--(31) node[left=1pt, midway, scale=.75]{\eqref{nu}} ; 
   \draw[->] (31)--(41) ;
   
   \draw[->] (22)--(32)  ;
   \draw[->] (32)--(42)  ;
                                
   \draw[->] (13)--(23) node[right=1pt, midway, scale=.75]{\eqref{nu}} ;  
   \draw[->] (33)--(43) node[right=1pt, midway, scale=.75]{$\simeq$} ;
  
   \draw[->] (11)--(22)  node[above=-1.5pt, midway, scale=.7]{\kern35pt\eqref{nu}} ;
   \draw[->] (22)--(23)  node[below=-2.9pt, midway, scale=.7]{\eqref{pf flat}\kern50pt};
   \draw[->] (31)--(32)  node[below=-2.3pt, midway, scale=.7]{\kern38pt\eqref{nu}} ;
   \draw[->] (41)--(42)  node[below=-2.3pt, midway, scale=.7]{\kern38pt\eqref{nu}} ;
   \draw[->] (41)--(43)  ;
   
   \node at (1.69,-2.5)[scale=.9]{\circled1} ;   
   \node at (2.46,-3.7)[scale=.9]{\circled2} ;
   
 \epic
 \]
\end{small}
\vskip-15pt
Commutativity of \circled1 is given by the second diagram in \cite[3.7.1.1]{li}. 

That of \circled2 follows from \eqref{pf flat} being (as stated above)
right-conjugate to the natural
isomorphism $\R\fst\R g_*\iso \R(fg)_*$, see e.g., the second diagram in \cite[3.3.7(a)]{li}, with
$(\fst, \>g_*, f^*\<\<, \>g^*)$ replaced by $(g^\flat\<\< f^\flat\<, (fg)^\flat\<, \>\R\fst \R g_*, \R(fg)_*)$.

Commutativity of the unlabeled subdiagrams is obvious.
\end{proof}

\end{cosa}

\begin{cosa}\label{bc} 
Given Corollary~\ref{qp duality}, one can translate many standard results about the pseudofunctor 
$(-)^{\<\times}$ (see e.g., \cite[\S4.7\kf]{li}) into corresponding results about the pseudofunctor
$(-)^\flat$ of~\ref{pf adj}, and ask for concrete interpretations.
In this section, an elaboration of \cite[p.\,167, 6.3]{RD}, tor-independent base change for $(-)^\flat$  is treated, both abstractly (Theorem~\ref{indt base change}) and concretely (Proposition~\ref{concbc}). 
In the following two sections, which elaborate \cite[p.\,174, 6.9]{RD}, further illustration is provided by an explication of the interaction of~$(-)^\flat$ with derived tensor (Proposition~\ref{flat and tensor}) and with derived hom (Proposition~\ref{flat and hom}). 
Later, section~\ref{affex} deals with the case of affine schemes, where the
foregoing examples can be described in equivalent commutative\kf-algebra terms.

Note, in perusing such examples with regard to a pseudo\kf-coherent finite map
$f=\phi\bar f\colon X\to Y\<$, that concretely representing a $\Dqc(X)$-map 
$\xi\colon F\to G$ is more or less the same (via~\ref{^* equivalence}) as concretely representing the 
$\Dqc(\fst\OX)$-map $\R\bar \fst\xi\>$, and that the latter involves more than doing the same for 
the $\Dqc(Y)$-map $\R\fst\xi\>$---because the natural map
\[
\Hom_{\Dqc\<(\fst\<\OX\<)}(\R\fst F,\R\fst G)\to\Hom_{\Dqc(Y)}(\R\fst F,\R\fst G)
\] 
isn't always injective.%
\footnote{\kf Let $k$ be a field and $R$ a finite\kf-dimensional local $k$-algebra with 
maximal ideal $m\ne0$ such that the natural map $k\to R/m$ is an isomorphism. The natural composite
\[
0\ne\Hom_k(m/m^2,k)\iso \Hom_R(m,k)\lto \textup{Ext}_R^1(k,k)
\]
is an isomorphism, so that the natural map
$\textup{Ext}^1_R(k,k)\to \textup{Ext}^1_k(k,k)=0$ is not injective, i.e., the natural map 
$\Hom_{\D(R)}(k,k[1])\to\Hom_{\D(k)}(k,k[1])$ is not injective.}
\medbreak

\begin{subcosa}\label{Indt square}

To any oriented commutative square of scheme\kf-maps
\begin{equation}\label{indt square}
\def\1{$X'$}
\def\2{$X$}
\def\3{$Y'$}
\def\4{$Y$}
\CD
 \bpic[xscale=.9, yscale=.7]

   \node(11) at (1,-1){\1};   
   \node(13) at (3,-1){\2}; 
     
   \node(31) at (1,-3){\3};
   \node(33) at (3,-3){\4};
   
    \draw[->] (11)--(13) node[above=1pt, midway, scale=.75] {$v$} ;  
      
   \draw[->] (31)--(33) node[below=1pt, midway, scale=.75] {$u$} ;
    
    \draw[->] (11)--(31) node[left=1pt, midway, scale=.75] {$g$} ;
   
    \draw[->] (13)--(33) node[right=1pt, midway, scale=.75] {$f$};
  \node at (2.05,-2)[scale=.85] {$\sigma$} ;
 \epic 
\endCD
\end{equation}
one associates the map
\begin{equation}\label{theta}
\theta_{\<\sigma}(E)\colon \LL u^*\R\fst E\to \R g_*\LL v^*\<\<E\qquad(E\in\Dqc(X))
\end{equation}                                                                                                                                                                                                                                                                                                                                                                                                                                                                                                                                                                                                                                                                                                                                                                                                                                                                    adjoint to the natural composite map
\[
\R\fst E\to\R\fst\R v_* \LL  v^*\<\<E\iso\R u_*\R g_*\LL v^*\<\<E.
\]
For a concrete local description of $\theta_{\<\sigma}$, see the end of \S\ref{concrete version beta} below.

As in \cite[\S3.10]{li}, the square $\sigma$ is called \emph{independent} if for all $E\in\Dqc(X)$, $\theta_{\<\sigma}(E)$
 is an isomorphism.
 
  If $\sigma$ is independent,  $f$ and $g$ finite and $f$ pseudo\kf-coherent (resp.~perfect), then 
 $g_*\CO_{X'}\cong g_*\LL v^*\OX\cong \LL u^*\!\fst\OX$ is pseudo\kf-coherent  (resp.~perfect), 
 i.e.,  $g$ is pseudo\kf-coherent (resp.~perfect).

If $\sigma$ is a fiber square (i.e., the associated map $X'\to X\times_Y Y'$ is an isomorphism), then 
independence of $\sigma$ is equivalent to \emph{tor-independence,} i.e.,  for all $y'\in Y'$ and $x\in X$ such that $y\set u(y') = f(x)$, 
\[
\tor_i^{\CO_{Y\!,\>y}}(\CO_{Y'\!,\>\>y'}\>,\>\CO_{\<X\<\<,\>x})=0 \quad\textup{for all }i>0.
\] 
For the proof,  one reduces as in \cite[3.10.3.2 and 3.10.3.3]{li} to where $Y\<$, $X\<$, $Y'$ and $X'$ 
are all affine, in which case \cite[3.10.3.1]{li} applies.%
\footnote{Misprint: references to (3.10) in \emph{loc.\,cit.} should be to (3.10.2.3).} 
 
Hence, if $\sigma$ is a fiber square then its being independent does not depend on its orientation;  and 
such a $\sigma$ is independent if  either $\<f$ or~$u$ is flat.
\end{subcosa}
\end{cosa}

The following Theorem~\ref{indt base change} is essentially contained in \cite[4.4.1]{li},
in whose proof
the assumption (at the beginning of Section~4.4) that all schemes are qcqs is needed only 
to ensure that $\R\fst$~has a right adjoint---which in the present circumstances is  known to be so (Corollary~\ref{qp duality}) 
without the said assumption. 
The proof here, though related to that of \emph{loc.\,cit.}, is\va1 more direct.\looseness=-1

In Theorem~\ref{indt base change} and Remark~\ref{beta-times}(a),  these  abbreviations are used:
\[
(-)_*\set\R(-)_*,\qquad (-)^*\set\LL (-)^*,
\qquad\sHom\set\R\>\sHom.
\]

\begin{subthm}[Tor-independent base change]\label{indt base change} 
Let\/~$\sigma$  be, as above, an independent fiber square, in which
$f$---hence $g$---is finite\va{.5} and pseudo\kf-coherent $($resp.~perfect\/$)$ and  $u$ has finite tor-dimension $($resp.~$u$ is arbitrary\/$),$  and let\/ $G\in\Dqcpl(Y)\ 
(resp.~\Dqc(Y)).$ With\/ $(-)^\flat$ and\/~$t^{}_{\<G}$ as in Prop.~\textup{\ref{represent},} \va{-.5}  the~adjoint of the composite map
\smash{\( 
g_*v^*\<\< f^\flat G\xto[\>\>\lift1.4,\theta_{\<\sigma}^{-\<1},\>]{} u^*\!\fst f^\flat G
 \xto[\lift1.3,u^*t^{}_{\<G},\!]{}u^*\mkern-.5mu G
\)}
is an \emph{isomorphism}\va6
 \[
 \beta_\sigma(G)\colon v^*\mkern-2.5mu f^\flat G\iso g^\flat u^*\mkern-.5mu G.
\] 
\end{subthm}

\begin{proof}
Since $g$ is affine, \cite[3.10.2.2]{li} makes it enough to show that $g_*\beta_\sigma(G)$ is an isomorphism---i.e., that
the top row of the following natural diagram,  with $\CH\set\sHom$,  composes to an isomorphism. 
\begin{small}
\[\mkern-4mu
\def\1{$g_*v^*\!f^\flat G$}
\def\2{$g_*g^\flat\< g_*v^*\!f^\flat G$}
\def\3{$g_*g^\flat u^{\<*}\<\<\fst f^\flat G$}
\def\4{$g_*g^\flat u^{\<*}\<G$}
\def\5{$g_*\CH_{\sX'}(\CO_{\sX'}, v^*\!f^\flat G\>)$}
\def\6{$\CH_{Y'}(g_*\CO_{\sX'}, g_*v^*\!f^\flat G\>)$}
\def\7{$u^{\<*}\<\<\fst f^\flat G$}
\def\8{$\CH_{Y'}(g_*\CO_{\sX'}, u^{\<*}\<\<\fst f^\flat G\>)$}
\def\9{$\CH_{Y'}(g_*\CO_{\sX'}, u^{\<*}G\>)$}
\def\ten{$u^{\<*}\CH_Y(\fst\OX,G\>)$}
\def\lvn{$\CH_{Y'}(u^{\<*}\<\<\fst\OX, u^{\<*}G\>)\ $}
 \bpic[xscale=3.8, yscale=1.1]

   \node(11) at (1.175,-1){\1};   
   \node(12) at (2.25,-1){\2}; 
   \node(13) at (3.05,-1){\3};   
   \node(14) at (3.8,-1){\4}; 
   
   \node(21) at (1.7,-2){\5};  
   
   \node(31) at (1.9,-3.2){\6};
   
   \node(41) at (1.175,-2.5){\7};   
   \node(42) at (3.05,-3.2){\8}; 
   \node(43) at (3.8,-2.5){\9}; 
 
   \node(51) at (1.175,-4){\ten};  
   \node(52) at (3.8,-4){\lvn};
   
    \draw[->] (11)--(12) ;  
    \draw[->] (12)--(13) node[below=1pt, midway, scale=.75] {$\via\>\theta_{\<\sigma}^{-\<1}$} ;  
    \draw[->] (13)--(14) ;  
    
    \draw[->] (42)--(43)  ;
   
   \draw[->] (51)--(52)  node[below=1pt, midway, scale=.75] {$\quad\rho(G)$} ;
    
    \draw[->] (11)--(41)  node[left, midway, scale=.75] {$\simeq$}
                                    node[right, midway, scale=.75] {$\theta_{\<\sigma}^{-\<1}$} ;
    \draw[->] (41)--(51) node[left, midway, scale=.75] {$\simeq$} ; 
     
    \draw[->] (21)--(31)  ; 
   
    \draw[->] (13)--(42)  node[right, midway, scale=.75] {$\simeq$};
     
    \draw[->] (14)--(43)  node[right, midway, scale=.75] {$\simeq$} ;
    \draw[<-] (43)--(52)  node[right, midway, scale=.75] {$\simeq$} 
                                    node[left, midway, scale=.75] {$\via\>\theta_{\<\sigma}^{-\<1}\!$} ; 
    
    \draw[double distance=2pt] (11)--(21)  ;     
  
   \draw[->] (12)--(2.1,-2.965)  node[right, midway, scale=.75] {$\mkern2mu\simeq$};
  
   \draw[->] (31)--(42) node[above=1pt, midway, scale=.75] {$\via\>\theta_{\<\sigma}^{-\<1}$} ; 
            
  \node at (2.495, -3.62) [scale=.9] {\circled1};
  \node at (1.8, -1.5) [scale=.9] {\circled2};
  \node at (2.62, -2.07) [scale=.9] {\circled3};
  \node at (3.425, -1.9) [scale=.9] {\circled4};         
  \epic
\]
\end{small}
\vskip-16pt
\indent
\emph{The natural map $\rho(G)$  is an isomorphism}: this assertion, being local,  results, e.g., from \cite[4.6.7]{li}
(in whose proof the untreated case where $E$ is strictly perfect---i.e., a bounded 
complex of\/ finite\kf-rank locally free\/ $\OY$-modules---and $H$ arbitrary is  easily disposed of by induction on the number of $n$ such that $E^n\ne 0$.) Thus it~suffices to show that the diagram commutes.
 
Using Proposition~\ref{via duality} and the second diagram in \cite[3.5.6(e)]{li}, 
one verifies---with a bit of patience---that
commutativity of subdiagram~\circled1 is given by~\cite[Lemma 4.6.4]{li} with $F\set\OX$ and 
$f^!$ replaced by $f^\flat$. 

Commutativity of \circled2 results from that of the natural functorial diagram\va{-1} 
 \[
\def\1{$g_*$}
\def\2{$g_*g^\flat\< g_*$}
\def\3{$\CH_{Y'}(\CO_{Y'}, g_*)$}
\def\5{$g_*\CH_{\sX'}(\CO_{\sX'}, \>\id)$}
\def\6{$\CH_{Y'}(g_*\CO_{\sX'}, g_*)$}
 \bpic[xscale=3, yscale=1.1]

   \node(11) at (1,-1){\1};   
   \node(13) at (3,-1){\2}; 
  
   \node(21) at (2,-2){\3};  
   
   \node(31) at (1,-3){\5};
   \node(33) at (3,-3){\6};
   
    \draw[->] (11)--(13) ;  
      
    \draw[->] (31)--(33) ;
    
    \draw[double distance=2pt] (11)--(31) ;
   
    \draw[->] (13)--(33) node[right=1pt, midway, scale=.75] {$\simeq$} ;
   
    \draw[double distance=2pt] (11)--(21)  ;     
  
    \draw[->] (33)--(21)  ; 
           
  \node at (2.6, -1.55) [scale=.9] {\circled2$_1$};
  \node at (1.55, -2.5) [scale=.9] {\circled2$_2$};
 
  \epic
\]
\vskip-2pt\noindent
Here, by the description in \ref{right adjoint} of the counit map $t$, commutativity of~\circled2$_1$  
means just that the natural composite \mbox{$g_*\to g_*g^\flat \<g_*\xto{\<u^{\<*}\<t\>} g_*$} is\va1 the identity; and
that of \circled2$_2$ is given by that of the first diagram\va1 in \cite[3.5.6(e)]{li}.\looseness=-1

Finally, commutativity of \circled 3 and \circled4 is clear.
\end{proof}

\begin{subrems}\label{beta-times} 
(a) If the schemes in Theorem~\textup{\ref{indt base change}} are qcqs, the following diagram, 
with $\xi$ as in~\textup{\ref{flat=times}}  and\/ $\beta_\sigma^{\<\times}(G)$ as in~\textup{\cite[4.4.3]{li},} 
commutes$\>:$\va{-1}
 \[
\def\1{$v^*\<\<f^{\<\times}\<G$}
\def\2{$g^{\<\times} u^*\mkern-.5mu G$}
\def\3{$v^*\<\<f^\flat G$}
\def\4{$g^\flat u^*\mkern-.5mu G$}
 \bpic[xscale=2.5, yscale=1.35]

   \node(11) at (1,-1){\3};   
   \node(12) at (2,-1){\4}; 
   
   \node(21) at (1,-2){\1};  
   \node(22) at (2,-2){\2};
   
    \draw[->] (11)--(12) node[above, midway, scale=.75] {$\beta_\sigma(G)$} ;  
    
    \draw[->] (21)--(22) node[below, midway, scale=.75] {$\beta^{\<\times}_\sigma(G)$} ;

    \draw[->] (11)--(21)  node[left, midway, scale=.75] {$v^*\<\xi^{}_{\<\<f}$}
                                    node[right, midway, scale=.75] {$\simeq$};
    
    \draw[->] (12)--(22) node[right=1pt, midway, scale=.75] {$\xi^{}_{\<g}$}
                                    node[left, midway, scale=.75] {$\simeq$};
  \epic
\]

\va{-6}
Indeed, $\beta_\sigma^{\<\times}(G)$, resp.~$\beta_\sigma(G)$, is by definition adjoint to\va{-1}
\[ 
g_*v^*\<\< f^{\<\times}\<G\xto[\lift1.4,\theta_{\<\sigma}^{-\<1},]{} u^*\!\fst f^{\<\times}\<G \xto[\lift1,\,u^*\<\tau^{}_{\<G},]{}u^*\mkern-.5mu G\>, \ \ \textup{resp.}\ \ 
g_*v^*\<\< f^\flat\<G\xto[\lift1.4,\theta_{\<\sigma}^{-\<1},]{} u^*\!\fst f^\flat\<G \xto[\lift1,\,u^*\<t^{}_{\<G},]{}u^*\mkern-.5mu G,
\]
from which definitions one readily derives the assertion via~\ref{flat=times}. \va1

(b) The special case of Theorem~\ref{indt base change} where the map $u$ is an open immersion
is equivalent to the composite map in Proposition~\ref{qc duality2} being an isomorphism,
cf.~~\cite[4.3.6]{li}. 

(c) A noteworthy result of Neeman \cite[Lemma 5.19]{Nm23} (using \emph{ibid.,} Convention 5.5) implies that when $u$ is flat and $g$ is perfect,  then $\beta_\sigma(G)$ is an isomorphism for all $G\in\Dqc(X)$. 

\end{subrems}

The rest of this section is devoted to realizing the map~$\beta_\sigma$---or equivalently, the map
$\bar g_*\beta_\sigma$---concretely. (As indicated before, this is rather more difficult than doing the same for
$g_*\beta_\sigma$, which was shown in
the proof of Theorem~\ref{indt base change} to be  
naturally isomorphic to the map~$\rho$, whose explicit description was indicated 
in the footnote in the proof of \ref{tensor}.) Locally, a more explicit such realization, in commutative\kf-algebra terms, is given in Proposition~\ref{concrete bc2}.

\smallskip
Until further notice,
the symbols $(-)_*$, $(-)^*$, $\otimes$ and $\sHom$ will have their ordinary (non-derived) meaning.

\begin{subcosa}
Let $Y$ be a ringed space, $\psi\colon\OY\to\CS$ an $\OY$-algebra, and $\oY$ the ringed space $(Y,\CS)$. The category~$\sA(\oY)$ of $\CS$-modules is naturally isomorphic to the category having as objects the
pairs $(\CN, m_{\CN})$ with $\CN$ an $\OY$-module and \mbox{$m_{\CN}\colon \CS\otimes_Y\CN\to \CN$} an $\OY$-homomorphism satisfying the usual conditions for scalar multiplication, and having the obvious morphisms. 

For example, when  $F$ is an $\CS$-module and $G$ an $\OY$-module, the $\CS$-module $\sHom_\psi(F,G)$ is specified as such a pair in \S\ref{2.3}.

Let $u\colon Y'\to Y$ be a map of ringed spaces. Let $\psi'\colon\CO_{Y'}\to\CS'$ be the $\CO_{Y'}$-algebra $u^*\<\psi$,  $\CS\to u_*u^*\<\CS=u_*\CS'$ the natural map, and 
\[
\bar u\colon \oY'\set(Y'\<\<,\CS')\lto (Y\<,\CS)=\colon\oY
\]
the corresponding map of ringed spaces. By definition, essentially, the direct image $\bar u_*(\CM,m_\CM)$
of an $\CS'$-module is the $\CS$-module $(u_*\CM, m_{\mkern.5mu u_*\CM})$ where $m_{\mkern.5mu u_*\CM}$ is the natural composite
\[
\CS\otimes_Y u_*\CM\lto u_*(u^*\<\CS\otimes_{Y'}\CM)\xto{\<\<u_*\mkern.5mu m_{\<\CM}\>} u_*\CM\>.
\]
The direct image of a map of $\CS'$-modules is, in these terms, specified in the obvious way.

One checks that the functor $\bar u_*$ has the left adjoint ${\bar u}^*$  given objectwise by 
\(
{\bar u}^*(\CN,m_\CN) = (u^*\<\<\CN, m_{\mkern.5mu u^*\<\<\CN})
\)
where $m_{u^*\<\<\CN}$ is the natural composite\va{-3}
\begin{equation}\label{u^*m}
u^*\<\CS\otimes_{Y'} u^*\<\<\CN\iso u^*(\CS\otimes_Y\CN\>\>)\xto{\<\<u^*m_\CN\>}u^*\<\<\CN,\\[-2pt]
\end{equation}
and mapwise in the natural way.\va1

Let $\phi\colon \oY=(Y\<,\CS)\to (Y,\OY)$ be the ringed-space map $(\id_Y,\psi)$, and let
$\phi'\colon \oY'=(Y'\<,\CS')\to (Y',\CO_{Y'})$  be $(\id_{Y'},\psi')$. With the preceding~ $\bar u^*$ it holds that $u^*\<\phi_*=\phi'_{\<*}\bar u^*$.
\end{subcosa}

\pagebreak[3]
\begin{sublem}\label{rho and fstOX}
For any complexes\/ $F\in\sA(\oY)$ and $G\in\sA(Y)$ there is a unique\/ $\CS'$-map\va{-3}
\[
\bar\rho^{}_0=\bar\rho^{}_0(F\<,G\>)\colon \bar u^*\sHom_\psi(F\<,G) \to \sHom_{\psi'}(\bar u^*\<\<F\<,u^*\mkern-.5mu G)\\[-1pt]
\]
such that\/ $\phi'_{\<*}\bar \rho^{}_0$ is the natural map\va{-3}
\[
u^*\<\phi_*\<\sHom_\psi(F,G)=u^*\sHom_Y(\phi_*F\<,G)\xto{\,\rho^{}_0\,} \sHom_{Y'}(u^*\<\phi_*F\<,u^*\mkern-.5mu G).
\]
\end{sublem}

\begin{proof} This follows easily from the standard explicit realization of $\rho^{}_0$ and of the scalar multiplication map in \S\ref{2.3}.  

More formally, the uniqueness of~$\bar\rho^{}_0$ is obvious, and its existence is given by commutativity of the border\va{.6} of the following natural diagram, in which $\CH\set\sHom$ and $E\set\phi_*F\<$, whose left (resp.~right) column composes to scalar multiplication by $u^*\<\CS=\CS'\>$:
\begin{small}
\[\mkern-2mu
\def\1{$u^*\CH_Y(E\<,G)\otimes^{}_{Y'}\<u^*\<\<\CS$}
\def\2{$\CH_{Y'}(u^*\<\<E\<,u^*\<G)\otimes^{}_{Y'}u^*\<\<\CS$}
\def\3{$u^*(\CH_Y(E\<,\>G)\<\otimes^{}_Y\<\CS)$}
\def\4{$u^*\CH_Y(\CS\otimes^{}_Y\<\<E\<,G)\<\otimes^{}_{Y'}\<\<u^*\<\<\CS$}
\def\5{$u^*(\CH_Y(\CS\otimes^{}_Y\<\<E\<,\>G)\otimes^{}_Y\<\CS)$}
\def\6{$\CH_{Y'}(u^*\<\<\CS\otimes^{}_{Y'}\<u^*\<\<E\<,\>u^*\<G)\otimes^{}_{Y'}\<\<u^*\<\<\CS$}
\def\7{$u^*(\CH_Y(\CS,\CH_Y(E\<,G))\otimes^{}_Y\<\CS)$}
\def\8{$\CH_{Y'}(u^*\<\<\CS,\CH_{Y'}(u^*\<\<E\<,u^*\<G))\otimes^{}_{Y'}\<\<u^*\<\<\CS$}
\def\9{$u^*\CH_Y(E\<,G)$}
\def\ten{$\CH_{Y'}(u^*\<\<E\<,u^*\<G)$}
\def\lvn{$u^*\CH_Y(\CS,\CH_Y(E\<,G))\otimes^{}_{Y'}\<\<u^*\<\<\CS$}
\def\twv{$\CH_{Y'}(u^*\CS,u^*\CH_Y(E\<,G))\otimes^{}_{Y'}\<\<u^*\<\<\CS$}
 \bpic[xscale=4.5, yscale=1.45]

   \node(11) at (1,-1){\1};   
   \node(13) at (2.65,-1){\2}; 
   
   \node(21) at (1,-2){\3};  
   \node(22) at (1.92,-2){\4};
   
   \node(31) at (1, -3){\5};  
   \node(33) at (2.65,-3){\6};
   
   \node(41) at (1,-4){\7}; 
   \node(42) at (1.8,-5){\lvn}; 
   \node(43) at (2.65,-4){\8};
   
   \node(51) at (1,-7){\9}; 
   \node(52) at (2.05,-6){\twv};
   \node(53) at (2.65,-7){\ten};

    \draw[->] (11)--(13) node[above=1pt,midway,scale=.75]{$\rho^{}_0\otimes\id$} ;   
   
    \draw[->] (51)--(53) node[below=1pt,midway,scale=.75]{$\rho^{}_0$} ;
       
   \draw[->] (11)--(21) node[left=1pt, midway, scale=.75]{$\simeq$} ;  
   \draw[->] (21)--(31) node[left=1pt, midway, scale=.75]{$\via \>m_{\<E}$} ; 
   \draw[->] (31)--(41) ;
   \draw[->] (41)--(51) ;
   
   \draw[->] (1.84,-2.2)--(1.84,-4.8) ;
   \draw[->] (1.84,-5.2)--(1.84,-5.8) ;
   
   \draw[->] (13)--(33) node[right=1pt, midway, scale=.75]{$\via \>m_{u^{\<*}\<\<E}$} ; 
   \draw[->] (33)--(43) ;
   \draw[->] (43)--(53) ;
   
    \draw[->] (11)--(22) node[right=1pt, midway]{$\lift1.5,\!\via \>m_{\<\<E},$} ; 
    \draw[->] (1.95,-2.2)--(33) ;
    \draw[->] (31)--(22) ;
    \draw[->] (41)--(1.75,-4.8) ;
    \draw[->] (1.94,-4.8)--(43) ;
    \draw[<-] (51)--(52) ;
    \draw[<-] (43)--(52) node[right=1pt, midway]{$\lift-.5,\mkern-7mu\via \rho^{}_0,$} ;;
    
     \node at (2.1, -1.5) [scale=.9] {\circled1};
     \node at (2.25, -3.5) [scale=.9] {\circled2};
     \node at (1.46, -5.5) [scale=.9] {\circled3}; 
 \epic
\]
\end{small}
\vskip-12pt
Commutativity of the unlabeled subdiagrams is clear. 

That subdiagram \circled1 commutes follows easily from the definition of~ $m_{u^{\<*}\<\<E}$ (see \eqref{u^*m}, with $\CN=E$).

Commutativity of \circled3 results from \cite[3.5.6(a)]{li} with $\big(f,A,B\big)$ replaced by 
$\big(u,\CS,\CH_Y(E\<,G)\big)$.

It suffices now to show commutativity of the following natural diagram,  
with $[-,-]\set\CH_Y(-,-)$ and $[-,-]'\set\CH_{Y'}(-,-)$, whose border is
adjoint to \circled2 without ``$\otimes_{Y'}\>\>u^*\<\CS.$" (Note that
$u^*[A,B]\to[u^*A,u^*B]'$ is adjoint to the natural composite map 
$[A,B\>]\lto[A,u_*u^*\mkern-.5mu B\>]\iso u_*[u^*A,u^*\mkern-.5mu B\>]',$ see \cite[(3.5.4.5)\:\emph{ff}\kf]{li}.)

\[
\def\1{$[S\otimes_Y\<E,G\>]$}
\def\2{$[S\otimes_YE,u_*u^*\mkern-.5mu G\>]$}
\def\3{$u_*[u^*(S\otimes_YE),u^*\mkern-.5mu G\>]'$}
\def\4{$[S,[E,u_*u^*\mkern-.5mu G\>]]$}
\def\5{$u_*[u^*S\otimes_{Y'}u^*E,u^*\mkern-.5mu G\>]'$}
\def\6{$[S,u_*[u^*E,u^*\mkern-.5mu G\>]'\>]$}
\def\7{$u_*[u^*S,[u^*E,u^*\mkern-.5mu G\>]'\>]'$}
\def\8{$[S,[E,G\>]]$}
\def\9{$[S,u_*u^*[E,G\>]]$}
\def\ten{$u_*[u^*S,u^*[E,G\>]]'$}
 \bpic[xscale=4.65, yscale=1.5]

   \node(11) at (1,-1){\1};   
   \node(12) at (2,-1){\2}; 
   \node(13) at (3,-1){\3};
   
   \node(22) at (2,-2){\4};  
   \node(23) at (3,-2){\5};
   
   \node(32) at (2, -3){\6};  
   \node(33) at (3,-3){\7};
   
   \node(41) at (1,-4){\8}; 
   \node(42) at (2,-4){\9};
   \node(43) at (3,-4){\ten};
   
   \draw[->] (11)--(12) node[above=1pt,midway,scale=.75]{$$} ;   
   \draw[->] (12)--(13) ;
   
   \draw[->] (32)--(33) ;
   
   \draw[->] (41)--(42) ;   
   \draw[->] (42)--(43) ;
       
   \draw[->] (11)--(41) node[left=1pt, midway, scale=.75]{$$} ;  
   
   \draw[->] (12)--(22) ; 
   \draw[->] (22)--(32) ;
   \draw[<-] (32)--(42) ;
      
   \draw[->] (13)--(23) node[right=1pt, midway, scale=.75]{$$} ; 
   \draw[->] (23)--(33) ;
   \draw[<-] (33)--(43) ;
 
    \draw[->] (41)--(22)  ; 

     \node at (1.45, -2.02) [scale=.9] {\circled4};
     \node at (2.46, -2.02) [scale=.9] {\circled5};
     \node at (1.55, -3.52) [scale=.9] {\circled6}; 
     \node at (2.46, -3.52) [scale=.9] {\circled7};
 \epic
\]
Commutativity of subdiagrams \circled4 and \circled7 is clear.
Commutativity of \circled6 results from the fact, noted above, that  $u^*[E,G\>]\to [u^*E,u^*\mkern-.5mu G\>]'$ is adjoint to $[E,G\>]\lto[E,u_*u^*\mkern-.5mu G\>]\iso u_*[u^*E,u^*\mkern-.5mu G\>]'.$
  Commutativity of \circled5 is given by \cite[Exercise 3.5.6(c)]{li}, with $(f,E,F,G)\set(u,S,E,u^*\mkern-.5mu G)$.
\end{proof}

\begin{subcosa}\label{decompose}
Let there be given an independent square of scheme\kf-maps
\[
\def\1{$X'$}
\def\2{$X$}
\def\3{$Y'$}
\def\4{$Y$}
\CD
 \bpic[xscale=.9, yscale=.8]

   \node(11) at (1,-1){\1};   
   \node(13) at (3,-1){\2}; 
     
   \node(31) at (1,-3){\3};
   \node(33) at (3,-3){\4};
   
    \draw[->] (11)--(13) node[above=1pt, midway, scale=.75] {$v$} ;  
      
   \draw[->] (31)--(33) node[below=1pt, midway, scale=.75] {$u$} ;
    
    \draw[->] (11)--(31) node[left=1pt, midway, scale=.75] {$g$} ;
   
    \draw[->] (13)--(33) node[right=1pt, midway, scale=.75] {$f$};
  \node at (2.05,-2)[scale=.85] {$\sigma$} ;
 \epic 
\endCD
\]
in which the maps $f$ and $g$ are affine. 
This $\sigma$ decomposes as the border of the commutative diagram of ringed-space maps
\[
\def\1{$X'$}
\def\2{$X$}
\def\3{$\oY'\set(Y'\<\<,\>g_*\CO_{\sX'})$}
\def\4{$\oY\set(Y\<,\>\fst\OX)$}
\def\5{$Y'$}
\def\6{$Y$}
 \bpic[xscale=2.5, yscale=1.65]

   \node(11) at (1,-1){\1};   
   \node(13) at (3,-1){\2}; 
   
   \node(21) at (1,-2){\3};  
   \node(23) at (3,-2){\4};
   
   \node(31) at (1,-3){\5};  
   \node(33) at (3,-3){\6};

    \draw[->] (11)--(13) node[above=1pt,midway,scale=.75]{$v$} ;   

    \draw[->] (21)--(23) node[above=1pt,midway,scale=.75]{$\bar u\set(u,\bar\psi)$} ;   

   \draw[->] (31)--(33) node[below=1pt, midway, scale=.75]{$u$} ;

   \draw[->] (11)--(21) node[left=1pt, midway, scale=.75]{$\bar g\set(g,\id)$} ;  
   \draw[->] (21)--(31) node[left=1pt, midway, scale=.75]{$\phi'\set(\id,\psi')$} ; 
   
   \draw[->] (23)--(33) node[right=1pt, midway, scale=.75]{$\phi\set(\id,\psi)$} ;
   \draw[->] (13)--(23) node[right=1pt, midway, scale=.75]{$\bar f\set(f,\id)$} ;  
  \epic
\]
where $\psi\colon\OY\to\fst\OX$ and $\psi'\colon\CO_{Y'}\to g_*\CO_{X'}$ are the maps associated with $f$ and $g$ respectively, and $\bar\psi$ is the natural composite map
\[
\fst\OX\lto \fst v_*\CO_{\sX'}\iso u_*g_*\CO_{\sX'}\>.
\]
The  ringed-space maps 
$\bar f$ and $\bar g$ are flat (see \S\ref{2.1}), so that the functors $\bar f^*$ 
and~$\bar g^*$ are exact, as are the functors $\phi_*$ and $\phi'_{\<*}$.

There results, for any  $E\in\Dqc(X)$  a natural commutative diagram (see \cite[3.7.2(ii)]{li}):\va{-8}
 \begin{equation}\label{transtheta}
\def\1{$\LL u^*\R(\phi\bar f\>)_*E$}
\def\2{$\R(\phi'\bar g)_{\<*}\LL v^*\<\<E$}
\def\3{$\LL u^*\<\phi_*\R\bar \fst E$}          
\def\4{$\phi'_{\<*}\LL \bar u^*\R\bar f_*E$}
\def\5{$\phi'_{\<*}\>\R\bar g_*\LL v^*\<\<E$}
\CD
 \bpic[xscale=3.5, yscale=1.5]

   \node(11) at (1,-1){\1};   
   \node(13) at (3,-1){\2}; 
   
   \node(21) at (1,-2){\3}; 
   \node(22) at (2.02,-2){\4};     
   \node(23) at (3,-2){\5};
   
    \draw[->] (11)--(13) node[below=1pt,midway,scale=.75]{$\theta_{\<\sigma}(E)$} 
                                   node[above,midway]{$\Iso$};
    \draw[->] (21)--(22) node[below=1pt,midway,scale=.75]{$\theta_1(\R\bar\fst E)$} ;   

   \draw[->] (22)--(23) node[below=1pt, midway, scale=.75]{$\phi'_{\<*}\theta_2(E)$} ;

   \draw[->] (11)--(21) node[left=1pt, midway, scale=.75]{$\simeq$} ;  
   \draw[->] (13)--(23) node[right=1pt, midway, scale=.75]{$\simeq$} ; 
  
 \epic
 \endCD
\end{equation}	

The definition \cite[3.7.2(i)(c)]{li} of $\theta_2$ implies that the following natural diagram 
commutes:\va{-8} 
\[
\def\1{$\bar g^*\LL \bar u^*\R\bar\fst E$}
\def\2{$\bar g^*\R\bar g_*\LL v^*\<\<E$}
\def\3{$\LL v^*\<\<\brf\R\bar f_*E$}          
\def\4{$\LL v^*\<\<E$}
 \bpic[xscale=3.5, yscale=1.35]

   \node(11) at (1,-1){\1};   
   \node(12) at (2,-1){\2}; 
   
   \node(21) at (1,-2){\3}; 
   \node(22) at (2.,-2){\4};     
     
    \draw[->] (11)--(12)  node[above=1pt,midway,scale=.75]{$\bar g^*\theta_2(E)$} ;
    \draw[->] (21)--(22) node[above=1pt,midway,scale=.75]{$\Iso$} 
                                    node[below=1pt,midway,scale=.75]{\ref{^* equivalence}} ;      
    
   \draw[->] (11)--(21) node[left=1pt, midway, scale=.75]{$\simeq$} ;  
   \draw[->] (12)--(22) node[left=1pt, midway, scale=.75]{$\simeq$} 
                                  node[right=1pt,midway,scale=.75]{\ref{^* equivalence}} ; 
  
 \epic
\]
So $\bar g^*\theta_2(E)$---whence, by ~\ref{^* equivalence}, 
$\theta_2(E)$---\emph{is an isomorphism,}
whence so is $\theta_1(\R\bar\fst E)$ (see \eqref{transtheta}), i.e.,  by ~\ref{^* equivalence}, \emph{so is\/ $\theta_1(H\>)$ for any} $H\in\Dqc(\oY)$.    
\end{subcosa}

For a ringed-space map $u\colon Y'\to Y\<$, an $\OY$-complex~$E$ is called $u^*\<$-\emph{acyclic} if the canonical map is an isomorphism $\LL u^*\<\<E\iso u^*\<\<E$.
For example, any K-flat $E$ is $u^*\<$-acyclic. 

For a scheme $Y$, a \emph{strictly perfect} $\OY$-complex is a bounded 
complex of\/ finite\kf-rank locally free\/ $\OY$-modules. Note that an $\OY$-complex is perfect~
if  locally---even \emph{globally} when $Y$ is affine---it is the target of a quasi-isomorphism with source a strictly perfect one \cite[p.\,122, 4.8, p.\,175, 2.2.10, p.\,163, 2.0, and p.\,96, 2.2]{Il}.

\begin{sublem}\label{barrho}
In\/ \textup{\ref{decompose},} let\/ $F\in\D(\oY),$ and assume\/ \emph{either} that
$\phi_*F$ is pseudo\kf-coherent,
$G\in\Dqcpl(Y)$ and $u$ has finite tor-dimension  
\emph{or} that $\phi_*F$ is perfect and\/ $G\in\Dqc(Y)$.  \va2

\noindent\textup{(i)} There is a unique bifunctorial\/ $\D(\oY')$-map
\[
\bar\rho=\bar\rho(F,G)\colon \LL\bar u^*\R\sHom_{\psi}(F\<,G)\to \R\>\sHom_{\psi'}(\LL \bar u^*\<\<F\<,
\>\LL u^*\mkern-.5mu G)
\]
such that  if\/ $G$ is\/ $u^*\<$-acyclic, the following natural diagram commutes\/$:$\va{-7} 

\[
\def\1{$\LL\bar u^*\sHom_{\psi}(F\<,G)$}
\def\2{$\bar u^*\sHom_{\psi}(F\<,G)$}
\def\3{$\LL\bar u^*\R\>\sHom_{\psi}(F\<,G)$}
\def\4{$\sHom_{\psi'}(\bar u^*\<\<F\<,u^*\mkern-.5mu G)$}
\def\5{$\R\>\sHom_{\psi'}(\bar u^*\<\<F\<,u^*\mkern-.5mu G)$}
\def\6{$\R\>\sHom_{\psi'}(\LL \bar u^*\<\<F\<,\LL u^*\mkern-.5mu G)$}
\def\7{$\R\>\sHom_{\psi'}(\LL \bar u^*\<\<F\<,u^*\mkern-.5mu G);$}
 \bpic[xscale=8, yscale=1.2]

   \node(11) at (1.525,-1){\1};   
   \node(12) at (2,-1){\2}; 
   
   \node(21) at (1,-1){\3};  
   \node(22) at (2,-2){\4};

   \node(32) at (2,-3){\5};
   
   \node(41) at (1,-4){\6}; 
   \node(42) at (2,-4){\7};

    \draw[->] (11)--(12) node[above=-.5,midway]{$\Iso$} 
                                   node[below=1pt,midway,scale=.75]{$\bar b$} ;   

    \draw[->] (41)--(42)  node[above=-.5,midway]{$\Iso$} 
                                    node[below=1pt,midway,scale=.75]{$\bar e$} ;
    
   \draw[->] (11)--(21) node[above,midway]{$\Iso$} 
                                  node[below=1pt, midway, scale=.75]{$\bar a$} ;
;  
   \draw[->] (21)--(41) node[left=1pt, midway, scale=.75]{$\bar\rho\>(F\<,G\>)$} ;

   \draw[->] (12)--(22) node[left=1pt, midway, scale=.75]{\textup{\ref{rho and fstOX}}} 
                                  node[right=1pt, midway, scale=.75]{$\bar\rho^{}_0(F\<,G\>)$}  ;  
   \draw[->] (22)--(32) node[left=1pt, midway, scale=.75]{$\simeq$} 
                                  node[right=1pt, midway, scale=.75]{$\bar c$} ; 
   \draw[->] (32)--(42) node[right=1pt, midway, scale=.75]{$\bar d$} ;
  
 \epic
\] 
and this\/ $\bar\rho\>(F\<,G\>)$ is an isomorphism.

\noindent\textup{(ii)} The following diagram commutes.
\[
\def\1{$\phi'_{\<*}\LL \bar u^*\R\>\sHom_\psi(F,G\>)$}
\def\2{$\phi'_{\<*}\R\>\sHom_{\psi'}(\LL\bar u^*\<\<F\<, \LL u^*\mkern-.5mu G\>)$}
\def\3{$\LL u^*\phi_*\R\>\sHom_\psi(F,G\>)$}
\def\4{$\LL u^*\R\>\sHom_Y(\phi_*F,G\>)$}
\def\5{$\R\>\sHom_{Y'}(\phi'_{\<*}\LL\bar u^*\<\<F\<, \LL u^*\mkern-.5mu G\>)$}
\def\6{$\R\>\sHom_{Y'}(\LL u^*\phi_*F\<, \LL u^*\mkern-.5mu G\>)$}
 \bpic[xscale=6, yscale=1.6]

   \node(11) at (1,-1){\1};   
   \node(12) at (2,-1){\2}; 
   
   \node(21) at (1,-2){\3};
   \node(22) at (2,-2){\5};
   
   \node(31) at (1,-3){\4};  
   \node(32) at (2,-3){\6};
 
    \draw[->] (11)--(12) node[above=.5pt, midway, scale=.75] {$\phi'_{\<*}\>\bar\rho$} ;  
    
    \draw[->] (31)--(32) node[below=.5pt, midway, scale=.75] {$\rho$} ;

    \draw[->] (11)--(21)  node[left, midway, scale=.75] {$\simeq$}
                                    node[right, midway, scale=.75] {$\theta_1^{-\<1}$};
    \draw[->] (21)--(31)  node[left, midway, scale=.75] {$\simeq$}
                                    node[right, midway, scale=.75] {\eqref{phiRHompsi}};
     
    \draw[->] (12)--(22)  node[right, midway, scale=.75] {$\simeq$}
                                    node[left, midway, scale=.75] {\eqref{phiRHompsi}};
    \draw[->] (22)--(32)  node[right, midway, scale=.75] {$\simeq$}
                                    node[left, midway, scale=.75] {$\via\>\theta_1$};

  \epic
\]
\end{sublem}

\begin{proof}
(i) First, \emph{the canonical map\/ $\sHom_{\psi}(F\<,G)\to\R\sHom_{\psi}(F\<,G)$ is an isomorphism 
$(\<$whence so is\/ $\bar a)$}---whether or not $G$ is $u^*$-acyclic.
For proving this, application\- of the functor~$\phi_*$ justifies replacing  ``$\>\psi$\kf" by~``$Y$" 
and ``$F$\kf" by~``$\phi_*F$\kf" (see \eqref{phiRHompsi0} and \eqref{phiRHompsi}). Moreover, the question being local, one can assume $\phi_*F$ to be a complex of locally free $\OY$-modules. 
Then one can proceed as in the second- and third-last paragraphs 
of \cite[\S4.6]{li}, with $(E,H)\set(\phi_*F,G)$.
(In line 3 of the third-last paragraph, ``isomorphism" should be ``quasi-isomorphism."
Also, when $\phi_*F$ is strictly perfect and $G$~is arbitrary, induct on the~number of degrees in which $F$ doesn't vanish.)

Likewise, $\bar c$ is an isomorphism.

Now every $\OY$-complex $G$ is
$\D(Y)$-isomorphic to a $u^*\<$-acyclic one, which can be assumed
bounded-below if $u$ has finite tor-dimension and $G\in\Dpl(Y)$ \cite[2.7.5,(vi) and (a)]{li}.
Thus to prove (i) one may assume that $G$ is $u^*$-acyclic, so that
$\bar e\>$ is an isomorphism, whence, via \cite[2.6.5]{li}, the existence and uniqueness of
a map $\bar\rho$ making the diagram commute.\va3

A similar inductive argument shows that if $G$ is $u^*$-acyclic, then for each integer~$n$, the canonical map is an isomorphism 
\[
H^n\LL u^*\sHom_Y(\phi_*F,G)\iso H^nu^*\sHom_Y(\phi_*F,G),
\]
i.e.,
$\sHom_Y(\phi_*F,G)=\phi_*\sHom_\psi(F,G\>)$ is $u^*$-acyclic. Also,
\[
H\set\sHom_{\psi}(F\<,G)\cong\R\sHom_{\psi}(F\<,G) \in \Dqc(\oY),
\] 
as follows via \cite[p.\,218, (2.2.4)]{EGA1} from the exactness of $\phi_*$ and the fact that, by~\eqref{phiRHompsi} and \ref{qcHom}, $\phi_*\R\sHom_\psi(F\<,G)\cong\R\sHom_Y(\phi_* F\<,G)\in\Dqc(Y)$. So as in the remarks right after \eqref{transtheta}, $\theta_1(H)$ is an isomorphism; and, as~noted right after \eqref{u^*m}, $u^*\phi_*=\phi_*'\bar u^*\<$.  Thus, from the natural diagram 
\[
\def\1{$\LL \bar u^*\phi_*H$}
\def\2{$\phi_*'\LL \bar u^*H$}
\def\3{$\bar u^*\phi_*H$}          
\def\4{$\phi_*' \bar u^*H,$}
 \bpic[xscale=3, yscale=1.4]

   \node(11) at (1,-1){\1};   
   \node(12) at (2,-1){\2}; 
   
   \node(21) at (1,-2){\3}; 
   \node(22) at (2,-2){\4};     
     
    \draw[->] (11)--(12)  node[above=1pt,midway,scale=.75]{$\Iso$} 
                                    node[below=1pt,midway,scale=.75]{$\theta_1(H)$} ;
     \draw[double distance=2] (21)--(22) ;      
    
   \draw[->] (11)--(21) node[left=1pt, midway, scale=.75]{$\simeq$} ;  
   \draw[->] (12)--(22) node[right=1pt,midway,scale=.75]{$\phi_*'\bar b$} ; 
  
 \epic
\]
which commutes (see \cite[Lemma 3.10.1.1]{li}), one gets that $\phi_*'\bar b$, 
and hence $\bar b$ itself, is  an isomorphism.

Finally, another similar induction shows that $\phi_*'\bar\rho^{}_0(F\<,G\>)=\rho^{}_0(F\<,G\>)$
is an isomorphism, whence so is $\bar\rho^{}_0(F\<,G\>)$.
Therefore, if $G$ is $u^*$-acyclic then $\bar\rho(F\<,G\>)\cong \bar\rho^{}_0(F\<,G\>)$ is an isomorphism, whence so is $\bar\rho(F\<,G\>)$ even if $G$ is not $u^*$-acyclic.\va1

(ii) This says that in the following natural diagram, where  \mbox{$\CH\set\sHom$,}
subdiagram \circled4 commutes. (The maps labeled $\theta_1^{-1}$ exist 
by the remarks after \eqref{transtheta} because the isomorphic complexes $\sHom_{\psi}(F\<,G)$ and $\R\sHom_{\psi}(F\<,G)$ are 
in $\Dqc(\oY)$, as follows\va{.5} via \cite[p.\,218, (2.2.4)]{EGA1} from the exactness of $\phi_*$ and the fact that, by \ref{qcHom}, $\phi_*\R\sHom_{\psi}(F\<,G)\cong\R\sHom_Y(\phi_* F\<,G)\in\Dqc(Y)$.)
\begin{small}
\[\mkern-4mu
\def\1{$\phi'_{\<*}\LL \bar u^*\CH_\psi(F\<,G\>)$}
\def\2{$\phi'_{\<*}\bar u^*\CH_\psi(F\<,G\>)$}
\def\3{$\LL u^*\<\<\phi_*\CH_\psi(F\<,G\>)$}
\def\4{$u^*\<\<\phi_*\CH_\psi(F\<,G\>)$}
\def\5{$\LL u^*\CH_Y(\phi_*F\<,G\>)$}
\def\6{$u^*\CH_Y(\phi_*F\<,G\>)$}
\def\7{$\phi'_{\<*}\LL \bar u^*\R\CH_\psi(F\<,G\>)$}
\def\8{$\phi'_{\<*}\CH_{\psi'}(\bar u^*\<\<F\<,u^*\mkern-.5mu G\>)$}
\def\9{$\LL u^*\<\<\phi_*\R\CH_\psi(F\<,G\>)$}
\def\ten{$\CH_{Y'}(\phi'_{\<*}\bar u^*\<\<F\<,u^*\mkern-.5mu G\>)$}
\def\lvn{$\LL u^*\R\CH_Y(\phi_*F\<,G\>)$}
\def\twv{$\CH_{Y'}(u^*\<\<\phi_*F\<,u^*\mkern-.5mu G\>)$}
\def\thn{$\phi'_{\<*}\R\CH_{\psi'}(\bar u^*\<\<F\<,u^*\mkern-.5mu G\>)$}
\def\frn{$\R\CH_{Y'}(\phi'_{\<*}\bar u^*\<\<F\<,u^*\mkern-.5mu G\>)$}
\def\ffn{$\R\CH_{Y'}(u^*\<\<\phi_*F\<,u^*\mkern-.5mu G\>)$}
\def\sxn{$\R\CH_{Y'}(\LL u^*\<\<\phi_*F\<,\LL u^*\mkern-.5mu G\>)$}
\def\svn{$\R\CH_{Y'}(\LL u^*\<\<\phi_*F\<,u^*\mkern-.5mu G\>)$}
\def\egn{$\R\CH_{Y'}(\phi'_{\<*}\LL \bar u^*\<\<F\<,\LL u^*\mkern-.5mu G\>)$}
\def\ntn{$\R\CH_{Y'}(\phi'_{\<*}\LL \bar u^*\<\<F\<, u^*\mkern-.5mu G\>)$}
\def\twy{$\phi'_{\<*}\R\CH_{\psi'}(\LL \bar u^*\<\<F\<,\LL u^*\mkern-.5mu G\>)$}
\def\twn{$\phi'_{\<*}\R\CH_{\psi'}(\LL \bar u^*\<\<F\<, u^*\mkern-.5mu G\>)\ $}
 \bpic[xscale=3.23, yscale=.965]

  \node(11) at (1,-.97){\1} ;   
  \node(16) at (3.93,-.97){\2} ; 
  
  \node(22) at (1.465,-2){\3} ;  
  \node(25) at (3.415,-2){\4} ; 
     
  \node(33) at (1.95,-3.03){\5} ;
  \node(34) at (2.9,-3.03){\6} ; 
  
  \node(42) at (1,-4){\7} ;   
  \node(46) at (3.93,-4){\8} ; 
  
  \node(52) at (1.465,-5.03){\9} ;  
  \node(55) at (3.42,-5.03){\ten} ; 
     
  \node(63) at (1.95,-6.06){\lvn} ;
  \node(64) at (2.9,-6.06){\twv} ;  
  
  \node(76) at (3.93,-7){\thn} ; 
  
  \node(85) at (3.42,-8.1){\frn} ; 
     
  \node(94) at (2.9,-9.2){\ffn} ;  

  \node(103) at (1.95,-10){\sxn} ;  
  \node(104) at (2.9,-11){\svn} ; 
  
  \node(112) at (1.95,-12){\egn} ;  
  \node(115) at (3.42,-12){\ntn} ; 
     
  \node(121) at (1,-13){\twy} ;
  \node(126) at (3.93,-13){\twn} ; 
 
    \draw[->] (11)--(16) node[above, midway, scale=.75] {$\phi'_{\<*}\bar b$} ;

    \draw[->] (22)--(25) ;
    
    \draw[->] (33)--(34) node[above, midway, scale=.75] {$b$} ;
    
    \draw[->] (103)--(104) node[above=-1, midway, scale=.75] {$\mkern15mu e$} ;
    
    \draw[->] (112)--(115) node[below, midway, scale=.75] {$$};
    
    \draw[->] (121)--(126) node[below=1pt, midway, scale=.75] {$\phi'_{\<*}\bar e$} ;
    
    \draw[->] (11)--(42) node[left, midway, scale=.75] {$\phi'_{\<*}\bar  a$} ;
    \draw[->] (42)--(121) node[left, midway, scale=.75] {$\phi'_{\<*}\bar\rho$} ;  
    
    \draw[->] (22)--(52) node[left=-1pt, midway, scale=.75] {$$} ;

    \draw[->] (33)--(63) node[left=1pt, midway, scale=.75] {$a$} ;
    \draw[->] (63)--(103) node[left=1pt, midway, scale=.75] {$\rho$} ;
    
    \draw[->] (34)--(64) node[right=1pt, midway, scale=.75] {$\rho^{}_0$} ;  
    \draw[->] (64)--(94) node[right=1pt, midway, scale=.75] {$c$} ;  
    \draw[->] (94)--(104) node[right=1pt, midway, scale=.75] {$d$} ;  
    
    \draw[->] (55)--(85) ;  
    \draw[->] (85)--(115) ;  

    \draw[->] (16)--(46) node[right=1pt, midway, scale=.75] {$\phi'_{\<*}\bar\rho^{}_0$} ;  
    \draw[->] (46)--(76) node[right=1pt, midway, scale=.75] {$\phi'_{\<*}\bar c$} ;  
    \draw[->] (76)--(126) node[right=1pt, midway, scale=.75] {$\phi'_{\<*}\bar d$} ;  
        
    \draw[->] (11)--(22) node[above=-1pt, midway, scale=.75] {$\mkern20mu\simeq$} 
                                   node[below=-4pt, midway, scale=.7] {$\theta^{-\<1}_1\mkern50mu$} ;
    \draw[double distance=2pt] (22)--(33) ;
    \draw[double distance=2pt] (16)--(25) ;
    \draw[double distance=2pt] (25)--(34) ;
    
    \draw[->] (42)--(52) node[above=-1pt, midway, scale=.75] {$\mkern20mu\simeq$} 
                                   node[below=-4pt, midway, scale=.7] {$\theta^{-\<1}_1\mkern50mu$} ;
    \draw[->] (52)--(63) node[above=-1pt, midway, scale=.75] {$\mkern20mu\simeq$} 
                                   node[below=-4pt, midway, scale=.75] {\eqref{phiRHompsi}\kern40pt} ;
    \draw[double distance=2pt] (46)--(55) ;
    \draw[double distance=2pt] (55)--(64) ;

    \draw[->] (76)--(85) node[below=-9pt, midway, scale=.75] {\rotatebox{36.5}
                                              {\kern16pt\eqref{phiRHompsi}}} 
                                   node[above=-2pt, midway, scale=.75] {\rotatebox{37}{$\simeq\mkern10mu$} } ;
     \draw[double distance=2pt] (85)--(94) ;
    \draw[->] (121)--(112) node[above=-1pt, midway, scale=.75] {$\simeq\mkern20mu$} 
                                   node[below=-3pt, midway, scale=.75] {\kern50pt\eqref{phiRHompsi}} ;
    \draw[<-] (112)--(103) node[left=1pt, midway, scale=.75] {$\simeq$} 
                                   node[right=1pt, midway, scale=.75] {$\via\theta_1$} ;
    \draw[<-] (126)--(115) node[above=-1pt, midway, scale=.75] {$\mkern20mu\simeq$} 
                                   node[below=-4pt, midway, scale=.75] {\eqref{phiRHompsi}\kern43pt} ;
    \draw[<- ](115)--(104) node[above=-1pt, midway, scale=.75] {$\mkern20mu\simeq$} 
                                   node[below=-4pt, midway, scale=.75] {$\via\theta_1\mkern60mu$} ;
    
   \node at (2.48,-1.5)[scale=.9]{\circled1} ;
   \node at (3.4,-3.54)[scale=.9]{\circled2} ;
   \node at (2.44,-7)[scale=.9]{\circled3} ;
   \node at (1.465,-7.98)[scale=.9]{\circled4} ;
    
  \epic
\]
\end{small}
\vskip-7pt\noindent
As before, one may assume that $G$ is $u^*$-acyclic, so that the maps $e$ and~$\bar e$ 
are isomorphisms. 

\pagebreak[3]
Commutativity of subdiagram~\circled1 is given by \cite[Lemma 3.10.1.1]{li} (with $(f,g,u,v)\set(\phi,\phi'\<\<,u,\bar u)$)---which holds over the category of arbitrary ringed spaces; and of~\circled2  by Lemma~\ref{rho and fstOX}. Subdiagram~\circled3 is just
diagram (4.6.7.1) toward the end of \cite[\S4.6]{li}, shown there to commute.
By (i), the outer border commutes. Diagram-chasing shows then that 
\circled4 commutes.\va3
\end{proof}

Let $\bar\rho=\bar \rho(\fst\OX,G\>)$ be as in~\ref{barrho}, and 
$\theta_2\colon\LL\bar u^*\R\fst\OX\iso \R\bar g_*\LL v^*\OX$ as in \eqref{transtheta} (an isomorphism by the remarks following \eqref{transtheta}). 
With notation as in \eqref{transtheta}, Proposition~\ref{concbc} below shows that applying~ 
$\>\bar g^*$ to the composite map
\begin{equation}\label{rho1}
\begin{aligned}
\bar\beta_\sigma(G)\colon \LL \bar u^*\<\<\pt G&=\LL\bar u^*\R\>\sHom_{\psi}(\bar\fst\OX,G)\\[2pt]
&\,\underset{\bar\rho}\lto\,
\R\>\sHom_{\psi'}(\LL \bar u^*\<\<\bar\fst\OX\<,\LL u^*\mkern-.5mu G)\\[-3pt]
&\underset{\via\theta_2}\iso
\R\>\sHom_{\psi'}(\bar g_*\CO_{\<X'},\LL u^*\mkern-.5mu G)
=\phi'{}^\flat \LL u^* G
\end{aligned}  
\end{equation}
gives a realization of the base\kf-change map $\beta_\sigma(G)$ of~\ref{indt base change}, a~realization that is concrete, modulo taking resolutions, as far as indicated by \ref{barrho}(i) and the explicit local description 
of~$\theta$ in \cite[Lemma 3.10.1.2]{li}.

 A more explicit local realization of $\beta_\sigma(G)$, in commutative\kf-algebra terms, results from Proposition~\ref{concrete bc2} below. \va1

\begin{sublem}\label{concrete bc} 
In the situation of\/ Theorem~\textup{\ref{indt base change},} $\beta_\sigma(G)$~is the unique\/  $\D(X')$-map\/ 
\mbox{$\beta(G)\colon v^*\!f^\flat G\to g^\flat u^*\mkern-.5mu G$} making the following  diagram
commute.
 \[
\def\1{$\R\bar g_*\LL v^*\mkern-2.5mu f^\flat G$}
\def\2{$\R\bar g_*g^\flat \LL u^*\mkern-.5mu G$}
\def\3{$\LL\bar u^*\<\pt G$}
\def\4{$\phi'{}^\flat \LL u^*\mkern-.5mu G$}
\def\5{$\LL\bar u^*\R\bar\fst f^\flat G$}
 \bpic[xscale=3, yscale=2]

   \node(11) at (1,-1){\1};   
   \node(12) at (1,-2){\2}; 
   
   \node(21) at (2.05,-1){\5};
   
   \node(31) at (3,-1){\3};  
   \node(32) at (3,-2){\4};
   
    \draw[->] (11)--(12) node[left=.5pt, midway, scale=.75] {$\R\bar g_*\beta(G)$} ;  
    
    \draw[->] (31)--(32)  node[right=.5pt, midway, scale=.75] {$\bar\beta_\sigma(G)$} 
                                    node[left=1pt, midway, scale=.75] {\eqref{rho1}} ;

    \draw[<-] (11)--(21)  node[above, midway, scale=.75] {$\Iso$}
                                    node[below, midway, scale=.75] {$\theta_2(f^\flat\<G)$} ;
    \draw[->] (21)--(31)  node[above, midway, scale=.75] {$\Iso$}
                                    node[below=1pt, midway, scale=.75] {\textup{\ref{^* equivalence}}} ;
     
    \draw[->] (12)--(32)  node[above=1pt, midway, scale=.75] {$\Iso$}
                                    node[below=1pt, midway, scale=.75] {\textup{\ref{^* equivalence}}} ;
  
  \epic
\]
\end{sublem}

\begin{proof} Uniqueness results from Proposition~\ref{^* equivalence}.

Applying $\Hr^0\R\Gamma(W,-)$ to the composite isomorphism in
Corollary~\ref{duality map}, and using the sentence right after \eqref{nu}, one gets that 
for $F\in\D(Y')$ the natural map
\[
\Hom_{\D(\oY')}(E, \phi'{}^\flat \<\<F\>)\lto \Hom_{\D(Y')}(\phi'_{\<*}E, \phi'_{\<*}\phi'{}^\flat \<\<F\>)
\]
has a left inverse, and so is injective. Hence, for Lemma~\ref{concrete bc} to hold it suffices that the border of the following diagram, in which
$\CH\set\R\>\sHom$,
$\bar\rho\set\bar\rho(\fst\OX,G)$, $(-)_*\>\set\R(-)_*$ and $(-)^*\>\set\LL (-)^*$,
commute:

\[\mkern-4mu
\def\1{$\phi'_{\<*}\bar u^*\mkern-1.5mu \pt G$}
\def\2{$\phi'_{\<*} \bar u^*\CH_\psi(\bar\fst\OX,G\>)\qquad$}
\def\3{$\phi'_{\<*}\bar u^*\<\<\bar\fst f^\flat G$}
\def\4{$\phi'_{\<*}\bar g_*v^*\<\< f^\flat G$}
\def\5{$\phi'_{\<*}\bar g_*g^\flat u^*\mkern-.5mu G$}
\def\6{$\phi'_{\<*}\CH_{\psi'}\<(\bar u^*\<\<\mkern-.5mu\bar\fst\OX\<, \<u^*\mkern-.5mu G\>)\qquad\quad $}
\def\7{$g_*v^*\<\< f^\flat G$}
\def\8{$g_*g^\flat u^*\mkern-.5mu G$}
\def\9{$\CH_{Y'}(g_*\CO_{\<X'},u^*\mkern-.5mu G\>)$}
\def\ten{$u^*\!\fst f^\flat G$}
\def\lvn{$\>\>u^*\<\<\phi_*\pt G$}
\def\twv{$u^*\CH_{Y}(\fst\OX,G\>)\>\>$}
\def\thn{$\CH_{Y'}(u^*\!\fst\OX,u^*\mkern-.5mu G\>)$}
\def\frn{$\CH_{Y'}(\<\phi'_{\<*}\bar u^*\<\<\bar\fst\OX,u^*\mkern-.5mu G\>)\ \>$}
\def\ffn{$\CH_{Y'}(u^*\<\<\phi_*\bar\fst\OX,u^*\mkern-.5mu G\>)$}
\def\sxn{$\phi'_{\<*}\CH_{\psi'}(\bar g_*\CO_{\<X'}\<,u^*\mkern-.5mu G\>)$}
 \bpic[xscale=3.6, yscale=1.75]

   \node(10) at (.6,-1){\4};   
   \node(13) at (3.5,-1){\2}; 
   
   \node(22) at (1.76,-1){\3};     
   
   \node(31) at (2.55,-1){\1};  
   \node(32) at (.6, -6){\5};  
   \node(33) at (3.5, -6){\6};
   \node(34) at (1.76, -6){\sxn};
     
   \node(41) at (.9, -2){\7};  
   \node(42) at (.9, -5){\8};  
   \node(53) at (2.93,-5){\frn}; 
   
   \node(51) at (1.76,-2){\ten}; 
   \node(43) at (1.76, -5){\9};
   \node(63) at (2.93,-4){\ffn};
   
   \node(50) at (2.55,-3){\lvn};  
   \node(60) at (1.76,-3){\twv};
   \node(62) at (1.76,-4){\thn};

    \draw[->] (10)--(32) node[ left=-.5pt,midway,scale=.75]{$\phi'_{\<*}\bar g_*\beta_\sigma$}  ;    
    
    \draw[<-] (34)--(33) node[below,midway,scale=.75]{$\via\theta_2$} ;
     
    \draw[->] (41)--(42) node[ right=1pt,midway,scale=.75]{$g_*\beta_\sigma$} ;
    \draw[->] (42)--(43) node[above=1pt,midway,scale=.75]{\kern-1pt\ref{fst fflat}\kern1pt} ;
   
    \draw[->] (50)--(60) node[below=1pt,midway,scale=.75]{\eqref{phiRHompsi} } ;
    \draw[->] (60)--(62) node[left=1pt,midway,scale=.75]{$\rho^{\mathstrut}$} ;
    \draw[double distance=2pt] (62)--(63) ;
    
   \draw[->] (22)--(51) node[left, midway, scale=.75]{$\theta_1^{-\<1}$} ;
   \draw[->] (31)--(50) node[right, midway, scale=.75]{$\theta_1^{-\<1}$} ;
                              
   \draw[double distance=2pt] (10)--(41) ;
   \draw[->] (41)--(51) node[below, midway, scale=.75]{$\theta_{\<\sigma}^{-\<1}$} ;
   
   \draw[double distance=2pt] (32)--(42) ;
   
   \draw[->] (13)--(33) node[right=1pt, midway, scale=.75]{$\phi'_{\<*}\bar \rho$} ;
   \draw[<-] (33)--(53) node[above=-8pt, midway, scale=.75]{\eqref{phiRHompsi}\kern45pt} ;
   \draw[<-] (43)--(53) node[below=1pt, midway, scale=.75]{$\via\theta_2$} ;
   \draw[<-] (53)--(63) node[right, midway, scale=.75]{$\via\theta_1$} ;
    
   \draw[<-] (22)--(10) node[above,midway,scale=.75]{$\phi'_{\<*}\theta_2^{-\<1}$} ;
   \draw[->] (22)--(31) node[above,midway,scale=.75]{\ref{^* equivalence}\kern.5pt} ; 
   \draw[double distance=2pt] (31)--(13)  ; 
     \draw[->] (32)--(34) node[below=1pt,midway,scale=.75]{\ref{^* equivalence}} 
                                   node[above,midway,scale=.75]{$\Iso$};
  
   \draw[->] (43)--(34) node[left=1pt,midway,scale=.75]{\eqref{phiRHompsi}} ;
   \draw[->] (51)--(50) node[above=-1.25pt,midway,scale=.75]{\kern35pt\ref{^* equivalence}} ;
   \draw[->] (51)--(60) node[left=1pt,midway,scale=.75]{\ref{fst fflat}} ;
   \draw[<-] (43)--(62) node[left,midway,scale=.75]{$\via\theta_{\<\sigma}$} ;
   
   \node at (2.93, -3.53)[scale=.9]  {\circled3};
   \node at (1.23, -1.52)[scale=.9]  {\circled1};
   \node at (2.35, -4.52)[scale=.9]  {\circled4};
   \node at (1.37, -3.53)[scale=.9]  {\circled2};
 
 \epic
\]

The unlabeled subdiagrams are easily seen to commute.  Commutativity of subdiagrams \circled1 and~\circled4  is given by that of the diagram~\eqref{transtheta}, with $E\set f^\flat G$ or $\OX$; commutativity of~\circled2 was shown in the proof of~\ref{indt base change}; and commutativity of~subdiagram~\circled3 is \ref{barrho}(ii). The conclusion follows.
\end{proof}

\begin{subprop}\label{concbc}
The following natural diagram commutes.
 \[
\def\1{$\LL v^*\<\<f^\flat G$}
\def\2{$g^\flat \LL u^*\mkern-.5mu G$}
\def\3{$\bar g^*\LL\bar  u^*\<\pt G$}
\def\4{$\bar g^*\phi'{}^\flat \LL u^* G$}
\def\5{$\LL v^*\<\<\bar f^*\pt G$}
 \bpic[xscale=2.5, yscale=2.25]

   \node(11) at (1,-1){\1};   
   \node(12) at (1,-2){\2}; 
   
   \node(21) at (1.95,-1){\5};
   
   \node(31) at (3,-1){\3};  
   \node(32) at (3,-2){\4};
   
    \draw[->] (11)--(12) node[left=.5pt, midway, scale=.75] {$\beta_\sigma(G)$} ;  
    
    \draw[->] (31)--(32)  node[right=.5pt, midway, scale=.75] {$\bar g^*\<\<\bar\beta_\sigma$} 
                                    node[left=1pt, midway, scale=.75] {\eqref{rho1}} ; 
    \draw[double distance = 2pt]  (11)--(21) ;
    
    \draw[->] (21)--(31)  node[above, midway, scale=.75] {$\Iso$} ;
     
    \draw[double distance = 2pt]  (12)--(32) ;
  
  \epic
\]

\end{subprop}

\begin{proof} 
The diagram, without ``$\>G,\>$" expands naturally to the following one,
where $(-)^*\set\LL (-)^*$ and  $(-)_*\set\R(-)_*$\,:
\[
\def\1{$v^*\<\<\bar f^*\<\pt$}
\def\2{$\bar g^*\bar  u^*\<\pt$}
\def\3{$v^*\<\<\bar f^*\<\<\bar\fst\bar f^*\<\pt$}
\def\4{$\bar g^*\bar  u^*\<\<\bar\fst\bar f^*\<\pt$}
\def\5{$\bar g^*\bar g_*v^*\<\<\bar f^*\<\pt$}
\def\6{$\bar g^*\bar g_*\bar g^*\<\phi'{}^\flat u^*$}
\def\7{$\bar g^*\<\phi'{}^\flat u^*$}
 \bpic[xscale=3.75, yscale=1.45]

   \node(11) at (1,-1){\1};   
   \node(13) at (3,-1){\2}; 
   
   \node(21) at (1.75,-1.42){\3};     
   \node(22) at (2.6, -1.85){\4};  
 
   \node(32) at (2, -3.1){\5};  
 
   \node(41) at (1, -4.15){\7};
   \node(42) at (2,-4.15){\6};
   \node(43) at (3,-4.15){\7};
     
    \draw[->] (11)--(13) node[above=1pt,midway,scale=.75]{$\Iso$}  ;    
    
    \draw[->] (21)--(22) node[above=1pt,midway,scale=.75]{$$} ;
    
    \draw[->] (41)--(42) node[above=1pt,midway,scale=.75]{$\Iso$} ; 
    \draw[->] (42)--(43) node[above=1pt,midway,scale=.75]{$\Iso$} ;
    
   \draw[->] (11)--(41) node[left, midway, scale=.75]{$\beta_\sigma$} ;
   \draw[->] (32)--(42) node[right, midway, scale=.75]{$\bar g^*\bar g_*\beta_\sigma$} ;
   \draw[->] (13)--(43) node[right, midway, scale=.75]{$\bar g^*\<\<\bar\beta_\sigma$} ;

    \draw[->] (11)--(21) node[above=1pt,midway,scale=.75]{$$}  ;    
    \draw[->] (11)--(32) node[above=1pt,midway,scale=.75]{$$}  ;
    \draw[->] (22)--(32) node[below=-2.1pt,midway,scale=.75]{$\mkern40mu \bar g^*\theta_2$} ;
    \draw[->] (13)--(22) node[above=1pt,midway,scale=.75]{$$} ;
    \draw[->] (41)--(42) node[above=1pt,midway,scale=.75]{$$} ;

   \node at (2, -2.3) [scale=.9] {\circled1};
   \node at (2.6, -3.12) [scale=.9] {\circled2};
 
 \epic
\]
Commutativity of subdiagram \circled1 results from \ref{^* equivalence} and adjointness of~$\theta_2$ to the natural composite  $\bar g^*\bar u^*\<\<\bar\fst\iso v^*\<\<\bar f^*\<\<\bar\fst\to v^*$ (see \cite[3.7.2(i)]{li}); and that of~\circled2 from Lemma~\ref{concrete bc}. That of the other two subdiagrams is clear.
\end{proof}


\begin{cosa}\label{tensor and flat}
For a pseudo\kf-coherent finite map $f \colon X\to Y\<$, and $F,G\in\Dqcpl(Y)$, with $G$ 
perfect (so that  $f^\flat\<\<F\Otimes{\sX} \LL f^*\<G \in\Dqcpl(X)$ and $F\Otimes{Y}G\in\Dqcpl(Y))$,
Proposition~\ref{flat and tensor}  provides a concrete representation of the map
\[
\chi=\chi(f,F,G)\colon f^\flat\<\<F\Otimes{\sX} \LL f^*\<G\lto f^\flat\<(F\Otimes{Y}G\>)
\]
that is defined abstractly to be 
adjoint under\/ \textup{\ref{right adjoint}} to the natural composite 
\[
\R\fst(f^\flat\<\< F\Otimes{\mkern.5mu\oY}\LL f^*\<G\>)
\underset{\lift1.2,p,}\iso\R\fst f^\flat\<\< F\Otimes{Y}G\lto F\Otimes{Y}G,\\[-3pt]
\]   
where\/ $p$ is the first  (projection) isomorphism in\/~\eqref{projf} with $E\set f^\flat F$.

The \emph{pseudofunctoriality} of $\chi$ is explicated in Proposition~\ref{pfchi}.
\va3

Let $(Y,\OY)$ be a ringed space, and
\begin{equation}\label{gam0}
\gamma^{}_0(E,F,G\>)\colon\sHom_Y(E,F\>)\otimes_Y G\lto
\sHom_Y(E,F\<\otimes_Y\< G\>)
\end{equation} 
the natural map of $\OY$-complexes. The next result is, \emph{mutatis mutandis,} an instance of \cite[Corollary 2.6.5]{li}.

\begin{sublem} \label{gamma}
There exists a unique trifunctorial map
\[
\gamma(E,F,G\>)\colon \<\R\>\sHom_Y(E,F\>)\Otimes{Y}G\to\R\>\sHom_Y(E,F\Otimes{Y}G\>)
\quad\<(E,\>F,\>G\in\D(Y))
\]
such that if $F$ is K-injective and\/ $G$ is K-flat then the following natural diagram commutes.
\[\mkern-4mu
\def\1{$\sHom_Y(E,F\>)\otimes_Y G$}
\def\2{$\sHom_Y(E,F\>)\Otimes{Y} G$}
\def\3{$\R\>\sHom_Y(E,F\>)\Otimes{Y} G$}
\def\4{$\sHom_Y(E,F\otimes_Y G\>)$}
\def\5{$\R\>\sHom_Y(E,F\<\otimes_{Y} G\>)$}
\def\6{$\R\>\sHom_Y(E,F\<\Otimes{Y}G\>)$}
  \bpic[xscale=4.5, yscale=1.57]

   \node(11) at (1,-1){\1};  
   \node(12) at (1.95,-1){\2};   
   \node(13) at (3,-1){\3}; 
   
   \node(21) at (1,-2){\4}; 
   \node(22) at (1.97,-2){\5};  
   \node(23) at (3,-2){\6};
     
    \draw[->] (11)--(12) node[above, midway, scale=.75] {$\Iso$} ;
    \draw[->] (12)--(13) node[above, midway, scale=.75] {$\Iso$} ;   
    
    \draw[->] (21)--(22) ;
    \draw[->] (22)--(23) node[above, midway, scale=.75] {$\Iso$} ;   
   
    \draw[->] (11)--(21) node[left=1pt, midway, scale=.75] {$\gamma^{}_0$} ;
     
    \draw[->] (13)--(23) node[right=1pt, midway, scale=.75] {$\gamma$} ;
  \epic
\]
\end{sublem}

\begin{subcor}\label{gam for O}
The map\/ $\gamma(\OY\<, F, \>G)$ factors naturally as
\[
\R\>\sHom_Y(\OY\<, F\>)\Otimes{Y}G\iso F\Otimes{Y}G\iso\R\>\sHom_Y(\OY\<, F\Otimes{Y}G\>).
\]
\end{subcor}
\begin{proof} Replace $F$ (respectively~$G$) by a quasi-isomorphic K-injective (respectively~K-flat) complex, and then
use the commutative diagram in~\ref{gamma} to reduce  to the corresponding easily-verified statement for $\gamma^{}_0$.
\end{proof}

\begin{subrem}\label{gamma0} As a relatively easy example of ``concrete vs.~abstract," it is readily shown 
that~$\gamma^{}_0$ is adjoint to the natural composite
\[
(\sHom_Y(E,F\>)\otimes_{Y}G\>) \otimes_{Y}E
\iso
(\sHom_Y(E,F\>)\otimes_{Y}E) \otimes_{Y}G
\lto
F\otimes_{Y}G,
\]
and hence that $\gamma$ is adjoint to the natural composite
\[
(\R\>\sHom_Y(E,F\>)\Otimes{Y}G\>) \Otimes{Y}E
\iso
(\R\>\sHom_Y(E,F\>)\Otimes{Y}E) \Otimes{Y}G
\lto
F\Otimes{Y}G.
\]
\end{subrem}

\vskip2pt

Now let $\psi\colon\OY\to\CS$, $\phi\colon(\oY,\CO_{\>\oY})=(Y,\CS)\to(Y,\OY\<)$ be as in \S\ref{2.3}, and $\pt(-)\set\R\>\sHom_\psi(\CO_{\>\oY},-)$ as in \ref{duality for phi}. \va2

Define the map of $\CO_{\>\oY}$-modules
\begin{equation}\label{bargamma}
\bar\gamma^{}_0(F,\,G\>)\colon\sHom_\psi(\CO_{\>\oY},F\>)\otimes_{\>\oY} \phi^*G
\,\lto \sHom_\psi(\CO_{\>\oY},F\<\otimes_Y\< G\>)
\end{equation}
to be the natural composite $\CO_{\>\oY}$-linear map
\begin{align*}
\phi_*(\sHom_\psi(\CO_{\>\oY},F\>)\otimes_{\>\oY} \phi^*G\>)
&\underset{\textup{\ref{pphi2}}}\iso
\phi_{\<*}\sHom_\psi(\CO_{\>\oY},F\>)\otimes_Y\<G\\
&\mkern 5mu=\!=\mkern 5mu
\sHom_Y(\phi_*\CO_{\>\oY},F\>)\otimes_Y G\\
&\,\xto[\lift1.2,\,\gamma^{}_0\>\>,]{}\,
\sHom_Y(\phi_*\CO_{\>\oY},F\<\otimes_Y\< G\>)\\
&\mkern 5mu=\!=
\phi_* \sHom_\psi(\CO_{\>\oY},F\<\otimes_Y\< G\>).
\end{align*}

\begin{sublem}\label{ft0} 
In the situation of\/ \textup{\S\ref{2.3}:}

{\rm (i)}  The map adjoint under\/ \textup{\ref{adjunction0}} to the composite map
\[
\phi_{\<*}(\pt\< F\Otimes{\mkern.5mu\oY}\LL\phi^*G\>)
\underset{\eqref{projphi}}\iso\phi_{\<*}\pt\< F\Otimes{Y}G\underset{\ref{counit}}\lto F\Otimes{Y}G
\qquad(F,\>G\in\D(Y))
\]   
is the unique\/ $\D(\oY)$-map\/ 
\[
\bar\chi\colon\pt\< F\Otimes{\mkern.5mu\oY}\LL\phi^*G\lto
\pt(F\Otimes{Y}G\>)
\] 
such that the following diagram commutes.
\[
\def\3{$\phi_{\<*}(\pt\< F\Otimes{\mkern.5mu\oY}\LL\phi^*G\>)$}
\def\4{$\phi_{\<*}\pt(F\Otimes{Y}G\>)$}
\def\6{$\phi_{\<*}\pt\< F\Otimes{Y}G$}
\def\9{$\R\>\sHom_Y(\fst\OX,F\>)\<\Otimes{Y} G$}
\def\ten{$\R\>\sHom_Y(\fst\OX,F\<\Otimes{Y}\<G\>)$}
  \bpic[xscale=5.2, yscale=1.65]
     
   \node(22) at (1.95,-2){\3};  
   \node(23) at (3,-2){\4} ;
   
     \node(32) at (1.95,-3){\6} ;  
   
   \node(42) at (1.95 ,-4){\9};   
   \node(43) at (3,-4){\ten}; 
  
   \draw[->] (22)--(23) node[above=-.7pt, midway, scale=.75] {$\phi_*\bar\chi$} ; 
   \draw[->] (42)--(43) node[above=1pt, midway, scale=.75] {\textup{\ref{gamma}}} 
                                  node[below=1pt, midway, scale=.75] {$\gamma$};
   
    \draw[->] (22)--(32) node[right=1pt, midway, scale=.75] {$\simeq$} 
                                   node[left=1pt,midway,scale=.75]{\eqref{projphi}};
    \draw[->] (32)--(42) node[left=1pt, midway, scale=.75] {\eqref{phiRHompsi}} 
                                   node[right=1pt, midway, scale=.75] {$\simeq$} ;
    
     \draw[->] (23)--(43) node[right=1pt, midway, scale=.75] {\eqref{phiRHompsi}} 
                                   node[left=1pt, midway, scale=.75] {$\simeq$} ;

    \node at (2.5,-3.03)[scale =.9]{\circled1} ;

  \epic
\]

\pagebreak[3]

{\rm(ii) (Explicit $\bar\chi$)} 
For K-injective\/ $F$ and \/K-flat $G,$ $\bar\chi$ factors naturally as\va2 
\begin{align*}
\pt\< &F\Otimes{\>\oY}\LL\phi^*G
=\R\>\sHom_\psi(\CO_{\>\oY},F\>)\Otimes{\>\oY} \LL\phi^*G\\
&\iso\sHom_\psi(\CO_{\>\oY},F\>)\otimes_{\>\oY} \phi^*G\\[-3pt]
&\xto[\!\eqref{bargamma}\!]{\bar\gamma^{}_0} \sHom_\psi(\CO_{\>\oY},F\<\otimes_Y\< G\>)
\lto \R\>\sHom_\psi(\CO_{\>\oY},F\<\otimes_Y\< G\>)
=\pt(F\Otimes{Y} G\>).
\end{align*}
\end{sublem}

\begin{proof}First, the uniqueness in  (i).
If  \circled1 commutes then, as made clear by~\ref{gam for O}, so does the following diagram, where going clockwise from upper right to lower left  gives the counit map from~\ref{counit}.\looseness=-1 

\[\mkern-4mu
\def\1{$\phi_*(\phi^\flat\<\<F\Otimes{\oY} \LL \phi^*G\>)$}
\def\2{$\phi_* \phi^\flat(F\Otimes{Y}G\>)$}
\def\3{$\phi_* \phi^\flat\<\<F\<\Otimes{Y} G$}
\def\5{$\R\>\sHom_Y(\phi_*\CO_{\>\oY},F\>)\<\Otimes{Y} G$}
\def\6{$\R\>\sHom_Y(\phi_*\CO_{\>\oY},F\<\Otimes{Y}\<G\>)$}
\def\7{$F\<\Otimes{Y}G$}
\def\8{$\R\>\sHom_Y(\OY\<,F\>)\<\Otimes{Y} G$}
\def\9{$\R\>\sHom_Y(\OY\<,F\<\Otimes{Y}G\>)$}
  \bpic[xscale=4.45, yscale=1.65]

   \node(11) at (1,-1){\1};   
   \node(13) at (3.03,-1){\2}; 
   
   \node(21) at (1,-2){\3};  
   \node(22) at (1.92,-2){\5} ;
   \node(23) at (3.03,-2){\6};
   
   \node(31) at (1,-3){\7} ;
   \node(32) at (1.92,-3){\8} ;
   \node(33) at (3.03,-3){\9} ;
   
    \draw[->] (11)--(13) node[above, midway, scale=.75] {$\phi_*\bar\chi$} ;  
    
    \draw[->] (21)--(22) node[above=-.5pt, midway, scale=.75] {$\Iso$}
                                   node[below=1pt, midway, scale=.75] {\eqref{phiRHompsi}}  ; 
    \draw[->] (22)--(23) node[below=1pt, midway, scale=.75] {$\gamma$} ; 
    
   \draw[<-] (31)--(32) node[above, midway, scale=.75] {$\Iso$} 
                                  node[below=1pt, midway, scale=.75] {\textup{natural}} ;
   \draw[<-] (32)--(33) node[above, midway, scale=.75] {$\Iso$}
                                  node[below=1pt, midway, scale=.75] {$\gamma^{-\<1}$}; 
   
    \draw[->] (11)--(21) node[right=1pt, midway, scale=.75] {$\simeq$} 
                                   node[left=1pt,midway,scale=.75]{\eqref{projphi}};
    \draw[->] (21)--(31) node[left=1pt, midway, scale=.75] {\ref{counit}} ;
    
    \draw[->] (22)--(32) node[right=1pt, midway, scale=.75] {$\via\psi$} ;
    
    \draw[->] (13)--(23) node[right=1pt, midway, scale=.75] {\eqref{phiRHompsi}} 
                                   node[left=1pt, midway, scale=.75] {$\simeq$} ;
    \draw[->] (23)--(33) node[right=1pt, midway, scale=.75] {$\via\psi$} ;
   
   \node at (1.93,-1.48)[scale =.9]{\circled1} ;
   
  \epic
\]
Thus any $\D(\oY)$-map~$\bar\chi$ such that \circled1 commutes must be the one adjoint 
under~\ref{adjunction0} to the natural composite
\[
\phi_{\<*}(\pt\< F\Otimes{\mkern.5mu\oY}\LL\phi^*G\>)
\underset{\eqref{projphi}}\iso\phi_{\<*}\pt\< F\Otimes{Y}G
\underset{\ref{counit}}\lto F\Otimes{Y}G.
\]   

This being so,  (i) and (ii) can be proved by showing, after replacing $F$ (resp.~$G$)
by a quasi-isomorphic K-injective (resp.~K-flat) $\OY$-complex, that for $\bar\chi$ as in (ii),  \circled1 commutes; and for this just note that each subdiagram
of the following natural diagram, where $\CH\set\sHom$, commutes---more or less by definition of the functorial maps involved, whence so does the border.\looseness=-1
\begin{small}
\[\mkern-5mu
\def\1{$\phi_{\<*}(\pt\< F\Otimes{\mkern.5mu\oY}\LL\phi^*G\>)$}
\def\2{$\phi_{\<*}(\CH_\psi\mkern-.5mu(\CO_{\>\oY}\mkern-.5mu,\mkern-.5mu F\>)\<\<\otimes_{\>\oY}\<\phi^*G\>)\quad$}
\def\3{$\phi_{\<*}\pt\< F\Otimes{Y}G$}
\def\4{$\phi_{\<*}\CH_\psi(\CO_{\>\oY},F\otimes_Y\<G\>)$}
\def\5{$\phi_{\<*}\pt(F\Otimes{Y}G\>)$}
\def\6{$\R\CH_Y(\phi_*\CO_{\>\oY},F\>)\Otimes{Y}G$}
\def\7{$\phi_{\<*}\CH_\psi(\CO_{\>\oY},F\>)\otimes_Y\<G$}
\def\8{$\CH_Y(\phi_*\CO_{\>\oY},F\>)\otimes_Y\<G$}
\def\ten{$\CH_Y(\phi_*\CO_{\>\oY},F\otimes_{Y}\<G\>)$} 
\def\lvn{$\R\CH_Y(\CO_{\>\oY},F\Otimes{Y}G\>)$}
\def\twv{$\phi_{\<*}\R\CH_\psi(\CO_{\>\oY},F\otimes_Y\<G\>)$}
\def\frn{$\R\CH_Y(\phi_*\CO_{\>\oY},F\Otimes{Y}G\>)$}
\def\ffn{$\phi_{\<*}(\R\CH_\psi\mkern-.5mu(\CO_{\>\oY}\mkern-.5mu,\mkern-.5mu F\>)\<\<\Otimes{\>\oY}\<\LL\phi^*G\>)$}
  \bpic[xscale=3.74, yscale=1.4]
  
   \node(01) at (1,-1){\1};
   \node(02) at (2.15,-1){\ffn};  
   \node(03) at (3.47,-1){\2};   
 
   \node(04) at (1,-2){\3};
   
   \node(11) at (3.47,-4){\4};   
   \node(12) at (1,-5){\5}; 
   
   \node(22) at (2.25,-2){\7};  
   \node(23) at (2.25,-2.9){\8} ;
   
   \node(31) at (1,-2.9){\6} ;
   \node(32) at (2.25,-4){\ten} ;  
   
   \node(42) at (3.47 ,-5){\twv};
    
   \node(51) at (1 ,-4){\frn};

   \draw[double distance=2pt] (01)--(02) ;
   \draw[->] (02)--(03) node[above, midway, scale=.75] {$\Iso$} ;
   \draw[->] (04)--(22) node[above=1pt, midway, scale=.75] {$\Iso$} ;
   \draw[->] (23)--(31) node[above=1pt, midway, scale=.75] {$\Iso$} ;        
                                 
     \draw[->] (01)--(04) node[right=1pt, midway, scale=.75] {$\simeq$}
                                  node[left=1pt, midway, scale=.75]{\eqref{projphi}}; 
     \draw[->] (04)--(31) node[left=1pt, midway, scale=.75] {\eqref{phiRHompsi}} 
                                    node[right=1pt, midway, scale=.75] {$\simeq$} ;
 
    \draw[->] (31)--(51) node[left=1pt, midway, scale=.75] {$\gamma$};

   \draw[->] (03)--(11) node[right=1.5pt, midway, scale=.75] {$\phi_*\<\bar\gamma^{}_0$} ; 
   
   \draw[double distance=2pt] (22)--(23) ;
  
   \draw[->] (11)--(42)  ;
   \draw[->] (42)--(12) node[above=1pt, midway, scale=.75] {$\Iso$} 
                                  node[below=1pt, midway, scale=.75] {\eqref{phiRHompsi}};
   \draw[->] (23)--(32) node[right=1pt, midway, scale=.75] {$\gamma^{}_0$} ; 
 
    \draw[->] (03)--(22) node[above, midway, scale=.75] {$\simeq\mkern15mu$}
                                     node[below=-3pt, midway, scale=.75]{\kern60pt\eqref{pphi2}} ;
    \draw[double distance=2pt] (11)--(32) ;
    \draw[->] (32)--(51)  ; 
    \draw[->] (12)--(51) node[left=1pt, midway, scale=.75] {\eqref{phiRHompsi}} 
                                   node[right=1pt, midway, scale=.75] {$\simeq$} ;  
   \node at (1.64,-3.45)[scale =.75] {\textup{(see \ref{gamma})}} ;
 \epic
\]
\vskip-10pt
\end{small}

\end{proof}

The following proposition addresses, both concretely and abstractly, the relation between $(-)^\flat$ and $\Otimes{}$.

\begin{subprop}\label{flat and tensor}
Let\/ $f \colon X\to Y$ be a pseudo\kf-coherent finite scheme-map, and $\psi\colon\OY\to\fst\OX$ the associated homomorphism, so that $f$ factors as\looseness=-1
\[
X\set(X\<,\OX)\xto{\!\bar f\,\set(f\<\<,\mkern1.5mu\id)}\oY\set(Y,\fst\OX)\xto{\!\phi\,\set(\id,\psi)} (Y,\OY)=:Y
\] 
$($see \textup{\S\ref{2.4})}.\va2  
For\/ $F,G\in\D(Y)$ let\/ $\chi=\chi(f,F,G)$ be the  natural composite map\va{-3}
\[
f^\flat\<\<F\Otimes{\sX} \LL f^*\<G\iso \bar{f\:\<}^{\!\<*}\<\<(\pt\< F\Otimes{\mkern.5mu\oY}\LL\phi^*G\>)
\xto[\textup{\ref{ft0}}]{\bar{f\:\<}^{\!\<*}\<\<\bar\chi\>\>} \bar{f\:\<}^{\!\<*}\<\<\pt(F\Otimes{Y}G\>)
  =f^\flat\<(F\Otimes{Y}G\>).
\]

{\rm (i)}
If\/ $F\in\Dqcpl(Y)$ and\/ $G\in\Dqc(Y)$ are such that\/ $F\Otimes{Y}G\in\Dqcpl(Y),$ then 
$\chi(f,F,G)$
is the\va{1.4} unique\/ 
$\D(X)$-map\/~$\chi\>'\colon f^\flat\<\<F\Otimes{\sX} \LL f^*\<G\to f^\flat\<(F\Otimes{Y}G\>)$
such that the following diagram commutes$\>:$\va{-1}
\[
\def\1{$\R\fst(f^\flat\<\<F\Otimes{\sX} \LL f^*\<G\>)$}
\def\2{$\R\fst f^\flat(F\Otimes{Y}G\>)$}
\def\3{$\R\fst f^\flat\<\<F\<\Otimes{Y} G$}
\def\5{$\R\>\sHom_Y(\fst\OX,F\>)\<\Otimes{Y} G$}
\def\6{$\R\>\sHom_Y(\fst\OX,F\<\Otimes{Y}G\>)\>;$}
  \bpic[xscale=3, yscale=1.25]

   \node(11) at (1,-1){\1};   
   \node(13) at (3,-1){\2}; 
   
   \node(21) at (1,-2){\3};  
   \node(22) at (1,-3){\5} ;
   \node(23) at (3,-3){\6};
     
    \draw[->] (11)--(13) node[above, midway, scale=.75] {$\R\fst\chi\>'$} ;  
    
    \draw[->] (21)--(22) node[left=1pt, midway, scale=.75] {$\simeq$} 
                                   node[right=1pt, midway, scale=.75] {$\textup{\ref{fst fflat}}$} ; 
    \draw[->] (22)--(23) node[below=1pt, midway, scale=.75] {$\gamma$}  ; 
    \draw[->] (11)--(21) node[right=1pt, midway, scale=.75] {\eqref{projf}}  
                                   node[left=1pt, midway, scale=.75] {$p$};
     
    \draw[->] (13)--(23) node[left=1pt, midway, scale=.75] {$\textup{\ref{fst fflat}}$} 
                                   node[right=1pt, midway, scale=.75] {$\simeq$} ;
   \node at (2.03,-2)[scale =.9]{\circled1} ;
   
  \epic
\]
\vskip-4pt
and $\chi$ corresponds under\/ \textup{\ref{represent}} to the  composite map\va{-2}
\[
\R\fst(f^\flat\<\< F\Otimes{\sX}\LL f^*\<G\>)
\underset{p}{\iso}\R\fst f^\flat\<\< F\Otimes{Y}G\xto[\lift1.3,\!t_{\<F}\Otimes{Y}\id\>,]{} F\Otimes{Y}G.
\]   
\vskip2pt

{\rm(ii)} If $f$ is perfect then\/ \textup{(i)} holds for all\/~$F,G\in\Dqc(Y),$ and 
both\/ $\bar\chi$ and\/~$\chi$ are isomorphisms.
\end{subprop}

\begin{proof} 
For diagram \circled1 to commute when $\chi'=\chi\>,$ \kf it clearly suffices that
the following natural diagram commute. \va{-3}
\begin{small}
\[\mkern-5mu
\def\1{$\phi_{\<*}\R\bar\fst\bar{f\:\<}^{\!\<*}\<\<\<(\pt\<\< F\<\Otimes{\mkern.5mu\oY}\<\LL\phi^{\<*}\<G\>)$}
\def\2{$\phi_{\<*}\R\bar\fst\bar{f\:\<}^{\!\<*}\<\<\pt(F\Otimes{Y}G\>)\ \;$}
\def\3{$\phi_{\<*}(\pt\< F\Otimes{\mkern.5mu\oY}\LL\phi^*G\>)$}
\def\4{$\phi_{\<*}\pt(F\Otimes{Y}G\>)$}
\def\5{$\phi_{\<*} \R\bar\fst\<(\bar{f\:\<}^{\!\<*}\<\<\<\pt\< F\<\Otimes{\sX}\<\bar{f\:\<}^{\!\<*}\<\LL\phi^{\<*}\<G\>)$}
\def\6{$\phi_{\<*}\pt\< F\Otimes{Y}G$}
\def\7{$\phi_{\<*}\R\bar\fst\bar{f\:\<}^{\!\<*}\<\<\pt\< F\Otimes{Y}G$}
\def\8{$\R\fst f^\flat\<\< F\Otimes{Y}G$}
\def\9{$\R\>\sHom_Y(\fst\OX,F\>)\<\Otimes{Y} G$}
\def\ten{$\R\>\sHom_Y(\fst\OX,F\<\Otimes{Y}\<G\>)\qquad\:$}
\def\lvn{$\R\fst(f^\flat\<\<F\Otimes{\sX} \LL f^*\<G\>)$}
\def\twv{$\R\fst f^\flat(F\Otimes{Y}G\>)$}
\def\thn {$\phi_{\<*}(\R\bar\fst\bar{f\:\<}^{\!\<*}\<\<\<\pt\< F\<\Otimes{\mkern.5mu\oY}
            \<\LL\phi^{\<*}\<G\>)$}
  \bpic[xscale=4.6, yscale=1.25]
  
   \node(02) at (1.95,.2){\lvn};   
   \node(03) at (3,.2){\twv}; 

   \node(12) at (1.95 ,-1){\1};   
   \node(13) at (3,-1){\2};

   \node(21) at (1,.2){\5};
   \node(22) at (1.95,-2){\3};  
   \node(23) at (3,-2){\4} ;
   
   \node(31) at (1,-3){\7} ;
   \node(20) at (1,-2){\thn} ;
   \node(32) at (1.95,-3){\6} ;  
   
   \node(41) at (1 ,-4){\8};
   \node(42) at (1.9 ,-4){\9};   
   \node(43) at (3,-4){\ten}; 
  
   \draw[->] (02)--(03) node[above, midway, scale=.75] {$\R\fst\chi$} ;
   
   \draw[->] (12)--(13) node[above=1pt, midway, scale=.75] 
                           {$\!\<\phi_{\<*}\R\bar\fst\bar{f\:\<}^{\!\<*}\<\<\< \bar\chi$} ;  
  
   \draw[->] (20)--(22) node[above=1pt, midway, scale=.75] {$\Iso$} 
                                  node[below=1pt, midway, scale=.75] {\ref{^* equivalence}};   
   \draw[->] (22)--(23) node[below=1pt, midway, scale=.75] {$\phi_*\bar\chi$} ; 
   
   \draw[->] (31)--(32) node[above=1pt, midway, scale=.75] {$\Iso$} 
                                  node[below=1pt, midway, scale=.75] {\ref{^* equivalence}};   
    
   \draw[->] (41)--(42)  node[above=1pt, midway, scale=.75] {$\Iso$}
                                   node[below=1pt, midway, scale=.75] {\textup{\ref{fst fflat}}} ;
   \draw[->] (42)--(43) node[below=1pt, midway, scale=.75] {$\gamma$} ;
   
    \draw[->] (21)--(20) node[left=1pt, midway, scale=.75] {$\simeq$} 
                                   node[right=1pt, midway, scale=.75] {\textup{\ref{project}}} ;
    \draw[->] (20)--(31) node[left=1pt, midway, scale=.75] {$\simeq$} 
                                   node[right=1pt,midway,scale=.75]{\eqref{projphi}};

    \draw[->] (02)--(12) node[left=1pt, midway, scale=.75] {$\simeq$} ;
    \draw[->] (12)--(22) node[left=1pt, midway, scale=.75] {$\simeq$} 
                                   node[right=1pt, midway, scale=.75] {\ref{^* equivalence}} ;
     
    \draw[double distance=2pt] (03)--(13);
   \draw[double distance = 2pt] (41)--(31)  ;
    
    \draw[->] (22)--(32) node[left=1pt, midway, scale=.75] {$\simeq$} 
                                   node[right=1pt,midway,scale=.75]{\eqref{projphi}};
    \draw[->] (32)--(1.95,-3.78) node[right=1pt, midway, scale=.75] {\eqref{phiRHompsi}} 
                                   node[left=1pt, midway, scale=.75] {$\simeq$} ;
    
    \draw[->] (13)--(23) node[left=1pt, midway, scale=.75] {$\textup{\ref{^* equivalence}}$} 
                                   node[right=1pt, midway, scale=.75] {$\simeq$} ;
    \draw[->] (23)--(43) node[left=1pt, midway, scale=.75] {\eqref{phiRHompsi}} 
                                   node[right=1pt, midway, scale=.75] {$\simeq$} ;
   
    \draw[->] (02)--(21) node[above=1pt, midway, scale=.75] {$\Iso$} ;
    \draw[<-] (12)--(21)  node[above=-1.5pt, midway, scale=.75] {\rotatebox{-21}{$\Iso\quad$}};
    \draw[double distance = 2pt] (41)--(31)  ;
   
   \node at (2.5,-3.04)[scale =.9]{\circled3} ;
   \node at (1.35,-1)[scale =.9]{\circled2} ; 
  \epic
\]

\end{small}

\vskip-2pt
\noindent But subdiagram \circled2 commutes, by \cite[3.4.7(i)]{li},
subdiagram \circled3 commutes, by \ref{ft0}, 
and all the unlabeled subdiagrams
obviously commute as well. \va2

As for the rest of (i),
unwinding the definitions of the maps involved, one verifies that the unlabeled subdiagrams of the next diagram commute; and via \ref{represent}(ii) and \ref{gam for O}, that going around clockwise from upper right to lower left  gives the map $t^{}_{\<\<F\>\Otimes{Y}G}\,$.\va{-7}
\[\mkern-3mu
\def\1{$\R\fst(f^\flat\<\<F\Otimes{\sX} \LL f^*\<G\>)$}
\def\2{$\R\fst f^\flat(F\Otimes{Y}G\>)$}
\def\3{$\R\fst f^\flat\<\<F\<\Otimes{Y} G$}
\def\5{$\R\>\sHom_Y(\fst\OX,F\>)\<\Otimes{Y} G$}
\def\6{$\R\>\sHom_Y(\fst\OX,F\<\Otimes{Y}G\>)\ $}
\def\7{$F\<\Otimes{Y}G$}
\def\8{$\R\>\sHom_Y(\OY\<,F\>)\<\Otimes{Y} G$}
\def\9{$\R\>\sHom_Y(\OY\<,F\<\Otimes{Y}G\>)$}
  \bpic[xscale=4.5, yscale=1.5]

   \node(11) at (1,-1){\1};   
   \node(13) at (3,-1){\2}; 
   
   \node(21) at (1,-2){\3};  
   \node(22) at (1.9,-2){\5} ;
   \node(23) at (3,-2){\6};
   
   \node(31) at (1,-3){\7} ;
   \node(32) at (1.9,-3){\8} ;
   \node(33) at (3,-3){\9} ;
   
    \draw[->] (11)--(13) node[above, midway, scale=.75] {$\R\fst\chi'$} ;  
    
    \draw[->] (21)--(22) node[above=-.5pt, midway, scale=.75] {$\Iso$}
    node[below=1pt, midway, scale=.75] {$\textup{\ref{fst fflat}}$}  ; 
    \draw[->] (22)--(23) node[below=1pt, midway, scale=.75] {$\gamma$} ; 
    
   \draw[<-] (31)--(32) node[above, midway, scale=.75] {$\Iso$} 
                                  node[below=1pt, midway, scale=.75] {\textup{natural}} ;
   \draw[<-] (32)--(33) node[above, midway, scale=.75] {$\Iso$}
                                  node[below=1pt, midway, scale=.75] {$\gamma^{-\<1}$}; 
   
    \draw[->] (11)--(21) node[right=1pt, midway, scale=.75] {$\simeq$} 
                                   node[left=1pt, midway, scale=.75] {$p$}  ;
    \draw[->] (21)--(31) node[left=1pt, midway, scale=.75] {$t_{\<F}\<\<\Otimes{Y}\<\<\id$} ;
    
    \draw[->] (22)--(32) node[right=1pt, midway, scale=.75] {$\via\psi$} ;
    
    \draw[->] (13)--(23) node[left=1pt, midway, scale=.75] {\ref{fst fflat}} 
                                   node[right=1pt, midway, scale=.75] {$\simeq$} ;
    \draw[->] (23)--(33) node[right=1pt, midway, scale=.75] {$\via\psi$} ;
   
   \node at (1.9,-1.48)[scale =.9]{\circled1} ;
   
  \epic
\]

\vskip-2pt
Thus, and in view of \ref{qcHom} with $F\set\fst\OX$, any $\chi'$ such that \circled1~commutes must  correspond
under\/ \textup{\ref{represent}} to the composite map
\[
\R\fst(f^\flat\<\< F\Otimes{\sX}\LL f^*\<G\>)
\underset{p}{\iso}\R\fst f^\flat\<\< F\Otimes{Y}G\xto[\lift1.3,\!t_{\<F}\Otimes{Y}\id\>,]{} F\Otimes{Y}G.
\]   
\vskip3pt

(ii) If $f$ is perfect, the proof of (i) is valid 
for all $F$ and $G$ in $\Dqc(Y)$.\va1

 For $\chi$ and $\bar\chi\>$ to be isomorphisms, it suffices that $\R\fst\chi$ be an isomorphism:
for if \mbox{$\R\fst\chi=\phi_* \R\bar\fst \chi$}
induces homology isomorphisms,\va{.6} then so does $\R\bar\fst\chi$, i.e.,  $\R\bar\fst\chi\cong\R\bar\fst\bar{f\:\<}^{\!\<*}\<\<\bar\chi\>$ is an
isomorphism, whence by Proposition~\ref{^* equivalence}, so are $\chi$ and $\bar\chi$. 

Since \circled2~commutes when $\chi'=\chi\>$,  therefore $\R\fst\chi$ is an isomorphism
if $\gamma$~is an isomorphism---which $\gamma$ \emph{is} when $f$ is perfect. (This is well-known:
the question being local, one can replace~$\fst\OX$ by an isomorphic bounded complex~$E$ of 
finite\kf-rank free $\OY$-modules, then by induction on the number of nonvanishing components of $E\<$, using the triangle \cite[p.\,70, (1)]{RD}, reduce\- to the trivial case where $E$ itself is a finite\kf-rank free $\OY$-module.)\va1
\end{proof}

\smallskip
The following variant of \cite[4.7.3.4, (a) and (d)]{li} contains the \emph{pseudofunctoriality of} $\chi$
(cf.~the part of \S5.7 in \cite{AJL11} that follows (5.7.3)). 

The proof of \ref{pfchi} that appears here  is abstract; a concrete treatment, via~\ref{ft0}(ii), is left\va1 to the curious reader.

\begin{subprop}[Transitivity of $\chi$]\label{pfchi} 
Let\/ $f,$ $F,$ $G$ and\/ $\chi$ be as in\/~\textup{\ref{flat and tensor}.} 
Let\/ $g \colon W\to X$ be a pseudo\kf-coherent finite scheme-map, assumed perfect if\/~$f$ is perfect. The following natural diagram, with $\chi'\set\chi(g,\OX,f^\flat\<\<F\>)$ and 
$E\set f^\flat\<\<F\Otimes{\sX}\LL f^*G,$ commutes.\va3
\begin{small}
\[\mkern-3mu
\def\1{$(g^\flat\OX\<\<\Otimes{W}\<\LL g^*\!f^\flat\< F)\<\<\Otimes{W}\<\LL g^*\LL f^*\<G$}
\def\2{$g^\flat\<\< f^\flat\< F\<\<\Otimes{W}\<\LL g^*\LL f^*\<G$}
\def\3{$(f\<g)^\flat\< F\<\Otimes{W}\<\LL g^*\LL f^*\<G\ $}
\def\4{$g^\flat\OX\<\<\Otimes{W}\LL g^*\<(f^\flat\< F\Otimes{\sX}\LL f^*\<G\>)\ \,$}
\def\5{$(f\<g)^\flat\< F\<\Otimes{W}\<\LL (f\<g)^{\<*}\<G\ $}
\def\6{$\ \ g^\flat (f^\flat\< F\Otimes{\sX}\LL f^*\<G\>)$}
\def\7{$g^\flat\<\< f^\flat(F\Otimes{Y}G)$}
\def\8{$(f\<g)^\flat(F\Otimes{Y}G\>)$}
 \bpic[xscale=4.27, yscale=1.4]

   \node(11) at (1,-1){\1} ;
   \node(12) at (2.1,-1){\2} ;   
   \node(13) at (3,-1){\3} ;
   
   \node(21) at (1,-2){\4} ;
   \node(23) at (3,-2){\5} ;
    
   \node(31) at (2.1,-2){\6} ;  
   \node(32) at (2.1,-3){\7} ;
   \node(33) at (3,-3){\8} ;
   
    \draw[->] (11)--(12) node[above, midway, ] {$\lift1.7,\via\chi'\>,$} ;
    \draw[->] (12)--(13) ;   
    
    \draw[->] (31)--(32)  node[right=1, midway, scale=.75] {$g^\flat\chi(f,F,G)$} ;
    \draw[->] (32)--(33)  ;
    
    \draw[->] (11)--(21) ;
                                     
    \draw[->] (12)--(31) node[right=1, midway, scale=.75] {$\chi(g,f^\flat\<\< F,\LL f^*\<\<G\>)$} ;
    \draw[->] (21)--(31)  node[below=1, midway, scale=.75]{$\chi(g,\<\OX\<,\<E)$} ;

    \draw[->] (13)--(23)  ;
    \draw[->] (23)--(33)  node[right=1, midway, scale=.75] {$\chi(fg,F,G\>)$} ;
    
    \node at (1.55,-1.52) [scale=.9]{\circled1} ;
    \node at (2.55,-2.02) [scale=.9]{\circled2} ;

  \epic
\]
\end{small}
\end{subprop}

\vskip-3pt
\emph{Remark.} It should be noted that under the assumptions of \ref{flat and tensor}(i), one has
\mbox{$E\cong f^\flat(F\Otimes{Y}G)\in\Dqcpl(X)$;} so in any case, $\chi(g,\OX,E)$ is
well-defined.

\begin{proof}
The commutativity of subdiagram \circled2 is equivalent to that of its adjoint, which, in view of 
\ref{flat and tensor}, is the border of the natural diagram\va3
\begin{small}
\[\mkern-3mu
  \def\1{$\R(f\<g)_*\big(g^\flat\<\< f^\flat\< F\Otimes{W}\LL g^*\LL f^*\<G\big)$}
  \def\2{$\R(f\<g)_*\big((f\<g)^\flat\< F\<\Otimes{W}\<\LL g^*\LL f^*\<G\big)\ $}
  \def\3{$\R\fst\R g_*\big(g^\flat\<\< f^\flat\< F\Otimes{W}\LL g^*\LL f^*\<G\big)$}
  \def\4{$\R\fst(\R g_*g^\flat\<\< f^\flat\< F\Otimes{\sX}\LL f^*\<G\>)$}
  \def\5{$\R(f\<g)_*\big(g^\flat (f^\flat\< F\Otimes{\sX}\LL f^*\<G\>)\big)$}
  \def\6{$\R\fst(f^\flat\< F\Otimes{\sX}\LL f^*\<G\>)$}
  \def\7{$\R(f\<g)_*\big((f\<g)^\flat\< F\<\Otimes{W}\<\LL (f\<g)^{\<*}\<G\big)\ $}
  \def\8{$\R\fst\R g_*g^\flat (f^\flat\< F\Otimes{\sX}\LL f^*\<G\>)$}
  \def\9{$\R\fst\R g_*g^\flat\<\< f^\flat\< F\<\<\Otimes{Y}\<\<G$}
  \def\ten{$\R(f\<g)_*g^\flat\<\< f^\flat(F\<\<\Otimes{Y}\<\<G)$}
  \def\lvn{$\R\fst f^\flat\<F\<\<\Otimes{Y}\<\<G$}
  \def\twv{$\mkern30mu\R(f\<g)_*(f\<g)^\flat\<F\<\<\Otimes{Y}\<\<G$}
  \def\thn{$\R\fst\R g_*g^\flat\<\< f^\flat(F\<\<\Otimes{Y}\<\<G)$}
  \def\frn{$\R\fst f^\flat(F\<\<\Otimes{Y}\<\<G)$}
  \def\ffn{$\mkern10mu F\<\<\Otimes{Y}\<\<G$}
 \bpic[xscale=2.8, yscale=1.4]

   \node(11) at (1,-1){\1} ;
   \node(14) at (4,-1){\2} ;   
   
   \node(22) at (1.9,-2){\3} ;
   \node(24) at (4,-2){\7} ;
   
   \node(33) at (2.8,-2.83){\4} ;
   
   \node(43) at (3.4,-3.66){\9} ;
   \node(44) at (4,-4.5){\twv} ;
   
   \node(51) at (1,-5.33){\5} ;
   \node(52) at (1.9,-4.5){\8} ; 
   \node(53) at (2.8,-5.33){\6} ;  
   
   \node(63) at (3.4,-6.16){\lvn} ;
   
   \node(71) at (1,-7){\ten} ;
   \node(72) at (1.9,-6.16){\thn} ; 
   \node(73) at (2.8,-7){\frn} ;  
   \node(74) at (4,-7){\ffn} ;

    \draw[->] (11)--(14) node[above, midway, ] {$$} ;
    
    \draw[->] (43)--(44) node[above, midway, ] {$$} ;
    
    \draw[->] (22)--(24) node[above, midway, ] {$$} ;   
    
    \draw[->] (51)--(52)  ;
    \draw[->] (52)--(53)  ;
   
    \draw[->] (71)--(72)  ;
    \draw[->] (72)--(73)  ;
    \draw[->] (3.27,-7)--(3.78,-7)  ;
    
    \draw[->] (11)--(51) ;
    \draw[->] (51)--(71) ;
                                     
    \draw[->] (22)--(52) ;
    \draw[->] (52)--(72) ;

    \draw[->] (33)--(53) ;
    \draw[->] (53)--(73) ;

    \draw[->] (43)--(63) ;
 
    \draw[->] (14)--(24) ;
    \draw[->] (24)--(44) ;
    \draw[->] (44)--(74) ;
%
    
    \draw[->] (22)--(33) ;
    \draw[->] (33)--(43) ;
    \draw[->] (53)--(63) ;
    \draw[->] (3.58,-6.4)--(74) ;
    
    \node at (2.35, -3.66) [scale=.9]{\circled3} ;
    \node at (3.72,-2.85) [scale=.9]{\circled4} ;
    \node at (3.72,-5.35) [scale=.9]{\circled5} ;
    \node at (3.4,-6.62) [scale=.9]{\circled6} ;

  \epic
\]
\end{small}
\vskip-5pt
In that diagram, the commutativity of \circled4 and \circled6 is given by \ref{flat and tensor}, 
that of \circled3 follows from \cite[Proposition 3.7.1]{li}, and, with notation as in the first paragraph
of \S\ref{pf adj},  \circled5 is the second (commutative) diagram in~\cite[3.3.7(a)]{li}. Commutativity of the unlabeled subdiagrams is easily checked.

Similarly, the commutativity of \circled1 results from that of all the subdiagrams of the following natural diagram. 
\begin{small}
\[\mkern-3mu
\def\1{$\R g_*\big((g^\flat\OX\<\<\Otimes{W}\LL g^*\!f^\flat\< F)\Otimes{W}\LL g^*\LL f^*\<G\big)$}
\def\2{$\R g_*\big(g^\flat\<\< f^\flat\< F\Otimes{W}\LL g^*\LL f^*\<G\big)$}
\def\3{$\R g_*(g^\flat\OX\<\<\Otimes{W}\LL g^*\!f^\flat\< F)\Otimes{\sX}\LL f^*\<G$}
\def\4{$\qquad\R g_*g^\flat\<\< f^\flat\<F\Otimes{W}\LL f^*\<G$}
\def\5{$\R g_*\big(g^\flat\OX\<\<\Otimes{W}\LL g^*\<(f^\flat\< F\Otimes{\sX}\LL f^*\<G\>)\big)$}
\def\6{$\R g_*g^\flat\OX\<\<\Otimes{\sX}\<\<f^\flat\< F\Otimes{\sX}\LL f^*\<G$}
\def\7{$\OX\!\Otimes{\sX}\<\<f^\flat\<F\<\Otimes{\sX}\<\LL f^*\<G$}
\def\8{$\quad\qquad f^\flat\<F\<\Otimes{\sX}\<\LL f^*\<G$}
 \bpic[xscale=3.93, yscale=1.4]

   \node(11) at (1,-1){\1} ;
   \node(13) at (3,-1){\2} ;   
   
   \node(21) at (1.35,-2){\3} ;
   \node(23) at (3,-2){\4} ;
   
   \node(31) at (1,-4){\5} ;
   \node(32) at (1.35,-3){\6} ;
   \node(33) at (2.6,-3){\7} ;
   \node(34) at (3,-4){\8} ;

    \draw[->] (11)--(13) node[above, midway, ] {$\lift1.4,\via\chi'\>,$} ;
    
    \draw[->] (2,-2)--(2.68,-2) node[above, midway, ] {$\lift1.4,\via\chi'\>,$} ;   
    
    \draw[->] (31)--(32)  ;
    \draw[->] (32)--(33)  ;
    \draw[->] (33)--(34)  ;
    
    \draw[->] (.65,-1.25)--(.65,-3.75) ;
                                     
    \draw[->] (21)--(32) ;
  
    \draw[->] (3.1,-2.25)--(3.1, -3.75)  ;

    \draw[->] (3.1,-1.25)--(3.1,-1.8)  ;
    
    \draw[->] (11)--(21) ;
    
    \node at (1,-2.52) [scale=.9]{\circled7} ;
    \node at (2.25,-1.48) [scale=.9]{\circled8} ;
    \node at (2.25,-2.5) [scale=.9]{\circled9} ;
  \epic
\]
\end{small}
The commutativity of subdiagram \circled7 is given, \emph{mutatis mutandis,} by \cite[3.4.7(iv)]{li}
(with $f$ replaced by $g$, $A\set \LL f^*\<G$, $B\set f^\flat\<F$ and $C\set g^\flat\OX$). The commutativity of \circled8 is clear, and that of \circled9 results from \ref{flat and tensor}.
So the border, and hence \circled1, commutes.
\end{proof}
\end{cosa}

\begin{cosa}\label{Hom and flat}
Let\/ $f\colon X\to Y$ be a pseudo\kf-coherent finite scheme-map\va1 (see \S\ref{quasi}),  let \mbox{$F\in\D(Y)$} be pseudo\kf-coherent, 
and let $G\in\Dqcpl(Y)$.  Let $f=\phi\bar f$ be as in~\S\ref{2.1}, $\pt$ as in~\ref{phiflat}, 
and $f^\flat\set\bar{f\:\<}^{\!\<*}\<\<\pt$ as in~\eqref{fflat}.

By~\ref{qcHom}, $\R\sHom_Y(\fst\OX,G)\in\Dqcpl(Y)$, so  $\pt G\in\Dqcpl(\oY)$. (See the proof of \ref{fst fflat}). Similarly, \eqref{projphi2} gives  
\[
\phi_*\LL\phi^*F\cong  F\Otimes{Y}\phi_*\CO_{\>\oY}=F\Otimes{Y}\fst\OX\in\Dqc(Y),
\]
so $\LL\phi^*F\in\Dqc(\oY)$.

Hence,  by~\ref{^* equivalence}, there are natural isomorphisms\va{-2}
\begin{gather*}
\LL\phi^*\<\<F\cong \R\bar\fst \bar{f\:\<}^{\!\<*}\LL\phi^*\<\<F\cong\R\bar\fst \LL f^*\<\<F,\qquad
\pt G\cong \R\bar\fst \bar{f\:\<}^{\!\<*}\!\pt G=\R\bar\fst f^\flat G.
\end{gather*}
\vskip-2pt
\noindent Moreover,  $\LL f^*\<\<F\in\Dqc(X)$ is pseudo\kf-coherent and $f^\flat G\in\Dqcpl(X)$, so \ref{qcHom} gives 
$\R\>\sHom^{}_\sX(\LL f^*\<\<F,f^\flat G\>)\in\Dqcpl(X)$,
and~\ref{tensor} gives an isomorphism\va{-2}
\begin{equation}\label{phi to f}
\bar{f\:\<}^{\!\<*}\R\>\sHom_{\>\oY}(\LL \phi^*\<\<F\<,\pt G\>)\iso
\R\>\sHom^{}_\sX(\LL f^*\<\<F\<,f^\flat G\>).\\[-5pt]
\end{equation}

\pagebreak[3]
\begin{subprop}\label{flat and hom} Under the preceding conditions$\>:$\va2

{\rm (i)}  The map adjoint under\/ \textup{\ref{adjunction0}} to the natural composite
\[
\phi_{\<*}\R\>\sHom_{\>\oY}(\LL\phi^*\<\<F\<,\pt G\>)\iso\R\>\sHom_Y(F\<,\phi_*\pt G\>)
\xto[\!\!\ref{counit}\!\!]{}\R\>\sHom_Y(F\<,G\>)
\]   
\vskip-3pt
\noindent is the unique\/ $\D(\oY)$-map\/ 
\[
\bar\zeta\colon \R\>\sHom_{\>\oY}(\LL\phi^*\<\<F\<,\pt G\>)\to\pt\R\>\sHom_Y(F\<,G\>)
\]  
making the next, natural, diagram commute---whence\/ $\bar\zeta$ is an isomorphism$\>:$

\[\mkern-4mu
\def\1{$\phi_{\<*}\R\>\sHom_{\oY}(\LL\phi^*\<\>F\<,\pt G\>)$}
\def\2{$\phi_{\<*}\pt\R\>\sHom_Y(F\<,G\>)$}
\def\3{$\R\>\sHom_Y(F\<,\phi_{\<*}\pt\<G\>)$}
\def\4{$\R\>\sHom_Y(F\<,\R\>\sHom_Y(\fst\OX\<,\>G\>)\>)$}
\def\5{$\R\>\sHom_Y(F\Otimes{Y}\<\fst\OX\<,\>G\>)$}
\def\6{$\R\>\sHom_Y(\fst\OX,\R\>\sHom_Y(F\<,G\>))$}  
  \bpic[xscale=3.7, yscale=1.75]
     
   \node(11) at (1,-1){\1};  
   \node(13) at (3,-1){\2} ;
   
   \node(21) at (1,-2){\3} ;  
   
   \node(31) at (1,-3){\4};   
   \node(32) at (2,-2.4){\5}; 
   \node(33) at (3,-3){\6};
  
    \draw[->] (11)--(13) node[above, midway, scale=.75] {$\phi_*\bar\zeta$} ;  
      
    \draw[->] (11)--(21) node[left=1pt, midway, scale=.75] {$\simeq$} ;
    \draw[->] (21)--(31) node[left=1pt, midway, scale=.75] {$\simeq$} 
                                   node[right, midway, scale=.75] {\eqref{phiRHompsi}} ;
                                   
    \draw[->] (13)--(33) node[left, midway, scale=.75] {{\eqref{phiRHompsi}}}
                                   node[right=1pt, midway, scale=.75] {$\simeq$} ;
                                   
   \draw[->] (31)--(32) node[above=-1.9pt, midway, scale=.75] {\rotatebox{14}{$\Iso$}} ;
   \draw[->] (32)--(33) node[above=-1.9pt, midway, scale=.75] {\rotatebox{-14}{$\Iso$}} ;
   
  \node at (2.03,-1.75) [scale=.9]{\circled1} ;
  
  \epic
\]

\pagebreak[3]
{\rm(ii) (Explicit $\bar\zeta\>$)} Suppose that\/ $G$ $($hence $\sHom_\psi(\fst\OX,G\>))$ is K-injective\va{.5} and that\/ $F$  is K-flat. Then\/ $\bar\zeta$~is the natural composite map
\begin{align*}
\R\>\sHom_{\>\oY}(\LL\phi^*\<\<F\<,\pt G\>)
&\<\<\iso\!
\sHom_{\>\oY}(\phi^*\<\<F\<,\sHom_\psi(\fst\OX,G\>)\>)\\[2pt]
&\<\underset{\lift1.4,\bar\zeta^{}_0,}{\<\iso}\!
\sHom_\psi(\fst\OX\<,\sHom_Y(F\<,G\>)\>)
\<\<\iso\!
\pt\R\>\sHom_Y(F\<,G\>),
\end{align*}
\noindent with $\bar\zeta^{}_0$  the natural composite map---$\fst\OX$-linear via the multiplication action of\/ $\fst\OX$ on itself,
\begin{align*}
\phi_*\sHom_{\>\oY}(\phi^*\<\<F\<,\sHom_\psi(\fst\OX,G\>)\>)&\iso
\sHom_Y(F\<,\phi_*\sHom_\psi(\fst\OX,G\>)\>)\\
&\mkern9.5mu=\!\!=\,
\sHom_Y(F\<,\sHom_Y(\fst\OX,G\>)\>)\\
&\iso
\sHom_Y(F\otimes_Y\<\fst\OX, G\>)\\
&\iso
\sHom_Y(\fst\OX,\sHom_Y(F\<, G\>)\>)\\
&\mkern9.5mu=\!\!=\,
\phi_*\sHom_\psi(\fst\OX,\sHom_Y(F\<, G\>)\>).
\end{align*}

\vskip3pt
{\rm(iii)} The composite isomorphism  
\begin{align*}
\R\>\sHom^{}_\sX(\LL f^*\<\<F,f^\flat G\>)
&\underset{\eqref{phi to f}}\iso
\bar{f\:\<}^{\!\<*}\R\>\sHom_{\>\oY}(\LL \phi^*\<\<F\<,\pt G\>)\\
&\underset{\bar{f\:\<}^{\!\<*}\<\<\bar\zeta}\iso 
\bar{f\:\<}^{\!\<*}\<\<\pt\R\>\sHom_Y(F\<,G\>)=
f^\flat\R\>\sHom_Y(F,G\>)
\end{align*}
is the unique\/ $\D(X)$-map~$\zeta$ making the following natural diagram commute$\>:$
\[\mkern-2mu
\def\1{$\R\fst f^\flat\R\>\sHom_Y(F,G\>)$}
\def\2{$\R\fst \R\>\sHom_\sX(\LL f^*\<\<F,\>f^\flat G\>)$}
\def\3{$\R\>\sHom_Y(\fst\OX,\R\>\sHom_Y(F,G\>));$}
\def\4{$\R\>\sHom_Y(F,\R\fst f^\flat G\>)$}
\def\5{$\R\>\sHom_Y(F,\R\>\sHom_Y(\fst\OX, G\>))$}
\def\6{$\R\>\sHom_Y(F\Otimes{Y}\<\fst\OX, \>G\>)$}
  \bpic[xscale=7.4, yscale=1.65]

   \node(11) at (2,-1){\1};   
   \node(12) at (1,-1){\2}; 
   
   \node(21) at (1,-2){\4};
   
   \node(31) at (1,-3){\5};
   \node(32) at (2,-3){\3};
   \node(30) at (1.5,-2.475){\6};

    \draw[<-] (11)--(12) node[above, midway, scale=.75] {$\R\fst\zeta$} ;

    \draw[->] (12)--(21) node[left=1pt, midway, scale=.75] {$\simeq$} ;
    \draw[->] (21)--(31) node[left=1pt, midway, scale=.75] {$\simeq$} 
                                   node[right=1pt, midway, scale=.75] {$\textup{\ref{fst fflat}}$} ;
    \draw[->] (11)--(32) node[left=1pt, midway, scale=.75] {$\textup{\ref{fst fflat}}$}
                                   node[right=1pt, midway, scale=.75] {$\simeq$} ;
                                   
   \draw[->] (31)--(30) node[above=-1.9pt, midway, scale=.75] {\rotatebox{14}{$\Iso$}} ;
   \draw[->] (30)--(32) node[above=-1.9pt, midway, scale=.75] {\rotatebox{-14}{$\Iso$}} ;
      
   \node at (1.52,-1.75)[scale=.9] {\circled2} ;
  
 \epic
\]
and this $\zeta$ corresponds under\/ \textup{\ref{represent}} to the  natural composite map
\[
\R\fst \R\>\sHom_\sX(\LL f^*\<\<F,\>f^\flat G\>)
\iso\R\>\sHom_Y(F,\R\fst f^\flat G\>)
\lto\R\>\sHom_Y(F\<,G\>).
\]
\end{subprop}

 \begin{proof}
First, uniqueness in (iii) and (i). Consider the natural diagram
\[\mkern-2mu
\def\1{$\R\fst f^\flat\R\>\sHom_Y(F,G\>)$}
\def\2{$\R\fst \R\>\sHom_\sX(\LL f^*\<\<F,\>f^\flat G\>)$}
\def\3{$\R\>\sHom_Y(\fst\OX,\R\>\sHom_Y(F,G\>))\quad $}
\def\4{$\R\>\sHom_Y(F,\R\fst f^\flat G\>)$}
\def\5{$\R\>\sHom_Y(F,\R\>\sHom_Y(\fst\OX, G\>))$}
\def\6{$\R\>\sHom_Y(F\Otimes{Y}\<\fst\OX, \>G\>)$}
\def\7{$\R\>\sHom_Y(\OY\<,\R\>\sHom_Y(F,G\>))$}
\def\8{$\R\>\sHom_Y(F,\R\>\sHom_Y(\OY\<, G\>))$}
\def\9{$\R\>\sHom_Y(F\Otimes{Y}\<\OY\<, \>G\>)$}
\def\ten{$\R\>\sHom_Y(F, G\>)$}
  \bpic[xscale=7.4, yscale=1.4]

   \node(11) at (2,-1){\1};   
   \node(12) at (1,-1){\2}; 
   
   \node(21) at (1,-2){\4};
   
   \node(31) at (1,-3.1){\5};
   \node(32) at (2,-3.1){\3};
   \node(30) at (1.5,-2.475){\6}; 
   
   \node(42) at (1,-4.3){\8};
   \node(43) at (2,-4.3){\7};
   \node(40) at (1.5,-3.675){\9};
   
   \node(50) at (1.5,-4.9){\ten};
   
    \draw[<-] (11)--(12) node[above, midway, scale=.75] {$\R\fst\zeta$} ;

    \draw[->] (12)--(21) node[left=1pt, midway, scale=.75] {$\simeq$} ;
    \draw[->] (21)--(31) node[left=1pt, midway, scale=.75] {$\simeq$} 
                                   node[right=1pt, midway, scale=.75] {$\textup{\ref{fst fflat}}$} ;
    \draw[->] (31)--(42) ;
                                    
    \draw[->] (30)--(40) ;                                
    \draw[->] (40)--(50) ;
                                   
   \draw[->] (11)--(32) node[left=1pt, midway, scale=.75] {$\textup{\ref{fst fflat}}$}
                                   node[right=1pt, midway, scale=.75] {$\simeq$} ;
   \draw[->] (32)--(43) ;
                                   
   \draw[->] (31)--(30) node[above=-1.9pt, midway, scale=.75] {\rotatebox{14}{$\Iso$}} ;
   \draw[->] (30)--(32) node[above=-1.9pt, midway, scale=.75] {\rotatebox{-14}{$\Iso$}} ;
   \draw[->] (42)--(50) node[above=-1.9pt, midway, scale=.75] {\rotatebox{-14}{$\Iso$}} ;
   \draw[<-] (43)--(40) node[above=-1.9pt, midway, scale=.75] {\rotatebox{-14}{$\Iso$}} ;
   \draw[->] (43)--(50) node[above=-1.9pt, midway, scale=.75] {\rotatebox{14}{$\Iso$}} ;
   \draw[->] (42)--(40) node[above=-1.9pt, midway, scale=.75] {\rotatebox{14}{$\Iso$}} ;

   \node at (1.52,-1.75)[scale=.9] {\circled2} ;
   \node at (1.39,-4.32) [scale=.9]{\circled3} ;
   \node at (1.585,-4.32) [scale=.9]{\circled4} ;
   
 \epic
\]
\vskip-5pt
\noindent
After replacing $G$ by a quasi-isomorphic K-injective complex, one can drop all the $\R$\kf s in \circled3 and \circled4
and check that the resulting subdiagrams---hence the original ones---are commutative.\va2

\begin{small}
More generally, in any closed category, using the definitions of the maps involved 
(see \cite[3.5.6(e) and 3.5.3(e)]{li}) one checks that subdiagrams \circled3 and \circled 4 are 
right-conjugate to the (clearly) commutative natural diagrams
\[
\def\7{$D\Otimes{Y}(F\Otimes{Y}\<\OY)$}
\def\8{$(D\Otimes{Y}F\>)\Otimes{Y}\<\OY$}
\def\9{$(D\Otimes{Y}\OY)\Otimes{Y}\<F$}
\def\ten{$D\Otimes{Y}F$}
  \bpic[xscale=7.4, yscale=1.4]

   \node(41) at (1,-4.3){\8};
   \node(42) at (2,-4.3){\9};
   
   \node(40) at (1.5,-3.675){\7};
   
   \node(50) at (1.5,-4.9){\ten};
   
  
    \draw[<-] (40)--(50) ;
                                    
   \draw[<-] (41)--(50) node[above=-1.2pt, midway, scale=.75] {\rotatebox{-14}{$\Iso$}} ;
   \draw[<-] (41)--(40) node[above=-1.2pt, midway, scale=.75] {\rotatebox{14}{$\Iso$}} ;
   \draw[<-] (50)--(42) node[above=-1.2pt, midway, scale=.75] {\rotatebox{14}{$\Iso$}} ;
   \draw[<-] (40)--(42) node[above=-1.2pt, midway, scale=.75] {\rotatebox{-14}{$\Iso$}} ;

   \node at (1.35,-4.3) [scale=.9]{\circled3$'$} ;
   \node at (1.65,-4.3) [scale=.9]{\circled4$'$} ;
   
 \epic
\]
\end{small}
\vskip-15pt
 It follows, in view of the definition of the counit $t$ for the adjunction $\R\fst\!\dashv f^\flat\<$ in~\ref{right adjoint}, and of \ref{qcHom},
that any $\D(X)$-map~$\zeta$ such that \circled2 commutes must be the one corresponding under \ref{represent}  to the natural composite\va{-2}
\[
\R\fst \R\>\sHom_\sX(\LL f^*\<\<F,\>f^\flat G\>)
\iso\R\>\sHom_Y(F,\R\fst f^\flat G\>)
\lto\R\>\sHom_Y(F\<,G\>),\\[-2pt]
\]
whence the uniqueness\va{1}  in (iii).%
\footnote
{In fact, $\zeta$ is right-conjugate to 
the projection isomorphism 
\[
\R\fst E\Otimes{\sX}F\iso \R\fst(E\Otimes{\sX}\LL f^*\<\<F\>),
\] 
cf.~\cite[Exercise 4.2.3(f)]{li}. Moreover, $\zeta^{-1}$ is adjoint to the natural composite map
\[
f^\flat\R\>\sHom_Y(F\<,G\>)\Otimes{X}\LL f^*\<\<F
\xto[\lift1.5,\stackrel{\scriptstyle\textup{\ref{flat and tensor}}}{\textup{(iii)}}\>,]{\chi}
f^\flat(\R\>\sHom_Y(F\<,G\>)\Otimes{Y}F\>)\lto f^\flat G,
\]\vskip1pt
cf.~\cite[Exercise 4.9.3(b)]{li} (in whose third last line ``$\>\bar{\!f\>}^!$" should be ``$f^!$").
} 

Similarly, using \ref{counit} one shows that any $\bar\zeta$ making \circled1 commute
is adjoint under~\ref{adjunction0} to the natural composite
\[
\phi_{\<*}\R\>\sHom_{\>\oY}(\LL\phi^*\<\<F\<,\pt G\>)\iso\R\>\sHom_Y(F\<,\phi_*\pt G\>)
\lto\R\>\sHom_Y(F\<,G\>),
\]   
whence the uniqueness in (i).\va2

This being so,  (i) and (ii) can be proved thus: assuming (as one may) that $G$ (resp.~$F$)~
is a K-injective (resp.~K-flat) $\OY$-complex, show that with~$\bar\chi$ as in (ii), diagram \circled1 commutes; and for this, note that each subdiagram
of the following natural diagram, where $\CH\set\sHom$,  commutes (more or less by definition of the functorial maps involved), whence so does the border.
\begin{small}
\[\mkern-5mu
\def\1{$\phi_*\R\CH_{\oY}(\LL\phi^*\<\<F\<,\pt G\>)$}
\def\2{$\phi_*\CH_{\oY}(\phi^*\<\<F\<,\CH_\psi(\fst\OX,G\>)\>)\ $}
\def\3{$\R\CH_Y(F\<,\phi_*\pt G\>)$}
\def\4{$\R\CH_Y(F\<,\phi_*\R\CH_\psi(\fst\OX,G\>)\>)$}
\def\5{$\CH_Y(F\<,\phi_*\CH_\psi(\fst\OX,G\>)\>)$}
\def\6{$\CH_Y(F\<,\CH_Y(\fst\OX,G\>)\>)$}
\def\7{$\CH_Y(F\otimes_Y\<\<\fst\OX, G\>)$}
\def\8{$\CH_Y(\fst\OX,\CH_Y(F\<, G\>)\>)$}
\def\9{$\phi_*\pt\R\CH_Y(F\<, G\>)$}
\def\ten{$\phi_*\CH_\psi(\fst\OX,\CH_Y(F\<, G\>)\>)$}
\def\lvn{$\R\CH_Y(F\<,\R\CH_Y(\fst\OX,G\>)\>)$}
\def\twv{$\R\CH_Y(F\Otimes{Y}\<\<\fst\OX,G\>)$}
\def\thn {$\R\CH_Y(\fst\OX,\R\CH_Y(F\<,G\>)\>)$}
  \bpic[xscale=4.4, yscale=1.3]
  
   \node(11) at (1,-1){\1};   
   \node(13) at (3,-1){\2}; 

   \node(21) at (1,-2){\3};
   \node(22) at (1.925 ,-2){\4};   
   \node(23) at (3,-2){\5}; 
   
   \node(31) at (1,-3){\lvn} ;
   \node(33) at (3,-3){\6} ;  
   
   \node(42) at (1 ,-4){\twv};
   \node(44) at (3,-4){\7}; 
   
   \node(52) at (1 ,-6){\9};   
   \node(53) at (3,-6){\ten}; 
   
   \node(61) at (1,-5){\thn} ;
   \node(63) at (3,-5){\8} ;  

   \draw[->] (11)--(13) node[above=1pt, midway, scale=.75] {$\Iso$} ;
   
   \draw[double distance=2pt] (21)--(22) ;  
   \draw[<-] (22)--(23) ; 
   
   \draw[<-] (31)--(33) ;
   
   \draw[<-] (42)--(44) ;
    
   \draw[<-] (52)--(53) ;
   
   \draw[<-] (61)--(63) ;
   
    \draw[->] (11)--(21) node[left=1pt, midway, scale=.75] {$\simeq$} ;
    \draw[->] (21)--(31) node[left=1pt, midway, scale=.75] {$\simeq$} 
                                   node[right, midway, scale=.75]{\eqref{phiRHompsi}} ; 
    \draw[->] (31)--(42) node[left=1pt, midway, scale=.75] {$\simeq$} ;
    \draw[->] (42)--(61) node[left=1pt, midway, scale=.75] {$\simeq$} ;

    \draw[->] (13)--(23) node[right=1pt, midway, scale=.75] {$\simeq$} ;
    \draw[double distance=2pt] (23)--(33);
    \draw[->] (33)--(44) node[right=1pt, midway, scale=.75] {$\simeq$} ;
    \draw[->] (44)--(63) node[right=1pt, midway, scale=.75] {$\simeq$} ;
    
     \draw[double distance=2pt] (63)--(53) ;
     \draw[->] (52)--(61) node[left=1pt, midway, scale=.75] {$\simeq$}
                                     node[right=1pt, midway, scale=.75]{\eqref{phiRHompsi}} ;
    \draw[->] (22)--(31) node[above=-1pt, midway, scale=.75] {$\simeq\mkern30mu$} 
                                   node[below, midway, scale=.75] {$\mkern50mu\eqref{phiRHompsi}$} ;
     \node at (1.925,-1.5)[scale =.85] {\textup{\cite[3.2.3(ii)]{li}}} ;
     \node at (1.925,-3.5)[scale =.85] {\textup{ \cite[2.6.1$^*$]{li}}} ;
     \node at (1.925,-4.5)[scale =.85] {\textup{\cite[2.6.1$^*$]{li}}} ;
  \epic
\]
\end{small}
\vskip-10pt
As for (iii),  to see that the following natural diagram expands \circled2, apply\- the first diagram in \cite[3.7.1.1]{li} to the leftmost column. Since by~(i),
subdiagram~\circled1 commutes, therefore for (iii) to hold---i.e., for the border to commute---it suffices that all the~other subdiagrams commute.
\begin{small}           
\[\mkern-4mu
\def\1{$\R\fst \R\>\CH_\sX(\LL f^*\<\<F\<,\>f^\flat G\>)$}
\def\3{$\phi_*\R\bar\fst \R\>\CH_\sX(\LL f^*\<\<F\<,\>f^\flat G\>)$}
\def\4{$\R\fst\bar{f\:\<}^{\!\<*}\R\CH_{\oY}(\LL \phi^*\<\<F\<,\pt G\>)$}
\def\5{$\R\fst f^\flat\R\>\CH_Y(F\<,G\>)$}
\def\6{$\phi_*\R\bar\fst\bar{f\:\<}^{\!\<*}\R\CH_{\oY}(\LL \phi^*\<\<F\<,\pt G\>) $}
\def\7{$\phi_*\R\bar\fst \bar{f\:\<}^{\!\<*}\pt \R\>\CH_Y(F\<,G\>)$}
\def\8{$\phi_*\R\>\CH_{\oY}(\LL\phi^*\<\<F\<,\R\bar\fst \bar{f\:\<}^{\!\<*}\pt G\>)$}
\def\9{$\phi_*\R\>\CH_{\oY}(\LL\phi^*\<\<F\<,\pt G\>)$}
\def\ten{$\phi_*\pt \R\>\CH_Y(F\<,G\>)$}
\def\lvn{$\R\>\CH_Y(F\<,\phi_* \R\bar\fst \bar{f\:\<}^{\!\<*}\pt  G\>)$}
\def\twv{$\R\>\CH_Y(F\<,\R\fst f^\flat G\>)$}
\def\thn{$\R\>\CH_Y(F\<,\phi_*\pt G\>)$}
\def\frn{$\R\>\CH_Y(F\Otimes{Y}\<\<\fst\OX, \>G\>)$}
\def\ffn{$\R\>\CH_Y(F\<,\R\>\CH_Y(\fst\OX, G\>))$}
\def\sxn{$\R\>\CH_Y(\fst\OX,\R\>\CH_Y(F\<,  G\>))$}
  \bpic[xscale=2.98, yscale=1.6]

   \node(11) at (1,-1){\1};   
   \node(13) at (2.45,-1){\4}; 
   \node(14) at (3.9,-1){\5};
   
   \node(22) at (1,-2){\3};   
   \node(23) at (2.45,-1.55){\6}; 
   \node(24) at (3.9,-2){\7};

   \node(32) at (1,-3){\8};
   \node(33) at (2.45,-3){\9};
   \node(34) at (3.9,-3){\ten}; 
   
   \node(42) at (1,-4){\lvn};
   \node(43) at (2.45,-4){\thn};
   
   \node(51) at (1,-5){\twv};
   \node(52) at (2.45,-5){\ffn};   
   \node(53) at (3.2,-4.5){\frn}; 
   \node(54) at (3.9,-5){\sxn};
   
    \draw[->] (11)--(13) node[above, midway, scale=.75] {$\Iso$} 
                                   node[below=1pt, midway, scale=.75] {$\eqref{phi to f}$} ; 
     \draw[->] (13)--(14) node[above, midway, scale=.75] {$\Iso$} 
                                   node[below=1pt, midway, scale=.75] 
                                                             {$\R\fst\bar{f\:\<}^{\!\<*}\<\<\bar\zeta\>\>$} ;
                                                              
    \draw[->] (32)--(33) node[below=1pt, midway, scale=.75] {{\textup{\ref{^* equivalence}}}}
                                   node[above, midway, scale=.75] {$\Iso$} ;
   \draw[->] (33)--(34) node[above, midway, scale=.75] {$\phi_*\bar\zeta$} ;    
   
   \draw[->] (42)--(43) node[below=1pt, midway, scale=.75] {{\textup{\ref{^* equivalence}}}}
                                   node[above, midway, scale=.75] {$\Iso$} ;

   \draw[->] (51)--(52)  node[above, midway, scale=.75] {$\Iso$} 
                                   node[below=1pt, midway, scale=.75] {$\eqref{fst fflat}$} ; 
                                   
     \draw[double distance=2pt] (11)--(22) ;
     \draw[->] (22)--(32) node[left=1pt, midway, scale=.75] {$\simeq$} ;
     \draw[->] (32)--(42) node[left=1pt, midway, scale=.75] {$\simeq$} ;
     \draw[double distance=2pt] (42)--(51) ;
       
     \draw[double distance=2pt] (13)--(23) ;
     \draw[->] (33)--(43) node[left=1pt, midway, scale=.75] {$\simeq$} ;
    
     \draw[double distance=2pt] (14)--(24) ; 
     \draw[->] (24)--(34) node[left=1pt, midway, scale=.75] {{\textup{\ref{^* equivalence}}}}
                                   node[right=1pt, midway, scale=.75] {$\simeq$} ;
                                   
     \draw[->] (23)--(33) node[left=1pt, midway, scale=.75] {{\textup{\ref{^* equivalence}}}}
                                   node[right=1pt, midway, scale=.75] {$\simeq$} ;
      \draw[->] (34)--(54)  node[right=1pt, midway, scale=.75] {$\simeq$} 
                                       node[left, midway, scale=.75] {\eqref{phiRHompsi}} ;

    \draw[->] (43)--(52) node[right, midway, scale=.75] {$\simeq$} 
                                   node[left, midway, scale=.75] {\eqref{phiRHompsi}} ;
   \draw[->] (22)--(23) node[above=-.5, midway, scale=.75] {\rotatebox{10}{$\Iso\mkern5mu$}} 
                                   node[below, midway, scale=.75] {$\mkern50mu\eqref{phi to f}$} ;
   \draw[->] (23)--(24) node[above=-.5pt, midway, scale=.75] {$\rotatebox{-10}{$\Iso$}$} 
                                   node[below, midway, scale=.75] {$\via\>\bar\zeta\mkern40mu$} ;
     
   \draw[->] (52)--(53) node[above=-2pt, midway, scale=.75] {\rotatebox{21}{$\Iso\mkern10mu$}} ;
   \draw[->] (53)--(54) node[above=-2pt, midway, scale=.75] {\rotatebox{-23}{$\ \Iso$}} ;
   
   \node at (3.15, -3.8)[scale=.9]{\circled1} ;
   \node at (1.75, -2.43)[scale=.9]{\circled5} ;
 \epic
\]
\end{small} 
Commutativity  of \circled5 results, in view of \cite[Example 3.5.2(d)]{li}, from the definition of the map $\rho$ 
given in the proof of Corollary~\ref{tensor}.

Commutativity of the remaining subdiagrams is obvious.
\end{proof}

\begin{subprop}[Transitivity of $\zeta$]\label{pfzeta} 
Let\/ $\zeta=\zeta(f,F\<,G)$  be as in\/~\textup{\ref{flat and hom}(iii).} 
Let\/ $g \colon W\to X$ be a pseudo\kf-coherent finite scheme-map.
Then, modulo natural isomorphisms,
\[
\zeta(gf,F\<,G)=g^\flat\zeta(f,F\<,G)\smallcirc \zeta(g,\LL f^*\<\<F\<,f^\flat G).
\]
\end{subprop}

\begin{proof}
Left to the reader.
\end{proof}
\end{cosa}

\begin{cosa}\label{excompl}
This section deals with the role played in concrete duality for a perfect affine map $f\colon X\to Y$ by the functorial \emph{trace~map} 
\[
\tr_{\<\<f}(G)\colon\R\fst\LL f^*\<G\to G\qquad(G\in\Dqc(Y))
\]
from \cite[p.\,154, 8.1]{Il}.  Dual to this map is the \emph{fundamental class} map  
\[
C_{\<\<f}(G)\colon \LL f^*\<G\lto f^\flat G\qquad(G\in\Dqc(Y)),
\]
an isomorphism if $f$ is \'etale.  

The trace map is ``\kf transitive" with respect to a composition of perfect affine maps, and so the map $C_{\<\<f}$ is \emph{pseudofunctorial,}
see Proposition~\ref{transfund}. 

For ``almost \'etale" $f\colon X\to Y\<,$ there results a
canonical pair involving a ``complementary sheaf\kern1.5pt" and a trace map, which pair represents the functor\- 
$\Hom_Y(\fst-,G\>)$ from~$\sA_{\qc}(X)$ to 
abelian groups, or, if $\fst\OX$ is locally free, the functor
$\Hom_{\D(Y)}(\R\fst-, \:G\>)$ from\/ $\Dqc(X)$ to abelian groups,
see Proposition~\ref{complementary}.

\va2

\begin{subcosa}\label{trprelim}
Let  $f\colon X\to Y\<$, $\bar f\colon X\to\oY$,  $\psi\colon\OY\to\fst\OX$, $\phi\colon\oY\to Y$ and $f^\flat\colon\D(Y)\to\D(X)$ be as in \S\ref{2.4}. Let  $F$ be a quasi-coherent
$\OX$-module and $G$ an $\OY$-module such that 
the $\OY$-module $\sHom_Y(\fst\OX,G\>)$---equivalently, 
the $\CO_{\>\oY}$-module $\sHom_\psi(\fst\OX,G\>)$---is quasi-coherent.\va1

One has the  iso\-morphism, \emph{affine duality for quasi-coherent sheaves,}
\begin{align*}
\fst\sHom_{\OX}(F,\bar{f\:\<}^{\!\<*}\sHom_\psi(\fst\OX, G\>))
&\iso\phi_*\sHom_{\<\fst\<\OX}\<\<(\bar\fst F,\sHom_\psi(\fst\OX, G\>))\\
&\iso \sHom_{\OY}\<\<(\fst F, G\>).
\end{align*}
that over open $U\subset Y$
is the standard isomorphism, with $S_U\set\Gamma(f^{-\<1}U,\OX)$,
$F_U\set\Gamma(f^{-\<1}U,F)$, $R_U\set\Gamma(U,\OY)$ and $G_U\set\Gamma(U,G)$,
\[
\Hom_{S_U\<}(F_U\>,\>\Hom_{R_U\<}(S_U,G_U))\iso\Hom_{R_U\<}(F_U\>,\>G_U).
\]

If $\>\R\sHom_Y(\fst\OX,G)\in\Dqc(Y)$, then this duality isomorphism arises from application of~$H^0$ to the isomorphism in Theorem~\ref{qc duality}.\va1

It follows that the pair consisting of 
$\bar{f\:\<}^{\!\<*}\sHom_\psi(\fst\OX, G\>)\ (=H^0\<\<f^\flat\< G\>)$
and the natural composite 
\begin{align*}
t'_G\colon\fst \bar{f\:\<}^{\!\<*}\sHom_\psi(\fst\OX, G\>)&\iso \phi_*\sHom_\psi(\fst\OX, G\>)\\
&\ =\!\!=\;\sHom_Y(\fst\OX,G\>)
\lto \sHom_Y(\OY\<,G\>)= G
\end{align*}
represents the functor $\Hom_Y(\fst-,G\>)$ from $\sA_\qc(X)$ to abelian groups. \va2

If $f$~is \emph{flat and locally finitely presentable,} that is, the $\OY$-module $\fst\OX$ is locally free of finite rank, 
then the natural map is an isomorphism
\[
\bar{f\:\<}^{\!\<*}\sHom_\psi(\fst\OX, G\>)=H^0\<\<f^\flat G\iso f^\flat G,
\] 
and Proposition~\ref{represent} implies that 
the pair above also represents the functor 
$\Hom_{\D(Y)}(\R\fst-, \:G\>)$ from\/ $\Dqc(X)$ to abelian groups. 

\pagebreak[3]
In this case, moreover, 
the natural map is an $\fst\OX$-isomorphism
\[
\sHom_\psi(\fst\OX, \OY)\otimes_{\oY} \phi^*\<G\iso\sHom_\psi(\fst\OX, G\>),
\]
whence the representing pair is naturally isomorphic to the pair
\[
(H^{\>0}\<\<f^\flat\OY\<\otimes_\sX \<f^*\<G\<,\,\> t'_{\OY}\!\otimes_Y\<\id_G\>).\\[2pt]
\]

If $f$ is \emph{finite and \'etale,} so that the $\CO_{\>\oY}$-module
$\sHom_\psi(\fst\OX, \OY\>)$ is free of rank one, generated by the usual \emph{trace map}
$\tr_{\<\<f}\colon \fst\OX\to\OY$, then
\[
H^{\>0}\<\<f^\flat\OY\otimes_\sX \<f^*\<G\cong \bar{f\:\<}^{\!\<*}\<\CO_{\>\oY}\otimes_\sX \<f^*\<G\cong\OX\otimes_\sX \<f^*\<G\cong f^*\<G,
\] 
and the representing pair is naturally isomorphic to the pair consisting of  $f^*\<G$ and the natural composite map
\[
 \R\fst f^*\<G \iso \fst\OX\otimes_Y G\xto{\!\tr_{\!f}\otimes\>\id_G}\OY\otimes_Y G\iso G.
\]
\end{subcosa}

\begin{subcosa}\label{hbar}
For dealing with more general $f$, recall from~\ref{gamma0} the trifunctorial map, over any 
ringed space $(Y,\OY)$,
\[
\gamma^{}_{\>Y}\colon\R\>\sHom_Y(L,M)\Otimes{Y} N\to \R\>\sHom_Y(L,M\Otimes{Y}N)\qquad(L,M,N\in\D(Y)),
\]
that is adjoint to the natural composition
\[
\R\>\sHom_Y(L,M)\Otimes{Y} N\Otimes{Y} L
\iso \R\>\sHom_Y(L,M)\Otimes{Y} L\Otimes{Y} N 
\lto M\Otimes{Y}N. 
\]

Elaborating \ref{gamma},  one finds that $\gamma^{}$ can be obtained from a K-injective 
resolution $M\to\>\>\overline{\<\<M}$ and a K-flat resolution~$\>\>\overline{\<\<N}\to N\<$ as the sheafification
\[
\sHom_Y( L, \>\>\overline{\<\<M}\>\>)\otimes_Y  \>\>\overline{\<\<N}\to 
\sHom_Y( L, \>\>\overline{\<\<M}\otimes_Y \>\>\overline{\<\<N}\>\>)
\]
of the map of presheaves of complexes that sends, for each open $U\subset Y\<$, any 
\[
(\alpha\colon  L^i|_U\to \>\>\overline{\<\<M}^{\>j}|_U)\otimes_{\CO_U} (\mu_k\in \>\>\overline{\<\<N}^{\>k}(U))
\qquad(i,j,k\in\mathbb Z)
\]
to the map taking $\lambda_i\in L^i(U)$ to 
$(-1)^{ik}\alpha(\lambda_i)\otimes_{\CO_U}  \mu_k.$

\begin{sublem}\label{^* and hbar}
\textup{(i)} For any ringed-space map $g\colon Y'\to Y$, the following natural diagram commutes.
\[
\def\1{$\LL g^*(\R\>\sHom_Y(L,M)\Otimes{Y} N)$}
\def\2{$\LL g^*\R\>\sHom_Y(L,M\Otimes{Y} N)$}
\def\3{$\LL g^*\R\>\sHom_Y(L,M)\Otimes{Y'} \LL g^*N$}
\def\4{$\R\>\sHom_{Y'}(\LL g^*\<\<L,\LL g^*\<(M\Otimes{Y} N))$}
\def\5{$\R\>\sHom_{Y'}(\LL g^*\<\<L,\LL g^*\<\<M)\Otimes{Y'} \LL g^*\<\<N$}
\def\6{$\R\>\sHom_{Y'}(\LL g^*\<\<L,\LL g^*\<\<M\Otimes{Y'} \LL g^*\<\<N)$}
  \bpic[xscale=6.9, yscale=1.6]
     
   \node(11) at (1,-1){\1};  
   \node(12) at (2,-1){\2} ;
   
   \node(21) at (1,-2){\3} ;  
   \node(22) at (2,-2){\4};   
   
   \node(31) at (1,-3){\5}; 
   \node(32) at (2,-3){\6};
  
    \draw[->] (11)--(12) node[above, midway, scale=.75] {$\LL g^*\<\gamma^{}_{\>Y}$} ;  
    \draw[->] (31)--(32) node[below, midway, scale=.75] {$\gamma^{}_{\>Y'}$} ;
      
    \draw[->] (11)--(21) node[left, midway, scale=.75] {$\simeq$} ;
    \draw[->] (21)--(31) ;
                                   
   \draw[->] (12)--(22) ;
    \draw[->] (22)--(32) node[right, midway, scale=.75] {$\simeq$} ;
    
  \epic
\]

\textup{(ii)} If the complex $L$ is perfect, then all the maps in this diagram are isomorphisms. 
\end{sublem}

\begin{proof}
(i) It suffices to show the commutativity of the adjoint diagram, i.e., of the border of the following natural diagram, in which $g^*$ stands for $\LL g^*\<$, $\CH$~for $\R\>\sHom$, $\otimes$ for $\Otimes{Y}$
and $\otimes'$ for $\Otimes{Y'}\>$, and $\alpha$ is adjoint to the identity map of $\CH(L,M\otimes N)$:
\begin{small}
\[\mkern-4mu
\def\1{$g^*(\CH_Y(L,M)\otimes N)\otimes'\<\< g^*\<\<L$}
\def\2{$g^*\CH_Y(L,M\otimes N)\otimes'\<\< g^*\<\<L$}
\def\3{$g^*(\CH_Y(L,M)\otimes N\otimes L)$}
\def\4{$g^*(\CH_Y(L,M\<\otimes\< N)\<\otimes\< L)$}
\def\5{$\ \ \CH_{Y'}(g^*\<\<L,g^*(M\otimes N))\otimes'\<\< g^*\<\<L$}
\def\6{$g^*\CH_Y(L,M)\otimes'\<\< g^*\<\<N\otimes'\<\< g^*\<\<L$}
\def\7{$g^*(\CH_Y(L,M)\otimes L\otimes N)$}
\def\8{$g^*(M\otimes N)$}
\def\9{$g^*\CH_Y(L,M)\otimes'\<\< g^*L\otimes'\<\< g^*N$}
\def\ten{$g^*(\CH_Y\<(L,\<M)\<\otimes\< L)\<\<\otimes'\<\< g^*\!N$}
\def\lvn{$\CH_{Y^{\<'}}\<(g^*L,g^*\<\<M)\<\otimes'\! g^*\<\<N\<\otimes'\! g^*\<\<L$}
\def\twv{$\CH_{Y^{\<'}}\<(g^*L,g^*\<\<M)\<\otimes'\! g^*\<\<L\<\otimes'\! g^*\<\<N$}
\def\thn{$g^*\<\<M\otimes'\<\< g^*\<\<N$}
  \bpic[xscale=3.85, yscale=1.5]
     
   \node(11) at (1,-1){\1};  
   \node(13) at (3,-1){\2} ;
   
   \node(22) at (1.87,-2){\4}; 
   \node(23) at (3,-2){\5}; 
     
    \node(31) at (1.6,-3){\3} ;  
   
   \node(41) at (1.6,-4){\7};
   \node(43) at (3.3,-4){\8};
   
   \node(52) at (1.6,-5.25){\9};   

   \node(61) at (1,-6){\6} ;  
   \node(63) at (2.5,-4.6){\ten}; 
   
   \node(71) at (1,-7){\lvn} ;  
   \node(72) at (3,-7){\twv};   
   \node(73) at (3.3,-5.25){\thn}; 
   
    \draw[->] (11)--(13) node[above, midway, scale=.75] {$\via \gamma^{}_{\>Y}$} ;  
    
    \draw[->] (31)--(22) node[above=-3, midway, scale=.75] {$\<\via \gamma^{}_{\>Y}\mkern48mu$} ;
    
    \draw[->] (41)--(43) node[below, midway, scale=.75] {$$} ;
    
     \draw[->] (71)--(72) node[below, midway, scale=.75] {$\via \gamma^{}_{\>Y'}$} ;
     \draw[->] (3.3,-6.8)--(3.3,-5.5) node[below, midway, scale=.75] {$$} ;
      
    \draw[->] (11)--(61) node[left, midway, scale=.75] {$$} ;
    \draw[->] (61)--(71) ;
                                   
    \draw[->] (31)--(41) ;
    \draw[->] (52)--(72) node[right, midway, scale=.75] {$$} ;
    
    \draw[->] (3.3,-1.2)--(3.3,-1.8) node[left, midway, scale=.75] {$$} ;
    \draw[->] (3.3,-2.2)--(3.3,-3.8) ;
    \draw[->] (43)--(73) ;

  \draw[->] (11)--(31) ;
  \draw[<-] (22)--(13) ;
  \draw[->] (22)--(43) node[below=1, midway, scale=.75] {$g^*\<\<\alpha$} ;
  \draw[->] (41)--(52) ;
  \draw[->] (1.92,-4.2)--(2.18,-4.35) ; ;
  \draw[->] (61)--(52) ;
  \draw[->] (52)--(63) ;
  \draw[->] (63)--(73) ;
 
  \node at (1.3,-3.52)[scale=.9]{\circled1} ;
  \node at (2.85,-2.75)[scale=.9]{\circled2} ;
  \node at (2.7,-5.75)[scale=.9]{\circled3} ;
 
  \epic
\]
\end{small}
\vskip-10pt

Here, the commutativity of the unlabeled subdiagrams is easily verified. 

\pagebreak[3]
The commutativity of \circled1 results from that of the diagram that is \emph{dual} (see \cite[3.4.5]{li})
to the second one following \cite[(3.4.2.1)]{li}.

That of \circled2 results from \cite[3.5.6(g)]{li}, with $C\set\CH_Y(L,M\<\otimes\< N)$, $D\set L$, and
$E\set M\<\otimes\< N$ (details left to the reader).

That of \circled3 is given by \cite[3.5.6(a)]{li}. 

Thus (i) is proven.

(ii) A \emph{strictly perfect} $\OY$-complex is a bounded complex of direct summands of finite\kf-rank free $\OY$-modules. An $\OY$-complex is perfect if locally it is the target of a quasi-isomorphism with source a strictly perfect one \cite[p.\,122, 4.8, p.\,163, 2.0, and p.\,96, 2.2]{Il}.
So to prove (ii), one can reduce, by localizing, to where $L$ is strictly perfect, and then by a simple induction on~the number $n$ of degrees in which $L$ doesn't vanish, to where $n=1$, in~which case the assertion is readily verified.
\end{proof}

\pagebreak[3]
\begin{subcosa}\label{tensorhom}
Let  $(Y,\OY)$ be a ringed space and $L$ a \emph{perfect} $\OY$-complex, so that, 
as in \ref{^* and hbar}(ii), the map~$\gamma^{}=\gamma^{}_{\>Y}$ is an isomorphism.\va1 

The trace map $\textup{\bf tr}^{}_{\<\<L}=\textup{\bf tr}^{}_{\<\<L/\OY}$ 
is defined to be the natural 
$\D(Y)$-composite
\[
\R\>\sHom_Y(L,L\>)
\underset{\gamma^{-1}}\iso \R\>\sHom_Y(L,\OY)\Otimes{Y}L
\lto \OY\<,
\]
see~\cite[p.\,154, 8.1]{Il}.%
\footnote{A similar definition holds in any closed monoidal category for an object $L$ such that, with $[-,-]\set$internal hom, the natural map is an isomorphism $[L,\mathbf 1]\otimes L \iso [L,L\>]$. For even greater generality, see e.g., \cite{PS}.}
There results, for any perfect $\OY$-algebra~$A$, the natural composite $\D(Y)$-map
\[
\tr^{\textup{alg}}_{\!A}\colon A\iso\R\>\sHom_{A}(A,A)
\xto{\mu_{\<A}\>}
\R\>\sHom_Y(A,A)\xto{\!\textup{\bf tr}_{\!A}\>}\OY.
\]

From~\ref{^* and hbar}(ii) and commutativity of subdiagram \circled2 (with $N\set\OY$) in the proof of ~\ref{^* and hbar}(i), one gets a natural identification
$\>g^*\textup{\bf tr}^{}_{\<\<L/\OY}\cong\textup{\bf tr}^{}_{\<\<g^{\<*}\!L/\CO_{Y'}}\>$. Together  with the natural identification $g^*\<\<\mu_{\<A}\cong\mu_{g^*\<\!A}$, this gives a natural identification 
 $\>g^*\tr^{\textup{alg}}_{\!A}\cong\tr^{\textup{alg}}_{\!g^*\<\!A}\>$.\va2

Recall that for any $\OY$-complex~$G$, the~natural
$\D(Y)$-map is an isomorphism $\sHom(L,G)\iso\R\>\sHom(L,G)$. (Proof: localizing on $Y$ and induction on the number of degrees in which $L$ doesn't vanish reduces the assertion to the trivial case $L=\OY$.) Consequently, and by the last part of \S\ref{hbar}, $\textup{\bf tr}^{}_{\!L}$~can be identified naturally with (the $\D(Y)$-image~of) the natural composite $\sA(Y)$-map
\[
\sHom_Y(L,L\>)\iso \sHom_Y(L,\OY)\otimes_{Y}L\lto \OY.
\]

In particular, if $L$ is a finite\kf-rank locally free $\OY$-module then $\textup{\bf tr}^{}_{\!L}$  can be identified naturally with (the $\D(Y)$-image~of) the the usual trace map $\tr^{}_{\!L}\colon\sHom(L\<,L\>)\to \OY$. Moreover, it follows from the explicit description of the map $\gamma^{}_Y$ in \ref{hbar} that there is a
natural identification
\[
\textup{\bf tr}^{}_{\!L[\>j]} =(-1)^j\tr^{}_{\<\<L}\qquad(j\in\mathbb Z).
\]
More generally, if $L$ is strictly perfect, so that the degree\kf-$i$ component $L^{\lift1.4,\<i,}$ is a finite\kf-rank locally free $\OY$-module for all $i$,  and vanishes for all but finitely many $i$,
then there is a natural identification
\begin{equation*}\label{tensorhom1}
\textup{\bf tr}^{}_{\<\<L}=\sum_i \>(-1)^i\>\tr^{}_{\<\<L^{\lift1.2,\<\<{\sss i},}}\>,
\tag{\ref{tensorhom}.1}
\end{equation*}
as is easily shown by  induction on the number of $i$ such that $L^{\lift1.4,\<i,}\ne0$.

\end{subcosa}

\begin{subcosa}
If $f\colon X\to Y$ is a \emph{perfect} affine  scheme\kf-map, so that $\R\fst\OX$ is a perfect
$\OY$-complex , then one has the 
$\D(Y)$-map $\tr_{\<\<f}\set\tr^{\textup{alg}}_{\<\<\fst\<\OX}\colon \fst\OX\to\OY\<,$
whence for any $G\in\Dqc(Y)$, the natural functorial composite 
\[
\tr_{\<\<f}(G)\colon\R\fst\LL f^*\<G\underset{\lift1.1,\eqref{projf},}\iso \fst\OX\Otimes{Y}G
\xto[\lift1.3,\tr_{\!f}\>\Otimes{}\>\>\id,]{}
\OY\Otimes{Y}G\iso G,
\]
\vskip2pt
\noindent whence, by  Corollary~\ref{right adjoint}, a natural functorial $\Dqc(X)$-map
\begin{equation}\label{fundclass}
C_{\<\<f}(G)\colon \LL f^*\<G\lto f^\flat G.
\end{equation}

For example, if $f$ is \emph{flat and finitely presentable,} so that $\fst\OX$ is locally free, then 
$\tr_{\<\<f}(\OY)\colon \fst\OX\to\OY$ identifies naturally with  the usual trace map; and if, also, $Y$ is noetherian, then $C_f(\OY\<)\colon\OX\to f^\flat\OY$ becomes~the (naive) \emph{fundamental class} (\cite[(0.2.3) with $n=0$, plus Example 2.6]{AJL14}).\looseness=-1 

For finite \emph{\'etale}~$f\<$,  $C_{\<\<f}$ is the \emph{isomorphism} at the end of \S\ref{trprelim}.
 
\end{subcosa}

The map $C_{\<\<f}\colon\LL f^*\to f^\flat$  is \emph{pseudofunctorial,} in the following sense.
\end{subcosa}

\begin{subprop} \label{transfund} Let\/ $W\xto{\,g\,\>} X\xto{\,f\,} Y$ $($hence $fg)$ be 
perfect affine~maps. 
The  following diagram of functors from\/ $\Dqc(Y)$ to\/ $\Dqc(W)$ commutes.
\[
\def\3{$g^\flat\<\< f^\flat$}
\def\4{$(fg)^\flat$}
\def\1{$\LL g^*\LL f^*$}
\def\2{$\LL(fg)^*$}
 \bpic[xscale=1.5, yscale=1]

   \node(11) at (1,-1){\1};   
   \node(13) at (3,-1){\2}; 
     
   \node(31) at (1,-3){\3};
   \node(33) at (3,-3){\4};
   
    \draw[->] (11)--(13) node[above, midway] {\Iso} 
                                   node[below=1pt, midway, scale=.75] {\textup{natural}} ;  
      
   \draw[->] (31)--(33) node[above, midway] {\Iso} 
                                  node[below=1pt, midway, scale=.75] {\eqref{pf flat}} ;
    
    \draw[->] (11)--(31) node[left, midway, scale=.75] {$\via \>C_{\<g}$} 
                                   node[right, midway, scale=.75] {\textup{and }$C_{\<\<f}$} ;
    \draw[->] (13)--(33) node[right=1pt, midway, scale=.75] {$C_{\<\<fg}$} ;
 
 \epic  
\]
\end{subprop}

\noindent\emph{Proof.} It suffices to prove commutativity of the dual diagram, which is the unlabeled subdiagram in the following natural diagram. 
\[
\def\1{$\R(fg)_{\<*}\>\LL g^*\LL f^*$}
\def\2{$\R(fg)_{\<*}\>\LL(fg)^*$}
\def\3{$\R\fst\R g_*\LL g^*\LL f^*$}
\def\4{$\R\fst\LL f^*$}
\def\5{$\id$}
\def\6{$\R\fst\R g_*g^\flat \LL f^*$}
\def\7{$\R(fg)_* (fg)^\flat$}
\def\8{$\R(fg)_*\> g^\flat\<\< f^\flat$}
\def\9{$\R\fst\R g_*g^\flat\<\< f^\flat$}
\def\ten{$\R\fst f^\flat$}
  \bpic[xscale=3.2, yscale=1.4]

   \node(11) at (2.5,-1){\1};   
   \node(14) at (4,-1){\2};
   
   \node(21) at (1,-1){\3};
   \node(22) at (2.5,-2){\8};   
   \node(23) at (3.5,-3){\7};

   \node(32) at (1.5,-3){\6};
   \node(33) at (2.5,-3){\9};

   \node(42) at (2.5,-4.3){\ten};
      
   \node(51) at (1,-5){\4};  
   \node(54) at (4,-5){\5};
   
    \draw[<-] (11)--(21) node[above=1, midway, scale=.75] {$\Iso$} ;

    \draw[->] (11)--(14) node[above=1, midway] {$\Iso$} ; 
    
    \draw[->] (22)--(23) node[above=-2, midway] {\rotatebox{-22}{$\Iso$}} 
                                   node[below=-10, midway, scale=.75] {\rotatebox{-22}{$\mkern15mu\via\>\eqref{pf flat}\mkern30mu$}} ;
                                                              
    \draw[->] (32)--(33) node[below=1, midway, scale=.75] {$\via\>C_{\!f}$} ;
    
   \draw[->] (51)--(54)  node[below=1, midway, scale=.75] {$\tr_{\!f}$}; 
                                   
      \draw[->] (11)--(22) node[left, midway, scale=.75] {$\via \>C_{\<g}$}
                                      node[right, midway, scale=.75] {$\textup{and }C_{\!f}$} ;

     \draw[->] (21)--(51) node[left=1, midway, scale=.75] {$\R\fst\tr_{\<g}$} ;
     
     \draw[->] (14)--(54) node[right=1, midway, scale=.75] {$\tr_{\!fg}$} ;
    
   \draw[->] (22)--(33) ;
   \draw[->] (21)--(32) node[right=1, midway, scale=.75] {$\via \>C_{\<g}$} ;
   \draw[->] (32)--(51) ;
   \draw[->] (51)--(42) node[above, midway, scale=.75] {$\via \>C_{\!f}\mkern55mu$} ;
   \draw[->] (42)--(3.915,-4.905);
   \draw[->] (23)--(3.93,-4.85) ;
   \draw[->] (33)--(42) ;
   
   \node at (1.9, -2.02)[scale=.9]{\circled1} ;
   \node at (1.9, -3.8)[scale=.9]{\circled3} ;
   \node at (3.05, -3.45)[scale=.9]{\circled2} ;
   \node at (1.2, -3.45)[scale=.9]{\circled4} ;
   \node at (2.5, -4.72)[scale=.9]{\circled5} ;

 \epic
\]
The commutativity of subdiagrams \circled1 and \circled3 is clear, subdiagrams \circled4 and \circled5 commute by definition, and the commutativity of \circled2 holds because the map \eqref{pf flat} is, by definition, dual to the composite map
\[
\R(fg)_* (fg)^\flat \to \R\fst\R g_*g^\flat\<\< f^\flat \to \R\fst f^\flat \to \id.
\]
Hence diagram chasing  shows it sufficient that the~border commute (i.e., that \emph{transitivity of the trace map} hold).

The border in question, applied to an arbitrary $G\in\Dqc(Y)$, expands to the border of the following natural diagram. 

\[
\def\2{$\R(fg)_{\<*}\>\LL g^*\LL f^*\<G$}
\def\1{$\R\fst\R g^{}_*\LL g^*\LL f^*\<G$}
\def\3{$\R(fg)_*\LL(fg)^*G$}
\def\4{$\R\fst\LL f^*\<G$}
\def\5{$G$}
\def\6{$(fg)_*\OW\<\Otimes{Y}G$}
\def\7{$\OY\<\Otimes{Y}G$}
\def\8{$\fst g^{}_*\OW\<\Otimes{Y}G$}
\def\9{$\fst\OX\<\Otimes{Y}G$}
\def\ten{$\R\fst(g^{}_*\OW\<\Otimes{\sX}\LL f^*\<G\>)$}
\def\lvn{$\R\fst(\OX\<\Otimes{\sX}\LL f^*\<G\>)$}
  \bpic[xscale=2.5, yscale=2]

   \node(11) at (2.55,-1){\2};  
   \node(12) at (1,-1){\1};
   \node(14) at (4,-1){\3};
   
   \node(21) at (1,-2){\ten};
   \node(22) at (2.55,-2){\8};
   \node(24) at (4,-2){\6};
   
   \node(31) at (1,-3){\lvn}; 
   \node(32) at (2.55,-3){\9};  
   \node(34) at (4,-3){\7};
      
   \node(41) at (1,-4){\4};  
   \node(44) at (4,-4){\5};
   
    \draw[<-] (11)--(12) node[above=.5, midway, scale=.75] {$\Iso$} ;
       
    \draw[->] (11)--(14) node[above=.5, midway, scale=.75] {$\Iso$} ; 
    \draw[->] (21)--(22) node[above=.5, midway, scale=.75] {$\Iso$} 
                                   node[below=.5, midway, scale=.75] {\eqref{projf}} ; 
    \draw[double distance=2] (22)--(24)  ;
    
    \draw[->] (31)--(32) node[above=.5, midway, scale=.75] {$\Iso$} 
                                   node[below=.5, midway, scale=.75] {\eqref{projf}} ;
    \draw[->] (32)--(34)  node[below=1, midway, scale=.75] {$\via\tr_{\!f}$};
    
    \draw[->] (41)--(44)  node[below=1, midway, scale=.75] {$\tr_{\!f}(G)$}; 

      \draw[->] (12)--(21) node[left=1, midway, scale=.75] {$\simeq$} ;
      \draw[->] (21)--(31) node[left, midway, scale=.75] {$\via\tr_{\<g}$} ;
      \draw[->] (31)--(41) node[left=1, midway, scale=.75] {$\simeq$} ;
      
      \draw[->] (22)--(32) node[left, midway, scale=.75] {$\via\tr_{\<g}$} ;

     \draw[->] (14)--(24) node[right=1, midway, scale=.75] {$\simeq$} ;
     \draw[->] (24)--(34) node[right=1, midway, scale=.75] {$\via\tr_{\!fg}$} ;
     \draw[->] (34)--(44) node[right=1, midway, scale=.75] {$\simeq$} ;
    
  \draw[->] (41)--(32) node[above=-2, midway] {\rotatebox{23}{$\Iso$}} ;

   \node at (2.5, -1.52)[scale=.9]{\circled7} ;
   \node at (3.25, -2.52)[scale=.9]{\circled8} ;

 \epic
\]
\vskip-3pt
\noindent Here, commutativity of the unlabeled subdiagrams is clear, and that of~\circled7 results readily from \cite[3.7.1]{li} with $F\set\OX$.  So to prove the commutativity of the border,
it suffices to prove that of \circled8 with all occurrences of~``$\Otimes{Y}\>G\>$" ignored (i.e., to prove
transitivity of the trace when $G=\OY$),
which  one can do by taking $A\set g_*\OW$ in the next lemma, where, for an $\OX$- or $\OY$-algebra $\CS$, $[-,-]_\CS\set\R\>\sHom_\CS(-,-)$
and $\otimes_\CS\set\Otimes{\CS}\>$.


\begin{sublem}\label{Lem:TransTrace} If\/ $A$ is a perfect\/ $\OX\<$-complex then\/ $\fst A$  is perfect
over both\/ $\fst\OX$ and\/ $\OY\<,$ and the following natural diagram commutes.
\[\mkern-3mu
\def\1{$\fst A$}
\def\2{$[\fst A,\fst A]_{\fst\<\OX}$}
\def\3{$[\fst A,\fst A]_{\OY}$}
\def\4{$\fst [A,A]_{\OX}$}
\def\5{$\fst [A,\OX]_{\OX}\!\otimes_{\fst\OX}\!\fst A$}
\def\6{$[\fst A,\OY]_{\OY}\!\otimes_{\OY}\!\fst A\mkern60mu$}
\def\7{$\fst( [A,\OX]_{\OX}\!\otimes_{\OX}\!A)$}
\def\8{$[\fst A,\fst\OX]_{\fst\<\OX}\!\otimes_{\fst\OX}\!\fst A$}
\def\9{$\fst\OX$}
\def\ten{$[\fst\OX,\fst\OX]_{\OY}$}
\def\lvn{$[\fst\OX,\OY]_{\OY}\!\otimes_{\OY}\!\fst\OX$}
\def\twv{$\OY$}
  \bpic[xscale=2.98, yscale=2]

   \node(11) at (1,0){\1};   
   \node(13) at (3.25,0){\1}; 
   \node(14) at (4.25,0){\3};
   
   \node(21) at (1,-1){\4};   
   \node(22) at (2.125,-2){\5}; 
   \node(23) at (3.25,-1){\2}; 
   \node(24) at (4.25,-2){\6};

   \node(31) at (1,-3){\7};
   \node(33) at (3.25,-3){\8};

   \node(41) at (1,-4.1){\9};
   \node(42) at (1.86,-4.1){\ten}; 
   \node(43) at (3.24,-4.1){\lvn};
   \node(44) at (4.25,-4.1){\twv};
 
    \draw[double distance=2] (11)--(13) ; 
    \draw[->] (13)--(14) ;
                                                              
    \draw[->] (41)--(42) ;
    \draw[->] (42)--(43) node[above, midway, scale=.75] {$\Iso$} ;
    \draw[->] (43)--(44) ;
                                   
     \draw[->] (11)--(21) ;
     \draw[->] (21)--(31) node[left=1pt, midway, scale=.75] {$\simeq$} ;
     \draw[->] (31)--(41)  ;
       
     \draw[->] (13)--(23) ;
     \draw[->] (23)--(33) node[left=1pt, midway, scale=.75] {$\simeq$} ;
     
     \draw[->] (14)--(24) node[right=1pt, midway, scale=.75] {$\simeq$} ;
     \draw[->] (24)--(44) ;
 
   \draw[->] (21)--(23) ;
   \draw[->] (22)--(1.3,-2.75) ;
   \draw[<-] (22)--(33) node[above=-1, midway, scale=.75] {\kern15pt$\simeq$} 
                                   node[below=-3, midway, scale=.75] {\textup{\ref{tensor}}\kern35pt};
   \draw[->] (33)--(41) ;
   
   \node at (2.06, -.52)[scale=.9]{\circled1} ;
   \node at (2.06, -1.52)[scale=.9]{\circled2} ;
   \node at (2.06, -3)[scale=.9]{\circled3} ;
   \node at (3.75, -2.52)[scale=.9]{\circled4} ;

 \epic
\]

\end{sublem}

\noindent\emph{Proof.}
To see that $\fst A$ is perfect over $\fst\OX$ and over $\OY$, one can localize on~$Y\<$, and so can assume that $X$ and $Y$ are affine schemes, in which case the assertions have a standard translation into the analogous, and simple, ones in commutative algebra.

As for the diagram, we first show subdiagrams \circled1, \circled2 and \circled3 commute,
so that ``trace is preserved under the equivalence in \ref{^* equivalence}."

It is left to the reader to verify that the natural map
\[
\alpha\colon\fst\R\>\sHom^{}_\sX(A,B)\to \R\>\sHom_Y(\fst A,\fst B\>)\qquad(A,B\in\D(X))
\]
associated to an \emph{arbitrary ringed-space map} $f\colon(X,\OX)\to(Y,\OY)$ via, e.g.,
\cite[(3.1.4) and 2.6.5]{li}, is characterized abstractly 
(cf.~\cite[(3.5.4.1)]{li} as being adjoint to the natural composite map
\[
\fst\R\>\sHom^{}_\sX(A,B)\Otimes{Y}\fst A\to
\fst(\R\>\sHom_Y(A,B\>)\Otimes{\sX}A)
\to \fst B.
\]
The commutativity of \circled3 is the special case\va2 where $B\set\OX$ and $\OY\set \fst\OX.$  

(Likewise, when other abstract properties of functorial maps, as in~\cite{li}, are used in the rest of this proof, it should be checked that the natural concrete interpretations of those maps do have those properties---so that the concrete result is an instance of an abstract one.)\va2

The commutativity of \circled1 is that of the second diagram 
in~\cite[3.7.1.1]{li} (where $\fst$ and $g^{}_*$ should be switched), derived from 
the natural composite map denoted here by $(X, A)\to(X,\OX) \to (Y,\fst\OX)$,  with $F=F'\set A$.

The commutativity of \circled2 is essentially the case $B\set\OX$, $C\set A$ of the next lemma, applied to the functor $\R\bar\fst\colon\D(\OX)\to\D(\fst\OX)$ associated
to the natural ringed-space map $\bar f\colon(X,\OX)\to (Y,\fst\OX)$.

\pagebreak[3]
\begin{sublem}\label{compatible} Let\/ $\fst\colon\mathbf X\to\mathbf Y $ be a symmetric monoidal functor between monoidal closed categories. With notation as in\/ \textup{\cite[\S3.5]{li},} for any\/ $A,$ $B$ and\/ $C$ in\/~$\mathbf X,$
the following natural diagram commutes.
\[
\def\1{$\fst\bigl([A,B]\otimes C\bigr)$}
\def\2{$\fst[A,B]\otimes \fst C$}
\def\3{$[\fst A,\fst B]\otimes \fst C$}
\def\4{$\fst[A,B\otimes C]$}
\def\5{$[\fst A,\fst(B\otimes C)]$}
\def\6{$[\fst A,\fst B\otimes \fst C]$}
  \bpic[xscale=3.75, yscale=1.3]

   \node(11) at (1,-1){\1} ;  
   \node(12) at (2,-1){\2} ;
   \node(13) at (3,-1){\3} ;
   
   \node(21) at (1,-2){\4} ;
   \node(22) at (2,-2){\5} ;
   \node(23) at (3,-2){\6} ;
   
    \draw[<-] (11)--(12) ;
    \draw[<-] (13)--(12) ;
       
    \draw[->] (21)--(22) ;
    \draw[->] (23)--(22) ;

      \draw[->] (11)--(21) ;
      \draw[->] (13)--(23) ;
      
 \epic
\]

\end{sublem}

\begin{proof}
\hskip-3pt%
\footnote{\kf Conceivably, the assertion is contained in \cite[Theorem 4.18]{le}.}
Expand the diagram, naturally,  as follows:
\begin{small}
\[\mkern-4mu
\def\1{$\fst\bigl([A,B]\<\<\otimes\<\< C\bigr)$}
\def\2{$\fst[A,B]\<\<\otimes\<\< \fst C$}
\def\3{$[\fst A,\fst B]\<\<\otimes\<\< \fst C$}
\def\4{$\fst[A,B\<\<\otimes\<\< C]$}
\def\5{$[\fst A,\fst(B\<\<\otimes\<\< C)]$}
\def\6{$[\fst A,\fst B\<\<\otimes\<\< \fst C)]$}
\def\7{$\fst[A,[A,B]\<\<\otimes\<\< C\<\<\otimes\<\< A]$}
\def\8{$[\fst A,[\fst A,\fst B]\<\<\otimes\<\< \fst C\<\<\otimes\<\< \fst A]$}
\def\9{$[\fst A,\fst([A,B]\<\<\otimes\<\< C\<\<\otimes\<\< A)]$}
\def\ten{$[\fst A,\fst[A,B]\<\<\otimes\<\< \fst C\<\<\otimes\<\< \fst A]$}
\def\lvn{$[\fst A,\fst([A,B]\<\<\otimes\<\< A\<\<\otimes\<\< C)]$}
\def\twv{$[\fst A,\fst[A,B]\<\<\otimes\<\< \fst A\<\<\otimes\<\< \fst C]$}
\def\thn{$[\fst A,[\fst A,\fst B]\<\<\otimes\<\< \fst A\<\<\otimes\<\< \fst C]$}
\def\frn{$\fst[A,[A,B]\<\<\otimes\<\< A\<\<\otimes\<\< C]$}
\def\ffn{$[\fst A,\fst([A,B]\<\<\otimes\<\< A)\<\<\otimes\<\< \fst C]$}
  \bpic[xscale=4.25, yscale=1.1]

   \node(11) at (1,-.9){\1};   
   \node(12) at (2.25,-.9){\2};
   \node(13) at (3,-.9){\3}; 
   
   \node(21) at (1,-2){\7};
   \node(22) at (2.25,-2){\ten};
    
   \node(32) at (1.7,-3){\9};
   \node(33) at (3,-3){\8}; 

   \node(42) at (1.7,-5.1){\lvn};   
   \node(43) at (2.25,-4){\twv};
   \node(44) at (3,-5.05){\thn}; 
      
   \node(51) at (1,-4){\frn};  
   \node(53) at (2.25,-5.9){\ffn};
   
   \node(81) at (1,-7){\4};   
   \node(82) at (1.7,-7){\5};
   \node(83) at (3,-7){\6}; 
   
    \draw[<-] (11)--(12) ;
    \draw[->] (12)--(13) ;   

    \draw[->] (81)--(82) ;
    \draw[<-] (82)--(83) ;         
                               
     \draw[->] (11)--(21) ;
     \draw[->] (21)--(51) ;
     \draw[->] (51)--(81) ;
     
     \draw[->] (12)--(22) ;

     \draw[->] (32)--(42) ;
     \draw[->] (42)--(82) ;
     
     \draw[->] (43)--(53) ;
     
     \draw[->] (13)--(33) ;
     \draw[->] (33)--(44) ;
     \draw[->] (44)--(83) ;

   \draw[->] (21)--(32) ;
   \draw[->] (22)--(32) ;
   \draw[->] (22)--(33) node[above=-1.5, midway, scale = .75]{$\mkern50mu\via\alpha$} ;
   \draw[->] (22)--(43) ;
   \draw[->] (43)--(42) ;
   \draw[->] (43)--(44) node[above=-1.5, midway, scale = .75]{$\mkern50mu\via\alpha$} ;
   \draw[->] (53)--(83) ;
   \draw[->] (53)--(42) ;
   \draw[->] (51)--(42) ;
   
  
   \node at (1.58, -1.65)[scale=.9]{\circled5} ;
   \node at (2, -3.52)[scale=.9]{\circled6} ;
   \node at (2.12, -4.7)[scale=.9]{\circled7} ;
   \node at (2.65, -5.5)[scale=.9]{\circled8} ;
 
 \epic
\]
\end{small}
Commutativity of the unlabeled subdiagrams is simple to verify. That of \circled6 and \circled7 follows at once from the definition of symmetric monoidal functor, see~e.g., \cite[3.4.2]{li}. That of~\circled8 
results from the above description of $\alpha$. 
That of \circled5 is equivalent to that of the adjoint diagram, that is, the border 

\pagebreak[3]
\noindent of the following natural diagram, where $D\set[A,B]$; and this border does commute, since
all the subdiagrams clearly do. 
\[
\def\1{$\fst(D\<\<\otimes\<\< C)\<\<\otimes\<\< \fst A$}
\def\2{$\fst D\<\<\otimes\<\< \fst C\<\<\otimes\<\< \fst A$}
\def\3{$\fst([A,D\<\<\otimes\<\< C\<\<\otimes\<\< A]\<\<\otimes\<\< A])$}
\def\4{$\fst(D\<\<\otimes\<\< C\<\<\otimes\<\< A)$}
\def\5{$[\fst A,D\<\<\otimes\<\< C\<\<\otimes\<\< A]\<\<\otimes\<\< \fst A]$}
\def\7{$\fst[A,D\<\<\otimes\<\< C\<\<\otimes\<\< A]\<\<\otimes\<\< \fst A$}
\def\9{$\fst(D\<\<\otimes\<\< C\<\<\otimes\<\< A)$}
\def\ten{$\fst D\<\<\otimes\<\< \fst C\<\<\otimes\<\< \fst A$}
 \bpic[xscale=4.5, yscale=1.15]
   \node(11) at (1,-1){\1};   
   \node(13) at (3,-1){\2};
  
   \node(23) at (3,-2){\4};
   
   \node(31) at (1,-3){\7}; 
   \node(32) at (2.07,-3){\3};
   \node(33) at (3,-3){\4};

    \draw[<-] (11)--(13) ;
    
    \draw[->] (31)--(32) ;
   \draw[->] (32)--(33) ; 
                               
     \draw[->] (11)--(31) ;
     \draw[->] (13)--(23) ;
     \draw[double distance =2] (23)--(33) ;
         
   \draw[->] (11)--(23) ;
   \draw[->] (23)--(32) ;

%
   
  \epic
\]

This completes the proof of Lemma~\ref{compatible}.
\end{proof}

To complete the proof of Lemma~\ref{Lem:TransTrace}, whence 
of Proposition~\ref{transfund}, one~needs 
subdiagram \circled4 to commute---which it does, by the next Lemma
(with $\CS\set\fst\OX$ and $\CT\set\fst A$).

\begin{sublem}\label{transtr} 
Let\/ $(Y,\OY)$ be a ringed space, let\/ $\CS$ be a perfect\/ $\OY$-algebra,
and let $E$ be a perfect\/ $\CS$-module. Then $E$ is a perfect $\OY$-module, and\/
$\textup{\bf tr}^{}_{\<\<E/\OY}$ factors naturally as\va{-10}
\begin{equation*}\label{transtr1}
\sHom_{\OY}(E,E)\lto \sHom_{\CS}(E,E)\xto{\textup{\bf tr}^{}_{\<\<E/\CS}}
\CS\xto{\!\tr^{\textup{alg}}_{\<\<\CS\</\OY}\!}\OY.
\tag{\ref{transtr}.1}
\end{equation*}

In particular, if\/ $\CT$ is a perfect\/ $\OY$-algebra, then
\[
\tr^{\textup{alg}}_{\CT/\OY}\<=\tr^{\textup{alg}}_{\<\<\CS/\OY}\!\smallcirc\>\tr^{\textup{alg}}_{\CT\</\CS}\colon \CT\to\OY.
\]

\end{sublem}

\noindent\emph{Proof.}
That $E$ is perfect over $\OY$ can be shown as in the beginning of the proof of \ref{Lem:TransTrace}.\va2

As for \eqref{transtr1}, compatibility of {\bf tr} with base change (see section~\ref{tensorhom}) allows one to assume that $E$ is $\D(\CS)$-isomorphic to a bounded complex 
of direct summands of finite\kf-rank free $\CS$-modules (see proof of~\ref{^* and hbar}(ii)). It~suffices then, in view of~ \eqref{tensorhom1}, to show that for any direct summand~$F$ of a finite\kf-rank free $\CS$-module, the following natural diagram commutes, where by previous considerations, the maps involved
can be identified with their nonderived precursors:
\begin{equation*}\label{transtr2}
\def\1{$\CS$}
\def\2{$[F\<,F\>]_{\OY}$}
\def\3{$[F\<,F\>]_{\CS}$}
\def\4{$[\CS,\CS\>]_{\OY}$}
\def\5{$[\CS,\OY]_{\OY}\<\<\otimes_{\OY}\<\CS$}
\def\6{$[F\<,\OY]_{\OY}\<\<\otimes_{\OY}\<\<F$}
\def\7{$[F\<,\CS\>]_{\CS}\otimes_\CS F$}
\def\8{$\OY$}
 \CD
  \bpic[xscale=2.5, yscale=2]

   \node(11) at (1,-1){\3};   
   \node(14) at (4,-1){\2};

   \node(24) at (4,-2){\6};
    
   \node(31) at (1,-2){\7};
     
   \node(41) at (1,-3){\1};  
   \node(42) at (1.8,-3){\4};   
   \node(43) at (3,-3){\5};
   \node(44) at (4,-3){\8}; 
   
    \draw[->] (11)--(14) ;    

    \draw[->] (41)--(42) ;
    \draw[->] (42)--(43) node[above=-1, midway, scale=.9]{$\Iso$} ; 
    \draw[->] (43)--(44) ;        
                               
     \draw[->] (11)--(31) node[left=1, midway, scale=.75]{$\simeq$} ;
     \draw[->] (31)--(41) ;
     
     \draw[->] (14)--(24) node[right=1, midway, scale=.75]{$\simeq$} ;
     \draw[->] (24)--(44) ;
        
  \epic
   \endCD \tag{\ref{transtr}.2}
\end{equation*}

If $F=F_1\oplus_\CS F_2\>$, then (one verifies) this diagram is the direct sum of the four natural diagrams
obtained by substituting $F_i$ for the first occurrence of~$F$ at each node and $F_{\<j}$
for the second occurrence, with $(i,j)=(1,1)$ or~$(1,2)$ or $(2,1)$ or $(2,2)$,  
the resulting arrow 
\[
[F_i,\CT\>\>]_{\CT}\otimes_{\>\CT} F_{\<j}\lto\CT\qquad (\CT\set\CS\textup{ or }\OY)
\]
representing the natural composite map
\[
[F_i,\CT\>\>]_{\CT}\otimes_{\>\CT} F_{\<j}\lto[F\<,\CT\>\>]_{\CT}\otimes_{\>\CT} F_{\<j}
\lto[F_{\<j},\CT\>\>]_{\CT}\otimes_{\>\CT} F_{\<j}\lto \CT,
\]
which vanishes when $i\ne j$. Hence,
\eqref{transtr2} commutes for $F$ if and only if it commutes for both $F_1$ and~$F_2$.
\looseness=-1

It follows that the question of commutativity of  \eqref{transtr2} reduces to the trivial case where $F=\CS$. \hfill{$\square\ \square\ \square$}

\vskip5pt
Our underlying theme, concrete realizations of abstract constructions, spurs~continuing on with interpretations, via traces for perfect affine maps, of maps involving $(-)^\flat$ via maps involving $(-)^*$.
Proposition~\ref{transfund} is just a first example. But this could be an endless process. 

Let some further examples, provided by the following assertions, suffice. As before, justification of these
assertions requires deploying the formalism of adjoint functors between closed categories (see e.g., \cite[\S\S3.5.5--3.5.6]{li}), and/or its scheme\kf-theoretic realization (see \cite[\S3.6.10]{li}).
\enlargethispage{-20pt}
\pagebreak[3]
\begin{subxrz}\label{redo}
Let $f\colon X\to Y$ be a perfect affine map, $F, G\in\Dqc(Y)$,
and $C_{\<\<f}$ as in \eqref{fundclass} (an isomorphism if $f$ is \'etale). 
\vskip2pt
(i)  Let $\sigma$ and $\beta_\sigma$ be as in Theorem~\ref{indt base change}.
Then the affine map~$g$ is perfect, and  the following natural diagram commutes.
\[
\def\5{$\LL v^*\!f^\flat G$}
\def\6{$g^\flat \LL u^*\mkern-.5mu G$}
\def\7{$\LL v^*\LL f^*\<G$}
\def\8{$\LL g^*\LL u^*\mkern-.5mu G$}
 \bpic[xscale=1.35, yscale=.95]
 
   \node(15) at (6,-3){\5};   
   \node(17) at (8,-3){\6}; 
     
   \node(35) at (6,-1){\7};
   \node(37) at (8,-1){\8};

   \draw[->] (15)--(17) node[above, midway, scale=.75] {$\Iso$}
                                  node[below=1pt, midway, scale=.75] {$\beta_\sigma$} ;  
      
   \draw[->] (35)--(37) node[above, midway, scale=.75] {$\Iso$} ;
    
    \draw[->] (35)--(15) node[left=1pt, midway, scale=.75] {$\LL v^*C_{\<\<f}$} ;
   
    \draw[->] (37)--(17) node[right=1pt, midway, scale=.75] {$C_{\<g}$} ;

 \epic 
\]
\vskip-2pt

(ii) Let $\chi\colon\<\< f^\flat \<\<F\<\Otimes{\sX}\<\LL f^*\<G\to \<f^\flat\<(F\<\Otimes{Y}\<\<G)$ 
be the isomorphism in~\ref{flat and tensor}(ii).
The following natural diagram commutes.
\[
\def\9{$f^\flat \<\<F\<\Otimes{\sX}\<\LL f^*\<\<G$}
\def\ten{$f^\flat(F\<\Otimes{Y}\<\<G)$}
\def\lvn{$\LL f^*\<\< F\<\Otimes{\sX}\<\LL f^*\<\<G$}
\def\twv{$\LL f^*\<(F\<\Otimes{Y}\<\<G)$}
 \bpic[xscale=1.25, yscale=.95]
 
   \node(12) at (0,-7){\9};   
   \node(14) at (2.8,-7){\ten}; 
     
   \node(32) at (0,-5){\lvn};
   \node(34) at (2.8,-5){\twv};
   
    \draw[->] (12)--(14) node[above, midway, scale=.75] {$\Iso$} 
                                   node[below=1, midway, scale=.75] {$\chi$} ;
      
   \draw[->] (32)--(34) node[above, midway, scale=.75] {$\Iso$} ;
    
    \draw[<-] (12)--(32) node[left=1pt, midway, scale=.75] {$\via\> C_{\<\<f}$} ;
   
    \draw[<-] (14)--(34) node[right=1pt, midway, scale=.75] {$C_{\<\<f}$};
  
   \epic 
\]

(iii) Assume $\R\>\sHom_Y(F,G)\in\Dqc(Y)$ and $\R\>\sHom_\sX(\LL f^*\<\<F,f^\flat G)\in\Dqc(Y)$ (for example, $F$ perfect, or $F$ pseudo\kf-coherent
and $G\in\Dqcpl(Y))$.

Let $\zeta\colon \R\>\sHom^{}_\sX(\LL f^*\<\<F,f^\flat G\>)\to
 f^\flat\R\>\sHom_Y(F,G\>)$ correspond via \ref{represent}(i) to the natural composite map
\[
\R\fst \R\>\sHom_\sX(\LL f^*\<\<F,\>f^\flat G\>)
\iso\R\>\sHom_Y(F,\R\fst f^\flat G\>)
\xto{\>t^{}_{\<G}\>}\R\>\sHom_Y(F\<,G\>).
\]

The following natural diagram commutes.
\[
\def\thn{$\R\>\sHom^{}_\sX(\LL f^*\<\<F,f^\flat G\>)$}
\def\frn{$f^\flat\R\>\sHom_Y(F,G\>)$}
\def\ffn{$\R\>\sHom^{}_\sX(\LL f^*\<\<F,\LL f^* G\>)$}
\def\sxn{$\LL f^*\R\>\sHom_Y(F,G\>)$}
 \bpic[xscale=1.25, yscale=.95]
   
   \node(16) at (6,-7){\thn};   
   \node(18) at (9.6,-7){\frn}; 
     
   \node(36) at (6,-5){\ffn};
   \node(38) at (9.6,-5){\sxn};
   
    \draw[->] (16)--(18)  node[below, midway, scale=.75] {$\zeta$}  ;
      
   \draw[<-] (36)--(38) ;
    
    \draw[<-] (16)--(36) node[left=1pt, midway, scale=.75] {$\via\> C_{\<\<f}$} ;
   
    \draw[<-] (18)--(38) node[right=1pt, midway, scale=.75] {$C_{\<\<f}$};

 \epic 
\]

\end{subxrz}

\begin{subcosa}\label{ex compl}
For additional  illustration, generalizing the last paragraph in \S\ref{trprelim}, 
Proposition~\ref{complementary} below gives, for any 
``almost \'etale" $f\colon X\to Y\<$, a concrete representation of the representing pair
$(H^0\<\<f^\flat\OY\<,\> t'_{\OY})$ in\/ \S\textup{\ref{trprelim}}.\va2

Consider a fiber square of scheme\kf-maps, with qcqs $f\>$:
\begin{equation*}\label{fib}
\mkern-60mu
\def\0{$V=$}
\def\1{$U\!\times_Y\<\< X$}
\def\2{$X$}
\def\3{$U$}
\def\4{$Y$}
\CD
 \bpic[xscale=2, yscale=1.35]

   \node(10) at (.4,-.975){\0} ; 
   \node(11) at (1,-1){\1} ;   
   \node(12) at (2,-1){\2} ; 
   
   \node(21) at (1,-2){\3} ;  
   \node(22) at (2,-2){\4} ;
   
    \draw[->] (11)--(12) node[above=1pt, midway, scale=.75] {$v$} ;  
 
    \draw[->] (21)--(22) node[below=1pt, midway, scale=.75] {$u$} ;

    \draw[->] (11)--(21) node[left=1pt, midway, scale=.75] {$g$} ;
    
    \draw[->] (12)--(22) node[right=1pt, midway, scale=.75] {$f$} ; 
                      
    \epic
  \endCD\tag*{(\ref{ex compl}.1)}
\end{equation*}
Assume that $u$ is  flat, or that  $f$ is affine.  
Then the natural map of functors 
from $\sA_{\qc}(X)$ to  $\sA_{\qc}(U)$
is an isomorphism $u^*\!\fst\iso g_*v^*$ \cite[9.3.3]{EGA1},
and one has the natural composite map
\begin{align*}
\varrho^{}_0=\varrho^{}_0(f,u)\colon \Hom_Y(\fst\OX,\OY)
&\,\lto
 \,\Hom_U(u^*\!\fst\OX, u^*\OY)\\
&\iso\Hom_U(g_*\OV, \OU),
\end{align*}
which sends $t\colon \fst\OX\to \OY$ to the natural composite
\[
g_*\OV\iso u^*\!\fst\OX\xto{\<u^{\<*}\<t\>\>} u^*\OY\iso\OU.
\]
That $s=\varrho^{}_0t$ is equivalent to the commutativity of the natural diagram
\[
\def\7{$u^*\!\fst\OX$}
\def\8{$g_*\OV$}
\def\9{$u^*\OY$}
\def\0{$\OU$}
 \bpic[xscale=2.4, yscale=1.5]
   \node(14) at (1,-1){\7} ;
   \node(15) at (2,-1){\8} ;
   
   \node(24) at (1,-2){\9} ;
   \node(25) at (2,-2){\0} ;
 
    \draw[->] (14)--(15)  node[above, midway, scale=.75] {$\Iso$} ; 
    \draw[->] (24)--(25) node[above, midway, scale=.75] {$\Iso$} ;  
    
    \draw[->] (14)--(24) node[left=1pt, midway, scale=.75] {$u^*t$};   
      
    \draw[->] (15)--(25) node[right=1pt, midway, scale=.75] {$s$}  ; 
  
  \epic
\]
or of its adjoint
\[
\def\1{$\fst\OX$}
\def\2{$u_*u^*\!\fst\OX$}
\def\3{$u_*g_*\OV$}
\def\4{$\OY$}
\def\5{$u_*u^*\OY$}
\def\6{$u_*\OU$}
 \bpic[xscale=2.4, yscale=1.55]
   \node(11) at (4,-1){\1} ;   
   \node(12) at (5,-1){\2} ;
   \node(13) at (6.08,-1){\3} ; 
   
   \node(21) at (4,-2){\4} ;  
   \node(22) at (5,-2){\5} ;
   \node(23) at (6.08,-2){\6} ;
 
    \draw[->] (11)--(12)  ;  
    \draw[->] (12)--(13)  node[above, midway, scale=.75] {$\Iso$} ;
   
    \draw[->] (21)--(22)  ;
    \draw[->] (22)--(23)  node[above, midway, scale=.75] {$\Iso$} ;
    
    \draw[->] (11)--(21) node[left=1pt, midway, scale=.75] {$t$} ;
      
    \draw[->] (13)--(23) node[right=1pt, midway, scale=.75] {$u_*s$} ; 
  
  \epic
\]

Suppose in addition that \va1

(i) \emph{$u$ is schematically surjective and qcqs,} in other words, the natural map \mbox{$\OY\to u_*\OU$} is 
injective---whence the map $\varrho^{}_0$ is injective and  $v_*\OV$ is quasi-coherent; and that 

(ii) \emph{$g$ is  finite and \'etale}---so that there is a unique $g_*\OV$-\emph{isomorphism}
\[ 
c^{}_{g}\colon g_*\OV\iso \sHom_U(g_*\OV,\OU)
\]
that sends $1\in\Gamma(U,g_*\OV)$ to the usual trace map 
\(
\tr^{}_g\colon g_*\OV\to \OU.\va1
\)

Denote base change to~any open subscheme $W\subset Y$ by ``subscript~$W.$" 
With $\tilde\varrho^{}_0$ the sheafification of the map of presheaves associating to 
any $W$ the map $\varrho^{}_0(f_W,u_W)$,
one has then the natural composite injective map
\[
\varrho'\colon\sHom_Y(\fst\OX,\OY)\underset{\lift1.4,\>\tilde\varrho^{}_0,\,}{\hookrightarrow} u_*\sHom_U(g_*\OV, \OU)
\underset{\lift1.2,c^{-\<1}_g\>,\>}\iso u_*g_*\OV=\fst v_*\OV.
\]
This map is in a natural way $\fst\OX$-linear, so
can also be represented, with notation as in \S\ref{trprelim}, as
\[
\varrho'\colon \sHom_\psi(\fst\OX,\OY)\hookrightarrow \bar\fst v_*\OV.
\]
Setting 
\[
\mathscr C_{\phi,u}\set\varrho'\sHom_\psi(\fst\OX,\OY),
\]
 one gets a natural $\OX$-isomorphism 
\[
H^{\>0}\<\<f^\flat\OY=\bar f^*\sHom_\psi(\fst\OX,\OY)\iso \bar{f\:\<}^{\!\<*}\mathscr C_{\phi,u}=\colon\mathscr C_{\<\<f,u} \subset v_*\OV.
\]

\pagebreak[3]
For an explicit description of the $\OX$-module $\mathscr C_{\<\<f,u}$, note that by the above description of the image of $\varrho^{}_0\>$, $\mathscr C_{\phi,u}\>$ is the sheafification of
the presheaf  $\>\mathscr C^{\>0}_{\phi,u}\>$ for which, with ``subscript~$W$" as above,
$\mathscr C^{\>0}_{\phi,u}(W)$ is the set of $r\in \Gamma(V_W,\OV)$ such that the natural composite\va{-10}
\[
f^{}_{W\<*}\CO_{\<X_W}\lto u^{}_{W\<*} \>u_{W}^*f^{}_{W\<*}\CO_{\<X_W}\iso u^{}_{W\<*}\>g^{}_{W\<*}\CO_{V_W}\xto{\!u^{}_{W\<*}\<\<\big(\<r\cdot \tr^{}_{g^{}_W}\<\big)}
u^{}_{W\<*}\CO_{U_W}\\[2pt]
\]
factors (necessarily uniquely) as  $f^{}_{W\<*}\CO_{\<X_W}\to\OW\hookrightarrow u^{}_{W\<*}\CO_{U_W}$.\va3

 The next proposition results.

\begin{subprop}\label{complementary}
Let\/  $f\colon X\to Y$ be an affine scheme-map such that the\/ $\OY$-module\/ $\sHom_Y(\fst\OX,\OY)$ is quasi-coherent, 
and let\/ \textup{\ref{fib}} be a fiber square in which $u$ is schematically surjective $($i.e., the associated map\/~\mbox{$\OY\to u_*\OU$} is injective$)$ and qcqs,
and the map\/~$g$ is finite and \'etale. 
The representing pair\/ 
$(H^{\>0}\<\<f^\flat\OY\<,\> t'_{\OY})$ in\/ \S\textup{\ref{trprelim}}\va{.6}
is naturally isomorphic~to the pair whose components are the 
``complementary sheaf''~$\mathscr C_{\<\<f,u}\subset v_*\OV$ $($see above$)$
and the restriction of\/ $u_*\tr_g$ to 
\(
\fst\mathscr C_{\<\<f,u}\subset \fst v_*\OV=u_*g_*\OV.
\)
\end{subprop}

\begin{subexs} (a) In~\ref{complementary}, if $u$ is the identity map then
$\mathscr C_{\phi,u} = \fst\OX$, giving the last paragraph in section~\ref{trprelim}.\va1

(b) Let  $R$ \kf be an integral domain with fraction field~$K\<$, and $R\to S$ a ring-homomorphism with $S$ finitely presentable\va{.5} as an $R$-module and $L\set S\otimes_RK$ a separable $K$-algebra, 
with trace map $\tr^{}_{\<\<L/K}\colon L\to K$. 
Let \ref{fib} be the scheme\kf-diagram corresponding to the natural diagram
\[
\def\1{$L$}
\def\2{$S$}
\def\3{$K$}
\def\4{$R$}
 \bpic[xscale=2, yscale=1.35]

   \node(11) at (1,-1){\1} ;   
   \node(12) at (2,-1){\2} ; 
   
   \node(21) at (1,-2){\3} ;  
   \node(22) at (2,-2){\4} ;
   
    \draw[<-] (11)--(12) node[above, midway, scale=.75] {$j$}  ;  
 
    \draw[<-] (21)--(22)  node[below, midway, scale=.75] {$i$}  ;

    \draw[<-]  (11)--(21) ;
    
    \draw[<-]   (12)--(22); 
                      
    \epic
\]
Then the complementary sheaf $\mathscr C_{\<\<f,u}$ is the sheafification of the $S$-module 
$\{x\in L\mid \tr^{}_{\<\<L/K}(x(jS))\subset R\}$.
\end{subexs}
\end{subcosa}
\end{cosa}

\begin{cosa}\label{regimm} 
This section treats  duality for a class of \emph{perfect 
closed immersions} of schemes, including all regular immersions.
Described, on this class,  is a concrete realization of the pseudofunctor~$(-)^\flat$ (see Proposition~\ref{KosReg} and Theorem  ~\ref{Kozpf1}), as well~as its interaction with $\Otimes{}$ and with $\R\sHom$ (see Proposition~\ref{ci tensor hom}). One prior version of such material can be found in~\cite[\S\S2.5--2.6]{Co}.\va2

Throughout, $f\colon X\to Y$ will be a closed immersion of schemes, and $\CI$~the kernel of the natural map $\OY\to \fst\OX$. The functor $f_*\colon\sA_\qc(X)\to\sA_\qc(Y)$ is ``extension by 0,"  and its natural left adjoint~$f^*$ associates to $G\in\sA_\qc(Y)$ the restriction to $X$ of $G/\CI G$. 	 So $f^*\mkern-2.5mu\fst$ is the identity functor of $\sA_\qc(X)$; and for $G\in\sA_\qc(Y)$, the unit map
$G\to \fst f^*\<G$ identifies naturally with the canonical surjection $G\twoheadrightarrow G/\CI G$.

\begin{subcosa} \label{I and Tor}

Define the $\OY$-isomorphism
\begin{equation}\label{I/I^2}
\triangledown\kern-1pt^{}_{\<\!f}\colon \CI\<\</\CI^{\>2}\iso 
\stor_1^{\OY}\!(\fst\OX,\fst\OX)=\fst H^{-\<1}\>\LL f^*\mkern-2.5mu\fst\OX,
\end{equation}
to be $-\partial^{-1}$ where $\partial$ is the usual connecting isomorphism
\[
\stor_1^{\OY}\<\<(\fst\OX,\fst\OX)\iso \stor_0^{\OY}\<\<(\fst\OX,\CI)= \CI\<\</\CI^{\>2}
\]
associated to the natural exact sequence $0\to\CI\to\OY\to\fst\OX\to 0$.\va1

Any  $\OY$-surjection $\vartheta\colon P\twoheadrightarrow\CI$  with $P$~\emph{flat} expands to a flat resolution \mbox{$P_\bullet:\cdots\to P'\to P \to \OY$} of $\fst\OX\cong\OY\!/\CI$, entailing natural isomorphisms
\[
\stor_1^{\OY}\<\<(\fst\OX,\fst\OX)\iso H_1(P_\bullet/\CI P_\bullet)
\iso\CI\<\</\CI^{\>2}
\]
whose composition one finds, by dissecting definitions, to be $-\partial=\triangledown\kern-1pt{}^{\>-1}_{\<\!f}$.\va1

For any flat  $u\colon Y'\to Y\<$, the projection $f'\colon X'\set X\times_Y Y'\to Y'$ being a closed immersion, one gets (via $P_\bullet$, for example) a natural identification\va{-2}
\begin{equation}\label{localize s}
\triangledown\kern-1pt^{}_{\<\!f'} = u^{\<*}\triangledown\kern-1pt^{}_{\<\!f}\>.
\end{equation}
\vskip-2pt
\enlargethispage*{3pt}
The natural composite map
\begin{equation}\label{torprod}
\begin{aligned}
 \LL f^*\mkern-2.5mu\fst \OX\Otimes{\sX}\> \LL f^*\mkern-2.5mu\fst \OX
&\iso\<\< \LL f^*\<(\fst \OX\Otimes{Y}\>\fst \OX)\\
&\,\lto\, \LL f^*\mkern-2.5mu\fst (\OX\Otimes{\sX}\OX)\<\<\iso\<\< \LL f^*\mkern-2.5mu\fst \OX. 
\end{aligned}
\end{equation}
makes the graded group
\mbox{$\oplus_{n\ge 0}\>\>H^{-n} \LL f^*\mkern-2.5mu\fst\OX$} \va{.6} into a strict (= alternating) 
graded $\OX$-algebra. 
(Localize, and see, e.g., \cite[p.\,201, Exercise 9(c)]{Bo}).%
\footnote{\,We'll need this only for ``Koszul-regular" $f$ (see \S\S\ref{K-reg},\ \ref{fundloc}), 
in which case  one can use---locally---the exterior-algebra structure on a Koszul complex that resolves $\fst\OX$.}\,)
Thus with  ${\lift1.7,{\pmb\bigwedge},}$ denoting ``exterior algebra," the isomorphism
\[
f^*\triangledown\kern-1pt^{}_{\<\!f}\colon f^*\<\big(\CI\<\</\CI^{\>2}\big)\iso 
f^*\mkern-2.5mu\fst H^{-\<1}\>\LL f^*\mkern-2.5mu\fst\OX = H^{-\<1}\>\LL f^*\mkern-2.5mu\fst\OX
\]  
extends uniquely to a homomorphism of graded $\OX$-algebras
\begin{equation}\label{fund1}
{\lift1.7,{\pmb\bigwedge},}_{\mkern-2mu \lift.9,X,}
\>f^*\<\<\big(\CI\<\</\CI^{\>2\>}\big)
\set
\oplus_{n\ge 0}\,{\lift1.7,{\bigwedge},\lift1.9,{\<\<n},}_{\mkern-10mu \lift.9,X,}\>f^*\<\<\big(\CI\<\</\CI^{\>2\>}\big)\lto 
\oplus_{n\ge 0}\>\>H^{-n}\LL f^*\mkern-2.5mu\fst \OX.
\end{equation}
\end{subcosa}

\vskip1pt
\begin{subcosa}\label{K-reg}
Suppose the $\OY$-ideal $\CI$ is generated by~a sequence of global sections\va{.5}  
$\mathbf t\set(t_1,t_2, \dots, t_d)$ that is
\emph{Koszul-regular},\va{1.2} i.e., with $K_i$ the $\OY$-complex that is 
\mbox{$\>\OY\xto{\lift.95,\<t_i\>\>,}\OY$} in degrees~\mbox{-1} and 0 
and that vanishes elsewhere,\va1 the Koszul complex 
$K(\mathbf t)\set\otimes_{i=1}^d K_i$ 
is  a finite free resolution of~$\fst\OX=H^0\<K(\mathbf t)$.\va{.6}
This~holds if the germ of~$\>\mathbf t$ at any point $y$ in the image of $f$ is
regular,\va1 see \cite[tag~063K]{Stacks}; and  conversely if $Y$ is locally noetherian \cite[tag~063L]{Stacks} or, for arbitrary $Y\<$, modulo ``smooth localization" \cite[tag~0629]{Stacks}.\va1

As a graded group,  $K(\mathbf t)$ identifies naturally with 
${\lift1.7,{\pmb\bigwedge},}_{\lift.9,Y,}\OY^{\lift1.4,d,}$;
and one checks that \va{.6} exterior-algebra multiplication 
\begin{equation}\label{mu}
\mu_{\>\mathbf t}\colon K(\mathbf t)\otimes_{\OY} K(\mathbf t)\to K(\mathbf t)
\end{equation}
is a map of \emph{flat $\OY\<$-complexes} that is $\D(Y)$-isomorphic (via the natural map $K(\mathbf t)\to\fst\OX$) to the natural composite map\looseness=-1
\[
\fst \OX\Otimes{Y}\>\fst \OX
\lto\fst (\OX\Otimes{\sX}\OX)\<\<\iso\<\<\fst \OX. 
\]
Using, e.g., \cite[(3.2.4.1)]{li},  one deduces a $\D(X)$-isomorphism 
from $f^*\<\<\mu_{\mathbf t}$ to~the composite map \eqref{torprod}.
Hence the natural isomorphism
\[
f^*\<\<K(\mathbf t)=
\oplus_{n =0}^d\< \big(H^{-n}f^*\<\<K(\mathbf t)\big)[n]\iso
\oplus_{n =0}^d\! \<\big(H^{-n}\LL f^*\mkern-2.5mu\fst\OX\big)[n]
\]
is an isomorphism of graded $\OX$-algebras. (Details left to the reader.)

Now let $P_\bullet\set K(\mathbf t)\to\fst\OX$ be the natural map, and 
let the map $\vartheta$ in~\ref{I and Tor} be the induced map $P\set P_0=\OY^d\twoheadrightarrow \CI$
(which sends the $i$-th canonical generator of $\OY^d$ 
to $t_i\ (1\le i\le d\>)$).  
From $H^{-1}\<K(\mathbf t)=0$ it follows that 
$f^*\vartheta\colon \CO_{\!X}^d\to f^*\CI=f^*\<(\CI\<\</\CI^{\>2})$ is an isomorphism, the resulting natural composite $\D(X)$-isomorphism \va2
\begin{equation}\label{fund1'}
\oplus_{n =0}^{d^{\mathstrut}}{\lift1.7,{\bigwedge},\lift1.9,{\<\<n},}_{\mkern-11mu \lift.9,X,}\>f^*\<\<\big(\CI\<\</\CI^{\>2\>}\big)[n]
\iso
\oplus_{n =0}^d\,{\lift1.7,{\bigwedge},\lift1.9,{\<\<n},}_{\mkern-11mu \lift.9,X,}\<
\big(\CO_{\!X}^{\lift1.3,d,}\big)[n] 
\iso
f^*\<\<K (\mathbf t)
\iso 
\LL f^*\mkern-2.5mu\fst\OX
\end{equation} 
being sent by  $H^{\<-1}$ to the isomorphism $f^*\triangledown\kern-1pt^{}_{\<\!f}$.
Furthermore, the preceding paragraph implies that application of the functor $\oplus_{n =0}^d \>H^{-n}$  to \eqref{fund1'} gives an \emph{isomorphism of graded $\OX$-algebras,} which, being determined by what it does in degree~1,  \emph{must be the map} \eqref{fund1} 
(\kern-.5pt which does not depend on the choice of the generating sequence~$\mathbf t$).
\end{subcosa}

\begin{subcosa}\label{fundloc}

Let $g\colon V\to W$ be a map of ringed spaces. 
For each $d\in\mathbb Z$ and \mbox{$F,G\in\D(W)$,} one has the natural map\va3
\[
\ext^{\>d}_W\<(F,G\>)
= \Hom_{\D(W)}\<(F,G[d\>]\>)
\lto\, \Hom_V\<(H^{-d}\LL g^*\<\<F,H^{0}\LL g^*\<G).
\]
Base\kf-changing to arbitrary open subsets of $W\<$, one gets a map of presheaves, whose sheafification is  a bifunctorial $\D(W)$-map
\[
\Psi(g,F,G,d\>)\colon \sExt^{\>d}_{W}(F,G\>)\lto 
g_*\sHom_V(H^{-d}\>\LL g^*\<\<F,H^{\>0}\LL g^*\<G\>).
\]

In particular, for $g$ the closed immersion $f\colon X\to Y\<$, $F\set\fst\OX=\OY\</\CI$, and $G$~an 
$\OY$-module, one has the $\D(Y)$-map
\[
\Psi(f,\fst\OX,G,d\>)\colon \sExt^{\>d}_Y(\fst\OX,G\>)\lto 
\fst\sHom^{}_\sX(H^{-d}\>\LL f^*\mkern-2.5mu\fst\OX,f^*\<G\>),
\]
from which one gets, via \eqref{fund1}, the natural composite map
\begin{align*}
 \Phi(f,G,d\>)\colon \sExt^{\>d}_{Y}(\fst\OX,G\>)&\,\lto \,
 \fst\sHom_\sX(f^*\<\<{\lift1.7,{\bigwedge},\lift2.4,{\<d},}_{\mkern-7mu Y^{\mathstrut}}\<\<\big(\CI\<\</\CI^{\>2\>}\big),f^*\<G\>)\\
 &\iso
\sHom_Y({\lift1.7,{\bigwedge},\lift2.4,{\<d},}_{\mkern-7mu Y^{\mathstrut}}\<\<\big(\CI\<\</\CI^{\>2\>}\big),G/\CI G\>),
 \end{align*}
sheafifying the natural composite map
\begin{equation}\label{unsheafify}
\begin{aligned}
\Hom_{\D(Y)}(\fst\OX,G[d\>])&\lto\Hom_Y\<(H^{-d}\>\LL f^*\<\<\fst\OX,H^{\>0}\LL f^*\<G\>)\\
&\lto \Hom_Y\<({\lift1.7,{\bigwedge},\lift2.4,{\<d},}_{\mkern-7mu Y^{\mathstrut}}\<\<\big(\CI\<\</\CI^{\>2\>}\big),G/\CI G\>).
\end{aligned}
\end{equation}

We'll say that $f$ is a \emph{Koszul-regular closed immersion of codimension} $d$ if $X$ is covered by open subsets of $\>Y$ over each of which the ideal $\CI$ is generated by 
a length~$d$, Koszul-regular, sequence of sections. Such an $f$ is perfect.

\begin{sublem}[cf.~{\cite[p.\,179, 7.2]{RD}}]\label{Kreg} 
If\/ $f\colon X\to Y$ is a K\kern-.5pt oszul-regular closed immersion of codimension $d,$ 
$\CI$ is the kernel of the associated map\/ $\OY\to\fst\OX,$ and\/ 
$G$ is an\/ $\OY$-module, then 
$\Phi(f,G,d\>)$ is an isomorphism
 \[
 \sExt^{\>d}_{Y}(\fst\OX,G\>)\iso
\sHom_Y({\lift1.7,{\bigwedge},\lift2.4,{\<d},}_{\mkern-7mu Y^{\mathstrut}}\<\<\big(\CI\<\</\CI^{\>2\>}\big),G/\CI G\>).
 \]
 \end{sublem}
 
\begin{proof}  As \eqref{fund1} is an isomorphism 
(see last paragraph in section~\ref{K-reg}), the assertion means that 
$\Psi\set\Psi(f,\fst\OX,G,d\>)$ \emph{is an isomorphism}.
The question being local (see \eqref{localize s}), one may assume that 
$\>\CI$ is generated~by a Koszul-regular sequence 
 $\mathbf t=(t_1,t_2,\dots,t_d)$ of sections. Then, by \cite[p.\,121, Cor.\,4.6]{Il}, the natural map, 
with $\K(Y)$ the homotopy category of $\OY$-complexes, 
\(
 \Hom^{\mathstrut}_{\>\K(Y)}(K(\mathbf t),G[d\>]\>) 
 \underset{\lift1.2,\alpha\>,}\lto\Hom_{\D(Y)}(K(\mathbf t),G[d\>]\>)
\)
sheafifies to an isomorphism; and since the composition of the natural sequence of maps 
\begin{align*}
 \Hom^{\mathstrut}_{\>\K(Y)}(K(\mathbf t),G[d\>]\>) 
&\underset{\lift1.2,\alpha\>,}\lto\Hom_{\D(Y)}(K(\mathbf t),G[d\>]\>)\\[-3pt]
&\underset{\lift1.2,\beta,}\lto
\Hom^{}_{\OX}\<\<(H^{-d}\LL f^*\<\<K(\mathbf t),H^{\>0}\LL f^*\<G)\\
&\!\iso\! \Gamma(Y,G)/\mathbf t\Gamma(Y,G),
\end{align*}
is just the natural isomorphism,
therefore $\Psi$, the sheafification of $\beta$, is indeed an isomorphism.
\end{proof}

For a Koszul-regular immersion $f\colon X\to Y$, the next proposition gives another representation of $f^\flat$ and the counit map $\fst f^\flat \to \id$.\va1

As a preliminary, note that there exists locally a  Koszul-regular sequence~$\mathbf t$ that generates the kernel $\CI$ of the natural map $\OY\to\fst\OX$; and one has natural isomorphisms
\[
\fst f^\flat \OY
\underset{\ref{fst fflat}}\iso
\R\>\sHom_Y(\fst\OX, \OY)
\iso
\sHom_Y(K(\mathbf t), \OY)\iso K(\mathbf t)[-d\>],\\[2pt]
\]
\vskip-3pt
\noindent
the last being inverse to the adjoint of the natural composite map
\[
K(\mathbf t)[-d\>]\otimes_Y K(\mathbf t) \iso (K(\mathbf t)\otimes_Y K(\mathbf t))[-d\>]
\xto[\eqref{mu}]{\<\mu_{\>\mathbf t}[-d\>]}K(\mathbf t)[-d\>]\lto \OY.\\[-3pt]
\]
Therefore, $\fst f^\flat \OY$---and hence $f^\flat \OY$---has vanishing cohomology in every degree
other than $d$. So there are natural global $\D(X)$-isomorphisms\va{-2} 
\[
(H^d\<f^\flat\OY\<\<)[-d\>]\iso f^\flat\OY\\[-4pt]
\] 
and \va{-4}
\[
\def\1{$(H^d\<\fst f^\flat\OY\<\<)[-d\>]$}
\def\2{$\big(H^{d}\R\>\sHom_Y(\fst\OX,\OY\<\<)\big)[-d\>]$}
\def\3{$\fst f^\flat\OY$}
\def\4{$\R\>\sHom_Y(\fst\OX,\OY\<\<)$}
 \bpic[xscale=5, yscale=1.2]
  \node(11) at (1,-1){\1} ;  
  \node(12) at (2,-1){\3} ;  
   
  \node(21) at (1,-2){\2} ; 
  \node(22) at (2,-2){\4} ;

    \draw[->] (11)--(12)  node[above, midway, scale=.75] {$\Iso$} ;
  
    \draw[<-] (11)--(21) node[left, midway, scale=.75] {$\simeq$} ;

    \draw[->] (12)--(22) node[right, midway, scale=.75] {$\simeq$} ;
 \epic
\]

\begin{subprop}\label{KosReg} Let\/ $f\colon X\to Y$ be a K\kern-.5pt oszul-regular closed immersion of codim\-ension $d,$ let $\CI$ be the kernel of the associated map\/ $\OY\to\fst\OX,$ and let\/ $\omega_{\<\<f}$ be the locally free\/ $\OX$-complex \va{-3}
\[
\omega_{\<\<f}\set \sHom^{}_\sX({\lift1.7,{\bigwedge},\lift2.4,{\<d},}_{\mkern-10mu X^{\mathstrut}}\<f^*\<\big(\CI\<\</\CI^{\>2\>}\big),\OX\<\<)[-d\>].
\]
Let\/ $t'$ be the natural composite map
\begin{align*}
\fst\omega_{\<\<f} 
&\iso\sHom_Y({\lift1.7,{\bigwedge},\lift2.4,{\<d},}_{\mkern-8mu Y^{\mathstrut}}
\<\<\big(\CI\<\</\CI^{\>2\>}\big),\OY\</\>\CI\>)[-d\>]\\[-1pt]
&\underset{\lift1,\ref{Kreg},}\iso
\sExt_Y^d(\fst\OX,\OY\>)[-d\>]\\[-2pt]
&\iso 
\big(H^{d}\R\>\sHom_Y(\fst\OX,\OY\<\<)\big)[-d\>] \\[1pt]
&\iso\R\>\sHom_Y(\fst\OX,\OY\<\<)\\
&\,\lto\,\R\>\sHom_Y(\OY\<,\OY\<\<)\iso\OY.
\end{align*}

\pagebreak[3]
The functorial map\va{-3} dual to the natural composite\/ $\Dqc(Y)$-map
\[
\fst(\omega_{\<\<f} \Otimes{\sX}\LL f^*\<G)
\underset{\textup{\ref{project}}\>}\iso
\fst\omega_{\<\<f} \Otimes{Y}G
\xto[\!\lift1.2,\via\> t',]{}
\OY\Otimes{Y}G\iso G\\[-3pt]
\]
is a\/ $\Dqc(X)$-isomorphism
\[
c^{\>\flat}_{\<\<f}(G)\colon \omega_{\<\<f}\Otimes{\sX} \LL f^*\<G\iso f^\flat\< G
\qquad(G\in\Dqc(Y).\\[-2pt]
\]
 \end{subprop}

\begin{proof}
Let $c^{\>\flat}\colon \omega_{\<\<f}\to f^\flat\OY$ be the natural composite $\OX$-isomorphism\va{3}
\begin{equation*}
\begin{aligned}\label{KosReg.1}
 \omega_{\<\<f}&\,=\!=\;\sHom_{\sX\<}^{\mathstrut^{\mathstrut}}({\lift1.7,{\bigwedge},\lift2.4,{\<d},}_{\mkern-10mu X^{\mathstrut}}\<f^*\<\<\big(\CI\<\</\CI^{\>2\>}\big),\OX\<\<)[-d\>]\\
&\iso
 f^*\sHom_Y({\lift1.7,{\bigwedge},
     \lift2.4,{\<d},}_{\mkern-7mu Y^{\mathstrut}}\<\<\big(\CI\<\</\CI^{\>2\>}\big),\OY\</\>\CI\>)[-d\>]\\
&\underset{\ref{Kreg}\>\>}
\iso 
f^*\mkern-.5mu \sExt^{\>d}_{Y}(\fst\OX,\OY\<)[-d\>]\\[-2pt]
&\iso 
f^*\<(H^d\<\fst f^\flat \OY)[-d\>]\\
&\iso
f^*\<(\fst H^d\<\<f^\flat\< \OY)[-d\>]\\
&\iso
(H^d\<f^\flat\< \OY)[-d\>]\iso f^\flat\OY.
\end{aligned}\tag{\ref{KosReg}.1}
\end{equation*}

Then with $\chi(f,\OY\<,G\>)\colon f^\flat \OY\Otimes{\sX}\LL f^*\<G\iso f^\flat G$
the  isomorphism from Proposition~\ref{flat and tensor}(ii), one has
\[
c^{\>\flat}_{\<\<f}(G\>)\set\chi(f,\OY\<,G)\smallcirc (c^{\>\flat}\Otimes{\sX}\id_{\LL f^*G}),
\] 
that is, the border of the following natural diagram commutes:
\[
\def\1{$\fst(\omega_{\<\<f}\Otimes{\sX} \LL f^*\<G)$}
\def\2{$\fst(f^\flat\OY\Otimes{\sX} \LL f^*\<G)$}
\def\3{$\fst f^\flat G$}
\def\4{$G$}
\def\5{$\fst\omega_{\<\<f}\Otimes{Y} G$}
\def\6{$\fst f^\flat\OY\Otimes{Y} G$}
\def\7{$\OY\Otimes{Y} G$}
 \bpic[xscale=5, yscale=1.6]
  \node(11) at (1,-1){\1} ;  
  \node(12) at (2,-1){\2} ;  
  \node(13) at (3,-1){\3} ;
   
  \node(22) at (2,-2){\6} ; 
  \node(23) at (3,-2){\4} ;
    
   \node(31) at (1,-3){\5} ;  
   \node(33) at (3,-3){\7} ;
    
    \draw[->] (11)--(12)  node[above, midway, scale=.75] {$\via c^{\>\flat}$} ;
    \draw[->] (12)--(13)  node[above, midway, scale=.75] {$\fst\chi$} ;   
     
    \draw[->] (31)--(33)  node[below, midway, scale=.75] {$\via\> t'$} ;

    \draw[->] (11)--(31) node[left=1pt, midway, scale=.75] {$\simeq$} 
                                  node[right, midway, scale=.75] {\eqref{projf}} ;

    \draw[->] (12)--(22) node[right=1pt, midway, scale=.75] {$\simeq$} 
                                  node[left, midway, scale=.75] {\eqref{projf}} ;
    
    \draw[->] (13)--(23)  ; 
    \draw[->] (33)--(23) node[right=1pt, midway, scale=.75] {$\simeq$} ;
      
    \draw[->] (31)--(22) node[above=1pt, midway, scale=.75] {$\via c^{\>\flat}\mkern7mu$} ;
    \draw[->] (22)--(33) ;
   
    \node at (1.47,-2){\circled1} ;
    \node at (2.6,-2){\circled2} ;
    \node at (2,-2.57){\circled3} ;
 \epic
\]
Indeed, the commutativity of subdiagram \circled1 is clear, that of \circled2 is given by the last assertion
 in \ref{flat and tensor}(i), and that of \circled3 (signifying that $c^{\>\flat}$ is dual to $t'$) is readily verified.

Thus $c^{\>\flat}_{\<\<f}(G\>)$ is an isomorphism, as asserted.
\end{proof}

The next lemma provides
an alternative, local,  description of the isomorphism in~\ref{Kreg}, hence of \eqref{KosReg.1} and
$c^{\>\flat}_{\<\<f}(\OY\<)$.

\begin{sublem}\label{alt c}
In Proposition~\textup{\ref{KosReg},} suppose that\/ $\CI$ is generated by a Koszul-regular sequence\/ $\mathbf t=(t_1,\dots,t_d)$ in $\Gamma(Y,\CI\>),$ and let 
$\vartheta\colon \OY^{\>d}\to\CI$ be the $\OY$-homomorphism taking the $i$\kf-th canonical generator of\/~$\OY^{\>d}$ to\/~$t_i$ $(1\le i\le d\>)$ $($so that, since $H^{-1}K(\mathbf t)=0,$ $f^*\vartheta$ is an isomorphism$)$. 

The following natural diagram of\/ $\OX$-isomorphisms commutes.
\[
\def\1{$\sHom_{\sX\<}({\lift1.7,{\bigwedge},\lift2.4,{\<d},}_{\mkern-10mu X^{\mathstrut}}\<
f^*\<\<\big(\CI\<\</\CI^{\>2\>}\big)\<,\OX\<\<)$}
\def\2{$\sHom_{\sX\<}({\lift1.7,{\bigwedge},\lift2.4,{\<d},}_{\mkern-10mu X^{\mathstrut}}\<\<
\big(\OX^{\>d}\big)\<,\OX\<\<)$}
\def\3{$f^*\sHom_Y({\lift1.7,{\bigwedge},\lift2.4,{\<d},}_{\mkern-7mu Y^{\mathstrut}}\<\<\big(\CI\<\</\CI^{\>2\>}\big)\<,\OY\</\>\CI\>)$}
\def\4{$H^d\>\sHom_{\sX\<}(f^*\<\<K_Y(\mathbf t),\OX)$}
\def\5{$f^*\mkern-.5mu \sExt^{\>d}_{Y}(\fst\OX,\OY\<)$}
\def\6{$H^d\<f^*\sHom_{Y}(K_Y(\mathbf t),\OY)$}
\def\7{$f^*\<\<H^d\<\fst f^\flat\OY$}
\def\8{$f^*\!H^d\>\sHom_{Y}(K_Y(\mathbf t),\OY)$}
  \bpic[xscale=6, yscale=1.6]

   \node(11) at (1,-1){\1} ;  
   \node(21) at (1,-2){\3} ;
   \node(31) at (1,-3){\5} ;
   \node(41) at (1,-4){\7} ;
    
   \node(12) at (2,-1){\2} ;
   \node(22) at (2,-2){\4} ;
   \node(32) at (2,-3){\6} ;
   \node(42) at (2,-4){\8} ;

     \draw[->] (11)--(12) node[above, midway, scale=.75] {$\via f^*\<\<\vartheta$} ;    
     \draw[<-] (42)--(41) ;
                                                                                                                                                                                                                                                                                                      
    \draw[<-] (11)--(21) ;
    \draw[<-] (21)--(31) node[left, midway, scale=.75] {\textup{\ref{Kreg}}} ;
    \draw[->] (31)--(41) ;

    \draw[<-] (12)--(22) ;
    \draw[<-] (22)--(32) ;
    \draw[<-] (32)--(42) ;
    
 \epic
\]
\end{sublem}

\begin{proof} Noting that $H^d\>\sHom_{Y}(K_Y(\mathbf t),\OY)$ is annihilated by~$\CI$, 
and that $H^d$~commutes with $\fst$,
 one checks that the adjoint diagram is isomorphic to the sheafification of the natural diagram
\[
\def\1{$\Hom_{\sX\<}({\lift1.7,{\bigwedge},\lift2.4,{\<d},}_{\mkern-10mu X^{\mathstrut}}\<
f^*\<\<\big(\CI\<\</\CI^{\>2\>}\big)\<,\OX\<\<)$}
\def\2{$\Hom_{\sX\<}({\lift1.7,{\bigwedge},\lift2.4,{\<d},}_{\mkern-10mu X^{\mathstrut}}\<\<
\big(\OX^{\>d}\big)\<,\OX\<\<)$}
\def\3{$\Hom_Y({\lift1.7,{\bigwedge},\lift2.4,{\<d},}_{\mkern-7mu Y^{\mathstrut}}\<\<\big(\CI\<\</\CI^{\>2\>}\big)\<,\OY\</\>\CI\>)$}
\def\4{$H^d\<\Hom_{\sX\<}(f^*\<\<K_Y(\mathbf t),\OX)$}
\def\5{$\Hom_{\D(Y)}(\fst\OX,\OY[d\>])$}
\def\6{$\fst f^*\!H^d\<\Hom_{Y}(K_Y(\mathbf t),\OY)$}
\def\7{$H^d\<\fst f^\flat\OY$}
\def\8{$H^d\<\Hom_{Y}(K_Y(\mathbf t),\OY)$}
  \bpic[xscale=6, yscale=1.6]

   \node(11) at (1,-1){\1} ;  
   \node(21) at (1,-2){\3} ;
   \node(31) at (1,-3){\5} ;
   \node(41) at (1,-4){\7} ;
    
   \node(12) at (2,-1){\2} ;
   \node(22) at (2,-2){\4} ;
   \node(32) at (2,-3){\8} ;
   \node(42) at (2,-4){\8} ;

     \draw[->] (11)--(12) node[above, midway, scale=.75] {$\via f^*\<\<\vartheta$} ;    
     \draw[->] (41)--(42) ;
                                                                                                                                                                                                                                                                                                      
    \draw[<-] (11)--(21) ;
    \draw[<-] (21)--(31) node[left, midway, scale=.75]{\eqref{unsheafify}} ;
    \draw[->] (31)--(41) ;

    \draw[<-] (12)--(22) ;
    \draw[<-] (22)--(32) ;
    \draw[double distance=2] (32)--(42) ;
 \epic
\]
So it suffices to see that this last diagram commutes---which one can do by pushing an
arbitrary $\D(Y)$-map $\fst\OX\to\OY[d\>]$ clockwise and counterclockwise around the diagram to the upper right corner.
\end{proof}
\end{subcosa}  

\begin{subcosa} \label{compose} Next, the setup for Theorem~\ref{Kozpf1}---pseudo\-functoriality 
of~$c^{\>\flat}_{\<\<-}$.

Let $X\xto{\,f\,} Y\xto{\,g\,}Z$  be Koszul-regular closed immersions of codim\-ensions $d$ and $e$ respectively, let $\CJ$ be the kernel of the natural map 
$\OZ\to g_*\OY$, and let
$\CL$ be the kernel of the natural map $\OZ\to g_*\fst \OX$---so that 
$\CJ\subset \CL$ and $\CI\set g^*(\CL/\<\<\CJ\>)$ is the kernel of the natural map $\OY\to\fst\OX$.
Then the map $gf$ is a Koszul-regular closed immersion of codimension $d+e$, 
see \cite[tag~067Q]{Stacks}.

\pagebreak[3]
The inclusion $\CJ\subset\CL$ induces an exact sequence\va{-2}
\[
0\to f^*\<g^*\<\big(\CJ/\CJ^2\big)\xto{\,i\,} f^*\<g^*\<\big(\CL/\<\CL^2\big)\xto{\,p\,} f^*(\CI/\CI^{\>2})\to 0
\]
of locally free $\OX$-modules of ranks $e$, $d+e$, and $d$, respectively, 
see \cite[tag~063N]{Stacks}, whence
a \emph{locally split} exact sequence, with $\upcheck E\set\sHom_{\OX}\<(E,\OX)$ for any $\OX$-module~$E$:\va{-3}
\[
0\to \upcheck{(f^*\<(\CI\<\</\CI^{\>2\>}))\!}\<
\xto{\,\upcheck p}\upcheck{(f^*\<\<g^*\<(\CL/\<\CL^{\>2\>}))\!}
\xto{\,\upcheck i}\upcheck{(f^*\<\<g^*\<(\CJ\<\</\<\<\CJ^{\>2\>}))\!}\to 0.
\]

Locally, there exists a right inverse $q$ of $p$, giving rise to the left inverse~$j$ of~$i$
such that $ij=\id-qp$, whence the isomorphism
\[
\upcheck{(f^*\<(\CI\<\</\CI^{\>2\>}))\!}\oplus \upcheck{(f^*\<\<g^*\<(\CJ\<\</\<\<\CJ^{\>2\>}))\!}\<
\underset{(\upcheck p\<\<\!,\,\upcheck j)}\iso \,
\upcheck{(f^*\<\<g^*\<(\CL/\<\CL^{\>2\>}))\!},
\]
\vskip-3pt
\noindent whence the standard isomorphism of presheaves (cf.~ \cite[Chap.~III, \S7.7\kf]{Bo70}), hence of sheaves:\va2
\begin{equation}\label{exact}
{\lift1.7,{\bigwedge},\lift2.4,{\<d},}_{\mkern-10mu X^{\mathstrut}}
 \<\< \big(\upcheck{(f^*\<(\CI\<\</\CI^{\>2\>}))\!} \> \big)^{\mathstrut}
\otimes^{}_\sX
{\lift1.7,{\bigwedge},\lift1.95,{\<e},}_{\mkern-10mu X^{\mathstrut}}
   \<\< \big(\upcheck{(f^*\<\<g^*\<(\CJ\<\</\<\<\CJ^{\>2\>}))\!} \> \big)
\iso
{\lift1.7,{\bigwedge},\lift2.4,{\<d+e},}_{\mkern-33mu X^{\mathstrut}}
  \mkern10mu  \big(\upcheck{(f^*\<\<g^*\<(\CL/\<\CL^{\>2\>}))\!} \> \big),
\end{equation}
readily seen to be \emph{independent of the choice of\/ $q$}, so that these local maps glue together into a natural global isomorphism of invertible $\OX$-modules. 
\end{subcosa}

\begin{subcosa}\label{explicit}
In more explicit algebraic terms, the maps\va1 $g$ and $f$ correspond locally to a pair of surjective ring homomorphisms 
$R\overset{\lift{.5},\varphi,\;}\twoheadrightarrow S \overset{\lift{.8},\xi,\ }\twoheadrightarrow T$, with kernels $J$ and~$I$ generated by Koszul-regular sequences $(r^{}_1,\dots,r_e)$ 
and~$(\bar s^{}_1,\dots,\bar s_d)$ respectively. Let $s_i\in R$ be such that  
\mbox{$\bar s_i=\xi(s_i)\ (1\le i\le d)$.} 
The $R$-sequence $(r^{}_1,\dots,r_e,s^{}_1,\dots, s_d)$ is 
Koszul-regular \cite[tag~0669]{Stacks}, 
and it generates the kernel $L$ of $\xi\varphi$. The inclusion $J\subset L$ induces an exact sequence
\begin{equation}\label{ciexact}
0\to T\otimes_S(J/J^2)\xto{\,\mathfrak i\,\>} L/\<L^2\xto{\,\mathfrak p\,} I/I^{\>2}\to 0
\end{equation}
of free $T$-modules of respective ranks $e$, $d+e$ and $d$ (see  just before~\eqref{fund1'}).
Let $\mathfrak j\colon L/L^2\to T\otimes_S(J/J^2)$ be the left inverse of $\mathfrak i$ such that 
\begin{align*}
\mathfrak j(r_m + L^2)&=1\otimes (r_m+J^2)&&\mkern-40mu(1\le m\le e),\\
\mathfrak j(s_n + L^2)&= 0 &&\mkern-40mu(1\le n\le d).
\end{align*}
Over $\spec(R)$, one checks, \eqref{exact} is the sheafification of the isomorphism\va3
\[
\lambda\colon\<{\lift1.7,{\bigwedge},\lift2.4,{\<d},}_{\mkern-7mu T^{\mathstrut}}
\<\<\Hom_{\>T}(I/I^2\<\<,T)
\otimes_{\mkern.5muT}
{\lift1.7,{\bigwedge},\lift1.95,{\<e},}_{\mkern-6mu T^{\mathstrut}}
\<\<\Hom_{\>T}(T\otimes_SJ/\<J^2\<\<,T)
\!\iso\!
{\lift1.7,{\bigwedge},\lift2.4,{\<d+e},}_{\mkern-29mu T^{\mathstrut}}
\mkern11mu \Hom_{\>T}(L/L^2\<,T),
\]
such that
\[
\lambda\big((\alpha^{}_1\wedge\dots\wedge\alpha^{}_d)\otimes(\beta^{}_1\wedge\dots\wedge\beta^{}_e)\big)=
(\alpha^{}_1\mathfrak p)\wedge\dots\wedge(\alpha^{}_d\>\>\mathfrak p)\wedge(\beta^{}_1\mathfrak j)\wedge\dots\wedge(\beta^{}_e\>\mathfrak j).
\]

\pagebreak[3]
\begin{subcosa}
Back in the global situation, if an $\OX$-module $E$ is locally free of rank $d$,\va{.8} there is an isomorphism
\(
{\lift1.7,{\bigwedge},\lift2.4,{\<\<d},}_{\mkern-10mu X^{\mathstrut}}\<\<\upcheck{(E\>})
\iso
\upcheck{({\lift1.7,{\bigwedge},\lift2.4,{\<d},}_{\mkern-10mu X^{\mathstrut}}\<E)\<\<},
\) 
\kf induced by the map $(\upcheck{E\>}){}^d \times E^{\>\>d} \to\OX$ that takes\va{.5} 
local sections $\big((\alpha_1,\dots,\alpha_d),(e_1,\dots,e_d)\big)$ to 
the determin\-ant $\det(\alpha_i\>e_j)$.\va{.6} 
Also, for a  finite\kf-rank locally free $\OY$-module~$F\<$, and  $\upcheck F\set\sHom_{\OY}\<(F,\OY)$,
there is a natural isomorphism $f^*\<\<(\upcheck F\>)\iso\upcheck{(f^*\<\<F\>)\<}$. 

Using such isomorphisms, one gets from \eqref{exact}---considered as a global isomorphism---a natural isomorphism of invertible $\OX$-modules\va2
\begin{equation}\label{exact'}
\upcheck{\big({\lift1.7,{\bigwedge},\lift2.4,{\<d},}_{\mkern-10mu X^{\mathstrut}}\<f^*\<\big(\CI\<\</\CI^{\>2\>}\big)\big)\!}\otimes^{}_\sX
f^*\<\upcheck{\big({\lift1.7,{\bigwedge},\lift2,{\<e},}_{\mkern-7mu Y^{\mathstrut}}\>g^*\<\big(\CJ\<\</\<\<\CJ^{\>2\>}\big)\big)\!}
\iso
\upcheck{\big({\lift1.7,{\bigwedge},\lift2.4,{\<d+e},}_{\mkern-33mu X^{\mathstrut}}\mkern12mu (g\<f)^{\<*}\<\big(\CL/\<\CL^{\>2\>}\big)\big)\!},
\end{equation}
\vskip2pt\noindent
whence, for the $\OX$-complexes\va2
\[\omega_{\<\<f}\set\upcheck{\big({\lift1.7,{\bigwedge},\lift2.4,{\<d},}_{\mkern-10mu X^{\mathstrut}}\<f^*\<\big(\CI\<\</\CI^{\>2\>}\big)\big)\!}[-d\>],\qquad
\omega_{g\<f}\set\upcheck{\big({\lift1.7,{\bigwedge},\lift2.4,{\<d+e},}_{\mkern-33mu X^{\mathstrut}}\mkern12mu (g\<f)^{\<*}\<\big(\CL/\<\CL^{\>2\>}\big)\big)\!}[-d-\<\<e]
\]
\vskip2pt\noindent
and the $\OY$-complex
\[
\omega_{g}\set\upcheck{\big({\lift1.7,{\bigwedge},\lift2,{\<e},}_{\mkern-7mu Y^{\mathstrut}}\>g^*\<\big(\CJ\<\</\<\<\CJ^{\>2\>}\big)\big)\!}[-e],
\]
a natural isomorphism 
\begin{equation}\label{tensoromega}
\omega_{\<\<f}\otimes_\sX \<f^*\<\omega_g \iso \omega_{gf},
\end{equation}
\vskip2pt\noindent
equal in degree $d+e$ to $(-1)^{de}$ times the isomorphism \eqref{exact'} 
(use the map $\theta_{ij}$ from \cite[(1.5.4)]{li},  with $i=-d,\:j=-e$).

In the local situation \ref{explicit},  routine manipulations show that \eqref{tensoromega} 
identifies naturally with the sheafification of the isomorphism of $\>T$-complexes\va3 
\begin{multline}\label{exact'aff}
h^{\mathstrut}\colon\mkern-6mu\Hom_{\>T}\!\big({\lift1.7,{\bigwedge},\lift2.4,{\<d},}_{\mkern-7mu T^{\mathstrut}}\<\<\big(I/I^{\>2\>}\big)\<,\<T\big)[-d\>]
\otimes_{\>T}\>
\Hom_S\!\big({\lift1.7,{\bigwedge},\lift2,{\<e},}_{\mkern-8mu S^{\mathstrut}}\<\<\big(J/J^{\>2\>}\big)\<,\<T\big)[-e]\\
\<\!\iso\!\!
\Hom_{\>T}\!\big({\lift1.7,{\bigwedge},\lift2.4,{\<d+e},}
 _{\mkern-29mu T^{\mathstrut}}\mkern10mu\big(\<L/\<L^{\>2\>}\big)\<,\<T\big)[-d-\<e]
\end{multline}
such that in degree $d+e$, with $r_{\!i}^L\set (r^{}_{\!i}+L^2)\in L/L^2$, etc.,\va2
\[
h(\alpha\otimes\beta) (r^L_{\!\lift1,1,}\wedge\dots\wedge r^L_{\!e}\wedge s^L_1\wedge\dots\wedge s^L_d)=
\alpha(\bar s^I_1\wedge\dots\wedge \bar s^I_d)\>\beta(r^J_{\!\lift1,1,}\wedge\dots\wedge r^J_{\!e}).
\]
\end{subcosa}


The \emph{pseudofunctoriality of} $c_{-}^\flat$ is given by the next theorem---essentially \cite[p.\,55, Theorem 2.5.1]{Co}, whose proof in \emph{ibid.,} section~2.6 is long and technical. The proof to be presented here is, \emph{mutatis mutandis}, the more direct one given in \cite[Appendix C.6]{NS}.

\begin{subthm}\label{Kozpf1} 
For\/ $G\in\Dqc(Z),$ the following\va{-1} natural\/ $\D(X)$-diagram commutes$\>:$
\[\mkern-3mu
\def\6{$\omega_{g\<f}\Otimes{\sX}\LL (gf)^*\<G$}
\def\7{$f^\flat\< g^\flat G $}
\def\8{$(gf)^\flat G$}
\def\9{$\omega_{\<\<f}\otimes_\sX \LL f^*\<\omega_g\Otimes{\sX}\LL f^*\LL g^*\<G$}
\def\0{$\omega_{g\<f}\Otimes{\sX}\LL f^*\LL g^*\<G$}
\def\ten{$f^\flat\OY\Otimes{\sX}\LL f^*\<\<g^\flat G$}
 \bpic[xscale=6, yscale=1.5]
   
   \node(11) at (1,-1){\9} ;
   \node(12) at (2,-1){\0} ; 
   
   \node(21) at (1,-2){\ten} ; 
   \node(22) at (2,-2){\6} ;
   
   \node(31) at (1,-3){\7} ;
   \node(32) at (2,-3){\8} ;
   
    \draw[->] (11)--(12) node[above, midway, scale=.75] {$\via\eqref{tensoromega}$ } ;
    
    \draw[->] (31)--(32) node[above, midway, scale=.75] {$\Iso$}
                                    node[below=1, midway, scale=.75] {\eqref{pf flat}} ;

     \draw[->] (11)--(21) node[left=1, midway, scale=.75] {$\ref{KosReg}$\textup{(i)}} 
                                     node[right, midway, scale=.75]{$\simeq$} ;
     \draw[->] (21)--(31) node[left=1, midway, scale=.75] {$\ref{flat and tensor}$\textup{(ii)}} 
                                         node[right, midway, scale=.75]{$\simeq$}  ;
     
     \draw[->] (12)--(22) node[left=1, midway, scale=.75] {$\simeq$} ;
     \draw[->] (22)--(32) node[right=1, midway, scale=.75] {$\ref{KosReg}$\textup{(i)}} 
                                     node[left=1, midway, scale=.75]{$\simeq$} ;

 \epic
\]
\end{subthm}

\begin{proof}[Proof of \textup{\ref{Kozpf1}}] It suffices to prove commutativity of the natural diagram\begin{small}
\[\mkern-3mu
\def\1{$(f^\flat\OY\<\<\Otimes{\sX}\LL f^*\<\<g^\flat\< \OZ)\Otimes{\sX}\LL f^*\LL g^*\<G$}
\def\2{$f^\flat \<\<g^\flat\< \OZ\Otimes{\sX}\LL f^*\LL g^*\<G$}
\def\3{$(gf)^\flat\< \OZ\<\Otimes{\sX}\<\LL f^*\LL g^*\<G\ $}
\def\4{$f^\flat\OY\<\<\Otimes{\sX}\LL f^*(g^\flat\< \OZ\Otimes{Y}\LL g^*\<G\>)$}
\def\5{$(gf)^\flat\< \OZ\<\Otimes{\sX}\<\LL (gf)^*\<G\ $}
\def\6{$\ f^\flat (g^\flat\< \OZ\Otimes{Y}\LL g^*\<G\>)$}
\def\7{$f^\flat\< g^\flat G $}
\def\8{$(gf)^\flat G$}
\def\9{$\omega_{\<\<f}\otimes_\sX \LL f^*\<\omega_g\Otimes{\sX}\LL f^*\LL g^*\<G$}
\def\0{$\omega_{g\<f}\Otimes{\sX}\LL f^*\LL g^*\<G$}
\def\ten{$f^\flat\OY\Otimes{\sX}\LL f^*\<\<g^\flat G$}
 \bpic[xscale=4.27, yscale=1.5]
   
   \node(01) at (1,0){\9} ;
   \node(03) at (2.95,0){\0} ; 
   
   \node(11) at (1,-1){\1} ;
   \node(12) at (2.25,-1){\2} ;   
   \node(13) at (2.95,-1.75){\3} ;
   
   \node(21) at (1,-2.25){\4} ;
   \node(23) at (2.95,-2.75){\5} ;
    
   \node(30) at (1,-3.75){\ten} ; 
   \node(31) at (2.25,-2.25){\6} ;  
   \node(32) at (2.25,-3.75){\7} ;
   \node(33) at (2.95,-3.75){\8} ;
   
    \draw[->] (01)--(03) ;
    
    \draw[->] (11)--(12) ;
    \draw[->] (12)--(13) ;   
    
    \draw[->] (30)--(32)  ;
    \draw[->] (31)--(32)  ;
    \draw[->] (32)--(33)  ;
    
    \draw[->] (01)--(11) ;
    \draw[->] (11)--(21) ;
    \draw[->] (21)--(30) ;
                                     
    \draw[->] (12)--(31)  ;
    \draw[->] (21)--(31)  ;
    
    \draw[->] (03)--(13) ;
    \draw[->] (13)--(23)  ;
    \draw[->] (23)--(33)  ;
    
    \node at (2.2,-.52) [scale=.9]{\circled1} ;
    \node at (1.64,-3) [scale=.9]{\circled2} ;

  \epic
\]
\end{small}

\vskip-7pt
The commutativity of \circled1 results from the following Lemma~\ref{Kozpf}. That of \circled2 is clear.
That of the other two subdiagrams is contained, \emph{mutatis mutandis,} in Proposition~\ref{pfchi}. The conclusion results.
\end{proof}

\vskip2pt
\begin{sublem}\label{Kozpf}
The following\va{-1} natural\/ $\D(X)$-diagram commutes$\>:$

\[
\def\0{$\omega_{\<\<f}\<\otimes_\sX\< f^*\<\omega_{\<g}$}
\def\1{$\omega_{\<\<f}\Otimes{\sX}\LL f^*\<\omega_{\<g}$}
\def\2{$\omega_{gf}$}
\def\3{$f^\flat\< g^\flat\OZ$}
\def\4{$(gf)^\flat\OZ$}
\def\5{$f^\flat\OY\Otimes{\sX}\LL f^*\<\<g^\flat\OZ$}
 \bpic[xscale=6.1, yscale=1.65]

   \node(10) at (1.55,-1){\0} ;
   \node(11) at (1,-1){\1} ;   
   \node(12) at (2,-1){\2} ;
   
   \node(101) at (1,-2){\5} ;
    
   \node(21) at (1,-3){\3} ;  
   \node(22) at (2,-3){\4} ;
 
    \draw[->] (11)--(10) node[above, midway, scale=.75] {$\Iso$} ;
    \draw[->] (10)--(12)  node[above, midway, scale=.75] {\kern-1pt\eqref{tensoromega}} ;   
    
    \draw[->] (21)--(22)  node[above, midway, scale=.75] {$\qquad\Iso$}
                                    node[below=1, midway, scale=.75] {\qquad\eqref{pf flat}} ;
    
    \draw[->] (11)--(101) node[left=1, midway, scale=.75] {$\ref{KosReg}$\textup{(i)}} 
                                     node[right, midway, scale=.75]{$\simeq$} ;
                                   
    \draw[->](101)--(21) node[left=1, midway, scale=.75] {$\ref{flat and tensor}$\textup{(ii)}} 
                                         node[right, midway, scale=.75]{$\simeq$} ;

    \draw[->] (12)--(22)  node[right=1, midway, scale=.75] {\textup{(\ref{KosReg}.1)}} 
                                    node[left, midway, scale=.75]{$\simeq$} ;
   \node at (1.55,-2)[scale=.9]{$\Delta$} ;
  
 \epic
\]
\end{sublem}

The strategy for proving \ref{Kozpf} is to reduce to the local situation \ref{explicit}, which is
disposed of in \S\ref{transci} below by means of arguments appearing in \cite[Appendix C.6.]{NS}.
The reduction is given by the following lemma,
with $\tilde Z$ the disjoint union of the members of an affine open covering of $Z$, over 
each of which both $\CI$ and $\CJ$ are generated by Koszul-regular sequences of sections, and with
$p\>$ the natural map. (Note that the vertices\va{.3} in $\Delta$ all have homology that vanishes in degrees other than $d+e$, so that $\Delta$ is essentially\va{.6} a diagram of quasi-coherent $\OX$-modules,
whence for any faithfully\va{.3} flat map $r\colon\tilde X\to X\<$, $\Delta$ commutes if $r^*\<\<\Delta$ does.)

\begin{sublem} \label{lem:Kozpf}
The situation being as in\/ \textup{\ref{compose},} let $p\colon\tilde Z\to Z$ be a flat scheme-map, and
\[
\def\0{$\tilde X$}
\def\1{$\tilde Y$}
\def\2{$\tilde Z$}
\def\3{$X$}
\def\4{$Y$}
\def\5{$Z$}
 \bpic[xscale=2, yscale=1.3]

   \node(11) at (1,-1){\0} ;
   \node(12) at (2,-1){\1} ;   
   \node(13) at (3,-1){\2} ;
   
   \node(21) at (1,-2){\3} ;
   \node(22) at (2,-2){\4} ;  
   \node(23) at (3,-2){\5} ;
 
    \draw[->] (11)--(12) node[above, midway, scale=.75] {$\tilde f$} ;
    \draw[->] (12)--(13) node[above, midway, scale=.75] {$\tilde g$} ;   
    
    \draw[->] (21)--(22) node[below=1, midway, scale=.75] {$f$} ;
    \draw[->] (22)--(23) node[below=1, midway, scale=.75] {$g$} ;
    
    \draw[->] (11)--(21) node[left=1, midway, scale=.75] {$r$} ;
                                   
    \draw[->](12)--(22) node[left=1, midway, scale=.75] {$q$} ;

    \draw[->] (13)--(23)  node[right=1, midway, scale=.75] {$p$} ;
 
 \epic
\]
a composite fiber square.  Then\/ $\tilde f$ and\/ $\tilde g$ are Koszul-regular immersions, and the inner rectangle in the following natural diagram is isomorphic to the outer one$\>:$
\[
\def\1{$\omega_{\<\<\tilde f}\otimes_{\<\<\tilde X} \tilde \<f^{\lift1.2,*,}\<\<\omega_{\tilde g}$}
\def\2{$\omega_{\tilde g\<\tilde f}$}
\def\3{$r^*\<(\omega_{\<\<f}\otimes_\sX \<f^*\<\omega_g)$}
\def\4{$r^*\<(f^\flat\OY\<\<\Otimes{\!X}\<\LL f^*\<\<g^\flat\OZ)\ $}
\def\5{$r^*\omega_{g\<f}$}
\def\6{$\tilde f^{\mkern1.5mu\flat}\CO_{\tilde Y}\Otimes{\!\tilde X}\LL \tilde f^{\lift1.2,*,}\<\<\tilde g^\flat\CO_{\<\<\tilde Z}$}
\def\7{$r^*\mkern-2.5mu f^\flat\< g^\flat\OZ$}
\def\8{$r^*\<(gf)^\flat\OZ$}
\def\lvn{$\tilde f^{\mkern1.5mu\flat}\< \tilde g^\flat\CO_{\<\<\tilde Z}$}
\def\twv{$(\tilde g\tilde f\>)^\flat\CO_{\<\<\tilde Z}$}
\def\thn{$\tilde f^{\mkern1.5mu\flat} q^*\< g^\flat\OZ$}
 \bpic[xscale=2.4, yscale=1.5]

   \node(11) at (.7,-.5){\1};   
   \node(15) at (4.75,-.5){\2}; 
      
   \node(22) at (2.1, -2){\3} ;    
   \node(24) at (3.5,-2){\5};
   
   \node(31) at (.7,-4.05){\6};
   
   \node(351) at (2.1,-3.35){\4} ;
   
   \node(42) at (2.1,-4.75){\7};
   \node(43) at (3.5,-4.75){\8};
   
   \node(52) at (1.45,-5.5){\thn};
   \node(53) at (.7,-6.25){\lvn};
   \node(55) at (4.75,-6.25){\twv};

    \draw[->] (11)--(15) node[below=1pt, midway, scale=.75] {\eqref{tensoromega}}
                                    node[above, midway, scale=.75] {$\Iso$}  ;  
    
    \draw[->] (22)--(24) node[below=1pt, midway, scale=.75] {$\via$\;\eqref{tensoromega}}
                                    node[above, midway, scale=.75] {$\Iso$}  ;   
       
    \draw[->] (42)--(43) node[below=1pt, midway, scale=.75] {$\via$\;\eqref{pf flat}}
                                    node[above, midway, scale=.75] {$\Iso$}  ;                                  

    \draw[->] (53)--(55) node[below=1pt, midway, scale=.75] {\eqref{pf flat}} 
                                    node[above, midway, scale=.75] {$\Iso$}  ;
    
    \draw[->] (11)--(31) node[right, midway, scale=.75] {$\ref{KosReg}$\textup{(i)}} 
                                   node[left, midway, scale=.75]{$\simeq$} ;
                                   
    \draw[->](31)--(53) node[right, midway, scale=.75] {$\ref{flat and tensor}\textup{(ii)}_{}$} 
                                   node[left, midway, scale=.75]{$\simeq$} ;

    \draw[->] (15)--(55)  node[left, midway, scale=.75] {$\ref{KosReg}$\textup{(i)}} 
                                    node[right, midway, scale=.75]{$\simeq$} ;
                                    
     \draw[->] (22)--(351) node[right, midway, scale=.75] {$\via\>\ref{KosReg}$\textup{(i)}} 
                                      node[left, midway, scale=.75]{$\simeq$} ;
                                   
    \draw[->](351)--(42) node[right, midway, scale=.75] {$\via\>\ref{flat and tensor}$\textup{(ii)}} 
                                     node[left, midway, scale=.75]{$\simeq$} ;

    \draw[->] (24)--(43)  node[left, midway, scale=.75] {$\ref{KosReg}$\textup{(i)}} 
                                    node[right, midway, scale=.75]{$\simeq$} ;
                                
    \draw[->] (22)--(11) node[above=-3, midway, scale=.75] {\rotatebox{-35}{$\ \Iso$}} ; 
    \draw[->] (24)--(15) node[above=-3, midway, scale=.75] {\rotatebox{35}{$\ \Iso$}} ;
    \draw[->] (42)--(52) node[above=-4, midway, scale=.75] {\rotatebox{33}{$\Iso\ $}} 
                                    node[below=-3, midway, scale=.75] {\kern33pt\textup{\ref{indt base change}}} ;
    \draw[->] (52)--(53) node[above=-4, midway, scale=.75] {\rotatebox{33}{$\Iso\ $}} 
                                    node[below=-3, midway, scale=.75] {\kern33pt\textup{\ref{indt base change}}} ;
    \draw[->] (43)--(55) node[above=-3, midway, scale=.75] {\rotatebox{-34}{$\ \Iso$}} 
                                    node[below=-3, midway, scale=.75] {\textup{\ref{indt base change}}\kern35pt} ;
  \node at (2.87,-1.27)[scale=.9]{\circled A} ;
  \node at (1.1,-3.37)[scale=.9]{\circled{\kf B}} ;
  \node at (3.97,-3.37)[scale=.9]{\circled{C\kf}} ;
  \node at (2.8,-5.55)[scale=.9]{\circled{\kf D}} ;

 \epic
\]
\end{sublem}

\begin{proof}
That $\tilde f$ and\/ $\tilde g$ are Koszul-regular immersions (of respective codimensions $d$ and $e$)
is easy to verify.

It needs then to be shown that subdiagrams \circled A, 
\circled{\kf B}, \circled{C\kf} and \circled{\kf D} commute.
(In other words,  the maps in \ref{Kozpf} are compatible with flat base change.)

Set $h\set gf\<$, $\tilde h\set \tilde g\tilde f$. These are Koszul-regular immersions of codimension 
$c\set d+e$, the kernel of the natural map $\OZ\to h_*\OX$ 
(resp.~$\CO_{\tilde Z}\to\tilde h_*\CO_{\<\tilde X}$) being $\CL$ (resp.~$\tilde\CL\set\CL\CO_{\<\<\tilde Z}$).
 
Using the definitions of the maps in \circled{C\kf}, and the easily-checked fact that, $h_*$ being exact,
the two natural composite maps $h^*\<\<H^ch_*\iso h^*\<h_*H^c\to H^c$ and 
$h^*\<\<H^ch_*\to H^ch^*\<h_*\to H^c$
are equal, one sees that commutativity of~\circled{C\kf}  is equivalent to that of the border of the natural diagram, with $\CH\set\sHom$,\begin{small}
\[\mkern-4mu
\def\1{$r^*\CH^{}_{\sX}\<\big(\<{\lift1.7,{\bigwedge},\lift2,{\<c},}_{\mkern-10mu X^{\mathstrut}}\<h^{\<*}\<\<\big(\CL/\<\CL^{\>2\>}\big)\<,\OX\<\big)$}
\def\3{$r^*h^{\<*}\CH^{}_{\<Z}\<\big(\<{\lift1.7,{\bigwedge},\lift2,{\<c},}_{\mkern-8mu Z^{\mathstrut}}\<\<
\big(\CL/\<\CL^{\>2\>}\big)\<,\>\OZ/\<\CL\big)$}
\def\5{$r^*h^{\<*}\sExt^{\>c}_{\<Z}(h_*\OX,\>\OZ)$}
\def\7{$r^*h^{\<*}\<H^ch_*h^\flat\OZ$}
\def\9{$r^*\<H^ch^{\<*}\<h_*h^\flat\OZ$}
\def\lvn{$r^*\<H^ch^\flat\OZ$}
\def\thn{$r^*h^\flat\OZ[c]$}
\def\2{$\CH_{\<\<\tilde X}\<\big(\<{\lift1.7,{\bigwedge},\lift2,{\<c},}_{\mkern-10mu \tilde X^{\mathstrut}}\<\tilde h^{\<*}\<\<\big(\tilde\CL/\tilde\CL^{\>2\>}\big)\<,\CO_{\!\tilde X}\<\big)$}
\def\4{$\tilde h^{\<*}\CH_{\<\tilde Z}\<\big(\<{\lift1.7,{\bigwedge},\lift2,{\<c},}_{\mkern-9mu \tilde Z^{\mathstrut}}\<\<\big(\tilde\CL/\tilde\CL^{\>2\>}\big)\<,\>\CO_{\<\<\tilde Z}/\<\<\tilde\CL\>\big)$}
\def\6{$\tilde h^{\<*}\sExt^{\>c}_{\<\tilde Z}(\tilde h_*\CO_{\!\tilde X},\>\CO_{\<\<\tilde Z})$}
\def\8{$\tilde h^{\<*}\<H^c\tilde h_*\tilde h^\flat\CO_{\<\<\tilde Z}$}
\def\ten{$H^c\tilde h^{\<*}\<\tilde h_*\tilde h^\flat\CO_{\<\<\tilde Z}$}
\def\twv{$H^c\tilde h^\flat\CO_{\<\<\tilde Z}$}
\def\frn{$\tilde h^\flat\CO_{\<\<\tilde Z}[c]$}
\def\ffn{$\CH_{\<\<\tilde X}\<\big(r^*\<\<{\lift1.7,{\bigwedge},\lift2,{\<c},}_{\mkern-10mu X^{\mathstrut}}\<h^{\<*}\<\<\big(\CL/\<\CL^{\>2\>}\big)\<,\>r^*\<\OX\<\big)$}
\def\sxn{$H^c \>r^*\<h^\flat\OZ$}
\def\svn{$H^c\>r^*\<h^{\<*}\<h_*h^\flat\OZ$}
\def\egn{$H^c\tilde h^{\<*}\<p^*\<h_*h^\flat\<\OZ$}
\def\ntn{$H^c\tilde h^{\<*}\<\tilde h_*r^*\<h^\flat\<\OZ$}
\def\tty{$\tilde h^{\<*}\<H^c\>p^*\<h_*h^\flat\<\OZ$}
\def\ton{$\tilde h^{\<*}\<H^c\tilde h_*r^*\<h^\flat\<\OZ$}
\def\ttw{$\tilde h^{\<*}\<p^*\<\<H^ch_*h^\flat\OZ$}
\def\tth{$\tilde h^{\<*}\<p^*\<\sExt^{\>c}_{\<Z}(h_*\OX,\>\OZ)\qquad\ $}
\def\tfr{$\tilde h^{\<*}\<p^*\<\<H^c\R\CH^{}_{\<Z}(h_*\OX\<,\OZ\<)$}
\def\tfv{$\tilde h^{\<*}\<\<H^c\R\CH_{\<\tilde Z}(\>p^*\<h_*\OX\<,p^*\<\OZ\<)\qquad$}
\def\tsx{$\tilde h^{\<*}\<H^c\R\CH_{\<\tilde Z}(\tilde h^{\<*}\CO_{\!\tilde X}\<,\CO_{\<\<\tilde Z}\<)$}
\def\tsv{$\quad \tilde h^{\<*}\<p^*\< \CH^{}_{\<Z}\<\big(\<{\lift1.7,{\bigwedge},\lift2,{\<c},}_{\mkern-8mu Z^{\mathstrut}}\<\<\big(\CL/\<\CL^{\>2\>}\big)\<,\>\OZ/\<\CL\big)$}
\def\teg{$\quad \tilde h^{\<*}\<\CH^{}_{\<\tilde Z}\<\big(p^*\< {\lift1.7,{\bigwedge},\lift2,{\<c},}_{\mkern-8mu Z^{\mathstrut}}\<\<\big(\CL/\<\CL^{\>2\>}\big)\<,\>p^*\< (\OZ/\<\CL)\big)$}
\def\tnn{$\tilde h^{\<*}\<H^cp^*\<\R\CH^{}_{\<Z}(h_*\OX\<,\OZ\<)$}
  \bpic[xscale=4.3, yscale=1.3]
 
   \node(11) at (.95,1.35){\1} ;   
   \node(12) at (2.02,1.35){\ffn} ;
   \node(13) at (3.05,1.35){\2} ; 
   \node(135) at (1.6,-.6){\tsv} ; 

   \node(21) at (1,.35){\3} ;
   \node(22) at (2.25,.35){\teg} ;
   \node(23) at (3,-.6){\4} ;   
  
   \node(31) at (1,-1.55){\5} ;
   \node(322) at (2.55,-2.5){\tsx} ;
   \node(32) at (1.6,-2.5){\tth} ;
   \node(33) at (3,-1.7){\6} ; 
   
   \node(315) at (1.6,-4){\tfr} ;
   \node(325) at (2.21,-1.7){\tfv} ; 
   \node(335) at (2.1,-3.25){\tnn} ;

   \node(41) at (1,-4.5){\7} ;
   \node(42) at (2.1,-4.75){\tty} ;
   \node(43) at (3,-4.5){\8} ; 
   
   \node(425) at (1.6,-5.5){\ttw} ;
   
   \node(51) at (1,-6){\9} ; 
   \node(52) at (1.6,-7){\svn} ;
   \node(53) at (2.55,-5.5){\ton} ;
      
   \node(62) at (2.1,-6.25){\egn} ;
   \node(63) at (3,-6){\ten} ;
   
   \node(72) at (2.55,-7){\ntn} ;
   \node(81) at (1,-8){\lvn} ;
   \node(82) at (2.1,-8){\sxn} ;
   \node(83) at (3,-8){\twv} ;  
   
   \node(91) at (1,-9){\thn} ;
   \node(93) at (3,-9){\frn} ;
    
    \draw[->] (11)--(12) node[above, midway, scale=.65] {$\Iso$}  ;
    \draw[->] (12)--(13) node[above, midway, scale=.65] {$\Iso$}  ;
   
    \draw[double distance=2] (322)--(33) node[above=1, midway, scale=.75] {$$}  ;
   
    \draw[->] (51)--(52) node[above=1, midway, scale=.75] {$$}  ;
                                  
    \draw[->] (81)--(82) node[above=1, midway, scale=.75] {$$}  ;
    \draw[->] (82)--(83) node[below=1pt, midway, scale=.75] {via\,\eqref{indt base change}} 
                                  node[above=1, midway, scale=.75] {$\Iso$}  ;
  
    \draw[->] (91)--(93) node[below=1pt, midway, scale=.75] {\eqref{indt base change}} 
                                    node[above=1, midway, scale=.9] {$\Iso$}  ;
                                       
     \draw[->](1,1.2)--(21) node[left, midway, scale=.75]{$\simeq$} ;
     \draw[->](21)--(31) node[left, midway, scale=.75]{\textup{\ref{Kreg}}} 
                                   node[right, midway, scale=.75]{$\simeq$} ;
                                   
     \draw[->](31)--(41) ;
     \draw[->](335)--(2.1,-1.95) ;
     \draw[->](135)--(32) node[left=-.6, midway]{\scalebox{.93}[1.05]{\lift4,\textup{\ref{Kreg}},}} 
                                   node[right=-.5, midway, scale=.94]{$\lift4.3,\simeq,$} ;
     \draw[->](41)--(51) ; 
     \draw[->](51)--(81) ; 
     \draw[->](81)--(91) node[left, midway, scale=.75]{$\simeq$} ;
     
     \draw[->](52)--(82) ;
     \draw[->](62)--(42) ;
     \draw[->](72)--(53) ; 
     \draw[->](335)--(42) ;
     
     \draw[->](3,1.2)--(23) node[right, midway, scale=.75]{$\simeq$} ;
     \draw[->](23)--(33) node[right, midway, scale=.75]{\textup{\ref{Kreg}}} 
                                   node[left, midway, scale=.75]{$\simeq$} ;
     \draw[->](33)--(43) ;
     \draw[->](43)--(63) ; 
     \draw[->](63)--(83) ; 
     \draw[->](83)--(93) node[right, midway, scale=.75]{$\simeq$} ;

      \draw[->](21)--(135) ;
      \draw[->](135)--(22) ;
      \draw[->](22)--(23) ;
      \draw[->](31)--(1.4, -2.25) ;
      \draw[double distance=2](32)--(315) ;
      \draw[->](315)--(335) ;
      \draw[->](325)--(322) ;
      \draw[->](315)--(425) ;
      \draw[->](52)--(62) ;
      \draw[->](62)--(72) ;
      \draw[->](72)--(63) node[above=-3, midway, scale=.7]{\rotatebox{34} {via}\kern3pt}
                                      node[below=-7, midway, scale=.7]{\kern10pt\rotatebox{34} {\kern3pt\eqref{indt base change}}} ;
      \draw[->](53)--(43) node[above=-3, midway, scale=.7]{\rotatebox{34} {via}\kern3pt}
                                      node[below=-7, midway, scale=.7]{\kern10pt\rotatebox{34} {\kern3pt\eqref{indt base change}}} ;
      \draw[->](322)--(43) ;
      \draw[->](72)--(82) ;
      \draw[->](41)--(425) ;
      \draw[->](425)--(42) ;
      \draw[->](42)--(53) ;
      \draw[->](82)--(91) ;  
      
     \node at (2.3,-.95)[scale =.85] {\circled1} ;
     \node at (2.5,-4)[scale =.85] {\circled2} ;
     \node at (2.1,-7.02)[scale =.85] {\circled3} ;
  \epic
\]
\vskip-2pt
\end{small}

Whether subdiagram \circled1 commutes is a local question. Hence one can assume that $\CL$ is generated by a Koszul-regular sequence $\mathbf t=(t_1,\dots,t_c),$ and replace $\R\CH^{}_{\<Z}(h_*\OX\<,\OZ\<)$ by
$\CH^{}_{\<Z}(K^{}_{\<\<Z}({\mathbf t}),\OZ\<)$ (resp.~$\R\CH^{}_{\<\tilde Z}(\tilde h_*\OX\<,\CO_{\<\<\tilde Z}\<)$ by
$\CH^{}_{\<Z}(K^{}_{\<\<Z}({p^*\mathbf t}),\OZ\<)$). Then Lemma~\ref{alt c} gives the commutativity of \circled1.

Subdiagram \circled2 commutes, since without ``$\>\tilde h^*\<H^c$"  it is the diagram, with $(f,g,u,v,G)\set (p,r,h,\tilde h, \CO_{\<\<\tilde Z})$, shown in the proof of~\ref{indt base change} to commute.

Commutativity of \circled3 holds because the natural map 
$p^*\<h_*\to\tilde h_*r^*$ is, by definition, adjoint to the natural composite map $\tilde h^*p^*\<h_*\iso r^*\<h^*h_*\to r^*$.

Checking commutativity of the unlabeled subdiagrams is straightforward.

Diagram chasing shows now that the border commutes, whence so does subdiagram \circled{C\kf}.\va2

Subdiagram \circled A involves only sheaves, so its commutativity is readily checked, locally, via \eqref{exact'aff}.\va2

As for \circled{\kf B},\va1 expand it naturally as follows, where $G\set g^\flat\OZ$---so that  
one has the isomorphism $q^*G\underset{\textup{\ref{indt base change}}}\iso \tilde g^\flat\CO_{\<\<\tilde Z}=\colon \tilde G$.
\begin{small}
\[\mkern-2mu
\def\1{$\omega_{\<\<\tilde f}\otimes_{\<\<\tilde X} \tilde f^{\lift1.2,*,}\!\omega_{\tilde g}$}
\def\2{$r^*\<\omega_{\<\<f}\otimes_{\<\<\tilde X} \tilde f^* \<q^*\<\omega_g$}
\def\3{$r^*\<\omega_{\<\<f}\otimes_{\<\<\tilde X} r^*\< \<f^*\<\omega_g$}
\def\4{$r^*\<(\omega_{\<\<f}\otimes_\sX \<\<f^*\<\omega_g)$}
\def\5{$\tilde f^{\mkern1.5mu\flat}\CO_{\tilde Y}\Otimes{\!\tilde X}\LL \tilde f^{\lift1.2,*,}\<\<\tilde G$}
\def\6{$r^*\!f^\flat\OY\<\otimes_{\<\<\tilde X} \LL\tilde f^* \<q^*\<G$}
\def\7{$r^*\!f^\flat\OY\<\otimes_{\<\<\tilde X} r^*\LL f^*\<G$}
\def\8{$r^*\<(f^\flat\OY\<\<\Otimes{\!X}\<\LL f^*\<\<G)\quad $}
\def\9{$\tilde f^{\mkern1.5mu\flat}\CO_{\tilde Y}\Otimes{\!\tilde X} \LL\tilde f^* \<q^*\<G$}
\def\ten{$\tilde f^{\mkern1.5mu\flat}\<\tilde G$}
\def\lvn{$\tilde f^{\mkern1.5mu\flat}\< q^*\< G$}
\def\twv{$r^*\!f^\flat G$}
  \bpic[xscale=3.345, yscale=2.2]
  
   \node(11) at (1,-1){\1};   
   \node(12) at (2,-1){\2};  
   \node(13) at (3,-1){\3};  
   \node(14) at (4,-1){\4}; 

   \node(21) at (1,-2){\5};
   \node(22) at (2,-1.6){\6};
   \node(23) at (3,-2){\7};
   \node(24) at (4,-2.4){\8};   
  
   \node(225) at (2,-2.4){\9};
   
   \node(31) at (1,-3){\ten} ;
   \node(32) at (2,-3){\lvn} ;  
   \node(34) at (4,-3){\twv} ;

   \draw[->] (12)--(11) ;
   \draw[->] (13)--(12) ;
   \draw[->] (14)--(13) node[below=5, midway, scale=.75] {} 
                                  node[above, midway, scale=.75] {} ;

   \draw[->] (22)--(21) ;
   \draw[->] (23)--(22) ;
   \draw[->] (24)--(23) ;  
    
   \draw[->] (23)--(225) ; 
   \draw[->] (225)--(21) ;
   
   \draw[->] (32)--(31) ;
   \draw[->] (34)--(32) ;
   
    \draw[->] (11)--(21) ;
    \draw[->] (21)--(31) ;
    
    \draw[->] (12)--(22) ;
    \draw[->] (225)--(32) ;

    \draw[->] (13)--(23) ;
    
    \draw[->] (14)--(24);
    \draw[->] (24)--(34) ;
 
    \node at (1.42,-1.35)[scale =.85] {\circled4} ;
    \node at (3,-2.53)[scale =.85] {\circled5} ;
     
  \epic
\]
\end{small}
\indent The commutativity of the unlabeled diagrams is easily checked. That of~\circled4 is shown by arguments like those used above to show commutativity of~\circled{C\kf}. For that of \circled5, it suffices to show the commutativity of the adjoint diagram, as per the hint  in \cite[4.7.3.4(c)]{li}, \emph{mutatis mutandis}.

Thus \circled{\kf B} commutes.\va2

The commutativity of \circled{\kf D} can be shown in the same way as that of the last diagram in \cite[4.6.8]{li}, thereby completing the proof of Lemma~\ref{lem:Kozpf}.
\end{proof}

For a Koszul-regular closed immersion $f\colon X\to Y$, the interaction of the isomorphism 
\[
 c^{\>\flat}_{\<\<f}(G)\colon \omega_{\<\<f}\Otimes{\sX} \LL f^*\<G\iso f^\flat\< G\qquad(G\in\Dqc(X))
 \]
 in \ref{KosReg}(i) with independent base change, $\Otimes{}$ and $\R\sHom$ is described in the following proposition.
  
\begin{subprop}\label{ci tensor hom} Let\/ $f\colon X\to Y$ be a Koszul-regular closed immersion of codimension\/ $d,$ 
and\/ $F,$ $G\in\Dqc(Y).$\va2

\textup{(i)} Let
\[
\def\1{$X'$}
\def\2{$X$}
\def\3{$Y'$}
\def\4{$Y$}
 \bpic[xscale=1, yscale=.8]

   \node(11) at (1,-1){\1};   
   \node(13) at (3,-1){\2}; 
     
   \node(31) at (1,-3){\3};
   \node(33) at (3,-3){\4};
   
    \draw[->] (11)--(13) node[above=1pt, midway, scale=.75] {$v$} ;  
      
   \draw[->] (31)--(33) node[below=1pt, midway, scale=.75] {$u$} ;
    
    \draw[->] (11)--(31) node[left=1pt, midway, scale=.75] {$g$} ;
   
    \draw[->] (13)--(33) node[right=1pt, midway, scale=.75] {$f$};
  \node at (2.05,-2)[scale=.85] {$\sigma$} ;
  
 \epic 
\]
be an independent square of scheme-maps\/ \textup{(\emph{see} \S\ref{Indt square})} in which\/ $g$ is affine. 

Then\/ $\sigma$ is a fiber square, and $g$ is a Koszul-regular closed immersion of codimension\/ $d.$ 

\pagebreak[3]

Moreover, with
\( 
\beta_\sigma(G)\colon \LL v^*\mkern-2.5mu f^\flat G\lto g^\flat\> \LL u^*\mkern-.5mu G
\)
adjoint to the natural composite map\/
\( 
\R g_*\LL v^*\<\< f^\flat G\iso \LL u^*\R \fst f^\flat G
 \lto\LL u^*\mkern-.5mu G, 
\)
the following natural\/ $\D(X')$-diagram commutes$\>:$
\[
\def\1{$\LL v^*(\omega_{\<\<f}\Otimes{\sX}\LL f^*\<G)$}
\def\2{$\LL v^*\omega_{\<\<f}\Otimes{\<X'}\LL v^*\LL f^*\<G$}
\def\3{$\omega_g\Otimes{\<X'}\LL g^*\LL u^*\mkern-.5mu G$}
\def\4{$\LL v^*\<\<f^\flat\<G$}
\def\5{$g^\flat\>\LL u^*\mkern-.5muG$}
 \bpic[xscale=4.5, yscale=1.5]

   \node(11) at (1,-1){\1}; 
   \node(12) at (2.02,-1){\2};   
   \node(13) at (3,-1){\3}; 
     
   \node(21) at (1,-2){\4};
   \node(23) at (3,-2){\5};
   
   \draw[->] (11)--(12) node[above, midway, scale=.75] {$\Iso$} ;  
   \draw[->] (12)--(13) node[above, midway, scale=.75] {$\Iso$} ;  
    
   \draw[->] (21)--(23) node[above=1pt, midway, scale=.9] {$\Iso$} 
                                   node[below=1pt, midway, scale=.75] {$\beta_\sigma(G)$} ;
    
    \draw[->] (11)--(21) node[left=1pt, midway, scale=.75] {$\LL v^*c^{\>\flat}_{\<\<f}(G)$} 
                                   node[right=1pt, midway, scale=.75]{$\simeq$} ;
   
    \draw[->] (13)--(23) node[right=1pt, midway, scale=.75] {$c^{\>\flat}_{\<\<g}(\LL u^*\mkern-.5muG)$}
                                   node[left=1pt, midway, scale=.75]{$\simeq$} ;

 \epic 
\]

\noindent\textup{(ii)} The following natural\/ $\D(X)$-diagram commutes$\>:$
\[
\def\1{$\omega_{\<\<f}\Otimes{\sX}\LL f^*(F\Otimes{Y}G\>)$}
\def\2{$\omega_{\<\<f}\Otimes{\sX}\LL f^*\<\<F\Otimes{\sX}\LL f^*\<G$}
\def\4{$f^\flat(F\Otimes{Y}G\>)$}
\def\5{$f^\flat\<\<F\Otimes{\sX}\LL f^*\<G$}
 \bpic[xscale=2.5, yscale=1.5]

   \node(11) at (3,-1){\1}; 
   \node(13) at (1,-1){\2}; 
     
   \node(21) at (3,-2){\4};
   \node(23) at (1,-2){\5};
   
   \draw[<-] (11)--(13) node[above, midway, scale=.75] {$\Iso$} ;  
    
   \draw[<-] (21)--(23)  node[above, midway, scale=.75] {\textup{\ref{flat and tensor}(ii)}}
                                    node[below=1pt, midway, scale=.75] {$\chi(f,F,G)$} ;
    
    \draw[->] (11)--(21) node[right=1pt, midway, scale=.75] {$c^{\>\flat}_{\<\<f}(F\Otimes{Y}G)$} 
                                    node[left=1pt, midway, scale=.75]{$\simeq$} ;
   
    \draw[->] (13)--(23) node[left, midway, scale=.75] {$c^{\>\flat}_{\<\<f}(F)\<\Otimes{\sX}\id$}
                                   node[right=1pt, midway, scale=.75]{$\simeq$} ;

  \epic 
\]
\noindent\textup{(iii)} If $F$ is pseudo-coherent and\/ $G\in\Dqcpl(Y),$ then the following natural\/  $\D(X)$-diagram, where the map $\xi$ is adjoint to the natural map
\[
\omega_{\<\<f}\Otimes{\sX}\R\sHom_Y(\LL f^*\<\<F\<,\>\LL f^*\mkern-.5mu G\>)\Otimes{\sX}\LL f^*\<\<F
\lto
\omega_{\<\<f}\Otimes{\sX}\LL f^*\mkern-.5mu G,
\]
commutes$\>:$\va2
\[
\def\1{$\omega_{\<\<f}\Otimes{\sX}\LL f^*\R\sHom_Y(F\<,G\>)$}
\def\2{$\omega_{\<\<f}\Otimes{\sX}\R\sHom_Y(\LL f^*\<\<F\<,\>\LL f^*\mkern-.5mu G\>)$}
\def\3{$\R\sHom_\sX(\LL f^*\<\<F\<,\>\omega_{\<\<f}\Otimes{\sX}\LL f^*\mkern-.5mu G\>)$}
\def\4{$f^\flat\R\sHom_Y(F\<,G\>)$}
\def\5{$\R\sHom_\sX(\LL f^*\<\<F\<,f^\flat\<G\>)$}
 \bpic[xscale=4, yscale=2.75]

   \node(11) at (1,-1){\1}; 
   \node(12) at (2.02,-1.53){\2};   
   \node(13) at (3.05,-1){\3}; 
     
   \node(21) at (1,-2){\4};
   \node(23) at (3.05,-2){\5};
   
   \draw[->] (11)--(12) node[above, midway, scale=.75] {$$} ;  
   \draw[->] (12)--(13) node[above=-1, midway, scale=.75] {$\xi\ $} ;  
    
   \draw[->] (21)--(23) node[above, midway, scale=.75] {\textup{\ref{flat and hom}(iii)}} 
                                   node[below=1pt, midway, scale=.75] {$\zeta^{-\<1}$} ;
    
    \draw[->] (11)--(21) node[left=1pt, midway, scale=.75] {$c^{\>\flat}_{\<\<f}(-)$} 
                                   node[right=1pt, midway, scale=.75]{$\simeq$} ;
   
    \draw[->] (13)--(23) node[right=1pt, midway, scale=.75] {$\via\>c^{\>\flat}_{\<\<f}(\LL f^*\mkern-.5mu G)$}
                                   node[left=1pt, midway, scale=.75]{$\simeq$} ;

 \epic 
\]

\end{subprop}

\begin{proof}
(i) Working locally on $Y\<$,  one can assume that the kernel of the natural map
$\OY\to\fst\OX$ is $\mathbf t\OY$, where $\mathbf t$~is a Koszul-regular sequence of length~$d$. 
Then $\fst\OX$ is resolved by the Koszul complex~$K_Y(\mathbf t)$, and, with $\mathbf t'$ the sequence
$\mathbf t\CO_{Y'}$, one has in $\D(Y')$ the natural isomorphisms 
\[
g_*\CO_{\<X'}\cong \LL g_*\LL v^*\OX \cong \LL u^*\R\fst\OX \cong u^*\<K_Y(\mathbf t)\cong K_{Y'}(\mathbf t').
\]
Therefore, $\mathbf t'$ is a Koszul-regular sequence of length~$d$, whence the projection 
$\tilde g\colon\tilde X=Y'\times_YX\to Y'$ is a Koszul-regular closed immersion of codimension\/~$d;$ and the natural map $\tilde g_*\CO_{\tilde X}\to g_*\CO_{X'}$ is an isomorphism, so that, $g$~being affine, the natural map
$X'\to\tilde X$ is an isomorphism. The first two assertions, which need only be verified locally, result.

As for the diagram in question, the definition of $c^{\>\flat}_{\<\<f}(G)$ gives the following natural expansion:
\[
\def\1{$\LL v^*(\omega_{\<\<f}\Otimes{\sX}\LL f^*\<G)$}
\def\2{$\LL v^*\omega_{\<\<f}\Otimes{\<X'}\LL v^*\LL f^*\<G$}
\def\3{$\omega_g\Otimes{\<X'}\LL g^*\LL u^*\mkern-.5mu G$}
\def\4{$\LL v^*\<\<f^\flat\<G$}
\def\5{$g^\flat\>\LL u^*\mkern-.5muG$}
\def\6{$\LL v^*(f^\flat\OY\Otimes{\sX}\LL f^*\<G)$}
\def\7{$\LL v^*\<\<f^\flat\OY\Otimes{\sX}\LL v^*\LL f^*\<G$}
\def\8{$g^\flat\CO_{Y'}\Otimes{\sX}\LL g^*\LL u^*\mkern-.5mu G$}
 \bpic[xscale=4.5, yscale=2.5]

   \node(11) at (1,-1){\1}; 
   \node(12) at (2.015,-1){\2};   
   \node(13) at (3,-1){\3}; 
    
   \node(115) at (1,-1.5){\6}; 
   \node(125) at (2.015,-1.5){\7};   
   \node(135) at (3,-1.5){\8}; 
   
   \node(21) at (1,-2){\4};
   \node(23) at (3,-2){\5};
   
   \draw[->] (11)--(12) ;  
   \draw[->] (12)--(13) ;  
   
   \draw[->] (115)--(125) ;  
   \draw[->] (125)--(135) ; 
    
   \draw[->] (21)--(23) ;
    
    \draw[->] (11)--(115) ;
    \draw[->] (115)--(21) ;
     
    \draw[->] (12)--(125) ;

     \draw[->] (13)--(135) ;                              
     \draw[->] (135)--(23) ;

  \node at (1.49, -1.27)[scale=.9]{\circled1} ;
  \node at (2.53, -1.27)[scale=.9]{\circled2} ;
  \node at (2.015, -1.79)[scale=.9]{\circled3} ; 
 \epic 
\]
The commutativity of \circled1 is clear. For that of \circled2, cf.~that of \circled{C\kf} in \ref{lem:Kozpf}. For that of \circled3, cf.~\cite[4.7.3.4(c)]{li} (with $E\set\OY$). \va 1

(ii) The diagram expands naturally as
\[
\def\1{$\omega_{\<\<f}\Otimes{\sX}\LL f^*(F\Otimes{Y}G\>)$}
\def\2{$\omega_{\<\<f}\Otimes{\sX}\LL f^*\<\<F\Otimes{\sX}\LL f^*\<G$}
\def\3{$f^\flat\OY\Otimes{\sX}\LL f^*(F\Otimes{Y}G\>)$}
\def\4{$f^\flat\OY\Otimes{\sX}\LL f^*\<\<F\Otimes{\sX}\LL f^*\<G$}
\def\5{$f^\flat(F\Otimes{Y}G\>)$}
\def\6{$f^\flat\<\<F\Otimes{\sX}\LL f^*\<G$}
 \bpic[xscale=2.5, yscale=2.5]

   \node(11) at (3,-1){\1}; 
   \node(13) at (1,-1){\2}; 
   
   \node(115) at (3,-1.5){\3}; 
   \node(135) at (1,-1.5){\4};
   
   \node(21) at (3,-2){\5};
   \node(23) at (1,-2){\6};
   
   \draw[<-] (11)--(13) node[above, midway, scale=.75] {$\Iso$} ;  
   
   \draw[<-] (115)--(135) node[above, midway, scale=.75] {$\Iso$} ;
    
   \draw[<-] (21)--(23)  node[above, midway, scale=.75] {\textup{\ref{flat and tensor}(ii)}}
                                    node[below=1pt, midway, scale=.75] {$\chi(f,F,G)$}  ;
    
    \draw[->] (11)--(115) node[right=1pt, midway, scale=.75] {$\via\>c^{\>\flat}_{\<\<f}(\OY)$} 
                                    node[left=1pt, midway, scale=.75]{$\simeq$} ;
    \draw[->] (115)--(21) node[right=1pt, midway, scale=.75] {$c^{\>\flat}_{\<\<f}(F\<\Otimes{\sX}G\>)$} 
                                    node[left=1pt, midway, scale=.75]{$\simeq$} ;

    \draw[->] (13)--(135) node[left, midway, scale=.75] {$\via\>c^{\>\flat}_{\<\<f}(\OY)$}
                                   node[right=1pt, midway, scale=.75]{$\simeq$} ;
     \draw[->] (135)--(23) node[left, midway, scale=.75] {$\via\>c^{\>\flat}_{\<\<f}(F)$}
                                   node[right=1pt, midway, scale=.75]{$\simeq$} ;

  \epic 
\]
Commutativity of the top half is clear; and that of the bottom half is left for the reader to verify (cf.~\cite[4.7.3.4(a)]{li}
with $E\set\OY$).\va2

(iii) It suffices to prove the commutativity of the adjoint diagram, i.e., of the border of the following 
natural diagram (where $\CH\set\sHom$):
\[
\def\1{$\omega_{\<\<f}\Otimes{\sX}\LL f^*\R\CH_Y(F\<,G\>)\Otimes{\sX}\LL f^*\<\<F$}
\def\2{$\omega_{\<\<f}\Otimes{\sX}\R\CH_\sX(\LL f^*\<\<F\<,\>\LL f^*\mkern-.5mu G\>)\Otimes{\sX}\LL f^*\<\<F$}
\def\3{$\omega_{\<\<f}\Otimes{\sX}\LL f^*\mkern-.5mu G$}
\def\4{$f^\flat\R\CH_Y(F\<,G\>)\Otimes{\sX}\LL f^*\<\<F$}
\def\5{$\R\CH_\sX(\LL f^*\<\<F\<,f^\flat\<G\>)\Otimes{\sX}\LL f^*\<\<F$}
\def\6{$f^\flat\<G$}
\def\7{$\omega_{\<\<f}\Otimes{\sX}\LL f^*(\R\CH_Y(F\<,G\>)\Otimes{Y}F)$}
\def\8{$f^\flat(\R\CH_Y(F\<,G\>)\Otimes{Y}F)$}
 \bpic[xscale=4, yscale=2.5]

   \node(11) at (1,-1){\1}; 
   \node(12) at (2,-.5){\2};   
   \node(13) at (3.05,-1){\3}; 
   
   \node(21) at (2,-1.5){\7};
   \node(22) at (2,-2){\8};
      
   \node(31) at (1,-2.5){\4};
   \node(32) at (2.22,-2.5){\5};
   \node(33) at (3.05,-2.5){\6};
   
   \draw[->] (11)--(12) ;  
   \draw[->] (12)--(13) ;  
    
   \draw[->] (21)--(22) ;
    
   \draw[->] (31)--(32) ;  
   \draw[->] (32)--(33) ; 
   
    \draw[->] (11)--(31) ;
   
    \draw[->] (13)--(33) ;
    
    \draw[->] (11)--(21) ;
    \draw[->] (21)--(13) ;
    \draw[->] (31)--(22) ;    
    \draw[->] (22)--(33) ;
  
    \node at(2,-1)[scale=.9] {\circled4} ; 
    \node at (1.3, -1.77)[scale=.9]{\circled5} ;
    \node at (2.7, -1.77)[scale=.9]{\circled6} ;
    \node at (2.015, -2.27)[scale=.9]{\circled7} ; 

 \epic 
\]

The commutativity of subdiagram \circled4 is given by \cite[3.5.6(g)]{li}, with $\alpha\colon[D,E\>]\otimes D\to E$ the natural map.
That of \circled5 is given by (ii), with $(\R\CH_Y(F\<,G\>),F)$ in place of $(F\<,G\>)$.
That of \circled6 is clear.

For \circled7, it suffices to prove the commutativity of its adjoint, and so of all the subdiagrams of the following
natural one:
\[\mkern-3mu
\def\1{$\R\fst(f^\flat\R\CH_Y(F\<,G\>)\Otimes{\sX}\LL f^*\<\<F\>)$}
\def\2{$\R\fst f^\flat(\R\CH_Y(F\<,G\>)\Otimes{Y}F)$}
\def\3{$\R\fst f^\flat\R\CH_Y(F\<,G\>)\Otimes{Y}F$}
\def\4{$\R\CH_Y(F\<,G\>)\Otimes{Y}F$}
\def\5{$G$}
\def\6{$\R\fst f^\flat\<G$}
\def\7{$\R\fst \R\CH_\sX(\LL f^*\<\<F\<,f^\flat\<G\>)\Otimes{Y}F$}
\def\8{$\R\fst (\R\CH_\sX(\LL f^*\<\<F\<,f^\flat\<G\>)\Otimes{\sX}\LL f^*\<\<F\>)$}
\def\9{$\R\CH_Y(F\<,\R\fst f^\flat\<G\>)\Otimes{Y}F$}
 \bpic[xscale=2.6, yscale=1.45]

   \node(11) at (1,-1){\1}; 
   \node(14) at (4,-1){\2};   
   
   \node(22) at (1.4,-2){\3}; 
   \node(23) at (3.4,-2){\4};
     
   \node(33) at (4.15,-3){\5};
   
   \node(42) at (1.4,-4){\7};
   \node(43) at (3.4,-4){\9};
   
   \node(51) at (1,-5){\8}; 
   \node(54) at (4.5,-5){\6}; 
   
   \draw[->] (11)--(14) ;  
   
   \draw[->] (22)--(23) ;  
    
   \draw[->] (42)--(43) ;
    
   \draw[->] (51)--(54) ;  
   
    \draw[->] (.35,-1.25)--(.35,-4.75) ;
   
    \draw[->] (22)--(42) ;
      
    \draw[->] (43)--(23) ;
      
    \draw[->] (4.5,-1.25)--(54) ;
    
    \draw[->] (11)--(22) node[above=-5, midway, scale=.75]{$\mkern85mu\eqref{projf}^{-\<1}$};
    \draw[->] (14)--(23) ;
    \draw[->] (.6,-4.75)--(42) node[above=-10, midway, scale=.75]{$\mkern97mu\eqref{projf}^{-\<1}$} ;    
    \draw[->] (43)--(54) ;
    \draw[->] (54)--(33) ;    
    \draw[->] (23)--(33) ;

    \node at(2.5,-1.52)[scale=.9] {\circled7$_1$} ; 
    \node at (2.5, -3)[scale=.9]{\circled7$_2$} ;
    \node at (2.5, -4.52)[scale=.9]{\circled7$_3$} ;
  
 \epic 
\]

The commutativity of the unlabeled subdiagrams is easily checked. That of~\circled7$_1$ results from
\ref{flat and tensor}(ii), and that of \circled7$_2$ (without ``$\Otimes{Y}F\>$") from \ref{flat and hom}(iii).

The commutativity of \circled7$_3$ results from that of its adjoint diagram, namely, with the abbreviations 
$\fst$ for $\R\fst$, $f^*$ for $\LL f^*\<,$ $[-,-]^{}_Z$ for $\R\CH_Z(-,-)$, and 
$\otimes_Z$ for~$\Otimes{Z}\ (Z=X\textup{ or }Y)$, from that of the natural diagram
\[
\def\1{$f^*(\fst [f^*\<\<F\<,f^\flat\<G\>]^{}_X\otimes F\>)$}
\def\2{$f^*([F\<,\fst f^\flat\<G\>]^{}_Y\otimes F)$}
\def\3{$[f^*\<\<F\<,f^\flat\<G\>]^{}_X\otimes f^*\<\<F$}
\def\4{$f^\flat\<G,$}
 \bpic[xscale=2.75, yscale=1.65]

   \node(11) at (1,-1){\1}; 
   \node(13) at (3,-1){\2};

   \node(21) at (1,-2){\3};
   \node(23) at (3,-2){\4};
   
   \draw[->] (11)--(13) ;  
   
   \draw[->] (21)--(23) ;
    
    \draw[->] (11)--(21) node[left=1pt, midway, scale=.75]{\eqref{projf}} ;
       
    \draw[->] (13)--(23) ;

  \epic 
\]
which commutativity follows, e.g., from \cite[3.5.5, 3.5.6(d) and 3.4.6.2]{li}.
\end{proof}

\end{subcosa}  

\end{cosa}

\begin{small}
\begin{cosa}\label{general affine}  

Corollary~\ref{qp duality} extends to \emph{all} affine maps 
of qcqs schemes if we replace~$\Dqc$ by $\D(\sA_\qc)$, see Corollary~\ref{anyaffinemap}. Corollary~\ref{moderate}
records that this replacement is unnecessary~for schemes~$X$ such that the natural functor $\D(\sA_\qc(X))\to\Dqc(X)$ is an equival\-ence of categories---for instance 
finite\kf-dimensional noetherian schemes \cite[p.\,191, 3.7\kf]{Il} or quasi-compact separated schemes \cite[p.\,230, 5.5]{BN} (but not arbitrary qcqs schemes,  see \cite[p.\,195, 0.3\kf]{Il}). 

Some details follow; the rest are left to the reader.\va2

For any qcqs scheme $X$, the inclusion functor 
$j^{}_{\<\<X}\colon\sA_\qc(X)\to\sA(X)$ has a right adjoint~$Q_\sX$, the \emph{quasi-coherator}
\cite[p.\,187, 3.2]{Il}. For example, if $X$ is affine one can take $Q_\sX$ to be the sheafification of
the global section functor.

\end{cosa}

\begin{sublem}\label{oQ} 
Let\/ $f\colon X\to Y$ be an affine map of qcqs schemes,
let\/ $\phi\colon \oY\to Y$  be as in \/ \eqref{phi} and let\/ $Q\set Q_Y$ be the quasi-coherator.\va2

{\rm(i)} The inclusion\/
$\bar{\jmath}\colon\sA_\qc(\oY)
\hookrightarrow \sA(\oY)$ has a right adjoint\/ $\oQ$
such that\va1 $\phi_*\oQ=Q\phi_*\>.$ 

{\rm(ii)} The derived functor\/ $\R\oQ$ is right-adjoint to\/ 
$\R\bar\jmath\colon\D(\sA_\qc(\oY))\to \D(\>\oY)$.
\end{sublem}

\begin{proof}
 (i) For $N\in\sA(Y)$, let $\epsilon_N\colon QN=j_YQN\to N$
be the counit map\va1 for the adjunction $j^{}_Y\!\dashv Q$.  
Set ${\oO}\set\phi_*\CO_{\>\oY}=\fst\OX,$ and $\>\otimes\>\set \otimes_{\OY}$ (nonderived).

For any $M\in\sA(\oY)$, scalar multiplication is an $\sA(Y)$-map 
\[
\mu\colon {\oO}\otimes\phi_*M\to \phi_*M. 
\]
Since $Q$ is right-adjoint to $j^{}_Y$, there is  a unique $\sA_\qc(Y)$-map~$\lambda$
making the following diagram commute:
\[
\CD
{\oO}\otimes Q\phi_*M @>\lambda>>Q\phi_*M\\
@V\id\otimes\>\> \epsilon VV @VV\epsilon V\\
{\oO}\otimes \phi_*M @>>\mu> \phi_*M
\endCD
\]
One checks that $\lambda$ 
makes $Q\phi_*M$ into a quasi-coherent ${\oO}$-module $\oQ M$; that this construction is functorial; and that the resulting functor $\oQ$ is as asserted in~(i).\va1

(ii) Since by (i),  $\oQ$ has an exact 
left adjoint,  therefore if~$\>G\>$ is a K-injective complex in~$\sA(\oY)$ then $\oQ G$ is  K-injective in~$\sA_\qc(\oY)$;
and  for any complex~$F\<$ in~$\sA_\qc(\oY)$, the functorial isomorphism of complexes of abelian groups 
\[
\Hom_{\sA(\oY)}(\bar\jmath\> F,G\>)\iso \Hom_{\sA_\qc(\oY)}(F,\>\oQ G\>)
\]
that results from~(i) gives a functorial derived-category isomorphism
\[
\R\<\<\Hom_{\D{(\oY)}}(\R \bar\jmath\> F\<,G\>)\iso \R\<\<\Hom_{\D(\sA_\qc(\oY))}(F\<,\>\R\oQ G\>),
\]
to which application of the homology functor $\textup{H}^0$ gives a functorial isomorphism
\[
\Hom_{\D(\>\oY)}(\R \bar\jmath\> F\<,G\>)\iso\Hom_{\D(\sA_\qc(\oY))}(F\<,\>\R\oQ G\>)
\]
that extends via K-injective resolution $G'\to G$ to arbitrary $G'\in\D(\oY)$.
\end{proof}

\begin{subcor}\label{anyaffinemap} With\/ $f$ as in~\textup{\ref{oQ}} and $\pt$ as in\/ \textup{\ref{duality for phi},} the functor\/ 
$\bar{f\:\<}^{\!\<*}\R\oQ\>\pt$ is right-adjoint to\/ $\R\fst\colon\D(\sA_\qc(X))\to\D(Y)$.
\end{subcor}

\begin{proof}
The functor $\R\fst$ factors as 
\[
\D(\sA_\qc(X))\xto{\R\bar\fst\>}\D(\sA_\qc(\oY))\xto{\:\R \bar\jmath\:}\D(\>\oY)\xto{\<\phi_*\>}\D(Y).
\]
So the assertion follows  from the paragraph just before \ref{locally},  Lemma~\ref{oQ}, and Corollary~\ref{adjunction0}.
\end{proof}

\begin{subcor}\label{moderate}
In~\textup{\ref{anyaffinemap},} if the natural functor\/ $\boldsymbol j^{}_{\!X}\colon\D(\sA_\qc(X))\to\Dqc(X)$ is an equivalence 
then\/ $\boldsymbol j^{}_{\!X}\<\bar{f\:\<}^{\!\<*}\R\oQ\>\pt$ is right-adjoint to\/ 
$\R\fst\colon\Dqc(X)\to\D(Y).$
\end{subcor}

\end{small}

\section{From commutative algebra to affine schemes}\label{affex} 

By Theorem~\ref{indt base change} with $u$ an open immersion, or by direct verification,
the foregoing constructions involving $f^\flat$ for finite pseudo\kf-coherent $f\colon X\to Y$ are compatible with open immersions on~$Y\<$, and so can be locally elucidated by making them more explicit when $Y$ and $X$ are affine schemes. We do this via an ``equivalence," given by \emph{sheafification,}  from the (concretely realized) duality theory for derived categories of modules over commutative rings  to the duality theory for $\Dqc$-categories over affine 
schemes---essentially a special case of the equivalence mentioned near the beginning of~\S\ref{pc finite}.
A quasi-inverse for this equivalence is provided by  the derived global section \mbox{functor.} 
Details appear in section~\ref{affine aspect}.

The underlying idea is, given a ring-homomorphism 
$\varphi\colon R\to S$ and a construction involving the corresponding scheme\kf-map
$f\colon \spec S\to\spec R$, to describe a concrete commutative\kf-algebra construction involving $\varphi$, whose sheafification is naturally isomorphic to the given one, and so constitutes a concrete realization. (Abusing terminology,  ``concrete" signifies ``concrete modulo choosing
K-injective or K-projective resolutions of complexes.")

For instance, sheafifying the right adjoint $\ush\varphi(-)\set\R\<\<\Hom_\varphi(S,-)$ of the restriction-of-scalars functor 
$\varphi_*\colon\D(S)\to\D(R)$ gives a right adjoint for $\R\fst$ (see Corollary~\ref{duality for phiaff} 
and  Proposition~\ref{affine duality}).

More such realizations are presented,  for the base\kf-change 
isomorphism~$\beta_\sigma$  of~Theorem~\ref{indt base change} (see Proposition~\ref{concrete bc2}),
and for the interaction of $\ush\varphi$ with~$\Otimes{}$ and $\R\<\Hom$ (see Propositions~\ref{concrete chi} 
and~\ref{concrete zeta}).

The algebraic version of pseudofunctoriality for Koszul-regular immersions (to which the scheme\kf-theoretic one has been reduced, see remarks immediately after Lemma~\ref{Kozpf}) is proved in section~\ref{transci}. This proof, essentially the one in \cite[Appendix C.6]{NS}, is more complicated than anything that came before.

\begin{cosa}\label{affine aspect}
For a commutative ring $T\<$, let $\sA(T)$ be the category of $\>T$-modules and~$\D(T)$ the corresponding derived category. 

Set $Z\set\spec T$. For the usual adjunction $\SH{\>\>T}\!\dashv \Gamma^{}_{\!\!Z}$ with
$\SH{\>\>T}\colon\sA(T)\to\sA(Z)$ the sheafification functor and 
$\Gamma^{}_{\!\!Z}\set\Gamma(Z,-)\colon \sA(Z)\to \sA(T)$ the global-section functor \cite[1.7.4]{EGA1},
the unit map $\id\iso\Gamma^{}_{\!\!Z}\>\SH{\>\>T}$ and 
counit map $\SH{\>\>T}\Gamma^{}_{\!\!Z}\to \id$ are the natural ones. 
%

Since $\Gamma^{}_{\!\!Z}$ has an exact left adjoint, it preserves
K-injectivity, and there results an adjunction $\SH{\>\>T}\!\<\dashv \R\Gamma^{}_{\!\!Z}$
of derived functors between $\D(T)$ and $\D(Z)$.  
Here the unit map is the natural functorial composite isomorphism\va{-1}
\begin{equation}\label{unitmap}
\SG\iso\Gamma^{}_{\!\!Z}\>\SH{\>\>T}\SG \iso\R\Gamma^{}_{\!\!Z}\>\SH{\>\>T}\SG
\qquad \SG\in\D(T)
\end{equation}
($\Gamma^{}_{\!\!Z}\>\SH{\>\>T}\to\R\Gamma^{}_{\!\!Z}\>\SH{\>\>T}$ is an isomorphism by  \cite[2.7.5, (ii)$\Rightarrow$(a)]{li}, dualized, with $d=0$,
and \cite[2.2.6]{li}));
and for $G\in\D(Z)$), if
$G\to J$ is a quasi-isomorphism of $\OZ$-complexes with $J$ K-injective,
then the counit~map\va{-1}
is the natural composite $\D(Z)$-map 
\[
\SH{\>\>T}\R\Gamma^{}_{\!\!Z}\>G \iso\SH{\>\>T}\<\Gamma^{}_{\!\!Z}J\lto J\iso G\>.
\]

The functor~
$\SH{\>\>T}$ factors naturally as 
\(
\sA(T)\xto{\bst}\sA_\qc(Z)\hookrightarrow\sA(Z),
\)
and $\bst$ is an equivalence---whence so is its derived functor
$\D(T)\to\D(\sA_\qc(Z))$, which will also be denoted $\bst$.

The derived~$\SH{\>\>T}$, considered as a functor from 
$\D(T)$ to $\Dqc(Z)\subset\D(Z)$, factors naturally as 
\begin{equation}\label{qcequiv}
\D(T)\xto[\lift1.3,\bst,\!]{\lift.45,\:\approx\:,} \D(\sA_\qc(Z))\xto[\>\>\lift1.3,\boldsymbol j^{}_{\!Z},]
{}\Dqc(Z).
\end{equation} 
Since $\boldsymbol j^{}_{\<\<Z}$ is an equivalence \cite[p.\,225, 5.1]{BN}, therefore so is this $\SH{\>\>T}\<$,
for which a quasi-inverse (i.e., right adjoint) is the restriction to $\Dqc(Z)$ of  $ \R\Gamma^{}_{\!\!Z}\>$.

\begin{subcosa}
Let $\varphi\colon R\to S$ be a homomorphism of commutative rings. Denote by
$\varphi_*\colon\sA(S)\to\sA(R)$ the exact restriction-of-scalars functor, and also (abusing notation)
its derived functor $\D(S)\to\D(R)$.  

Let\va{-3}
\[
 \spec S =:X\xto{\,f\,} Y\set \spec R
\]
be the scheme\kf-map corresponding to $\varphi$. 
For any $S$-complex~$\SEE$ 
 the natural $\D(Y)$-maps are isomorphisms
\begin{equation}\label{sheafify_*}
\SH{\<R}\varphi_*\SEE\xto[\upsilon_{\<\varphi}\<(\SEE)]{\Iso} \fst\>\SH{S} \SEE\xto[q^{}_{\<\<f}(\SEE)]{\Iso}\R\fst\> \SH{S}\SEE.
\end{equation}
The first isomorphism is elementary: it is same as the natural isomorphism 
$\widetilde \SEE_R\cong\bar \fst \widetilde \SEE_S$ from Example~\ref{locally}, with
scalars restricted from $\fst\OX$ to $\OY$. In more detail (see  \cite[1.7.7(ii)]{EGA1}), $\upsilon_{\<\varphi}\colon\SH{\<R}\varphi_*\iso \fst\>\SH{S}$ is the functorial map such~that 
$\Gamma^{}_{\<\<\!Y}\upsilon_\varphi$ is
the natural composite functorial $R$-isomorphism
\begin{equation}\label{factor}
\Gamma^{}_{\<\<\!Y}\SH{\<R}\varphi_*\iso\varphi_*\iso\varphi_*\Gamma^{}_{\!\!X}\SH{S}
=\Gamma^{}_{\<\<\!Y}\<\fst\>\SH{S}\<,
\end{equation}
whence $\upsilon_{\<\varphi}$ is the natural composite isomorphism
\begin{equation*}\label{factor'}
\SH{\<R}\varphi_*\iso\SH{\<R}\varphi_*\Gamma^{}_{\!\!X}\SH{S}
=\SH{\<R}\Gamma^{}_{\<\<\!Y}\<\fst\>\SH{S}\< \iso \fst\>\SH{S}\<.
\tag*{(\ref{factor})$'$}
\end{equation*}
And,  the map $q^{}_{\<\<f}(\SEE^{\>0})$ is an isomorphism for any $S$-module~$\SEE^{\>0}$
\cite[1.3.2]{EGA3}\kf,
that is, such an $\SEE^{\>0}$ is $\fst$-acyclic \cite[2.2.6]{li}; so by the dualized version of \cite[2.7.5, (ii)$\Rightarrow$(a)]{li},  $q^{}_{\<\<f}(\SEE)$ is an isomorphism for any $S$-complex~$\SEE$.\va2

The isomorphisms in \eqref{sheafify_*} are pseudofunctorial, in that for 
any homomorphism $\xi\colon S\to T$ of commutative rings, with corresponding scheme\kf-map
\(
 \spec T =:V\xto{\lift.7,\,g\,,} X\set \spec S,
\)
the following natural diagram commutes: 
\begin{equation}\label{pfgamma}
\def\1{$\SH{\<R}\<\varphi_*\xi_*$}
\def\3{$\SH{\<R}\<(\xi\varphi)_*$}
\def\4{$\fst\>\SH{S}\xi_*$}
\def\6{$\R\fst\>\SH{S}\xi_*$}
\def\7{$\fst g_*\>\SH{\>T}$}
\def\8{$\R\fst g_*\>\SH{\>T}$}
\def\9{$\R\fst \R g_*\>\SH{\>T}$}
\def\ten{$\R(fg)_*\>\SH{\>T}$}
\def\lvn{$(fg)_*\>\SH{\>T}$}
 \CD
  \bpic[xscale=3, yscale=1.35]

   \node(11) at (1,-1){\1};   
   \node(13) at (3,-1){\3};
   
   \node(21) at (1,-3){\4};   

   \node(31) at (1,-4){\6};
   \node(32) at (2,-3){\7};
   \node(33) at (3,-3){\lvn};
   
   \node(41) at (1,-5){\8};
   
   \node(51) at (2,-5){\9};
   \node(53) at (3,-5){\ten};

    \draw[double distance=2pt] (11)--(13) ; 
 
    \draw[->] (32)--(33) node[above=1pt, midway, scale=.75] {$\Iso$} ;
   
   \draw[->] (51)--(53) node[above=1pt, midway, scale=.75] {$\Iso$} ; 
                                   
    \draw[->] (11)--(21) node[left=1pt, midway, scale=.75] {$\upsilon_{\<\varphi}$} ;
    \draw[->] (21)--(31) node[left=1pt, midway, scale=.75] {$q^{}_{\<\<f}$} ;
    \draw[->] (31)--(41) node[left=1pt, midway, scale=.75] {$\R\fst\<\upsilon^{}_{\<\xi}$} ;
    \draw[->] (41)--(51) node[below, midway, scale=.75]{$\R\fst q^{}_g$} ;
                                                 
    \draw[->] (13)--(33) node[right=1pt, midway, scale=.75] {$\upsilon^{}_{\<\xi\varphi}$} ;
    \draw[->] (33)--(53) node[right=1, midway, scale=.75]{$q^{}_{\<\<f\<g}$} ;
    
    \draw[->] (21)--(32)  node[above=-1, midway, scale=.75] {$\fst\<\upsilon_{\<\xi}$} ;
    \draw[->] (32)--(51)    ;

    \node at (2,-2.02)[scale=.9]{\circled1} ;
    \node at (1.5,-4.02)[scale=.9]{\circled2} ;
    \node at (2.5,-4.02)[scale=.9]{\circled3} ;
 \epic
 \endCD
\end{equation}

Commutativity of \circled1 (i.e., pseudofunctoriality of $\upsilon$) is readily checked
via application of the equivalence $\Gamma^{}_{\<\<\!Y}$ and use of  \eqref{factor}, or otherwise; that of~ \circled2 is clear; and that of \circled3 is rudimentary, see.~e.g., \cite[(3.6.4.1)]{li}.\va2
\end{subcosa}

\pagebreak[3]
\begin{subcosa} (The reader is advised to proceed directly to section~\ref{ushvarphi}, returning here only as needed).\va2

The functor $\R\fst\colon\D(X)\to\D(Y)$ has a monoidal structure (see \cite[Definition 3.4.2]{li}), given by the natural composite map 
\[
\OY\lto\fst\OX\iso\R\fst\OX
\] 
and by the natural bifunctorial map (see e.g., \cite[3.2.4(ii)]{li})
\[
 \R\fst E\Otimes{Y}\R\fst F\lto\R\fst(E\Otimes{\sX}F\>)\qquad(E,F\in\D(X)).
\]

The functor $\varphi_*\colon\D(S)\to\D(R)$ has an analogous monoidal structure.

The exact functor $\SH{S}$ also has a monoidal structure.  For its description,
recall that for every $S$-complex~$\SG$, there is a quasi-isomorphism 
\mbox{$\pi^{}_{\SG}\colon\bar\SG\to \SG$}  such that 
$\bar\SG$ is a direct limit of bounded-above
flat $S$-complexes, and so is K-flat (see \cite[Prop.~2.5.5]{li} and its proof). Then 
$\SH{S}\bar\SG$ is a direct limit of bounded-above 
flat $\OX$-complexes, and so is a K-flat $\OX$-complex; and $\SH{S}\pi^{}_{\SG}$ is a quasi-isomorphism.  Therefore, one can
declare the said monoidal structure on $\SH{S}$ to be given by the natural isomorphism 
$\OX\iso\SH{S}\<S$ and by the unique bifunctorial $\D(X)$-isomorphism
(see \cite[2.6.5(ii)]{li}, dualized)
\begin{equation}\label{sheafify Otimes}
\SH{S}\<\SEE\Otimes{\sX}\<\SH{S}\<\SF\iso\SH{S}\<(\SEE\Otimes{\<S}\SF\>)
\qquad(\SEE,\SF\in\D(S))
\end{equation}
that makes the following otherwise natural diagram commute for all $\SEE,\>\SF$:
\[
\def\1{$\SH{S}\<\SEE\Otimes{\sX}\<\SH{S}\<\SF$}
\def\2{$\SH{S}\<(\SEE\Otimes{\<S}\SF\>)$}
\def\3{$\SH{S}\<\SEE\otimes_{\sX}\<\SH{S}\<\SF$}
\def\4{$\SH{S}\<(\SEE\otimes_{\<S}\SF\>)$}
  \bpic[xscale=3.9, yscale=1.5]

   \node(11) at (1,-1){\1};   
   \node(12) at (2,-1){\2}; 
      
   \node(21) at (1,-2){\3};
   \node(22) at (2,-2){\4};
       
    \draw[->] (11)--(12) node[above,midway, scale=.75]{$\Iso$} ; 
    \draw[->] (21)--(22) node[above,midway, scale=.75]{$\Iso$} ;  
    
    \draw[->] (11)--(21) ;
    \draw[->] (12)--(22) ; 
   \epic
\]

Similar remarks apply to the functor $\SH{\<R}$.

The isomorphism $\upsilon_{\<\varphi}\colon\SH{\<R}\varphi_*\iso\R\fst\>\SH{S}$ from \ref{factor'} is compatible with the monoidal structures on the functors involved:  
\begin{sublem}\label{mon_*}
For\/ $\SEE,$ $\SF\in\D(S),$ the following natural diagram commutes$:$
\[
\def\1{$\R \fst\> \SH{\<S}(\SEE\Otimes{\<\<S}\SF\>)$}
\def\2{$\R \fst (\SH{\<S}\SEE\Otimes{\<\<X}\SH{\<S}\SF\>)$}
\def\3{$\R \fst\> \SH{\<S}\SEE\Otimes{Y}\R \fst\> \SH{\<S}\SF$}
\def\4{$\SH{\<R}\< \varphi_*(\SEE\Otimes{\<\<S}\SF\>)$}
\def\5{$\SH{\<R}( \varphi_*\SEE\Otimes{\<\<R} \varphi_*\SF\>)$}
\def\6{$\SH{\<R}\< \varphi_*\SEE\Otimes{Y}\SH{R}\<\varphi_*\SF$}
  \bpic[xscale=3.9, yscale=1.5]

   \node(11) at (1,-1){\6};   
   \node(12) at (2.037,-1){\5}; 
   \node(13) at (3,-1){\4};   
   
   \node(21) at (1,-2){\3};
   \node(22) at (2.037,-2){\2};
   \node(23) at (3,-2){\1};
    
    \draw[->] (11)--(12) node[above,midway, scale=.75]{$\Iso$} ; 
    \draw[->] (12)--(13) ;
    
    \draw[->] (21)--(22) ;  
    \draw[->] (22)--(23) node[above,midway, scale=.75]{$\Iso$} ;
    
    \draw[->] (11)--(21) node[left,midway, scale=.75]{$\simeq$} 
                                   node[right,midway, scale=.75]{$\upsilon_{\<\<\varphi}\<\Otimes{Y}\<\<\upsilon_{\<\varphi}$} ;
    \draw[->] (13)--(23) node[right,midway, scale=.75]{$\simeq$} 
                                   node[left,midway, scale=.75]{$\upsilon_{\<\varphi}$} ;
 
   \epic
\]
\end{sublem}

\begin{proof} 
In view of the above quasi-isomorphisms 
$\pi^{}_{\SEE}\colon\bar\SEE\to \SEE$ and $\pi^{}_{\SF}\colon\bar\SF\to \SF$, one can assume that all the complexes $\SEE$, $\SF$, $ \SH{S}\SEE$ and $\SH{S}\SF$ are K-flat. 

It's enough then to prove commutativity of the natural diagram          
\[\mkern-4mu
\def\1{$\SH{\<R}\< \varphi_*\SEE\Otimes{Y}\SH{R}\<\varphi_*\SF$}
\def\2{$\SH{\<R}\<( \varphi_*\SEE\Otimes{\<\<R} \varphi_*\SF\>)$}
\def\3{$\SH{\<R}\< \varphi_*(\SEE\Otimes{\<\<S}\SF\>)$}
\def\4{$\SH{\<R}\< \varphi_*\SEE\otimes_{Y}\SH{R}\<\varphi_*\SF$}
\def\5{$\fst\>\SH{S}\SEE\Otimes{Y}\<\<\fst\>\SH{S}\SF$}
\def\6{$\SH{\<R}\<( \varphi_*\SEE\otimes_{\<R} \varphi_*\SF\>)$}
\def\7{$\SH{\<R}\< \varphi_*(\SEE\otimes_{S}\SF\>)$}
\def\8{$\fst\>\SH{S}\SEE\otimes_{Y}\<\<\fst\>\SH{S}\SF$}
\def\9{$\fst(\SH{S}\SEE\otimes_{\<X}\SH{S}\SF\>)$}
\def\ten{$\fst\>\SH{S}\<(\SEE\otimes_{S}\SF\>)$}
\def\lvn{$\R\fst(\SH{S}\SEE\otimes_{\<X}\SH{S}\SF\>)$}
\def\twv{$\R\fst\>\SH{S}\<(\SEE\otimes_{S}\SF\>)$}
\def\thn{$\R\fst\>\SH{S}\SEE\Otimes{Y}\R\fst\>\SH{S}\SF$}
\def\frn{$\R \fst (\SH{\<S}\SEE\Otimes{\<\<X}\SH{\<S}\SF\>)$}
\def\ffn{$\R \fst\> \SH{\<S}(\SEE\Otimes{\<\<S}\SF\>)$}
  \bpic[xscale=4.5, yscale=1.5]

   \node(11) at (1,-1){\1};   
   \node(12) at (2,-1){\2}; 
   \node(13) at (3,-1){\3};
   
   \node(21) at (1.5,-1.75){\4}; 
   
   \node(31) at (1,-2.5){\5};   
   \node(32) at (2,-2.5){\6}; 
   \node(33) at (3,-2.5){\7};

   \node(41) at (1.5,-3.25){\8};
   \node(42) at (2,-4){\9};
   \node(43) at (3,-4){\ten}; 
   
   \node(52) at (2,-4.75){\lvn};
   \node(53) at (3,-4.75){\twv};
   
   \node(61) at (1,-5.5){\thn};
   \node(62) at (2,-5.5){\frn}; 
   \node(63) at (3,-5.5){\ffn};   
     
    \draw[->] (11)--(12) node[above=1, midway, scale=.75] {$\Iso$} ; 
     \draw[->] (12)--(13) ;
                                                              
    \draw[->] (32)--(33) ;
      
   \draw[->] (42)--(43) node[above=1, midway, scale=.75] {$$};
   
   \draw[->] (52)--(53) ;

   \draw[->] (61)--(62) ; 
   \draw[->] (62)--(63) ;
                                   
     \draw[->] (11)--(31) node[left=1, midway, scale=.75] {$\simeq$} ;
     \draw[->] (31)--(61) node[left=1, midway, scale=.75] {$\simeq$} ;
           
     \draw[->] (21)--(41) node[left=1, midway, scale=.75] {$\simeq$} ;
    
     \draw[->] (12)--(32) node[right=1, midway, scale=.75] {$\simeq$} ;
     \draw[->] (42)--(52) node[right=1, midway, scale=.75] {$\simeq$} ;
     \draw[->] (62)--(52) node[right=1, midway, scale=.75] {$\simeq$} ;      

     \draw[->] (13)--(33) node[right=1, midway, scale=.75] {$\simeq$} ;
     \draw[->] (33)--(43) node[right=1, midway, scale=.75] {$\simeq$} ;
     \draw[->] (43)--(53) node[right=1, midway, scale=.75] {$\simeq$} ;
     \draw[->] (63)--(53) node[right=1, midway, scale=.75] {$\simeq$} ;      

   \draw[->] (11)--(21) ;
   \draw[->] (21)--(32) ;
   \draw[->] (31)--(41) ;
   \draw[->] (41)--(42) ;
   
   \node at (1.6, -1.42)[scale=.9]{\circled1} ;
   \node at (2.52, -1.75)[scale=.9]{\circled2} ;
   \node at (2.25, -3.27)[scale=.9]{\circled3} ;
   \node at (1.4, -4.375)[scale=.9]{\circled4} ;
   \node at (2.52, -5.165)[scale=.9]{\circled5} ;
   
 \epic
\]
  
The commutativity of the unlabeled subdiagrams is clear.
 
Subdiagrams \circled1 and \circled5 commute by the  description of the map \eqref{sheafify Otimes}.

The commutativity of \circled3 can be easily be checked after application of the equivalence 
$\Gamma^{}_{\!\!X}$.

Subdiagram \circled4 expands naturally as follows, 
with $E\set\SH{S}\SEE$ and $F\set\SH{S}\SF$: 

\begin{small}
\[\mkern-3mu
\def\1{$\fst(E\otimes_\sX F\>)$}
\def\2{$\fst E\otimes_Y \fst F$}
\def\3{$\fst(f^*\mkern-2.5mu\fst E\otimes_\sX f^*\mkern-2.5mu\fst F\>)$}
\def\4{$\fst f^*(\fst E\otimes_Y \fst F\>)$}
\def\5{$\R\fst(f^*\mkern-2.5mu\fst E\otimes_\sX f^*\mkern-2.5mu\fst F\>)$}
\def\6{$\R\fst f^*(\fst E\otimes_Y \fst F\>)$}
\def\7{$\R\fst(E\otimes_\sX F\>)$}
\def\8{$\R\fst(f^*\mkern-2.5mu\fst E\Otimes{\sX} f^*\mkern-2.5mu\fst F\>)$}
\def\9{$\R\fst \LL f^*(\fst E\otimes_Y \fst F\>)$}
\def\ten{$\fst E\Otimes{Y} \fst F$}
\def\lvn{$\R\fst(\LL f^*\mkern-2.5mu\fst E\Otimes{\<\sX} \LL f^*\mkern-2.5mu\fst F\>)$}
\def\twv{$\R\fst \LL f^*(\fst E\Otimes{Y} \fst F\>)$}
\def\thn{$\R\fst(E\Otimes{\sX} F\>)$}
\def\frn{$\R\fst(\LL f^*\R\fst E\Otimes{\<\sX} \LL f^*\R\fst F\>)$}
\def\ffn{$\R\fst \LL f^*(\R\fst E\Otimes{Y} \R\fst F\>)$}
\def\sxn{$\R\fst E\Otimes{Y} \R\fst F$}
  \bpic[xscale=3.43, yscale=1.2]

   \node(11) at (1,-1){\1};   
   \node(14) at (4,-1){\2}; 
   
   \node(22) at (1.6,-2){\3}; 
   \node(23) at (2.9,-2){\4};  
   
   \node(32) at (1.6,-3){\5}; 
   \node(33) at (2.9,-3){\6}; 
   
   \node(41) at (1,-3.75){\7};   
   \node(42) at (1.6,-4.5){\8};
   \node(43) at (3.13,-4.5){\9};  
   \node(44) at (4,-5){\ten};
   
   \node(52) at (1.95,-6){\lvn};   
   \node(53) at (3.35,-6){\twv}; 
   
   \node(62) at (1.95,-7){\frn};   
   \node(63) at (3.35,-7){\ffn};
   
   \node(71) at (1,-8){\thn};  
   \node(74) at (4,-8){\sxn};

    \draw[<-] (11)--(14) ;
    
    \draw[<-] (22)--(23) ; 
    
    \draw[<-] (32)--(33) ;
    
    \draw[<-] (52)--(53) ; 
    
    \draw[<-] (62)--(63) ; 
         
    \draw[<-] (71)--(74) ;
     
    \draw[->] (11)--(41) ;
    \draw[->] (71)--(41) ;
    
    \draw[->] (22)--(32) ;
    \draw[<-] (32)--(42) ; 
    \draw[<-] (42)--(52) ; 
    \draw[<-] (62)--(52) ;

    \draw[->] (23)--(33) ;
    \draw[<-] (33)--(43) ; 
    \draw[<-] (43)--(53) ; 
    \draw[<-] (63)--(53) ; 
        
    \draw[<-] (14)--(44) ;
    \draw[<-] (74)--(44) ;
    
   \draw[<-] (11)--(22) ;
   \draw[->] (32)--(41) ;
   \draw[->] (42)--(71) ;
   \draw[->] (62)--(71) ;
   
   \draw[<-] (23)--(14) ;
   \draw[<-] (43)--(14) ;
   \draw[->] (44)--(53) ;
   \draw[->] (74)--(63) ;

   \node at (2.5, -1.52)[scale=.9]{\circled4$^{}_1$} ;
   \node at (3.27, -2.52)[scale=.9]{\circled4$^{}_2$} ;
   \node at (2.37, -4.55)[scale=.9]{\circled4$^{}_3$} ;
   \node at (1.6, -6.52)[scale=.9]{\circled4$^{}_4$} ;
   \node at (2.6, -7.55)[scale=.9]{\circled4$^{}_5$} ;
   
  \epic
\]
\end{small}
\hskip 8.4pt The commutativity of the unlabeled subdiagrams is clear.

The commutativity of subdiagram \circled4$^{}_1$ is given, e.g., by \cite[(3.4.5.2)]{li} (taking into account \emph{ibid.,} 3.1.9 and 3.4.4(a)). That of \circled4$^{}_5$ holds by definition of its bottom arrow
\cite[3.2.4(ii)]{li}; 
of \circled4$^{}_2$ by that of \cite[(3.2.1.3)]{li}; 
of \circled4$^{}_4$ by that of \cite[(3.2.1.2)]{li}; and
of \circled4$^{}_3$ by that of \cite[(3.2.4.1)]{li}.\va1

Finally, subdiagram \circled2 without``$\>\SH{\<R}\<$" expands naturally as follows:
\[
\def\1{$\varphi_*\SEE\Otimes{\<\<R} \varphi_*\SF$}
\def\2{$\varphi_*\LL\varphi^*(\varphi_*\SEE\Otimes{\<\<R} \varphi_*\SF\>)$}
\def\3{$\varphi_*(\LL\varphi^*\varphi_*\SEE\Otimes{\<S} \LL\varphi^*\varphi_*\SF\>)$}
\def\4{$\varphi_*(\SEE\Otimes{\<\<S}\SF\>)$}
\def\5{$\varphi_*\LL\varphi^*(\varphi_*\SEE\otimes_{\<R} \varphi_*\SF\>)$}
\def\6{$\varphi_*\SEE\otimes_R\varphi_*\SF$}
\def\7{$\varphi_*\varphi^*(\varphi_*\SEE\otimes_R\varphi_*\SF\>)$}
\def\8{$\varphi_*(\varphi^*\varphi_*\SEE\otimes_{S} \varphi^*\varphi_*\SF\>)$}
\def\9{$\varphi_*(\SEE\otimes_{S}\SF\>)$}
\def\0{$\varphi_*(\varphi^*\varphi_*\SEE\Otimes{\<\<S} \varphi^*\varphi_*\SF\>)$}
  \bpic[xscale=3.45, yscale=1.5]

   \node(11) at (1,-1){\1};   
   \node(14) at (4,-1){\4}; 
   
   \node(22) at (1.85,-2){\2}; 
   \node(23) at (3.15,-2){\3};  
   
   \node(32) at (1.57,-3){\5}; 
   \node(33) at (3,-3){\0}; 
   \node(34) at (4,-3){\4};
   
   \node(42) at (1.85,-4){\7};
   \node(43) at (3.15,-4){\8};  
   
   \node(51) at (1,-5){\6};   
   \node(54) at (4,-5){\9}; 
    
    \draw[->] (11)--(14) ;
    
    \draw[->] (22)--(23) ; 
    
    \draw[->] (33)--(34) ;
    
    \draw[->] (42)--(43) ;
     
    \draw[->] (51)--(54) ;
     
    \draw[->] (11)--(51) ;
    
    \draw[->] (22)--(32) ;
    \draw[->] (32)--(42) ; 
    
    \draw[->] (23)--(33) ;
    \draw[->] (33)--(43) ;
    
    \draw[double distance=2] (14)--(34) ;
    \draw[->] (34)--(54) ;

   \draw[->] (11)--(22) ;
   \draw[->] (23)--(14) ;
   \draw[->] (51)--(32) ;
   \draw[->] (51)--(42) ;
   \draw[->] (43)--(54) ;
   
   \node at (2.455, -1.52)[scale=.9]{\circled2$^{}_1$} ;
   \node at (2.28, -3.05)[scale=.9]{\circled2$^{}_2$} ;
   \node at (3.58, -2.52)[scale=.9]{\circled2$^{}_3$} ;
   \node at (1.58, -3.63)[scale=.9]{\circled2$^{}_4$} ;
   \node at (2.475, -4.528)[scale=.9]{\circled2$^{}_5$} ;
   
  \epic
\]

The commutativity of the unlabeled subdiagrams is clear.\!
Subdiagram~\circled2$^{}_1$ commutes by definition, cf.~\cite[3.2.4(ii)]{li}. Checking the commutativity of its nonderived version \circled2$^{}_5$ is left to the reader.
The commutativity of~\circled2$^{}_3$ (respectively~\circled2$^{}_4$) is given by that of \cite[(3.2.1.2)]{li} (respectively~\cite[(3.2.1.3)]{li}. The commutativity of \circled2$^{}_2$ is given by the commutative\kf-algebra counterpart (proved similarly) of \cite[(3.2.4.1)]{li}. Thus \circled2 commutes.

This concludes the proof of Lemma~\ref{mon_*}.
\end{proof}

\end{subcosa}

\begin{subcosa}\label{ushvarphi}

With $^\op$ denoting ``opposite category," the functor \va{-1}
\[
\Hom_\varphi\colon \sA(S)^\op\times\sA(R)\to\sA(S)\\[-5pt]
\]
is given by\va{-1}
\[
 \Hom_\varphi(\SEE,\SG\>)\set \Hom_R(\varphi_*\SEE, \>\SG\>)\qquad \bigl(\SEE\in\sA(S), \,\SG\in\sA(R)\bigr),
\] 
$ \Hom_R(\varphi_*\SEE, \>\SG\>)$ being an $S$-module in the usual way (cf.~\S\ref{2.3}).
\va5

The proof of the next proposition and its corollaries, being similar to that of Proposition~\ref{adjass0} and its corollaries,  is left to the reader.

\pagebreak

\begin{subprop}\label{adjass0aff}
There is a unique trifunctorial\/ $\D(S)$-isomorphism
\begin{multline*}
\bar\alpha(\SEE,\SF,\SG\>)\colon\R\<\<\Hom_\varphi(\SEE\Otimes{\<S}\<\SF, \>\SG\>)\iso\R\<\<\Hom_{\>S}(\SEE, \R\<\<\Hom_\varphi(\SF,\SG\>))\\   
\big(\SEE,\SF\in\D(\>S),\;\SG\in\D(R)\big)  
\end{multline*} 
such that the following natural diagram, with $\Hr\set\Hom$ and\/ $\bar\alpha^{}_0(\SEE,\SF,\SG\>)$ the standard isomorphism of\/ $S$-complexes, commutes\/$.$ 
\[\mkern-4mu
\def\1{$\Hr_\varphi(\SEE\otimes_{S}\<\SF, \>\SG\>)$}
\def\2{$\R \Hr_\varphi(\SEE\otimes_{S}\<\SF, \>\SG\>)$}
\def\3{$\R \Hr_\varphi(\SEE\Otimes{S}\<\SF, \>\SG\>)$}
\def\4{$\Hr_{S}\<\big(\<\SEE, \Hr_{\varphi}(\SF,\SG\>)\<\big)$}
\def\5{$\R \Hr_{S}\<\big(\<\SEE, \Hr_{\varphi}(\SF,\SG\>)\<\big)$}
\def\6{$\R \Hr_{S}\<\big(\<\SEE, \R \Hr_{\varphi}(\SF,\SG\>)\<\big)$}
 \bpic[xscale=4.25, yscale=1.6]

   \node(11) at (1,-1){\1};   
   \node(12) at (1.972,-1){\2}; 
   \node(13) at (3,-1){\3};
   
   \node(21) at (1,-2){\4};  
   \node(22) at (1.972,-2){\5};
   \node(23) at (3,-2){\6}; 
   
    \draw[->] (11)--(12) ;  
    \draw[->] (12)--(13) ;
 
    \draw[->] (21)--(22) ;  
    \draw[->] (22)--(23) ;

    \draw[->] (11)--(21) node[right=1pt, midway, scale=.75] {$\simeq$}
                                   node[left=1pt, midway, scale=.75] {$\bar\alpha^{}_0(\SEE,\SF,\SG\>)$} ;
    \draw[->] (13)--(23) node[left=1pt, midway, scale=.75] {$\simeq$}
                                   node[right=1pt, midway, scale=.75] {$\bar\alpha(\SEE,\SF,\SG\>)$} ; 
                                       
    \epic
\]
  \end{subprop}

\begin{subcor}\label{adjassaff}
There is a unique trifunctorial $\D(R)$-isomorphism
\begin{multline*}
\bar\alpha_\varphi(\SEE,\SF,\SG\>)\colon\R\<\<\Hom_{\<R}\<\<\big(\varphi_*(\SEE\Otimes{\<S}\SF\>), \>\SG\big)\iso\varphi_*\R\<\<\Hom_{S}(\SEE, \R\<\<\Hom_\varphi(\SF,\SG\>)\big)\\   
\big(\SEE,\SF\in\D(\>S),\;\SG\in\D(R)\big)  
\end{multline*} 
such that the following natural\/ $\D(R)$-diagram, with $\Hr\set\Hom$, commutes\/$.$ 
\[\mkern-5mu
\def\1{$\Hr_{\<R}\<\big(\varphi_*(\SEE\otimes_{\<S}\<\SF\>), \>\SG\big)$}
\def\2{$\R \Hr_{\<R}\<\big(\varphi_*(\SEE\otimes_{\<S}\<\SF\>), \>\SG\big)$}
\def\3{$\R \Hr_{\<R}\<\big(\varphi_*(\SEE\Otimes{\<S}\<\SF\>), \>\SG\big)$}
\def\4{$\varphi_*\Hr_{S}\<\big(\<\SEE, \Hr_{\varphi}(\SF,\SG\>)\<\big)$}
\def\5{$\varphi_*\R \Hr_{S}\<\big(\<\SEE, \Hr_{\varphi}(\SF,\SG\>)\<\big)$}
\def\6{$\varphi_*\R \Hr_{S}\<\big(\<\SEE, \R \Hr_{\varphi}(\SF,\SG\>)\<\big)$}
 \bpic[xscale=4.37, yscale=1.6]

   \node(11) at (1,-1){\1};   
   \node(12) at (1.98,-1){\2}; 
   \node(13) at (3,-1){\3};
   
   \node(21) at (1,-2){\4};  
   \node(22) at (1.98,-2){\5};
   \node(23) at (3,-2){\6}; 
   
    \draw[->] (11)--(12) ;  
    \draw[->] (12)--(13) ;
 
    \draw[->] (21)--(22) ;  
    \draw[->] (22)--(23) ;

    \draw[->] (11)--(21) node[right=1pt, midway, scale=.75] {$\simeq$}
                                   node[left=-.4pt, midway, scale=.75] {$\varphi_*\bar\alpha_0(\<\SEE,\SF,\SG)$} ;
    \draw[->] (13)--(23) node[left=1pt, midway, scale=.75] {$\simeq$}
                                   node[right=1pt, midway, scale=.75] {$\bar\alpha_\varphi(\SEE,\SF,\SG\>)$} ; 
                                       
    \epic
\]
This entails the functorial $R$-isomorphism
\begin{equation*}\label{AdjAss}
\textup{H}^0\bar\alpha_\varphi(\SEE,\SF,\SG)
\colon\!\<\Hom_{\D(R)}\!\<\<\big(\varphi_*(\SEE\Otimes{\>S}\<\SF\>), \SG\big)
\!\iso\!\varphi_*\<\Hom_{\D(S)}\<\big(\SEE, \R\<\<\Hom_\varphi(\SF,\SG)\big).
\end{equation*}
\end{subcor}

Let\va{-1}
\[
\ush\varphi\colon\D(R)\to\D(S)\\[-2pt]
\] 
be the functor $\R\<\<\Hom_\varphi(S,-)$.

\begin{subcor}\label{duality for phiaff}
For\/ $\SEE\in\D(\>S)$ and $\SG\in\D(R),$ one has the  bifunctorial $\D(R)$-isomorphism
\[
\bar\alpha_\varphi(\SEE,S,\SG)\colon\R\<\<\Hom_R(\varphi_*\SEE, \>\SG\>)\iso\varphi_*\R\<\<\Hom_{\>S}(\SEE, \ush\varphi \SG\>),
\]
In particular, there is an adjunction
\begin{equation}\label{*sh}
\varphi_*\!\<\dashv \ush\varphi,
\end{equation}
given by the functorial\/ $R$-isomorphism
\begin{equation*}
\textup{H}^0\bar\alpha_\varphi(\SEE,S,\SG\>)\colon\Hom_{\D(R)}\!\<\<\big(\varphi_*\SEE, \>\SG\big)\iso
\varphi_*\<\Hom_{\D(S)}\<(\SEE, \>\ush\varphi \SG\big).
\end{equation*}
\end{subcor}

Here, the counit at $\SG\in\D(R)$ is the natural composite map 
\begin{equation}\label{counit*sh}
\mathsf t^{}_{\varphi,\SG}\colon\varphi_*\ush\varphi\SG\iso \R\<\<\Hom_R(\varphi_*S, \SG\>) 
\lto\R\<\<\Hom_R(R, \SG\>)\iso\SG\>;
\end{equation}
and the unit at $\SEE\in\D(S)$ is the natural composite map 
\[
\mathsf u^{}_{\varphi,\SEE}\colon\SEE\iso
\R\<\<\Hom_S(S,\SEE) \lto\R\<\<\Hom_\varphi(S, \varphi_*\SEE) =\ush\varphi\varphi_*\SEE,
\]
i.e., the natural composite map \va{-1}
\[
\SEE\xto{\>\>\eta'_{\SEE}\>\>} \Hom_\varphi(S, \varphi_*\SEE) \lto\R\<\<\Hom_\varphi(S, \varphi_*\SEE) =\ush\varphi\varphi_*\SEE
\]
where for $e\in \SEE$, $\eta'_{\SEE}\mkern.5mu e$ is the $R$-homomorphism taking $s\in S$ to $se$.
(Cf.~\ref{counit unit}.)\va4

Arguing as in the proof of Proposition~\ref{duality map}, one gets:

\begin{subprop}\label{duality for varphi}
The inverse of the above isomorphism\/ $\bar\alpha_\varphi(\SEE,S,\SG\>)$ factors naturally as
\[
\varphi_*\R\<\<\Hom_S(\SEE,\ush\varphi\SG)\lto
\R\<\<\Hom_R(\varphi_*\SEE,\varphi_*\ush\varphi\SG)\lto
\R\<\<\Hom_R(\varphi_*\SEE,\SG).
\]
\end{subprop}

\medskip
The  equivalences $\SH{}$ and~$\R\Gamma$ (section~\ref{affine aspect}) transform   
$\ush\varphi$ into a right adjoint of  $\R\fst$, as follows. \va3

One has, for $E\in\Dqc(X)$ and $G\in\D(Y)$, the natural isomorphisms\va2
\begin{equation}
\begin{aligned}\label{dualityisos}
\Hom_{\Dqc(X)}^{\mathstrut}(E\<, \>\SH{S}\mkern-.5mu\ush{\varphi}\R\Gamma^{}_{\<\<\!Y}\>G\>)
&\iso 
\Hom_{\D(S)}(\R\Gamma^{}_{\<\<\!X}\<E, \ush{\varphi}\R\Gamma^{}_{\<\<\!Y}\>G\>)\\[4pt]
&\underset{\eqref{*sh}}\iso
\Hom_{\D(R)}(\varphi_*\R\Gamma^{}_{\<\<\!X}\<E, \R\Gamma^{}_{\<\<\!Y}\>G\>)\\
&\iso
\Hom_{\D(Y)}(\SH{\<R}\varphi_*\R\Gamma^{}_{\<\<\!X}\<E, G\>)\\[5pt]
&\underset{\eqref{sheafify_*}}\iso
\Hom_{\D(Y)}(\R\fst\>\SH{S}\<\R\Gamma^{}_{\<\<\!X}E, G\>)\\
&\iso
\Hom_{\D(Y)}(\R\fst E\<, G\>).
\end{aligned}
\end{equation}
\vskip2pt\noindent
Hence: 
\pagebreak[3]

\begin{subprop}\label{affine duality} \textup{(Cf.~\ref{anyaffinemap}.)}
The above-defined functor
\[
\SH{S}\mkern-.5mu\ush{\varphi}\R\Gamma^{}_{\<\<\!Y}\colon\D(Y)\to\Dqc(X)
\]
is right-adjoint to\/ $\R\fst\>,$
with unit at\/ $E\in\Dqc(X)$ the natural composite map 
\begin{align*}
E\iso\SH{S}\R\Gamma^{}_{\<\<\!X}E&\,\lto\,\SH{S}\ush\varphi\varphi_*\R\Gamma^{}_{\<\<\!X}E\\
&\iso\SH{S}\ush\varphi\R(\varphi_*\Gamma^{}_{\<\<\!X})E\\
&\ =\!\!=\,\,\SH{S}\ush\varphi\R(\Gamma^{}_{\<\<\!Y}\fst)E
\iso\SH{S}\ush\varphi\R\Gamma^{}_{\<\<\!Y}\R\fst E
\end{align*}
and counit at\/ $G\in\D(Y)$ the natural composite map 
\[
\ G\lot\SH{\<R}\R\Gamma^{}_{\<\<\!Y}G\lot
\SH{\<R}\varphi_*\ush\varphi\R\Gamma^{}_{\<\<\!Y}G
\underset{\eqref{sheafify_*}}\osi
\R\fst\>\SH{S}\ush\varphi\R\Gamma^{}_{\<\<\!Y}G.
\]
\end{subprop}

\begin{proof}
The first assertion results at once from \eqref{dualityisos}. Verifying that the unit and counit are as stated 
is straightforward (if slightly tedious).
\end{proof}
\end{subcosa}

\begin{subcosa}\label{generalized flat}
\emph{Till this section~\textup{\ref{affex}} ends,  set} $f^\flat G\set
\SH{S}\mkern-.5mu\ush{\varphi}\R\Gamma^{}_{\<\<\!Y}G \ (G\in\D(Y))$.\va2

Proposition~\ref{affine duality} shows that this is
consistent with the previous meaning of $f^\flat \<G$ when $G$ is as in 
Proposition~\ref{represent}\kf; but now $f^\flat$ will be right-adjoint to the \emph{entire} functor
$\R\fst\colon\Dqc(X)\to\D(Y)$ for \emph{any} map $f$ of affine schemes 
(cf.~Corollary~\ref{anyaffinemap}). The unit map $E\to f^\flat\R\fst E$ and 
the counit map $\R\fst f^\flat G\to G$ are then as in Proposition~\ref{affine duality}.\va1

\end{subcosa}

\begin{subcosa}
As for pseudofunctoriality, given \va{-1} ring homomorphisms
$R\xto{\varphi} S \xto{\xi\>\>} T$, with corresponding scheme\kf-maps
\(
W\xto{\ g\ }X \xto{\ f\ }Y,
\) 
define\va1 (abstractly) the functorial isomorphism 
\begin{equation}\label{pi}
\pi_{\xi\<,\varphi}\colon\ush{\xi}\ush{\varphi}\iso\ush{(\xi\varphi)}
\end{equation}
to be, analogously to ~\eqref{pf flat}, naturally right-conjugate to the identity map
$(\xi\varphi)_*=\varphi_*\xi_*$. (For a concrete realization, see Proposition~\ref{pfush}.)
One checks that such $\pi$ make $\ush{(-)}$ into a \emph{contravariant pseudofunctor}.

It follows that $(-)^\flat$ is made into a \emph{pseudofunctorial} right adjoint 
of~$\R(-)_*$ by composite natural isomorphisms of the form
\[
\pi'_{\xi\<,\varphi}\colon g^\flat \<\<f^\flat=\SH{\>T}\<\ush{\xi}\R\Gamma^{}_{\<\<\!X}\>\SH{S}\<\ush{\varphi}\R\Gamma^{}_{\<\<\!Y}
\iso
\SH{\>T}\<\ush{\xi}\<\ush{\varphi}\R\Gamma^{}_{\<\<\!Y}
\>\underset{\via\pi_{\xi\<,\varphi}}\iso\>
\SH{\>T}\<\ush{(\xi\varphi)}\R\Gamma^{}_{\<\<\!Y}=(fg)^\flat.
\]

One has then a natural commutative diagram of isomorphisms
\begin{equation}\label{sheafify pf}\mkern-25mu
\def\1{$\SH{\>T}\mkern-.5mu\ush{\xi}\R\Gamma^{}_{\<\<\!X}\>
\SH{S}\mkern-.5mu\ush{\varphi}\R\Gamma^{}_{\<\<\!Y}$}
\def\2{$\SH{\>T}\mkern-.5mu\ush{\xi}\ush{\varphi}\R\Gamma^{}_{\<\<\!Y}$}
\def\3{$\SH{\>T}\mkern-.5mu\ush{(\xi\varphi)}\R\Gamma^{}_{\<\<\!Y}$}
\def\4{$g^\flat\<\< f^\flat$}
\def\5{$(f\<g)^\flat$}
\CD
 \bpic[xscale=7, yscale=1.6]

   \node(11) at (1,-1){\1}; 
   \node(115) at (1.54,-1){\2};  
   \node(12) at (2,-1){\3}; 
   
   \node(21) at (1,-2){\4};  
   \node(22) at (2,-2){\5};
   
    \draw[->] (11)--(115) node[above, midway, scale=.75] {$\Iso$} ;
    \draw[->] (115)--(12) node[above, midway, scale=.75] {$\Iso$}
                                   node[below, midway, scale=.75] {$\SH{\>T}\<(\pi_{\xi\<,\varphi})$} ;

    \draw[->] (21)--(22) node[above, midway, scale=.95] {$\Iso$}
                                   node[below=1pt, midway, scale=.75] {\eqref{pf flat}} ;
    

    \draw[->] (11)--(21) node[left=1pt, midway, scale=.75] {$\simeq$};  
    \draw[->] (12)--(22) node[right=1pt, midway, scale=.75] {$\simeq$};

    \epic
  \endCD
\end{equation}

\end{subcosa}
\noindent
Thus the next proposition, which contains  a concrete realization of  
$\pi_{\xi\<,\varphi}$, entails, \va{-1} for pseudo\kf-coherent maps of affine schemes, 
a concrete realization of~the pseudofunctoriality isomorphism $g^\flat\<\< f^\flat\iso(f\<g)^\flat$
in~\eqref{pf flat}.

\begin{subprop}\label{pfush}
The\/ $\D(T)$-isomorphism
\[
\pi_{\xi\<,\varphi}(\SG)\colon\R\<\<\Hom_\xi(T,\R\<\<\Hom_\varphi(S,\SG))\iso
\R\<\<\Hom_{\xi\varphi}(T,\SG)\qquad(\SG\in\D(R))
\]
is the unique functorial map $\alpha$ such that
the following natural diagram commutes for all\/ $R$-complexes $\SG$.\va{-1}
\[
\def\1{$\Hom_\xi(T,\Hom_\varphi(S,\SG))$}
\def\2{$\Hom_{\xi\varphi}(T,\SG)$}
\def\3{$\R\<\<\Hom_\xi(T,\R\<\<\Hom_\varphi(S,\SG))$}
\def\4{$\R\<\<\Hom_{\xi\varphi}(T,\SG)$}
 \bpic[xscale=4.5, yscale=1.45]

   \node(11) at (1,-1){\1};   
   \node(12) at (2,-1){\2}; 
   
   \node(21) at (1,-2){\3};  
   \node(22) at (2,-2){\4};

    \draw[->] (21)--(22) node[below, midway, scale=.75] {$\alpha(\SG)$} ;
     
    \draw[->] (11)--(12) ;  
    
    \draw[->] (11)--(21) ;
    \draw[->] (12)--(22) ;
  \epic
\]
\end{subprop}
\begin{proof}
It suffices to verify commutativity of the second diagram in \cite[3.3.7(a)]{li},
with $\beta$ the identity map of $(\xi\varphi)_*=\varphi_*\xi_*$.

For this purpose, one may assume that the $R$-complex $\SG$ is K-injective, whence so is 
the $S$-complex $\Hom_\varphi(S,\SG)$. 
Hence the assertion follows 
from commutativity of the following natural diagram of $R$-complexes.
\[
\def\1{$(\xi\varphi)_*\mkern-1.5mu  \Hom_\xi(T,\Hom_\varphi(S,\SG))$}
\def\2{$(\xi\varphi)_*\mkern-1mu \Hom_{\xi\varphi}(T,\SG)$}
\def\3{$\varphi_*\mkern-1mu \Hom_S(\xi_*T,\Hom_\varphi(S,\SG))$}
\def\4{$\<\Hom_R((\xi\varphi)_*T,\SG)$}
\def\5{$\varphi_*\mkern-1mu \Hom_\varphi(S,\SG)$}
\def\6{$\<\Hom_R(\varphi_*S,\SG)$}
\def\7{$\<\Hom_R(R,\SG)=\SG$}
\def\8{$\varphi_*\mkern-1mu \Hom_S(S,\Hom_\varphi(S,\SG))$}
\bpic[xscale=7.3, yscale=1.35]

   \node(11) at (1,-1){\1};   
   \node(12) at (2,-1){\2}; 
   
   \node(21) at (1,-2){\3};  
   \node(22) at (2,-2.5){\4};  
   
   \node(31) at (1,-3){\8}; 
   
   \node(41) at (1,-4){\5};  
   \node(415) at (1.5,-4){\6};
   \node(42) at (2,-4){\7};  

    \draw[->] (11)--(12) ;  
    
    \draw[double distance=2pt] (41)--(415) ;
    \draw[->] (415)--(42) ; 
    
    \draw[double distance=2pt] (11)--(21) ;
    \draw[->] (21)--(31) ;
    \draw[->] (31)--(41) node[left=1pt, midway, scale=.75] {$\simeq$} ;
   
    \draw[double distance=2pt] (12)--(22) ;
    \draw[->] (22)--(42) ;

  \epic
\]
\vskip-3pt\noindent
To verify this commutativity, one checks that that, by traveling around the diagram \emph{either} clockwise \emph{or} counterclockwise from upper left to lower right, an $S$-homomorphism 
$\lambda\colon\xi_*T\to\Hom_\varphi(S,\SG)$  goes to $[\lambda(1_T)](1_S)$.
\end{proof}

Recall, from ~\ref{adjass0aff}, 
the map $\bar\alpha(\SEE,\SF,\SG)\ \big(\SEE,\SF\in\D(\>S),\;\SG\in\D(R)\big) $.
\begin{subcor}
It holds that \/ $\xi_*\pi_{\xi\<,\varphi}(\SG)= \bar\alpha(\xi_*T,S,\SG).$ \hfill{$\square$}
\end{subcor}
\end{cosa}

\begin{cosa} \label{concrete version beta}
This section is devoted to a description of  the (concrete) commutative\kf-algebra version of the func\-torial map~$\beta_\sigma$  in~Theorem~\ref{indt base change}.\va2

Let 
\[
\def\1{$S'$}
\def\2{$S$}
\def\3{$R\>'$}
\def\4{$R$}
\def\5{$X'$}
\def\6{$X$}
\def\7{$Y'$}
\def\8{$Y$}
 \bpic[xscale=.9, yscale=.65]

   \node(11) at (1,-1){\1};   
   \node(13) at (3,-1){\2}; 
     
   \node(31) at (1,-3){\3};
   \node(33) at (3,-3){\4};
   
    \draw[<-] (11)--(13) node[above=1pt, midway, scale=.75] {$\nu$} ;  
      
   \draw[<-] (31)--(33) node[below=1pt, midway, scale=.75] {$\mu$} ;
    
    \draw[<-] (11)--(31) node[left=1pt, midway, scale=.75] {$\xi$} ;
   
    \draw[<-] (13)--(33) node[right=1pt, midway, scale=.75] {$\varphi$};
  \node at (2.05,-2)[scale=.85] {$\hat\sigma$} ;

   \node(15) at (6,-1){\5};   
   \node(17) at (8,-1){\6}; 
     
   \node(35) at (6,-3){\7};
   \node(37) at (8,-3){\8};
   
    \draw[->] (15)--(17) node[above=1pt, midway, scale=.75] {$v$} ;  
      
   \draw[->] (35)--(37) node[below=1pt, midway, scale=.75] {$u$} ;
    
    \draw[->] (15)--(35) node[left=1pt, midway, scale=.75] {$g$} ;
   
    \draw[->] (17)--(37) node[right=1pt, midway, scale=.75] {$f$};
  \node at (7.05,-2)[scale=.85] {$\sigma$} ;
 \epic 
\]
\vskip-1pt
\noindent be, respectively,  a commutative diagram of maps of commutative rings and the corresponding commutative diagram of maps of affine schemes.

\pagebreak[3]
The functor 
\begin{equation}\label{phi^*}
\varphi^*(-)\set S\otimes_\varphi -
\end{equation}
from $\sA(R)$ to $\sA(S)$ (notation akin to that in \S\ref{pphi}) is left-adjoint to $\varphi_*\>$, with
unit at $N\in\sA(R)$ the natural map
\[
\eta_\varphi(N)\colon N\to\varphi_*S\otimes_RN=\varphi_*(S\otimes_\varphi N)=\varphi_*\varphi^*N,
\]
and with  counit at $M\in\sA(S)$ the scalar multiplication map 
\[
\epsilon_\varphi(M)\colon \varphi^*\varphi_*M=S\otimes_\varphi \varphi_*M\to M.
\]

\pagebreak[3]
As~in the proof of \cite[Proposition 3.2.1]{li}, it follows that the derived functor 
\begin{equation}\label{Lphi^*}
\LL\varphi^*(-)\set S\Otimes{\varphi} -
\end{equation}
from $\D(R)$ to $\D(S)$
is left-adjoint to $\varphi_*\>$, the adjunction isomorphism\va2
\begin{equation}\label{AdjIso}
\alpha\colon \Hom_{\D(S)}^{\mathstrut}(\LL\varphi^*\<\SG,\SEE)\iso\Hom_{\D(R)}(\SG,\varphi_*\SEE)
\end{equation}
being the unique bifunctorial map making the following natural diagram,
 
\pagebreak[3]\noindent
where $\K(T)$ is the category of homotopy classes of maps of \mbox{$T$-complexes,} commute (see \cite[Corollary 3.2.2]{li}):
\[
\def\1{$\Hom_{\K(S)}(\varphi^*\<\SG,\SEE)$}
\def\2{$\Hom_{\D(S)}(\varphi^*\<\SG,\SEE)$}
\def\3{$\Hom_{\D(S)}(\LL\varphi^*\<\SG,\SEE)$}
\def\4{$\Hom_{\K(R)}(\SG,\varphi_*\SEE)$}
\def\5{$\Hom_{\D(R)}(\SG,\varphi_*\SEE)$}
  \bpic[xscale=4, yscale=1.75]

   \node(11) at (1,-1){\1};   
   \node(12) at (1.98,-1){\2}; 
   \node(13) at (3,-1){\3};
   
   \node(21) at (1,-2){\4};
   \node(23) at (3,-2){\5};   
  
    \draw[->] (11)--(12) ;  
    \draw[->] (12)--(13) ;  
    
    \draw[->] (21)--(23) ;
  
    \draw[->] (11)--(21);
    \draw[->] (13)--(23) node[right=.5pt, midway, scale=.75] {$\alpha^{}$}
                                   node[left=1pt, midway, scale=.75] {$\simeq$} ;
 \epic
\]
In this adjunction, the unit at $\SG$ is the natural $\D(R)$-composition \va{-1}
\[
\bar\eta_\varphi(\SG)\colon \SG\iso \SG'\xto{\!\via\eta\>\>}\varphi_*\varphi^*\SG'\iso\varphi_*\LL\varphi^*\SG
\]
for any K-flat resolution $\SG'\to\SG$
(i.e., $R$-quasi-isomorphism\va2 with $\SG'$ K-flat);
\noindent and the counit at 
$\SEE$ is the natural $\D(S)$-composition\va{-3}
\[
\bar\epsilon_\varphi(\SEE)\colon\LL \varphi^* \varphi_*\SEE\lto \varphi^*\varphi_*\SEE\xto{\<\<\epsilon_\varphi(\SEE)\>\>}\SEE.
\]

\va2

The functorial $\D(R\>')$-map 
\[
\theta_{\hat\sigma}(\SEE)\colon\LL\mu^*\<\varphi_*\SEE
\lto\xi_*\LL\nu^*\SEE\qquad(\SEE\in\D(S))
\]
is defined (abstractly) to be the adjoint of the  $\D(R)$-map
\begin{equation}\label{adjtheta}
\varphi_*\bar\eta_\nu(\SEE)\colon\varphi_*\SEE\lto
\varphi_*\nu_*\LL\nu^*\SEE=\mu_*\xi_*\LL\nu^*\SEE\>.
\end{equation} 

In terms that are explicit---up to choices of a K-flat $S$-resolution \mbox{$\SEE'\to\SEE$} and a K-flat 
$R$-resolution $\SG\to\varphi_*\SEE'$---one can 
describe $\theta_{\hat\sigma}(\SEE)$ as being
the natural composite $\D(R\>')$-map
\begin{align*}
 \LL\mu^*\<\varphi_*\SEE
&\,=\!\!=\,R\>'\>\Otimes{\mu}\>\>\varphi_*\SEE
\iso R\>'\>\Otimes{\mu }\>\>\varphi_*\SEE'\\
&\iso R\>'\>\otimes_{\mu }\>\>\SG
\lto R\>'\>\otimes_\mu \>\>\varphi_*\SEE'
\lto\xi_*(S'\>\otimes_\nu\>\SEE')\\
&\iso\xi_*\LL\nu^*\SEE\>.
\end{align*}
Indeed, it suffices to consider the case $\SEE'=\SEE$, where, with $\SG\to\varphi_*\SEE$ 
the preceding K-flat resolution,
one  finds readily that this assertion\pagebreak[3] results from  
the commutativity\- (straightforward to verify)  of the natural $\D(R)$-diagram
\[\mkern-3mu
\def\1{$\varphi_*\SEE$}
\def\2{$\varphi_*(\nu_*S'\otimes_S\SEE)$}
\def\3{$\varphi_*\nu_*(S'\otimes_\nu\SEE)$}
\def\4{$\varphi_*\nu_*\nu^*\SEE$}
\def\5{$\mu_*\LL\mu^*\varphi_*\SEE$}
\def\6{$\SG$}
\def\7{$\mu_*\mu^*\SG$}
\def\8{$\mu_*R\>'\otimes_R\SG$}
\def\9{$\mu_*(R\>'\otimes_\mu\SG)$}
\def\ten{$\mu_*(R\>'\Otimes{\mu}\varphi_*\SEE)$}
\def\lvn{$\mu_*(R\>'\otimes_\mu\varphi_*\SEE)$}
\def\twv{$\mu_*\xi_*(S'\otimes_\nu\SEE)$}
\def\thn{$\mu_*\xi_*\nu^*\SEE$}
\def\frn{$\mu_*R\>'\otimes_R\varphi_*\SEE$}
  \bpic[xscale=3.53, yscale=1.45]

   \node(11) at (1,-1){\1};   
   \node(12) at (2.1,-1){\2}; 
   \node(13) at (3.13,-1){\3};
   \node(14) at (4,-1){\4};
   
   \node(21) at (1,-1.95){\5};   
   \node(22) at (1.53,-1.95){\6};   
   \node(23) at (2.65,-2){\frn};   
   
   \node(32) at (2.1,-2.92){\8}; 
   
   \node(41) at (1.53,-2.92){\7};   
   \node(42) at (1,-5){\9}; 
    
   \node(51) at (1,- 3.3){\ten};
   \node(52) at (2.12,-5){\lvn};
   \node(53) at (3.13,-5){\twv};   
   \node(54) at (4,-5){\thn};
  
    \draw[->] (11)--(12) ;  
    \draw[double distance=2] (12)--(13) ;
    \draw[double distance=2] (13)--(14) ;
    
    \draw[->] (52)--(53) ;
    \draw[double distance=2] (53)--(54) ;
    
    \draw[->] (11)--(21) node[left,midway,scale=.75]
                                 {$\bar\eta_\mu$};
    \draw[double distance=2] (21)--(51) ;
    
    \draw[->] (22)--(41) node[right,midway,scale=.75]
                                 {$\eta_\mu$};   
    
    \draw[->] (22)--(32) ;
    \draw[double distance=2] (32)--(42) ;
    \draw[->] (42)--(52) ;

    \draw[double distance=2] (13)--(53) ;
    \draw[double distance=2] (14)--(54) ;
    
    \draw[->] (11)--(22)  node[above=-5pt, midway, scale=.9] 
                                 {\rotatebox{-36}{$\mkern10mu\Iso$}};
    \draw[->] (11)--(23) ;
    \draw[->] (21)--(41) node[above=-5pt, midway, scale=.9] 
                                 {\rotatebox{-36}{$\mkern10mu\Iso$}} ;
    \draw[double distance=2] (23)--(52); 
    \draw[->] (32)--(23) ;
    \draw[double distance=2] (41)--(42); 
    \draw[<-] (42)--(51) node[left=1, midway, scale=.75] {$\simeq$} ; 
    
 \epic
\]


\begin{subcosa}\label{local theta}
It will be shown next that for any\/ $E\in\Dqc(X)$ and $\SEE\set\R\Gamma^{}_{\!\!X}E$ (so that $E\cong\SH{S}\<\SEE$, see  end of \S\ref{affine aspect}),
\emph{the sheafification\/ $\SH{R\>'}\<\theta_{\hat\sigma}(\SEE)$ is naturally isomorphic to the map
\(
\theta_{\<\sigma}(E)\colon \LL u^*\R\fst\>E\to \R g_*\LL v^*\<\<E
\)
in}\va1 \eqref{theta}.  

This follows from \cite[Example 3.10.1(a)]{li}, but here a somewhat different argument
will be given, whose motivation is
that sheafification preserves adjointness of maps
and that the sheafification of the adjoint \eqref{adjtheta} of~$\theta_{\hat\sigma}(\SEE)$ is naturally isomorphic to the natural composite map
\[
\R\fst E\lto \R\fst\R v_*\LL v^*E\iso \R u_*\R g_*\LL v^*E,
\]
which is, by definition, the adjoint of $\theta_{\<\sigma}\>$.

The precise formulation---Proposition ~\ref{conctheta} and its proof---needs some
preliminaries.\va2

\pagebreak[3]
The standard functorial $\OX$-isomorphism, for $R$-complexes~$\SG$, 
\begin{equation}\label{sheafify^*0}
\varsigma(\SG)=\varsigma_\varphi(\SG)\colon f^*\<\SH{\<R}\SG\iso \SH{S}\<\varphi^*\SG,
\end{equation}
see \cite[1.7.7(i)]{EGA1}, is adjoint to the natural composite map\va{-2}
\[
\SH{R}\SG\lto\SH{R}\varphi_*\varphi^*\SG\xto[\!\upsilon(\varphi^*\SG)\!]{\Iso}\fst\>\SH{\<S}\<\varphi^*\SG
\]
\vskip-3pt
\noindent with $\upsilon\colon\SH{\<R}\varphi_*\iso \fst\>\SH{S}$ the functorial isomorphism in ~\ref{factor'}.
In other words,  the following natural functorial  diagram commutes:
\begin{equation}\label{adjconj}
\def\1{$\SH{\<R}$}
\def\2{$\fst f^*\<\SH{\<R}$}
\def\3{$\SH{\<R}\varphi_*\varphi^*$}
\def\4{$\fst\>\SH{\<S}\<\varphi^*$}
 \CD
  \bpic[xscale=2.3, yscale=1.3]

   \node(11) at (1,-1){\1}; 
   \node(12) at (2,-1){\2}; 
     
   \node(21) at (1,-2){\3};
   \node(22) at (2,-2){\4};
   
    \draw[->] (11)--(12) ;  
    \draw[->] (21)--(22) node[below=1pt, midway, scale=.75] {$\upsilon$} 
                                   node[above, midway, scale=.75] {$\Iso$} ;   
       
    \draw[->] (11)--(21) ;
    \draw[->] (12)--(22) node[right=1pt, midway, scale=.75] {$\fst\varsigma$} 
                                    node[left=1pt, midway, scale=.75] {$\simeq$} ;
   \epic 
 \endCD
\end{equation}
There results commutativity of the natural diagram
\[
\def\1{$f^*\<\SH{\<R}\<\varphi_*$}
\def\2{$f^*\mkern-2.5mu\fst f^*\<\SH{\<R}\<\varphi_*$}
\def\4{$f^*\<\SH{\<R}\varphi_*\varphi^*\<\varphi_*$}
\def\5{$f^*\mkern-2.5mu\fst\>\SH{\<S}\<\varphi^*\<\varphi_*$}
\def\6{$\SH{\<S}\<\varphi^*\<\varphi_*$}
\def\8{$f^*\!\fst\>\SH{\<S}$}
\def\9{$\SH{\<S}$}
  \bpic[xscale=4.3, yscale=1.55]

   \node(11) at (1,-1){\1}; 
   \node(12) at (2.03,-1.5){\2};  
   \node(13) at (3,-1){\1}; 
     
   \node(21) at (1.32,-2){\4};
   \node(22) at (2.03,-2.27){\5};
   \node(23) at (3,-2.27){\6};
   
   \node(31) at (1,-3){\1};
   \node(32) at (2.03,-3){\8};
   \node(33) at (3,-3){\9};

  
   \draw[double distance=2pt] (11)--(13) ;
   \draw[->] (11)--(12) ;  
   \draw[->] (12)--(13)  ;
   
   \draw[->] (21)--(22) node[below, midway, scale=.75] {$f^*\upsilon$}
                                   node[above=-.7pt, midway, scale=.75] {\rotatebox{-7}{$\Iso$}} ;   
   \draw[->] (22)--(23)  ;
   
   \draw[->] (31)--(32) node[below, midway, scale=.75] {$f^*\upsilon$}
                                  node[above, midway, scale=.75] {$\Iso$} ;   
   \draw[->] (32)--(33)  ;

    \draw[double distance=2pt] (11)--(31) ;
    \draw[->] (11)--(21) ;
    \draw[->] (21)--(31) ;
    
    \draw[->] (12)--(22) node[right=1pt, midway, scale=.75] {$f^*\mkern-2.5mu\fst\varsigma$} 
                                    node[left=1pt, midway, scale=.75] {$\simeq$} ;
    \draw[->] (22)--(32) ;
    
    \draw[->] (13)--(23) node[right=1pt, midway, scale=.75] {$\varsigma$} 
                                   node[left=1pt, midway, scale=.75] {$\simeq$} ;
    \draw[->] (23)--(33) ;

  \epic 
\]
\vskip-3pt
\noindent whose border can be represented as the (commutative) natural diagram
\begin{equation}\label{adjconj'}
\def\1{$\SH{\<S}$}
\def\2{$f^*\!\fst\>\SH{\<S}$}
\def\3{$\SH{\<S}\<\varphi^*\<\varphi_*$}
\def\4{$f^*\<\SH{\<R}\<\varphi_*$}
 \CD
  \bpic[xscale=3, yscale=1.3]

   \node(11) at (1,-1){\1}; 
   \node(12) at (2,-1){\2}; 
     
   \node(21) at (1,-2){\3};
   \node(22) at (2,-2){\4};
   
    \draw[<-] (11)--(12) ;  
    \draw[<-] (21)--(22) node[below=1pt, midway, scale=.75] {$\varsigma$} 
                                   node[above, midway, scale=.75] {$\Iso$} ;   
       
    \draw[<-] (11)--(21) ;
    \draw[<-] (12)--(22) node[right=1pt, midway, scale=.75] {$f^*\upsilon$} 
                                    node[left=1pt, midway, scale=.75] {$\simeq$} ;
   \epic 
 \endCD
\end{equation}

As for derived versions of the foregoing, 
upon replacing $\SG$ by a quasi-isomorphic direct limit of bounded-above flat $R$-complexes (see the remarks preceding \eqref{sheafify Otimes}), one gets 
a natural functorial $\D(X)$-isomorphism
\begin{equation}\label{sheafify^*}
\bar\varsigma(\SG)\colon\LL f^*\<\SH{\<R}\<\SG\iso \SH{S}\<\LL\varphi^*\SG
\qquad(\SG\in\D(R))
\end{equation}
such that the following natural functorial diagram commutes
\begin{equation*}\label{Adjconj}
\def\1{$\SH{\<R}$}
\def\2{$\R\fst \LL f^*\<\SH{\<R}$}
\def\3{$\SH{\<R}\varphi_*\LL\varphi^*$}
\def\4{$\R\fst\>\SH{\<S}\<\LL\varphi^*,$}
 \CD
  \bpic[xscale=3, yscale=1.25]

   \node(11) at (1,-1){\1}; 
   \node(12) at (2,-1){\2}; 
     
   \node(21) at (1,-2){\3};
   \node(22) at (2,-2){\4};
   
    \draw[->] (11)--(12) ;  
    \draw[->] (21)--(22) node[below=1pt, midway, scale=.75] {\eqref{sheafify_*}} 
                                   node[above, midway, scale=.75] {$\Iso$} ;   
       
    \draw[->] (11)--(21) ;
    \draw[->] (12)--(22) node[right=1pt, midway, scale=.75] {$\R\fst\bar\varsigma$} 
                                    node[left=1pt, midway, scale=.75] {$\simeq$} ;
   \epic 
 \endCD
 \tag*{(\ref{adjconj})$'$}
\end{equation*}
from which one deduces, as above, commutativity of the natural diagram
\begin{equation*}\label{Adjconj'}
\def\1{$\SH{\<S}$}
\def\2{$\LL f^*\R\fst\>\SH{\<S}$}
\def\3{$\SH{\<S}\<\LL\varphi^*\<\varphi_*$}
\def\4{$\LL f^*\<\SH{\<R}\<\varphi_*$}
 \CD
  \bpic[xscale=3, yscale=1.25]

   \node(11) at (1,-1){\1}; 
   \node(12) at (2,-1){\2}; 
     
   \node(21) at (1,-2){\3};
   \node(22) at (2,-2){\4};
   
    \draw[<-] (11)--(12) ;  
    \draw[<-] (21)--(22) node[below=.5pt, midway, scale=.75] {$\bar\varsigma$} 
                                   node[above, midway, scale=.75] {$\Iso$} ;   
       
    \draw[<-] (11)--(21) ;
    \draw[<-] (12)--(22) node[right=1pt, midway, scale=.75] {\eqref{sheafify_*}} 
                                    node[left=1pt, midway, scale=.75] {$\simeq$} ;
   \epic 
 \endCD
 \tag*{(\ref{adjconj'})$'$}
\end{equation*}

\medbreak

\begin{subprop}\label{conctheta}
The following diagram, in which\/ $\SEE$ is any\/ $S$-complex, commutes.\looseness=1
\[
\def\1{$\SH{\<R\>'}\<\LL\mu^*\varphi_*\SEE$}
\def\2{$\SH{\<R\>'}\<\xi_*\>\LL\nu^*\SEE$}
\def\3{$\LL u^*\SH{\<R}\varphi_*\SEE$}
\def\4{$\R g_*\>\SH{\<S'}\<\LL\nu^*\SEE$}
\def\5{$\LL u^*\R\fst\>\SH{\<S}\<\SEE$}
\def\6{$\R g_*\LL v^*\SH{\<S}\<\SEE$}
  \bpic[xscale=3.75, yscale=1.6]

   \node(11) at (1,-1){\1};   
   \node(12) at (2,-1){\2}; 
      
   \node(21) at (1,-2){\3};
   \node(22) at (2,-2){\4};
   
   \node(31) at (1,-3){\5};
   \node(32) at (2,-3){\6};   
  
    \draw[->] (11)--(12) node[above=1pt, midway, scale=.75] {$\SH{R\>'}\<\theta_{\hat\sigma}(\SEE)$} ;  
       
    \draw[->] (31)--(32) node[below=1pt, midway, scale=.75] 
      {$\theta_{\<\sigma}(\SH{S}\SEE)$} ;
    
    \draw[->] (11)--(21) node[right=1pt, midway, scale=.75] {$\simeq$}
                                   node[left=1pt, midway, scale=.75]{\eqref{sheafify^*}} ;
    \draw[->] (21)--(31) node[right=1pt, midway, scale=.75] {$\simeq$}
                                   node[left=1pt, midway, scale=.75]{\eqref{sheafify_*}} ; 
    
    \draw[->] (12)--(22) node[left=1pt, midway, scale=.75] {$\simeq$}
                                   node[right=1pt, midway, scale=.75]{\eqref{sheafify_*}} ;
    \draw[->] (22)--(32) node[left=1pt, midway, scale=.75] {$\simeq$}
                                   node[right=1pt, midway, scale=.75]{\eqref{sheafify^*}} ;
  \epic
\]

\end{subprop}

\begin{proof} Keeping in mind the adjoint \eqref{adjtheta} of 
$\theta_{\hat\sigma}$, expand the diagram naturally, without ``$\,\SEE$," thus:
\[\mkern-3mu
\def\1{$\SH{\<R\>'}\<\LL\mu^*\varphi_*$}
\def\2{$\SH{\<R\>'}\<\xi_*\>\LL\nu^*$}
\def\3{$\LL u^*\SH{\<R}\varphi_*$}
\def\4{$\R g_*\>\SH{\<S'}\<\LL\nu^*$}
\def\5{$\LL u^*\R\fst\>\SH{\<S}$}
\def\6{$\R g_*\LL v^*\SH{\<S}$}
\def\7{$\SH{\<R\>'}\<\<\LL\mu^*\<\varphi_*\nu_*\LL\nu^*$}
\def\8{$\SH{\<R\>'}\<\LL\mu^*\<\<\mu_*\xi_*\>\LL\nu^*$}
\def\9{$\LL u^*\<\SH{\<R}\varphi_*\nu_*\LL\nu^*$}
\def\ten{$\LL u^*\SH{\<R}\mu_*\xi_*\>\LL\nu^*\quad$}
\def\lvn{$\LL u^*\R\fst\>\SH{\<S}\nu_*\LL\nu^*$}
\def\twv{$\LL u^*\R u_*\>\SH{\<R\>'}\<\xi_*\>\LL\nu^*$}
\def\thn{$\LL u^*\R\fst \R v_*\>\SH{\<S'}\<\<\LL\nu^*$}
\def\frn{$\LL u^*\R u_*\R g_*\>\SH{\<S'}\<\<\LL\nu^*$}
\def\ffn{$\LL u^*\R\fst \R v_*\LL v^*\<\SH{\<S}$}
\def\sxn{$\LL u^*\R u_*\R g_*\LL v^*\<\SH{\<S}$}
  \bpic[xscale=3.4, yscale=1.6]

   \node(11) at (1,-1.5){\1};   
   \node(14) at (4,-1.5){\2}; 
   
   \node(22) at (1.95,-1.5){\7};
   \node(23) at (3.05,-1.5){\8};

   \node(31) at (1,-3.5){\3};
   \node(32) at (1.95,-2.5){\9};
   \node(33) at (3.05,-2.5){\ten};
   \node(34) at (4,-3.5){\4};
   
   \node(41) at (1,-4.5){\5};
   \node(42) at (1.95,-3.5){\lvn};
   \node(43) at (3.05,-3.5){\twv};

   \node(52) at (1.95,-4.5){\thn};
   \node(53) at (3.05,-4.5){\frn};
   
   \node(62) at (1.95,-5.5){\ffn};
   \node(63) at (3.05,-5.5){\sxn};

   \node(74) at (4,-4.5){\6};   
  
    \draw[double distance=2pt] (22)--(23) ;
   
    \draw[double distance=2pt] (32)--(33) ;
        
    \draw[->] (52)--(53) node[above, midway, scale=.75] {$\Iso$} ;
    
    \draw[->] (62)--(63) node[above, midway, scale=.75] {$\Iso$} ;
       
    \draw[->] (11)--(31) node[right=1pt, midway, scale=.75] {$\simeq$}
                                   node[left, midway, scale=.75]{\eqref{sheafify^*}} ;
    \draw[->] (31)--(41) node[right=1pt, midway, scale=.75] {$\simeq$}
                                   node[left=1pt, midway, scale=.75]{\eqref{sheafify_*}} ; 
                                   
     \draw[->] (22)--(32) node[left=1pt, midway, scale=.75] {$\simeq$}
                                   node[right, midway, scale=.75]{\eqref{sheafify^*}} ;                              
     \draw[->] (23)--(33)  node[right=1pt, midway, scale=.75] {$\simeq$}
                                   node[left, midway, scale=.75]{\eqref{sheafify^*}} ; 
     
     \draw[->] (32)--(42) node[right, midway, scale=.75]{\eqref{sheafify_*}};                              
     \draw[->] (33)--(43) node[right=.5, midway, scale=.75] {$\simeq$}
                                    node[left, midway, scale=.75]{\eqref{sheafify_*}};
     
     \draw[->] (42)--(52) node[right, midway, scale=.75]{\eqref{sheafify_*}};                              
     \draw[->] (43)--(53) node[right=1pt, midway, scale=.75] {$\simeq$}
                                    node[left, midway, scale=.75]{\eqref{sheafify_*}} ;
     
     \draw[->] (52)--(62) node[left=1pt, midway, scale=.75] {$\simeq$}
                                   node[right, midway, scale=.75]{\eqref{sheafify^*}} ;                              
     \draw[->] (53)--(63)  node[right=1pt, midway, scale=.75] {$\simeq$}
                                   node[left, midway, scale=.75]{\eqref{sheafify^*}} ;
   
    \draw[->] (14)--(34) node[left=1pt, midway, scale=.75] {$\simeq$}
                                   node[right, midway, scale=.75]{\eqref{sheafify_*}} ;
    \draw[->] (34)--(74) node[left=1pt, midway, scale=.75] {$\simeq$}
                                   node[right, midway, scale=.75]{\eqref{sheafify^*}} ;
                                   
     \draw[->] (11)--(22) ;                              
     \draw[->] (23)--(14) ;
     
     \draw[->] (31)--(32) ;                              
     \draw[->] (43)--(14) ;
     
     \draw[->] (41)--(42) ; 
     
     \draw[->] (53)--(34) ; 
     
     \draw[->] (41)--(62) ;                              
     \draw[->] (63)--(74) ;
     
    \node at (3.427,-2.02)[scale=.9]{\circled1} ;
    \node at (2.48,-3.52)[scale=.9]{\circled2} ;   
    \node at (1.375,-4.52)[scale=.9]{\circled3} ;    
                         
 \epic
\]

Commutativity of the unlabeled subdiagrams is clear.

Commutativity of~\circled1 results from that of diagram~\ref{Adjconj'}, 
with $(\mu,u)$ in place of $(\varphi,f)$.ƒ
   
Commutativity of~\circled3 results from that of diagram~\ref{Adjconj}, 
with $(\nu,v)$ in place of $(\varphi,f)$.

Commutativity of~\circled2 results from that of the two diagrams obtained from~\eqref{pfgamma} by making the respective substitutions $(\xi,g)\mapsto(\nu,v)$ and $(\varphi,f)\mapsto(\mu,u)$.
\end{proof}

\begin{subcosa}
Now for the commutative\kf-algebra version of the map $\beta_\sigma$ in~\ref{indt base change}.\va1

Assume that the natural $\D(R)$ composite 
\(
R\>'\Otimes{\<\<R}S\to R\>'\otimes_R S\to  S'
\)
is an isomorphism,
i.e., $\textup{Tor}_i^R(R\>',S)=0$ for all $i>0$ and the natural map is an isomorphism $R\>'\otimes_R S\iso S'$; 
equivalently, assume  $\sigma$ to be a tor-independent fiber square (see \cite[(3.10.2)(ii)$'$]{li}). 
Thus we can, and do, identify the maps $\xi$ and $\nu$ with the respective canonical maps $R\>'\to R\>'\otimes_R\<S\>$ and
$S\to R\>'\otimes_R\< S$.
Then $\sigma$ is an independent square (see the remarks following
\eqref{theta}), so for any $S$-complex~$\SEE$ the above map 
$\theta_{\<\sigma}(\SH{S}\SEE)$ is an isomorphism,  
whence, by~\ref{conctheta} and \eqref{qcequiv},  so is $\theta_{\hat\sigma}(\SEE)$.

As in~\ref{indt base change}---but more generally (see~\ref{generalized flat}), the 
$\D(X')$-map $\beta_\sigma(G)$ is defined for all $G\in\D(Y)$ to be adjoint to 
the natural composite map
\[ 
\R g_*\LL v^*\<\< f^\flat G\xto[\>\>\lift1.35,\theta_{\<\sigma}^{-\<1},\>]{} \LL u^*\R\fst f^\flat G 
\lto \LL u^*\mkern-.5mu G.
\]

Let\/ $\beta_{\hat\sigma}(\SG)\ (\SG\in\D(R))$ be the $\D(S')$-map 
adjoint to the natural composite
\[ 
\xi_*\LL \nu^*\<\< \ush\varphi \SG\xto[\>\>\lift1.35,\theta_{\<\hat\sigma}^{-\<1},\>]{} 
\LL \mu^*\varphi_* \ush\varphi \SG \lto \LL \mu^*\SG,
\]
that is, $\beta_{\hat\sigma}(\SG)$ is the natural composite map
\[
\LL\nu^*\<\<\ush\varphi \SG 
\lto 
\ush\xi\xi_*\>\LL\nu^*\<\<\ush\varphi \SG
\xto{\ush\xi\theta^{-1}_{\hat\sigma}\<\<} 
\ush\xi\LL\mu^*\varphi_*\ush\varphi \SG
\lto
\ush\xi\LL\mu^*\SG,
\]
a map that is concrete to the extent that up to taking K-flat and K-injective resolutions,  
it has previously been explicitly described. (See the lines just after~\eqref{*sh} and~\eqref{adjtheta}.) \va2

\emph{For\/ $G\in\Dqc(Y),$ $\beta_\sigma(G)$ is isomorphic to the sheafification of 
$\beta_{\hat\sigma}(\R\Gamma^{}_{\!\!Y} G)$.} 
That follows from the next proposition, with $\SG\set\R\Gamma^{}_{\!\!Y} G$.
\end{subcosa}

Application of $\SH{S}\mkern-.5mu\ush{\varphi}$ to the unit isomorphism \eqref{unitmap} produces a natural functorial isomorphism
\begin{equation}\label{affdual}
\vartheta_{\<\!f}(\SG)\colon \SH{S}\mkern-.5mu\ush{\varphi}\SG\iso 
\SH{S}\mkern-.5mu\ush{\varphi}\R\Gamma^{}_{\<\<\!Y}\SH{\<R}\SG=
f^\flat\SH{\<R}\SG
\qquad(\SG\in\D(R)).
\end{equation}

\begin{subprop}\label{concrete bc2}
For any\/ $\SG\in\D(R),$ the following diagram 
commutes.
\[
\def\1{$\SH{\<S'}\<\LL\nu^*\<\ush\varphi \SG$}
\def\2{$\SH{\<S'}\<\ush\xi\LL\mu^*\SG$}
\def\3{$\LL v^*\SH{\<S}\ush\varphi \SG$}
\def\4{$g^\flat\SH{\<R\>'}\<\LL\mu^*\SG$}
\def\5{$\LL v^*\<\<f^\flat\SH{\<R} \SG$}
\def\6{$g^\flat\LL u^*\<\SH{\<R}\SG$}
  \bpic[xscale=3.75, yscale=1.6]

   \node(11) at (1,-1){\1};   
   \node(12) at (2,-1){\2}; 
      
   \node(21) at (1,-2){\3};
   \node(22) at (2,-2){\4};
   
   \node(31) at (1,-3){\5};
   \node(32) at (2,-3){\6};   
  
    \draw[->] (11)--(12) node[above=1pt, midway, scale=.75] {$\SH{S'}\<\<\beta_{\hat\sigma}(\SG)$} ;  
       
    \draw[->] (31)--(32) node[below=1pt, midway, scale=.75] {$\beta_\sigma(\SH{\<R}\SG)$} ;
    
    \draw[->] (11)--(21) node[right=1pt, midway, scale=.75]{$\simeq$}
                                   node[left=1pt, midway, scale=.75]{\eqref{sheafify^*}} ;
    \draw[->] (21)--(31) node[right=1pt, midway, scale=.75]{$\simeq$}
                                   node[left=1pt, midway, scale=.75]{$\LL v^*\vartheta_{\<\!f}$} ;

    \draw[->] (12)--(22) node[left=1pt, midway, scale=.75]{$\simeq$}
                                   node[right=1pt, midway, scale=.75]{$\vartheta_{\!g}$} ;
    \draw[->] (22)--(32) node[left=1pt, midway, scale=.75] {$\simeq$}
                                   node[right=1pt, midway, scale=.75]{\eqref{sheafify^*}} ;
  \epic
\]

\end{subprop}


\begin{proof}
Expand the diagram naturally, without ``$\SG$," thus:
 \[\mkern-2mu
\def\1{$\SH{\<S'}\<\LL\nu^*\<\ush\varphi $}
\def\2{$\SH{\<S'}\<\ush\xi\LL\mu^*$}
\def\3{$\LL v^*\SH{\<S}\<\ush\varphi $}
\def\4{$g^\flat\SH{\<R\>'}\<\LL\mu^*$}
\def\5{$\LL v^*\!f^\flat\SH{\<R}$}
\def\6{$g^\flat\LL u^*\<\SH{\<R}$}
\def\7{$\SH{\<S'}\<\ush\xi\xi_*\LL\nu^*\<\ush\varphi$}
\def\8{$\SH{\<S'}\<\<\ush\xi\LL\mu^*\<\varphi_*\ush\varphi$}
\def\9{$g^\flat\<\SH{\<R\>'}\<\<\xi_*\LL\nu^*\<\ush\varphi$}
\def\ten{$g^\flat\<\SH{\<R\>'}\<\<\LL\mu^*\<\varphi_*\ush\varphi$}
\def\lvn{$g^\flat\R g_*\>\SH{\<S'}\<\<\LL\nu^*\<\ush\varphi$}
\def\twv{$g^\flat\LL u^*\<\SH{\<R}\<\varphi_*\ush\varphi$}
\def\thn{$g^\flat\R g_*\LL v^*\<\SH{\<S}\<\<\ush\varphi$}
\def\frn{$g^\flat\LL u^*\R\fst\>\SH{\<S}\<\ush\varphi$}
\def\ffn{$g^\flat\R g_*\LL v^*\!f^\flat\<\SH{\<R}$}
\def\sxn{$g^\flat\LL u^*\R\fst f^\flat\<\SH{\<R}$}
  \bpic[xscale=3.4, yscale=1.7]

   \node(11) at (1,-1.5){\1};   
   \node(14) at (3.95,-1.5){\2}; 
   
   \node(22) at (1.9,-1.5){\7};
   \node(23) at (3,-1.5){\8};

   \node(41) at (1,-3.75){\3};
   \node(32) at (1.9,-2.5){\9};
   \node(33) at (3,-2.5){\ten};
   \node(34) at (3.95,-3.75){\4};
   
   \node(42) at (1.9,-3.75){\lvn};
   \node(43) at (3,-3.75){\twv};

   \node(52) at (1.9,-5){\thn};
   \node(53) at (3,-5){\frn};
   
   \node(62) at (1.9,-6){\ffn};
   \node(63) at (3,-6){\sxn};

   \node(71) at (1,-6){\5};
   \node(74) at (3.95,-6){\6};   
  
     \draw[->] (22)--(23) node[above, midway, scale=.75] {$\via\theta_{\!\hat\sigma}^{-\<1}$} ;
   
    \draw[->] (32)--(33) node[below, midway, scale=.75] {$\via\theta_{\!\hat\sigma}^{-\<1}$} ;
        
    \draw[->] (52)--(53) node[above, midway, scale=.75] {$\via\theta_{\!\sigma}^{-\<1}$} ;
    
    \draw[->] (62)--(63) node[below, midway, scale=.75] {$\via\theta_{\!\sigma}^{-\<1}$} ;
       
    \draw[->] (11)--(41) node[right=1pt, midway, scale=.75] {$\simeq$}
                                   node[left, midway, scale=.75]{\eqref{sheafify^*}} ;
    \draw[->] (41)--(71) node[right=1pt, midway, scale=.75] {$\simeq$}
                                   node[left=.5pt, midway, scale=.75]{$\via\vartheta_{\<\!f}$} ;
                                   
     \draw[->] (22)--(32) node[left=1pt, midway, scale=.75] {$\simeq$}
                                   node[right, midway, scale=.75]{$\vartheta_{\!g}$} ;
     \draw[->] (23)--(33)  node[right=1pt, midway, scale=.75] {$\simeq$}
                                   node[left, midway, scale=.75]{$\vartheta_{\!g}$} ;
     
     \draw[->] (32)--(42) node[left=1pt, midway, scale=.75] {$\simeq$}
                                    node[right, midway, scale=.75]{\eqref{sheafify_*}};                              
     \draw[->] (33)--(43) node[right=1pt, midway, scale=.75] {$\simeq$}
                                   node[left, midway, scale=.75]{\eqref{sheafify^*}} ;
     
     \draw[->] (42)--(52) node[left=1pt, midway, scale=.75] {$\simeq$}
                                   node[right, midway, scale=.75]{\eqref{sheafify^*}} ;                              
     \draw[->] (43)--(53) node[right=1pt, midway, scale=.75] {$\simeq$}
                                    node[left, midway, scale=.75]{\eqref{sheafify_*}} ;
     
     \draw[->] (52)--(62) node[left=1pt, midway, scale=.75] {$\simeq$}
                                   node[right, midway, scale=.75]{$\via\vartheta_{\<\!f}$} ;
     \draw[->] (53)--(63)  node[right=1pt, midway, scale=.75] {$\simeq$}
                                   node[left, midway, scale=.75]{$\via\vartheta_{\<\!f}$} ;
   
    \draw[->] (14)--(34) node[left=1pt, midway, scale=.75] {$\simeq$}
                                   node[right, midway, scale=.75]{$\vartheta_{\!g}$} ;
    \draw[->] (34)--(74) node[left=1pt, midway, scale=.75] {$\simeq$}
                                   node[right, midway, scale=.75]{\eqref{sheafify^*}} ;
                                   
     \draw[->] (11)--(22) ;                              
     \draw[->] (11)--(42) ;   
          
     \draw[->] (23)--(14) ;
     
     \draw[->] (33)--(34) ;
     
     \draw[->] (41)--(52) ; 
     \draw[->] (43)--(74) ;
      
      \draw[->] (63)--(74) ;
  
     \draw[->] (71)--(62) ;      
  
    \node at (1.52,-2.02)[scale=.9]{\circled1} ;   
    \node at (2.45,-3.78)[scale=.9]{\circled2} ;   
    \node at (3.38,-5.52)[scale=.9]{\circled3} ;
    
 \epic
\]

Here, commutativity of the unlabeled subdiagrams is clear.

Commutativity of \circled2 results from Proposition~\ref{conctheta}.

Commutativity of \circled3, verifiable after dropping ``$g^\flat\LL u^*\>$," 
follows easily from the description of the counit map $\R\fst f^\flat\to\id$ in Proposition~\ref{affine duality}. Similarly,
\[
\def\1{$\SH{\<R\>'}$} 
\def\2{$\SH{\<R\>'}\xi_*\ush\xi$}
\def\4{$\R g_*g^\flat\<\SH{\<R\>'}$}
\def\5{$\R g_*\mkern.5mu\SH{\<S'}\<\<\ush\xi$}
  \bpic[xscale=3, yscale=1.5]

   \node(11) at (1,-1){\1}; 
   \node(12) at (2.,-1){\2};

   \node(21) at (1,-2){\4};
   \node(22) at (2,-2){\5};

    \draw[<-] (11)--(12) ;  
   
   \draw[<-] (21)--(22) node[below, midway, scale=.75]{$\vartheta_{\!g}$}
                                   node[above, midway, scale=.75]{$\Iso$} ;   
     
    \draw[<-] (11)--(21) ;
    \draw[<-] (12)--(22) node[right=1pt, midway, scale=.75] {\eqref{sheafify_*}} 
                                    node[left=1pt, midway, scale=.75] {$\simeq$} ;
      
  \epic 
\]
\vskip-2pt
commutes, whence so does
\[
\def\1{$g^\flat\SH{\<R\>'}\xi_*$} 
\def\2{$g^\flat\SH{\<R\>'}\xi_*\ush\xi\xi_*$}
\def\4{$g^\flat\R g_*g^\flat\<\SH{\<R\>'}\<\<\xi_*$}
\def\5{$g^\flat\R g_*\mkern.5mu\SH{\<S'}\<\<\ush\xi\xi_*$}
\def\6{$g^\flat\R g_*\mkern.5mu\SH{\<S'}$}
\def\8{$\SH{\<S'}\<\<\ush\xi\xi_*$}
\def\9{$\;\SH{\<S'},$}
  \bpic[xscale=4.3, yscale=1.75]

   \node(11) at (1,-1){\1}; 
   \node(12) at (2.03,-1.5){\2};  
   \node(13) at (3,-1){\1}; 
     
   \node(21) at (1.32,-2){\4};
   \node(22) at (2.03,-2.27){\5};
   \node(23) at (3,-2){\6};
   
   \node(31) at (1,-3){\1};
   \node(32) at (2.03,-3){\8};
   \node(33) at (3,-3){\9};

  
   \draw[double distance=2pt] (11)--(13) ;
   \draw[<-] (11)--(12) ;  
   \draw[<-] (12)--(13)  ;
   
   \draw[<-] (21)--(22) node[below, midway, scale=.75] {$\vartheta_{\!g}$}
                            node[above=-.7pt, midway, scale=.75] {\rotatebox{-7}{$\Iso$}} ;   
   \draw[<-] (22)--(23)  ;
   
   \draw[<-] (31)--(32) node[below=1pt, midway, scale=.75] {$\vartheta_{\!g}$}
                                        node[above, midway, scale=.75] {$\Iso$} ;   
   \draw[<-] (32)--(33)  ;

    \draw[double distance=2pt] (11)--(31) ;
    \draw[<-] (11)--(21) ;
    \draw[<-] (21)--(31) ;
    
    \draw[<-] (12)--(22) node[right=1pt, midway, scale=.75] {\eqref{sheafify_*}} 
                                    node[left=1pt, midway, scale=.75] {$\simeq$} ;
    \draw[<-] (22)--(32) ;
    
    \draw[<-] (13)--(23) node[right=1pt, midway, scale=.75] {\eqref{sheafify_*}}
                                   node[left=1pt, midway, scale=.75] {$\simeq$} ;
    \draw[<-] (23)--(33) ;

  \epic 
\]
giving commutativity of \circled1.
\end{proof}

\begin{subrem}
For $\beta_{\hat\sigma}(\SG)$ to be an isomorphism,  further conditions are needed---for example,
that $\varphi_* S$ is a pseudo\kf-coherent  $R$-module (whence \mbox{$S'$~is 
a} pseudo\kf-coherent  $R\>'$-module, see paragraph preceding \ref{qcHom}), and the~$R$-module $\mu_*R\>'$ has finite tor-dimension (whence the map $u$ has finite tor-dimension),
and $\SG\in\Dpl\<(R)$.
\end{subrem}

\begin{subrem} Paste the commutative diagrams in Propositions~\ref{concrete bc2} 
(without~``$\>\SG\>$") and~\ref{concbc} 
(with ``$\>G\>$" replaced by ``$\>\SH{\<R}\<$") along their common edge to get
the following natural diagram, the commutativity of whose border expresses
the relation between the concrete realizations of $\beta_\sigma$ that are indicated in those two places.
Conversely, a more concrete proof---avoiding adjoints---that the border commutes would, together with~\ref{concbc}, give~\ref{concrete bc2}.\va{-2}
\[
\def\1{$\SH{\<S'}\<\LL\nu^*\<\ush\varphi $}
\def\2{$\SH{\<S'}\<\ush\xi\LL\mu^*$}
\def\3{$\LL v^*\SH{\<S}\<\ush\varphi $}
\def\4{$g^\flat\SH{\<S}\<\LL\mu^*$}
\def\5{$\LL v^*\<\<f^\flat\SH{\<R} $}
\def\6{$g^\flat\LL u^*\<\SH{\<R}$}
\def\7{$\LL v^*\<\<\bar f^*\<\pt \SH{\<R}$}
\def\8{$\bar g^*\LL\bar  u^*\<\pt \SH{\<R}$}
\def\9{$\bar g^*\phi'{}^\flat \LL u^* \SH{\<R}$}
  \bpic[xscale=3.75, yscale=1.4]

   \node(11) at (1,-1){\1};   
   \node(12) at (2,-1){\2}; 
      
   \node(21) at (1,-2){\3};
   \node(22) at (2,-2){\4};
   
   \node(31) at (1,-3){\5};
   \node(32) at (2,-3){\6};   
   
   \node(41) at (1,-3.9){\7};
   
   \node(51) at (1,-4.8){\8};
   \node(52) at (2,-4.8){\9};

    \draw[->] (11)--(12) node[above=1pt, midway, scale=.75] {$\SH{S'}\<\<\beta_{\hat\sigma}$} ;  
    
    \draw[->] (31)--(32) node[below=1pt, midway, scale=.75] {$\beta_{\sigma}(\SH{\<R})$} ;
     
    \draw[->] (51)--(52) node[below=1pt, midway, scale=.75] {$\bar g^*\<\bar\beta_{\sigma}$} ;    
    
    \draw[->] (11)--(21) node[right=1pt, midway, scale=.75] {$\simeq$}
                                   node[left=1pt, midway, scale=.75]{\eqref{sheafify^*}} ;
    \draw[->] (21)--(31) node[right=1pt, midway, scale=.75] {$\simeq$}
                                   node[left=1pt, midway, scale=.75]{$\via\vartheta_{\<\!f}$} ; 
    \draw[double distance=2pt] (31)--(41) ;
    \draw[->] (41)--(51)  node[right=1pt, midway, scale=.75]{$\simeq$} ;
    
    \draw[->] (12)--(22) node[left=1pt, midway, scale=.75] {$\simeq$}
                                   node[right=1pt, midway, scale=.75]{$\vartheta_{\!g}$} ;
    \draw[->] (22)--(32) node[left=1pt, midway, scale=.75] {$\simeq$}
                                   node[right=1pt, midway, scale=.75]{$\via\,$\eqref{sheafify^*}} ;
    \draw[double distance=2pt] (32)--(52) ;
  \epic
\]

\vskip-3pt

\end{subrem}
\end{subcosa}
\end{cosa}

\begin{cosa}\label{concrete version chi}
This section~\ref{concrete version chi} is devoted to proving Proposition~\ref{concrete chi}, which gives, for maps of affine schemes, a concrete representation of a generalization of the map $\chi$ 
in~\ref{flat and tensor}.\va2

Again, $\varphi\colon R\to S$ is a homomorphism\va{1.5} of commutative rings, with
corresponding scheme\kf-map\va2
\(
 \spec S =:X\xto{\lift1.15,{\,f\,},} Y\set \spec R.
\)
 
Let $\SF$ and $\SG$ be $R$-complexes, $F\set\SH{\<R}\SF$, $G\set\SH{\<R}\<\SG$.\va{.6}  
There is a natural bifunctorial $\D(S)$-map
\begin{equation}\label{chi0}
\begin{aligned}
\chi^{}_0(\SF,\SG)\colon\ush\varphi\SF\Otimes{\<S}\LL\varphi^*\SG
&=\R\<\<\Hom_\varphi(S,\SF\>)\Otimes{\<S}(S\Otimes{\varphi}\SG)\\
&\lto 
\R\<\<\Hom_\varphi(S,\SF\Otimes{\<\<R}\SG)=\ush\varphi(\SF\Otimes{\<\<R}\SG)
\end{aligned}
\end{equation}
given by \cite[Corollary 2.6.5]{li}, in which, when
$\SF$ is K-injective and $\SG$ is~K-flat (whence $S\otimes_\varphi \SG$ is K-flat over~$S$),
take $\zeta$ to be the natural $\D(S)$-isomorphism\va{-2}
\[
\Hom_\varphi(S,\SF\>)\otimes_S(S\otimes_\varphi\SG)
\iso\R\<\<\Hom_\varphi(S,\SF\>)\Otimes{\<S}(S\Otimes{\varphi}\SG)
\]
and $\beta$ the natural composite $\D(S)$-map\va{-2}
\begin{multline*}
\Hom_\varphi(S,\SF\>)\otimes_S(S\otimes_\varphi\SG)
\iso\Hom_\varphi(S,\SF\>)\otimes_\varphi\SG\\
\lto\Hom_\varphi(S,\SF\otimes_R\SG)
\lto\,\R\<\<\Hom_\varphi(S,\SF\otimes_R\SG)
\iso\R\<\<\Hom_\varphi(S,\SF\Otimes{R}\SG).
\end{multline*}
As in \emph{loc.\,cit.,} $\chi^{}_0(\SF,\SG)$ is the unique bifunctorial map that equals
$\beta\smallcirc\zeta^{-\<1}$
whenever $\SF$ is K-injective and $\SG$ is K-flat; to this extent, $\chi^{}_0$ is concrete.\va1

The next proposition gives that the sheafification of 
$\chi^{}_0(\SF,\SG)$ is naturally isomorphic to the\/ $\Dqc(X)$-map\va{-2}
\[
\tilde\chi(F,G\>)\colon
f^\flat\<\< F\Otimes{\sX}\LL f^*\<G
\lto 
f^\flat(F\Otimes{Y}G\>)
\]
that is adjoint to the natural composite map\va{-1}
\[
\R\fst(f^\flat\<\< F\Otimes{\sX}\LL f^*\<G\>)
\underset{\lift1.1,p^{}_{\!f},}{\iso}\R\fst f^\flat\<\< F\Otimes{Y}G\lto F\Otimes{Y}G\\[-2pt]
\]   
with $p^{}_{\!f}\set p^{}_{\!f}(F,G)$ the projection isomorphism in~\ref{projf}---thereby making
$\chi^{}_0$ a concrete representation of $\tilde\chi$.
If $S$ is perfect as an $R$-module, or if $\SF$ and $\SF\Otimes{\<\<R}\SG$ are in $\Dpl\<(R)$,
then $\tilde\chi$ is the map $\chi$ in~\ref{flat and tensor}.

\begin{subprop}\label{concrete chi}
The following diagram commutes.
\[
\def\1{$\SH{\<S}(\ush\varphi\SF\Otimes{\<S}\LL\varphi^*\SG)$}
\def\2{$\SH{\<S}\ush\varphi(\SF\Otimes{\<\<R}\SG)$}
\def\3{$\SH{\<S}\ush\varphi\SF\Otimes{\sX}\SH{\<S}\LL\varphi^*\SG$}
\def\4{$f^\flat\SH{\<R}(\SF\Otimes{\<\<R}\SG)$}
\def\5{$f^\flat\< F\Otimes{\sX}\LL f^*\<G$}
\def\6{$f^\flat(F\Otimes{\<Y} G)$}
  \bpic[xscale=3.75, yscale=1.4]

   \node(11) at (1,-1){\1};   
   \node(12) at (2,-1){\2}; 
      
   \node(21) at (1,-2){\3};
   \node(22) at (2,-2){\4};
   
   \node(31) at (1,-3){\5};
   \node(32) at (2,-3){\6};   
  
    \draw[->] (11)--(12) node[above=1pt, midway, scale=.75] {$\SH{S}\<\chi^{}_0$} ;  
       
    \draw[->] (31)--(32) node[below=1pt, midway, scale=.75] {$\tilde\chi$} ;
    
    \draw[->] (11)--(21) node[right=1pt, midway, scale=.75] {$\simeq$}
                                   node[left=1pt, midway, scale=.75] {$\eqref{sheafify Otimes}$} ;
    \draw[->] (21)--(31) node[right=1pt, midway, scale=.75] {$\simeq$}
                                   node[left=1pt, midway, scale=.75]
                                            {$\textup{\eqref{affdual}}\Otimes{\sX}
                                            \<\<\textup{\eqref{sheafify^*}}$} ; 
    
    \draw[->] (12)--(22) node[left=1pt, midway, scale=.75] {$\simeq$}
                                   node[right=1pt, midway, scale=.75]{\textup{\eqref{affdual}}} ;
    \draw[->] (22)--(32) node[left=1pt, midway, scale=.75] {$\simeq$}
                                   node[right=1pt, midway, scale=.75] 
                                   {$\eqref{sheafify Otimes}$}  ;
  \epic
\]
\end{subprop}

\begin{proof}  Lemma~\ref{dualchi} below provides an abstract description of~$\chi^{}_0$ resembling that of~$\tilde\chi$. The lemma uses a commutative\kf-algebra analog of $p^{}_{\!f}$, 
\begin{equation}\label{projaff}
p_{\varphi}(\SEE,\SG)\colon \varphi_*\SEE\Otimes{\<\<R}\SG\iso \varphi_*(\SEE\Otimes{\<S}\LL\varphi^*\SG)
\qquad(\SEE\in\D(S),\,\SG\in\D(R)),
\end{equation}
defined as the unique bifunctorial $\D(R)$-isomorphism that is equal, when~$\SG$, 
hence~$\varphi^*\SG$, is K-flat, to the natural composite isomorphism
\[
\varphi_*\SEE\Otimes{\<\<R}\SG\iso \varphi_*(\SEE\otimes_{\<S}(S\otimes_\varphi \SG))
= \varphi_*(\SEE\otimes_{\<S}\varphi^*\SG)\iso \varphi_*(\SEE\Otimes{\<S}\LL\varphi^*\SG),
\]
see \cite[2.6.5]{li}. Lemma~\ref{dualproj} will show that the sheafification of $p_\varphi$ is a concrete realization 
of~$p^{}_{\!f}$. After that, the proof of  Proposition~\ref{concrete chi} will quickly be concluded. 

\begin{sublem} \label{dualchi}
For all\/ $\SF,\SG\in\D(R),$ the map\/ $\chi^{}_0(\SF,\SG)$ is adjoint to the natural composite map\va{-3}
\[
\varphi_*(\ush\varphi\SF\Otimes{\<S} \LL\varphi^*\SG)
\xto{p_\varphi^{-\<1}\,} \varphi_*\ush\varphi\SF\Otimes{\<\<R}\SG
\lto \SF\Otimes{\<\<R}\SG.
\]
\end{sublem}

\begin{proof} It is straightforward to reduce to showing that 
when $\SF$ is K-injective and $\SG$ is K-flat 
(so that $\chi^{}_0=\beta\smallcirc\zeta^{-1}$),
the border of the following natural diagram, in which $\Hr\set\Hom$, commutes:\va{-2}
\[\mkern-4mu
\def\1{$\varphi_*\big(\ush\varphi\SF\otimes_{S}\varphi^*\SG\big)$}
\def\2{$\varphi_*\big(\Hr_\varphi(S,\SF\>)\!\otimes_{S}\!(\<S\<\otimes_\varphi\<\<\SG)\big)$}
\def\3{$\varphi_*\Hr_\varphi(S,\SF\otimes_{\<R}\<\<\SG)$}
\def\4{$\Hr_R(S,\SF\>)\<\<\otimes_R\<\<\SG$}
\def\5{$\Hr_R(S,\SF\otimes_{\<R}\<\<\SG)$}
\def\6{$\varphi_*\Hr_\varphi(S,\SF\>)\<\<\otimes_R\<\<\SG$}
\def\7{$\varphi_*\ush\varphi\SF\otimes_R\<\<\SG$}
\def\8{$\Hr_R(R,\SF\>)\<\<\otimes_R\<\<\SG$}
\def\9{$\Hr_R(S,\SF\>)\otimes_{\<R}\<\<\SG$}
\def\ten{$\SF\otimes_{\<R}\<\<\SG$}
\def\lvn{$\Hr_R(R,\SF\otimes_{\<R}\<\<\SG)$}
\def\twv{$\varphi_*\big(\Hr_\varphi(S,\SF\>)\otimes_\varphi\<\<\SG\big)$}
\def\thn{$\varphi_*\big(\Hr_\varphi(S,\SF\>)\<\<\otimes_{S}\varphi^*\SG\big)$}
  \bpic[xscale=3.78, yscale=2]

   \node(11) at (1,-.5){\1} ;   
   \node(22) at (2.275, -1.25){\thn} ;
   \node(14) at (3.4,-.5){\7}; 
 
   \node(21) at (1,-1.25){\2};
   \node(23) at (3.4,-1.25){\6};
   
   \node(31) at (1,-2){\twv};
   \node(32) at (2.275,-2){\4};
   \node(33) at (3.4,-2){\8};
   
   \node(41) at (1,-2.75){\3};
   \node(42) at (1.85,-2.75){\5};   
   \node(43) at (2.7,-2.75){\lvn}; 
   \node(44) at (3.4,-2.75){\ten}; 
  
    \draw[->] (11)--(14)  node[above, midway, scale=.75] {$p_\varphi^{-\<1}$} ;

    \draw[double distance=2] (21)--(22) ;
    \draw[->] (22)--(23) node[above, midway, scale=.75] {$p_\varphi^{-\<1}$} ;

    \draw[double distance=2] (31)--(32) ;

    \draw[->] (32)--(33) ;
   
    \draw[double distance=2] (41)--(42) ;
    \draw[->] (42)--(43) ;    
    \draw[->] (43)--(44) ;    
    
    \draw[->] (11)--(21) node[left=1, midway, scale=.75] 
                                                  {$\varphi_*(\zeta^{-\<1})$} ;   
    \draw[->] (21)--(31) ;    
    \draw[->] (31)--(41) ;

    \draw[->] (14)--(23) ; 
    \draw[->] (23)--(33) ;
    \draw[->] (33)--(44) ;

   \draw[->] (11)--(22)  ; 
   \draw[double distance=2] (32)--(23) ;
   \draw[->] (32)--(42) ;
   \draw[->] (33)--(43) ;
  
   \node at (2.275,-1.64)[scale=.9]{\circled1} ;
 
 \epic
\]

\vskip-3pt

One readily checks commutativity of the unlabeled subdiagrams; and 
that of \circled1 follows from the above description of $p_\varphi\>$. The assertion results.  
\end{proof}

\begin{sublem}\label{sheafify projn} The map\/ $p_{\varphi}(\SEE,\SG)$ is adjoint to the natural composite map
\[
\LL\varphi^*(\varphi_*\SEE\Otimes{\<\<R}\SG)\iso 
\LL\varphi^*\varphi_*\SEE\Otimes{\<S}\LL\varphi^*\SG\lto
\SEE\Otimes{\<S}\LL\varphi^*\SG.
\]
\end{sublem}
\begin{proof} One may assume the $R$-complex $\SG$ to be K-flat, in which case the assertion amounts to commutativity of the border of the natural diagram
\[
\def\1{$\varphi_*\SEE\otimes_{\<R}\SG$}
\def\2{$\varphi_*\LL\varphi^*(\varphi_*\SEE\otimes_{\<R}\SG)$}
\def\3{$\varphi_*\varphi^*(\varphi_*\SEE\otimes_{\<R}\SG)$}
\def\4{$\varphi_*(\varphi^*\varphi_*\SEE\otimes_{\<S}\varphi^*\SG)$}
\def\5{$\varphi_*(\SEE\Otimes{\<S}\varphi^*\SG)$}
\def\6{$\varphi_*(\LL\varphi^*\varphi_*\SEE\Otimes{\<S}\LL\varphi^*\SG)$}
  \bpic[xscale=3.75, yscale=1.4]

   \node(11) at (1,-1){\1};   
   \node(13) at (3,-1){\2}; 
      
   \node(22) at (2,-2){\3};
   
   \node(32) at (2,-3){\4};
   
   \node(41) at (1,-4){\5};
   \node(43) at (3,-4){\6};   
  
    \draw[->] (11)--(13) ;  
       
    \draw[<-] (41)--(43) ;
    
    \draw[->] (11)--(41) node[left=1pt, midway, scale=.75] {$p_{\varphi}$} ;
    \draw[->] (22)--(32) ;
    \draw[->] (13)--(43) ; 

    \draw[->] (11)--(22) ;
    \draw[->] (13)--(22) ; 
    \draw[->] (32)--(41) ;
    \draw[->] (32)--(43) ; 

    \node at (2,-1.5)[scale=.9]{\circled1} ;
    \node at (1.45,-2.5)[scale=.9]{\circled2} ;
    \node at (2.55,-2.5)[scale=.9]{\circled3} ;
    \node at (2,-3.53)[scale=.9]{\circled4} ;
 \epic
\]
The commutativity of \circled2 is easily checked. For that of \circled 1 (resp.~\circled4, resp.~\circled3)
cf.~\cite[(3.2.1.3), resp.~(3.2.1.2), resp.~(3.2.4.1)]{li}. The assertion results.
\end{proof}

\pagebreak[3]
The sheafification $\SH{\<R}p_\varphi$ is naturally isomorphic to the projection map $p^{}_{\!f}\>$:

\begin{sublem}\label{dualproj} The following diagram, with\/ $E\set\SH{\<S}\SEE,$ $G\set\SH{\<R}\SG,$ and\/ $p^{}_{\!f}$~ as in \eqref{projf}$,$ commutes.
\[
\def\1{$\SH{\<R}(\varphi_*\SEE\Otimes{\<\<R}\SG)$}
\def\2{$\SH{\<R}\varphi_*\SEE\Otimes{Y}\SH{\<\<R}\SG$}
\def\3{$\R\fst\>\SH{\<S}\SEE\Otimes{Y}\SH{\<\<R}\SG$}
\def\4{$\R\fst E\Otimes{Y}G$}
\def\5{$\SH{\<R}\varphi_*(\SEE\Otimes{\<\<S}\LL\varphi^*\SG)$}
\def\6{$\R\fst\>\SH{\<S}\<(\SEE\Otimes{\<\<S}\LL\varphi^*\SG)$}
\def\7{$\R\fst(\SH{\<S}\SEE\Otimes{\<\<X}\SH{\<S}\LL\varphi^*\SG)$}
\def\8{$\R\fst(E\Otimes{\<\<X}\LL f^*\<G)$}
  \bpic[xscale=4.5, yscale=1.4]

   \node(11) at (1,-1){\1};   
   \node(12) at (2,-1){\5}; 
      
   \node(21) at (1,-2){\2};
   \node(22) at (2,-2){\6};
   
   \node(31) at (1,-3){\3};
   \node(32) at (2,-3){\7};   
   
   \node(41) at (1,-4){\4};
   \node(42) at (2,-4){\8}; 
  
    \draw[->] (11)--(12) node[above=1pt, midway, scale=.75] {$\SH{\<R}p_\varphi$} ;  
       
    \draw[->] (41)--(42) node[above, midway, scale=.9] {$\Iso$}
                                   node[below=1, midway, scale=.75] {$p^{}_{\!f}$} ;
    
    \draw[->] (11)--(21) node[right=1pt, midway, scale=.75] {$\simeq$}
                                   node[left=1pt, midway, scale=.75] {$\eqref{sheafify Otimes}$} ;
    \draw[->] (21)--(31) node[right=1pt, midway, scale=.75] {$\simeq$}
                                   node[left=1pt, midway, scale=.75]
                                            {$\textup{\eqref{sheafify_*}}$} ; 
    \draw[double distance=2](31)--(41);
    
    \draw[->] (12)--(22) node[left=1pt, midway, scale=.75] {$\simeq$}
                                   node[right=1pt, midway, scale=.75] 
                                    {\textup{\eqref{sheafify_*}}} ;
    \draw[->] (22)--(32) node[left=1pt, midway, scale=.75] {$\simeq$}
                                   node[right=1pt, midway, scale=.75] 
                                   {$\eqref{sheafify Otimes}$}  ;
    \draw[->] (32)--(42) node[left=1pt, midway, scale=.75] {$\simeq$}
                                   node[right=1pt, midway, scale=.75] 
                                   {$\eqref{sheafify^*}$}  ;
  \epic
\]
\end{sublem}

\begin{proof} Keeping in mind \ref{sheafify projn} and the analogous property of $p^{}_{\!f}$
(a property which defines $p^{}_{\!f}$), expand the diagram, naturally, as follows:
\begin{small}
\[\mkern-3mu
\def\1{$\SH{\<R}(\varphi_*\SEE\Otimes{\<\<R}\SG)$}
\def\2{$\SH{\<R}\varphi_*\SEE\<\Otimes{Y}\<\SH{\<R}\SG$}
\def\3{$\R\fst\>\SH{\<S}\SEE\<\Otimes{Y}\<\SH{\<R}\SG$}
\def\4{$\R\fst E\Otimes{Y}G$}
\def\5{$\SH{\<R}\varphi_*(\SEE\Otimes{\<\<S}\LL\varphi^*\SG)$}
\def\6{$\R\fst\>\SH{\<S}\<(\SEE\Otimes{\<\<S}\LL\varphi^*\SG)$}
\def\7{$\R\fst(\SH{\<S}\SEE\<\Otimes{\<\<X}\<\SH{\<S}\LL\varphi^*\SG)\:$}
\def\8{$\R\fst(E\Otimes{\<\<X}\LL f^*\<G)$}
\def\9{$\SH{\<R}\varphi_*\LL\varphi^*(\varphi_*\SEE\Otimes{\<\<R}\SG)$}
\def\ten{$\R\fst\>\SH{\<S}\LL\varphi^*(\varphi_*\SEE\Otimes{\<\<R}\SG)$}
\def\lvn{$\R\fst\LL f^*\SH{\<R}(\varphi_*\SEE\<\Otimes{\<\<R}\<\SG)$}
\def\twv{$\R\fst\LL f^*(\SH{\<R}\varphi_*\SEE\<\Otimes{Y}\<\SH{\<R}\SG)$}
\def\thn{$\R\fst\LL f^*(\R\fst E\Otimes{Y}G)$}
\def\frn{$\SH{\<R}\varphi_*(\LL\varphi^*\varphi_*\SEE\Otimes{\<S}\LL\varphi^*\SG)$}
\def\ffn{$\R\fst\>\SH{\<S}(\LL\varphi^*\varphi_*\SEE\Otimes{\<S}\LL\varphi^*\SG)$}
\def\sxn{$\R\fst(\SH{\<S}\LL\varphi^*\varphi_*\SEE\<\Otimes{\<\<X}\<\SH{\<S}
      \LL\varphi^*\SG)$}
\def\svn{$\R\fst(\LL f^*\<\SH{\<R}\varphi_*\SEE\<\Otimes{\<\<X}\<\LL f^*\<\SH{\<R}\SG)$}
\def\egn{$\R\fst(\LL f^*\R\fst E\Otimes{\<\<X}\LL\fst G)$}
  \bpic[xscale=3.25, yscale=2]

   \node(11) at (1,-1){\1};  
   \node(12) at (1.75,-.5){\9}; 
   \node(13) at (3.25,-.5){\frn}; 
   \node(14) at (4,-1){\5}; 
  
   \node(21) at (1,-2){\1}; 
   \node(22) at (1.75,-1.5){\ten};
   \node(23) at (3.25,-1.5){\ffn};

   \node(31) at (1,-3){\2};
   \node(32) at (1.75,-2.5){\lvn};
   \node(33) at (3.25,-2.5){\sxn};
   \node(34) at (4,-2){\6};
   
   \node(42) at (1.75,-3.5){\twv};
   \node(43) at (3.25,-3.5){\svn};
   \node(44) at (4,-3){\7};

   \node(51) at (1,-4){\4};
   \node(52) at (1.75,-4.5){\thn};
   \node(53) at (3.25,-4.5){\egn};
   \node(54) at (4,-4){\8};
  
    \draw[->] (11)--(12) ;
    \draw[->] (12)--(13) ;
    \draw[->] (13)--(14) ;
    
    \draw[->] (21)--(32) ;
    \draw[->] (22)--(23) ;
    \draw[->] (23)--(34) ;
    
    \draw[->] (31)--(42) ;
    \draw[->] (42)--(43) ;
    
    \draw[->] (33)--(44) ;

    \draw[->] (51)--(52) ;
    \draw[->] (52)--(53) ;
    \draw[->] (53)--(54) ;

    \draw[double distance=2] (11)--(21) ;
    \draw[->] (21)--(31) ;
    \draw[->] (31)--(51) ; 
                                  
    \draw[->] (12)--(22) ;
    \draw[->] (22)--(32) node[right=1pt, midway, scale=.75] {$\via\eqref{sheafify^*}$} ;
    \draw[->] (32)--(42) ;
    \draw[->] (42)--(52) node[right=1pt, midway, scale=.75] {$\via\eqref{sheafify_*}$} ; 
      
    \draw[->] (13)--(23) ;
    \draw[->] (23)--(33) ;
    \draw[->] (33)--(43) node[left=1pt, midway, scale=.75] {$\via\eqref{sheafify^*}$} ;                             
    \draw[->] (43)--(53) node[left=1pt, midway, scale=.75] {$\via\eqref{sheafify_*}$} ;                          
    
    \draw[->] (14)--(34) ;
    \draw[->] (34)--(44) ;
    \draw[->] (44)--(54) ;

    \node at (1.535,-1.03)[scale=.9]{\circled1} ;
    \node at (2.44,-2.52)[scale=.9]{\circled2} ;   
    \node at (3.455,-4.01)[scale=.9]{\circled3} ;    
 \epic
\]
\end{small}

The commutativity of \circled1 follows from that of \eqref{adjconj}$'$,
and that of \circled3 from \eqref{adjconj'}$'$. The commutativity of \circled2 results from the following formal consequence of Lemma~\ref{mon_*}.
\end{proof}

\begin{sublem}\label{mon^*} The following
diagram of natural isomorphisms commutes$\>:$
\[
\def\1{$\LL f^*\SH{\<R}(\SF\Otimes{\<\<R}\SG)$}
\def\2{$\LL f^*(\SH{\<R}\SF\Otimes{Y}\SH{\<R}\SG)$}
\def\3{$\LL f^*\SH{\<R}\SF\Otimes{\sX}\LL f^*\SH{\<R}\SG$}
\def\4{$\SH{\<S}\LL \varphi^*(\SF\Otimes{\<\<R}\SG)$}
\def\5{$\SH{\<S}(\LL \varphi^*\SF\Otimes{\<S}\LL \varphi^*\SG)$}
\def\6{$\SH{\<S}\LL \varphi^*\SF\Otimes{\<X}\SH{\<S}\LL \varphi^*\SG$}
  \bpic[xscale=4.25, yscale=1.5]

   \node(11) at (1,-1){\1};   
   \node(12) at (1.973,-1){\2}; 
   \node(13) at (3,-1){\3};   
   
   \node(21) at (1,-2){\4};
   \node(22) at (1.98,-2){\5};
   \node(23) at (3,-2){\6};                                                                                                                                                                                                                                                                                                                                                                                                                                                                                                                                                                                                                                                                                                                                                                                                                                                                                                                                                                                                                                                                                                                                                   
    
    \draw[->] (11)--(12) ;  
    \draw[->] (12)--(13) ;
    
    \draw[->] (21)--(22) ;  
    \draw[->] (22)--(23) ;
    
    \draw[->] (11)--(21) node[left=1pt, midway, scale=.75] {\eqref{sheafify^*}} ;
    \draw[->] (13)--(23) node[right=1pt, midway, scale=.75] {$\via\eqref{sheafify^*}$} ;
   \epic
\]
\end{sublem}

\begin{proof} It is enough  to show commutativity of the adjoint diagram, which expands naturally as:
\[\mkern-3mu
\def\1{$\SH{\<R}(\SF\Otimes{\<\<R}\SG)$}
\def\2{$\SH{\<R}\SF\<\Otimes{Y}\<\SH{\<R}\SG$}
\def\3{$\R\fst\LL f^*(\SH{\<R}\SF\<\Otimes{Y}\<\SH{\<R}\SG)$}
\def\4{$\R\fst(\LL f^*\SH{\<R}\SF\<\Otimes{\<X}\<\LL f^*\SH{\<R}\SG)$}
\def\5{$\ \SH{\<R}\varphi_*\LL\varphi^*\SF\<\Otimes{Y}\<\SH{\<R}\varphi_*\LL\varphi^*\SG$}
\def\6{$\R\fst\LL f^*\SH{\<R}\SF\<\Otimes{Y}\<\R\fst\LL f^*\SH{\<R}\SG$}
\def\7{$\ \ \SH{\<R}(\varphi_*\LL\varphi^*\SF\<\Otimes{\<\<R}\varphi_*\LL\varphi^*\SG)\quad$}
\def\8{$\R\fst\>\SH{S}\<\LL \varphi^*\SF\<\Otimes{Y}\<\R\fst\>\SH{S}\<\LL \varphi^*\SG$}
\def\9{$\SH{\<R}\varphi_*\LL\varphi^*(\SF\Otimes{\<\<R}\SG)$}
\def\ten{$\SH{\<R}\varphi_*(\LL\varphi^*\SF\<\Otimes{\<S}\LL\varphi^*\SG)$}
\def\lvn{$\R\fst\>\SH{S}\LL\varphi^*(\SF\Otimes{\<\<R}\SG)$}
\def\twv{$\R\fst\>\SH{S}\<(\LL\varphi^*\SF\Otimes{\<S}\LL\varphi^*\SG)$}
\def\thn{$\R\fst(\SH{S}\LL\varphi^*\SF\Otimes{\<\<X}\SH{S}\LL\varphi^*\SG)$}
\def\frn{$\R\fst\LL f^*\SH{\<R}(\SF\Otimes{\<\<R}\SG)$}
  \bpic[xscale=4.45, yscale=2.2]

   \node(11) at (1,-1){\1};  
   \node(12) at (1.95,-1){\2}; 
   \node(125) at (2.5,-.5){\3}; 
   \node(13) at (3,-1){\4}; 
  
   \node(22) at (1.95,-2){\5};
   \node(225) at (2.5,-1.5){\6}; 
  
   \node(32) at (1.95,-3.25){\7};
   \node(325) at (2.5,-2.5){\8}; 

   \node(40) at (1,-3.45){\frn}; 
   \node(41) at (1.5,-3.8){\9};  
   \node(42) at (1.95,-4.3){\ten}; 
   
   \node(51) at (1,-5){\lvn};  
   \node(52) at (1.95,-5){\twv};    
   \node(53) at (3,-5){\thn}; 
  
    \draw[->] (11)--(12) ;
    
    \draw[->] (41)--(42) ;
    
    \draw[->] (51)--(52) ;
    \draw[->] (52)--(53) ;

    \draw[->] (11)--(40) ;
    \draw[->] (11)--(41) ;
    \draw[->] (41)--(51) ; 
    \draw[->] (40)--(51) ;
                                  
    \draw[->] (12)--(22) ;
    \draw[->] (22)--(32) ;
    \draw[->] (32)--(42) ;
    \draw[->] (42)--(52) ; 
    
    \draw[->] (225)--(325) ;
      
    \draw[->] (3.05,-1.19)--(3.05,-4.81) ;
    
    \draw[->] (11)--(32) ;
    \draw[->] (12)--(125) ;
    \draw[->] (125)--(13) ;
    \draw[->] (12)--(225) ;
    \draw[->] (225)--(13) ;
    \draw[->] (22)--(325) ;
    \draw[->] (325)--(53) ;
         
    \node at (2.4,-1.03)[scale=.9]{\circled1} ;   
    \node at (2.24,-1.77)[scale=.9]{\circled2} ;   
    \node at (1.51,-2.75)[scale=.9]{\circled3} ; 
    \node at (1.19,-3.13)[scale=.9]{\circled4} ;
    \node at (2.35,-3.8)[scale=.9]{\circled5} ;    
                     
 \epic
\]

Commutativity of the unlabeled subdiagrams is clear. 

Commutativity of \circled1 follows from the fact that the natural map 
\[
\LL f^*(\SH{\<R}\SF\<\Otimes{Y}\<\SH{\<R}\SG)\to
\LL f^*\SH{\<R}\SF\<\Otimes{\<X}\<\LL f^*\SH{\<R}\SG
\] 
is adjoint to the natural composite map 
\[
\SH{\<R}\SF\Otimes{Y}\SH{\<R}\SG\to
\R\fst\LL f^*\SH{\<R}\SF\Otimes{Y}\R\fst\LL f^*\SH{\<R}\SG\to
\R\fst(\LL f^*\SH{\<R}\SF\Otimes{Y}\LL f^*\SH{\<R}\SG),
\]
cf. \cite[beginning of \S3.4.5]{li}.  That of \circled3 holds for analogous reasons.

Commutativity of \circled2 and of \circled4 follows from \ref{Adjconj}. 

Commutativity of \circled5 is given by Lemma~\ref{mon_*}.
\end{proof}
 
Now, to prove Proposition~\ref{concrete chi}, use Lemma~\ref{dualchi} to
expand the diagram in question, naturally, as follows:
\begin{small}
\[\mkern-4mu
\def\1{$\SH{\<S}(\ush\varphi\SF\Otimes{\<S}\LL\varphi^*\SG)$}
\def\2{$\SH{\<S}\ush\varphi\varphi_*(\ush\varphi\SF\Otimes{\<S}\LL\varphi^*\SG)$}
\def\3{$\SH{\<S}\ush\varphi\<(\varphi_*\ush\varphi\SF\<\Otimes{\<\<R}\<\<\SG)$}
\def\4{$\SH{\<S}\ush\varphi\<(\SF\<\Otimes{\<\<R}\<\<\SG)$}
\def\5{$\SH{\<S}\ush\varphi\SF\Otimes{\<X}\SH{\<S}\LL\varphi^*\SG$}
\def\6{$\quad f^\flat\SH{\<R}\varphi_*(\ush\varphi\SF\Otimes{\<S}\LL\varphi^*\SG))$}
\def\7{$f^\flat\SH{\<R}(\varphi_*\ush\varphi\SF\<\Otimes{\<\<R}\<\<\SG)$}
\def\8{$f^\flat\SH{\<R}(\SF\<\Otimes{\<\<R}\<\<\SG)$}
\def\9{$f^\flat\<\R\fst\SH{\<S}(\ush\varphi\SF\Otimes{\<S}\LL\varphi^*\SG)$}
\def\ten{$f^\flat(\SH{\<R}\varphi_*\ush\varphi\SF\<\Otimes{Y}\<G\>)$}
\def\lvn{$f^\flat\<\R\fst(\SH{\<S}\ush\varphi\SF\Otimes{\<X}\SH{\<S}\LL\varphi^*\SG)$}
\def\twv{$f^\flat(\R\fst\SH{\<S}\<\ush\varphi\SF\<\Otimes{Y}\<G\>)$}
\def\thn{$f^\flat \<\<F\Otimes{\<X}\LL f^*\<G$}
\def\frn{$f^\flat\R\fst(f^\flat \<\<F\Otimes{\<X}\LL f^*\<G)$}
\def\ffn{$f^\flat(\R\fst f^\flat \<\<F\Otimes{\<Y}G)$}
\def\sxn{$f^\flat(F\Otimes{\<Y}G)$}
\def\svn{$f^\flat\R\fst(\SH{\<S}\ush\varphi\SF\Otimes{\<X}\LL f^*\<G)$}
  \bpic[xscale=3.38, yscale=1.65]

   \node(11) at (1,-1.55){\1};   
   \node(12) at (1.93,-1){\2};
   \node(13) at (3.2,-1){\3};
   \node(14) at (4,-1.55){\4}; 

   \node(22) at (1.93,-2){\6};
   \node(23) at (3.2,-2){\7};
   
   \node(31) at (1,-3.45){\5};
   \node(32) at (1.93,-3){\9};
   \node(33) at (3.2,-3){\ten};
   \node(34) at (4,-3.45){\8};
   
   \node(42) at (1.93,-4){\lvn};
   \node(43) at (3.2,-5){\twv};
   
   \node(51) at (1,-5.43){\thn};
   \node(52) at (1.93,-6){\frn};
   \node(53) at (3.2,-6){\ffn};
   \node(54) at (4,-5.43){\sxn};   
   \node(55) at (1.93,-5){\svn};
  
    \draw[->] (11)--(12) ;
    \draw[->] (12)--(13) ;
    \draw[->] (13)--(14) ;
   
    \draw[<-] (22)--(23) node[above, midway, scale=.75] {$\Iso$}
                                   node[below=2, midway, scale=.75] {$\>\via \>p_\varphi$} ;
    \draw[->] (23)--(34) ;
        
    \draw[->] (51)--(52) ;
    \draw[<-] (52)--(53) node[above, midway, scale=.75] {$\Iso$}
                                   node[below=1, midway, scale=.75]{$\via \>p^{}_{\!f}$} ;
    \draw[->] (53)--(54) ;
    \draw[->] (43)--(55) node[above, midway, scale=.75] {$\Iso$}
                                   node[below=1, midway, scale=.75]{$\via \>p^{}_{\!f}$} ;
    \draw[->] (11)--(31) node[right=1pt, midway, scale=.75] {$$}
                                   node[left, midway, scale=.75]{} ;
    \draw[->] (31)--(51) ;          
                                   
     \draw[->] (12)--(22)  ;                              
     \draw[->] (22)--(32)  ; 
     \draw[->] (32)--(42)  ;                              
     \draw[->] (42)--(55)  ;       
     \draw[->] (55)--(52)  ; 
     
     \draw[->] (13)--(23)  ;                              
     \draw[->] (23)--(33)  ; 
     \draw[->] (33)--(43)  ;                                                     
     \draw[->] (43)--(53) node[right=1pt, midway, scale=.75] {\eqref{affdual}} ;
        
    \draw[->] (14)--(34) node[right=1pt, midway, scale=.75] {$$}
                                   node[left, midway, scale=.75]{} ;
    \draw[->] (34)--(54) ;          
                                   
     \draw[->] (11)--(32) ;                              
     \draw[->] (31)--(42) ;
     \draw[->] (33)--(54) ;
     
    \node at (1.64,-1.65)[scale=.9]{\circled1} ;
    \node at (2.58,-3.52)[scale=.9]{\circled2} ;   
    \node at (3.42,-4.44)[scale=.9]{\circled3} ;    
                         
 \epic
\]
\par
\end{small}
Commutativity of \circled3 (resp.~\circled1) is essentially the same as commutativity of \circled3 (resp.~\circled1) in the proof of Proposition~\ref{concrete bc2}.

The commutativity of \circled2 is given by Lemma~\ref{dualproj} 
(with $\SEE\set\ush\varphi \SF$).
\end{proof}
\end{cosa}

%
%
%
%
%
%

\begin{cosa}\label{concrete version zeta}
This section is devoted to proving Proposition~\ref{concrete zeta}, which gives, for maps of affine schemes, a concrete representation of a generalization of the map $\zeta$ in~\ref{flat and hom}(iii).\va2

Once again, $\varphi\colon R\to S$ is a homomorphism of commutative rings, with
corresponding scheme\kf-map\va2
\(
 \spec S =:X\xto{\,f\,} Y\set \spec R.
\)
The global-section functor from $\sA(Y)$ to $\sA(R)$ is denoted by $\Gamma^{}_{\!\<\<Y}$,\va1 and the sheafification functor
from $\sA(R)$ to $\sA_\qc(Y)$ by $\bsr\>$.\va1

The composite \emph{quasi-coherator} functor
\[
Q_Y\colon \sA(Y)\xto{\Gamma^{}_{\!\<\<Y}}\sA(R)\xto{\bsr}\sA_\qc(Y)
\] 
is right-adjoint to the 
inclusion functor $j^{}_Y\colon\sA_\qc(Y)\hookrightarrow\sA(Y)$.
Since $j^{}_Y$ is fully faithful, the unit map for this adjunction is an isomorphism 
\[
A\iso Q^{}_Yj^{}_YA=Q^{}_Y\<A\qquad (A\in\sA_\qc(Y)).
\]
The counit  is the natural map
\[
j^{}_YQ_YB=\SH{\<R}\Gamma^{}_{\!\!Y}B\lto B\qquad (B\in\sA(Y)).
\]

Since $j^{}_Y$ is exact, therefore $Q_Y$ preserves K-injectivity, and there results an adjunction
$\boldsymbol j^{}_{\<\<Y}\!\dashv\R Q_Y$ between the derived functors. In fact,
$\boldsymbol j^{}_{\<\<Y}$ gives an equivalence $\D(\sA_\qc(Y))\xto{\lift.5,{\:\approx\:},}\Dqc(Y)\subset\D(Y)$
\cite[p.\,225, 5.1]{BN}, with quasi-inverse the restriction  to $\Dqc(Y)$ of the right adjoint $\R Q_Y$. 
So if $C\in\Dqc(Y)$ then the counit map 
is an isomorphism $\R Q_YC=\boldsymbol j^{}_{\<\<Y}\R Q_YC\iso C$.

Set
\begin{align*}
\sHom_Y^\qc&\set Q_Y\sHom^{}_Y\colon \sA(Y)^\op\times\sA(Y)\to\sA_\qc(Y),\\[2pt]
\R\>\sHom_Y^\qc&\set \R Q_Y\R \sHom^{}_Y\colon \D(Y)^\op\times\D(Y)\to\D(\sA_\qc(Y)).
\end{align*}
For $\OY$-complexes $F$ and $G$, the natural bifunctorial map 
\[
\sHom_Y^\qc(F\<,G\>)\to\R\>\sHom_Y^\qc(F\<,G\>)
\]
is an isomorphism if $F$ is K-flat and $G$ is K-injective (so that $\sHom_Y(F\<,G\>)$ is K-injective).

For any $R$-complexes $\SF$ and $\SG$, and  $F\set\SH{\<R}\SF$, $G\set\SH{\<R}\SG$,
there is a natural $\sA_\qc(Y)$-isomorphism
\[
\kappa(\SF,\SG)\colon\SH{\<R}\!\Hom_R(\SF,\SG) \iso \SH{\<R}\Gamma^{}_{\!\!Y}\<\sHom_Y(F,G)
=\sHom^\qc_Y(F\<,G\>).
\]
Furthermore, there is a natural composite $\D(\sA_\qc(Y))$-isomorphism
\[
\boldsymbol\kappa(\SF,\SG)\colon
\SH{\<R}\<\R\<\<\Hom_R(\SF,\SG) \iso
\SH{\<R}\<\R\Gamma^{}_{\!\!Y}\R\>\sHom_Y(F,G)\iso
\R\> \sHom^\qc_Y(F\<,G\>).
\]
whose first component is an isomorphism
because the natural equivalences
\[
\D(R)\xto[\,\lift1.3,\bsr,]{\lift.5,{\,\approx\,},} \D(\sA_\qc(Y))
\xto[\>\lift1.5,\boldsymbol j^{}_{\<\<Y},]{\lift.5,{\,\approx\,},} \Dqc(Y).\\[-2pt]
\]
give rise, for any $n\in\mathbb Z$, to natural isomorphisms
\begin{align*}
 H^n\R\<\<\Hom_R(\SF,\SG)\cong\Hom_{\D(R)}(\SF,\SG[n])&\cong \Hom_{\D(Y)}(F,G[n])\\
 &\cong H^n\R\<\<\Hom_Y(F,G)\\
 & \cong H^n\R\Gamma^{}_{\!\!Y}\R\>\sHom_Y(F,G).
\end{align*}
\vskip2pt
Next, there is a natural bifunctorial isomorphism of \mbox{$S$-complexes}
\[
\Hom_S(S\otimes_\varphi \>\SF\<, \Hom_\varphi(S,\<\SG))
\<\!\iso\!\<\Hom_\varphi(S\otimes_{\varphi}\>\SF\<,\<\SG)
\<\!\iso\!\<\Hom_\varphi(S,\Hom_R(\SF\<,\<\SG)).
\]
From this, with $\SF$ K-flat and $\SG$ K-injective, one gets (via, e.g., \cite[2.6.5]{li}) a natural 
bifunctorial $\D(S)$-isomophism, concrete up to choice of K-flat or K-injective resolutions,
\begin{align*}
\zeta_0\colon\R\<\<\Hom_S(\LL\varphi^*\SF,\ush\varphi\SG)&=
\R\<\<\Hom_S(S\Otimes{\varphi}\< \SF\<, \R\<\<\Hom_\varphi(S,\SG))\\
&\iso\<\<\R\<\<\Hom_\varphi(S,\R\<\<\Hom_R(\SF\<,\SG))= \ush\varphi\R\<\<\Hom_R(\SF,\SG).
\end{align*}

\begin{subprop} \label{concrete zeta}
The sheafification $\SH{S}\<\zeta_0$ is naturally isomorphic to the map\/
$($that must then be an \emph{isomorphism)}
\[
\zeta\colon
\R\>\sHom^\qc_\sX(\LL f^*\<\<F,f^\flat G\>)
\lto f^\flat\R\>\sHom_Y(F,G\>)
\]
which is adjoint to the natural  composite map
\begin{align*}
\R\fst\R\>\sHom^\qc_\sX(\LL f^*\<\<F\<, f^\flat G\>) &\<\lto\< \R\fst\R\>\sHom_\sX(\LL f^*\<\<F\<, f^\flat G\>) \\
&\<\lto\< \R\>\sHom_Y(\R\fst\LL f^*\<\<F\<, \R\fst f^\flat G\>) \<\lto\< \R\>\sHom_Y(F\<,  G\>).
\end{align*}
\end{subprop}

\noindent\emph{Remark.} 
If $C\set\R\>\sHom_\sX(\LL f^*\<\<F\<, f^\flat G\>)\in\Dqc(Y)$ (which is so under the hypotheses of ~\ref{flat and hom}(iii)) then the map $\R\fst\R Q_YC\<\lto\< \R\fst C$ in~\ref{concrete zeta} is an isomorphism (see above). 
Keeping in mind \cite[3.2.4(i)]{li} and the paragraph preceding \cite[3.5.5]{li}, one finds then that the map $\zeta$ in
~\ref{flat and hom}(iii) is a special case of the map $\zeta$ in~\ref{concrete zeta}.

\begin{proof} The strategy is, to relate the sheafification of 
the adjoint of~$\zeta_0$ to the adjoint of $\zeta$.

\begin{sublem}\label{adj zeta0}
The map\/ $\zeta_0$ is adjoint to the natural  composite map
\begin{align*}
\varphi_*\R\<\<\Hom_S(\LL \varphi^*\SF\<, \ush\varphi \SG\>) 
\lto \R\<\<\Hom_R(\varphi_*\LL \varphi^*\SF\<, \varphi_* \ush\varphi \SG\>) \lto 
\R\<\<\Hom_R(\SF\<,  \SG\>).
\end{align*}
\end{sublem}
\begin{proof}
The assertion means that the border of the following natural diagram,  in which $\Hr=\Hom$, and whose top row composes to $\zeta_0$, commutes.
\begin{small}
\[\mkern-4mu
\def\1{$\R\Hr_S(\LL\varphi^*\SF,\ush\varphi\SG)$}
\def\2{$\R\Hp(\LL\varphi^*\SF,\SG)$}
\def\3{$\ush\varphi\R\Hr_{\<R}(\SF,\SG)$}
\def\4{$\ush\varphi\varphi_*\R\Hp(\LL\varphi^*\SF,\SG)$}
\def\5{$\ush\varphi\varphi_*\ush\varphi\R\Hr_{\<R}(\SF,\SG)$}
\def\6{$\ush\varphi\varphi_*\R\Hr_S(\LL\varphi^*\SF,\ush\varphi\SG)$}
\def\7{$\ush\varphi\R\Hr_{\<R}(\varphi_*\LL\varphi^*\SF,\varphi_*\ush\varphi\SG)$}
\def\8{$\ush\varphi\R\Hr_{\<R}(\varphi_*\LL\varphi^*\SF,\SG)$}
  \bpic[xscale=4.3, yscale=1.4]

   \node(11) at (1,-1){\1};   
   \node(12) at (2.02,-1){\2};
   \node(13) at (3,-1){\3};
 
   \node(23) at (2.5,-2){\5};
 
   \node(32) at (2.02,-3){\4};

   \node(41) at (1,-4){\6};
   \node(42) at (2.02,-4){\7};
   \node(43) at (3,-4){\8};
   
    \draw[->] (11)--(12) ;
    \draw[->] (12)--(13) ;
   
    \draw[->] (41)--(42) ;
    \draw[->] (42)--(43) ;
    
    \draw[->] (11)--(41) node[left=1pt, midway, scale=.75] {$\eta$} ;
    \draw[->] (12)--(32) node[left=1pt, midway, scale=.75] {$\eta$} ;          
                                   
    \draw[->] (43)--(13)  ;                              
                                    
     \draw[->] (41)--(32) node[above=-.5, midway, scale=.75] {$\ush\varphi\mu\mkern20mu$} ;                              
     \draw[->] (32)--(43) node[above=.5, midway, scale=.75] {$\mkern20mu\ush\varphi\lambda$} ;
     \draw[->] (32)--(23) ;
     \draw[<-] (2.55,-1.8)--(2.85,-1.2) node[above=-1, midway, scale=.75] {$\eta\mkern20mu$}; 
     \draw[->] (2.62,-1.8)--(2.92,-1.2) node[below=1, midway, scale=.75] {$\mkern5mu\epsilon$}; 
     
    \node at (2.02,-3.52)[scale=.9]{\circled1} ;
    \node at (2.65,-2.6)[scale=.9]{\circled2} ;   
                         
 \epic
\]
\end{small}
The unlabeled subdiagrams obviously commute, and $\epsilon\smallcirc\eta=\id$; so it's enough to show that \circled1 and \circled2 commute.

For \circled1 to commute, it suffices, by Proposition~\ref{duality for varphi}, that 
$\lambda\mu$ be inverse to the isomorphism 
$\bar\alpha_\varphi(\LL\varphi^*\SF,S,\SG)$ from Proposition~\ref{duality for phiaff}---which it~is: one~may assume that  $\SG$ is K-injective, so that all $\R$s can be dropped, making the assertion into an easily verified one in commutative algebra.

As for \circled2, one may assume $\SG$ to be K-injective and $\SF$ to be K-flat, 
thereby reducing to verifying commutativity of the natural diagram
\[
\def\1{$\varphi_*\Hr_\varphi(S\otimes_R \SF,\SG)$}
\def\2{$\varphi_*\Hr_\varphi(S,\Hr_R(\SF,\SG))$}
\def\3{$\Hr_\varphi(\varphi_*(S\otimes_R \SF\>),\SG)$}
\def\4{$\Hr_R(\SF,\SG)$}
  \bpic[xscale=4, yscale=1.4]

   \node(11) at (1,-1){\1};   
   \node(12) at (2,-1){\2}; 
   
   \node(21) at (1,-2){\3};
   \node(22) at (2,-2){\4};
                                                                                                                                                                                                                                                                                                                                                                                                                                                                                                                                                                                                                                                                                                                                                                                                                                                                                                                                                                                                                                                                                                                                                 
    \draw[->] (11)--(12) ;  
    
    \draw[->] (21)--(22) ;  
     
    \draw[->] (11)--(21)  ;
    \draw[->] (12)--(22) ;
    
  \epic
\]
To check this commutativity, send any map $\lambda\colon S\otimes_R \SF\to\SG$
from the upper left to the lower right corner, both clockwise and counterclockwise, and note that either way produces the map taking $x\in\SF$ to $\lambda(1\otimes x)$.
\end{proof}

Now, the adjunction $\boldsymbol j^{}_{\<\<Y}\!\dashv\R Q_Y$ entails a natural commutative diagram, with vertical arrows induced by counit maps,
\begin{small}
\[\mkern-4mu
\def\1{$\R\fst\R\>\sHom^\qc_\sX(\LL f^*\<\<F\<, f^\flat G\>)$}
\def\2{$\R\>\sHom^\qc_Y(\R\fst\LL f^*\<\<F\<, \R\fst f^\flat G\>)$}
\def\3{$\R\>\sHom^\qc_Y(F\<,  G\>)$}
\def\6{$\R\fst\R\>\sHom_\sX(\LL f^*\<\<F\<, f^\flat G\>)$}
\def\7{$\R\>\sHom_Y(\R\fst\LL f^*\<\<F\<, \R\fst f^\flat G\>)$}
\def\8{$\R\>\sHom_Y(F\<,  G\>)$}
  \bpic[xscale=4.4, yscale=1.4]

   \node(11) at (1,-1){\1};   
   \node(12) at (2.08,-1){\2};
   \node(13) at (3,-1){\3};

   \node(41) at (1,-2){\6};
   \node(42) at (2.08,-2){\7};
   \node(43) at (3,-2){\8};
   
    \draw[->] (11)--(12) ;
    \draw[->] (12)--(13) ;
   
    \draw[->] (41)--(42) ;
    \draw[->] (42)--(43) ;
    
    \draw[->] (11)--(41)  ;
    \draw[->] (12)--(42)  ;          
    \draw[->] (13)--(43)  node[right, midway, scale=.75]{$\epsilon_{\boldsymbol j\<\<,\R Q}$};
    
  \epic
\]
\end{small}\par
\noindent Also, $\Hom_{\D(Y)}(E,\epsilon_{\boldsymbol j\<\<,\R Q})$ is an isomorphism for any $E\in\Dqc(X)$,\va{.5} and so $\epsilon_{\boldsymbol j\<\<,\R Q}$ itself is an isomorphism.\va1 

Hence, Proposition~\ref{concrete zeta} results from the commutativity of the border of the following natural diagram, in which $\Hr\set\R\<\<\Hom$ and $\CH\set\R\>\sHom$:
\begin{small}
\[\mkern-4mu
\def\1{$\SH{\<S}\<\<\Hr_S(\LL\varphi^*\SF\<,\ush\varphi\SG)$}
\def\2{$\SH{\<S}\<\ush\varphi\Hr_R(\SF\<,\SG)$}
\def\3{$\quad\SH{\<S}\<\<\ush\varphi\<\<\varphi_*\Hr_S(\LL\varphi^*\SF\<,\ush\varphi\SG)$}
\def\4{$\SH{\<S}\<\<\ush\varphi\Hr_R(\varphi_*\LL\varphi^*\SF\<,\varphi_*\ush\varphi\SG)$}
\def\5{$\quad f^\flat\<\SH{\<R}\varphi_*\Hr_S(\LL\varphi^*\SF\<,\ush\varphi\SG)$}
\def\6{$f^\flat\<\SH{\<R}\Hr_R(\varphi_*\LL\varphi^*\SF\<,\varphi_*\ush\varphi\SG)$}
\def\8{$\CH^\qc_\sX(\SH{\<S}\<\<\LL\varphi^*\SF\<,\SH{\<S}\<\<\ush\varphi\SG)$}
\def\9{$f^\flat\R\fst\SH{\<S}\<\Hr_S(\LL\varphi^*\SF\<,\ush\varphi\SG)$}
\def\ten{$f^\flat\CH^\qc_Y(\SH{\<R}\<\varphi_*\LL\varphi^*\SF\<,\>\SH{\<R}\<\varphi_*\ush\varphi\SG)\ $}
\def\lvn{$\CH^\qc_\sX(\LL f^*\<\<F,f^\flat G\>)$}
\def\twv{$\ f^\flat\R\fst\CH_\sX^\qc(\SH{\<S}\<\LL\varphi^*\SF\<,\>\SH{\<S}\<\ush\varphi\SG)$}
\def\thn{$\ f^\flat\CH^\qc_Y(\R\fst\>\SH{\<S}\<\LL\varphi^*\SF\<,\>
  \R\fst\> \SH{\<S}\<\ush\varphi\SG)\ $}
\def\frn{$f^\flat\<\SH{\<R}\<\Hr_R(\SF\<,\SG)$}
\def\ffn{$f^\flat\R\fst\CH^\qc_\sX(\LL f^*\<\<F,f^\flat G\>)$}
\def\sxn{$f^\flat\CH^\qc_Y(\R\fst\LL f^*\<\<F,\R\fst f^\flat G\>)$}
\def\svn{$f^\flat\CH^\qc_Y(F,G\>)$}
  \bpic[xscale=3.35, yscale=1.5]

   \node(11) at (1,-1){\1};   
   \node(14) at (4,-1){\2}; 

   \node(22) at (1.7,-2){\3};
   \node(23) at (3.25,-2){\4};
   
   \node(32) at (1.7,-3){\5};
   \node(33) at (3.25,-3){\6};
   
   \node(40) at (1,-4){\1};
   \node(41) at (1,-5){\8};
   \node(44) at (4,-4){\frn};
   
   \node(52) at (2.4,-4){\9};
   \node(53) at (3.25,-5){\ten};
  
   \node(62) at (1.7,-6){\twv};
   \node(63) at (2.5, -7){\thn};
   
   \node(71) at (1,-7){\lvn};
   \node(74) at (4,-7){\svn};
  
   \node(82) at (1.7,-8){\ffn};
   \node(83) at (3.25,-8){\sxn};

    \draw[->] (11)--(14) node[above, midway, scale=.75] {$$} ;
  
    \draw[->] (22)--(23) node[above, midway, scale=.75] {$$} ;
    
    \draw[->] (32)--(33) node[above, midway, scale=.75] {$$} ;
     
    \draw[->] (40)--(52) node[above, midway, scale=.75] {$$} ;
    
    \draw[->] (62)--(63) node[above, midway, scale=.75] {$$} ;

    \draw[->] (82)--(83) node[above, midway, scale=.75] {$$} ;

    \draw[double distance = 2] (11)--(40) node[right=1pt, midway, scale=.75] {$$} ;
    \draw[->] (40)--(41) node[right=1pt, midway, scale=.75] {$$} ;
    \draw[->] (41)--(71) node[right=1pt, midway, scale=.75] {$$} ;
    
    \draw[->] (22)--(32) node[right=1pt, midway, scale=.75] {$$} ;
    \draw[->] (32)--(52) node[right=1pt, midway, scale=.75] {$$} ; 
    \draw[->] (52)--(62) node[right=1pt, midway, scale=.75] {$$} ;
    \draw[->] (62)--(82) node[right=1pt, midway, scale=.75] {$$} ;
    
    \draw[->] (23)--(33) node[right=1pt, midway, scale=.75] {$$} ;
    \draw[->] (33)--(53) node[right, midway, scale=.75] {$f^\flat\boldsymbol\kappa$} ; 
    \draw[->] (53)--(63) node[right=1pt, midway, scale=.75] {$$} ;
    \draw[->] (63)--(83) node[right=1pt, midway, scale=.75] {$$} ;
       
    \draw[->] (14)--(44) node[right=1pt, midway, scale=.75] {$$} ;
    \draw[->] (44)--(74) node[right=1pt, midway, scale=.75] {$f^\flat\boldsymbol\kappa$} ;
   
     \draw[->] (11)--(22) ; 
                                  
     \draw[->] (41)--(62) ;
     \draw[->] (71)--(82) ;

     \draw[->] (23)--(14) ; 
     \draw[->] (33)--(44) ; 
     \draw[->] (53)--(74) ;
     \draw[->] (83)--(74) ;
     
     \node at (2.47,-1.5)[scale=.9]{\circled0} ; 
 
    \node at (1.35,-2.5)[scale=.9]{\circled1} ;
    
    \node at (2.7,-4.5)[scale=.9]{\circled2} ; 
    
    \node at (3.25,-6.5)[scale=.9]{\circled3} ;
                           
 \epic
\]
\end{small}

Commutativity of the unlabeled subdiagrams is simple to check.

Commutativity of \circled0 is given by Lemma~\ref{adj zeta0}.

Commutativity of \circled1  is essentially the same as that of subdiagram \circled1 in the proof of 
Proposition~\ref{concrete bc2}.

Commutativity of \circled3 follows easily from that of the diagram in~\ref{Adjconj} and of subdiagram \circled3 in the proof of 
Proposition~\ref{concrete bc2}.

As for the remaining subdiagram \circled2, let us show, more generally, that for any $S$-complexes~$\SA$ and~$\SB$, the following natural $\Dqc(Y)$-diagram commutes.
\[
\def\5{$\quad \<\SH{\<R}\<\varphi_*\Hr_S(\SA\>,\SB)$}
\def\6{$\<\SH{\<R}\Hr_R(\varphi_*\SA\>,\varphi_*\SB)$}
\def\9{$\R\fst\SH{\<S}\<\Hr_S(\SA\>,\SB)$}
\def\ten{$\CH^\qc_Y(\SH{\<R}\<\varphi_*\SA\>,\>\SH{\<R}\<\varphi_*\SB)\ $}
\def\twv{$\ \R\fst\CH_\sX^\qc(\SH{\<S}\<\SA\>,\>\SH{\<S}\<\SB)$}
\def\thn{$\ \CH^\qc_Y(\R\fst\>\SH{\<S}\<\SA\>,\>
  \R\fst\> \SH{\<S}\<\SB)\ $}
  \bpic[xscale=5, yscale=1.5]

   \node(32) at (1,-3){\5};
   \node(33) at (2,-3){\6};
      
   \node(52) at (1,-4){\9};
   \node(53) at (2,-4){\ten};
  
   \node(62) at (1,-5){\twv};
   \node(63) at (2, -5){\thn};

   \draw[->] (32)--(33) ;
   
   \draw[->] (62)--(63) ;
   
   \draw[->] (32)--(52) ;
   \draw[->] (52)--(62) ;
   
   \draw[->] (33)--(53) ;
   \draw[->] (53)--(63) ;

 \epic
\]

The natural isomorphism
$
\Hom(\SH{\<R}\SEE,\< \CH^\qc_Y(F,G))\<\<\!\iso\!\!\Hom(\SH{\<R}\SEE, \<\CH_Y(F,G))
\\(\SEE\in\D(R)\textup{;} \;F,G\in\D(Y)),
$
allows one to replace $\CH^\qc$ in the preceding dia\-gram by $\CH$, whereupon it suffices to show commutativity of the adjoint diagram, which is, up to obvious isomorphisms, the border of the following natural diagram:
\begin{small}
\[\mkern-4mu
\def\1{$\SH{\<R}\<\varphi_{\<*}\Hr_S(\SA\>,\SB)\<\Otimes{Y}\<\SH{\<R}\<\varphi_{\<*}\SA$}
\def\2{$\SH{\<R}\Hr_R(\varphi_{\<*}\SA\>,\varphi_{\<*}\SB)\<\Otimes{Y}\<\SH{\<R}\<\varphi_{\<*}\SA$}
\def\3{$\quad\SH{\<R}(\varphi_{\<*}\Hr_S(\SA\>,\SB)\<\Otimes{\<\<R}\<\varphi_{\<*}\SA)$}
\def\4{$\SH{\<R}(\Hr_R(\varphi_{\<*}\SA\>,\varphi_{\<*}\SB)\<\Otimes{\<\<R}\<\varphi_{\<*}\SA)$}
\def\5{$\R\fst\SH{\<S}\<\Hr_S(\SA\>,\SB)\<\Otimes{Y}\<\R\fst\SH{S}\SA$}
\def\6{$\SH{\<R}\<\varphi_{\<*}(\Hr_S(\SA\>,\SB)\<\Otimes{\<S}\<\SA)$}
\def\7{$\CH_Y\<(\SH{\<R}\<\varphi_{\<*}\SA\>,\>\SH{\<R}\<\varphi_{\<*}\SB)\<\<\Otimes{Y}\<\<\SH{\<R}\<\varphi_{\<*}\SA$}
\def\8{$\quad\R\fst(\SH{\<S}\<\Hr_S(\SA\>,\SB)\<\Otimes{\<\<X}\<\SH{S}\SA)$}
\def\9{$\R\fst\SH{S}(\Hr_S(\SA\>,\SB)\<\Otimes{\<S}\<\SA)$}
\def\ten{$\SH{\<R}\<\varphi_{\<*}\SB$}
\def\lvn{$\R\fst(\CH_\sX(\SH{\<S}\SA\>,\SH{\<S}\SB)\<\Otimes{\<\<X}\<\SH{S}\SA)$}
\def\twv{$\R\fst\CH_\sX(\SH{\<S}\<\SA\>,\>\SH{\<S}\<\SB)\<\Otimes{Y}\<\R\fst\SH{S}\SA$}
\def\thn{$\CH_Y(\R\fst\>\SH{\<S}\<\SA\>,\R\fst\>\SH{\<S}\<\SB)\<\Otimes{Y}\<\R\fst\SH{S}\SA$}
\def\frn{$\R\fst\SH{S}\SB$}
  \bpic[xscale=2.7, yscale=1.5]

   \node(11) at (1,-1){\1};   
   \node(14) at (4,-1){\2}; 

   \node(22) at (1.7,-2){\3};
   \node(23) at (3.5,-2){\4};
   
   \node(31) at (1,-3){\5};
   \node(33) at (2.465,-3){\6};
   \node(34) at (3.98,-3){\7};
   
   \node(42) at (1.7,-4.2){\8};
   \node(43) at (3.45,-4.2){\9};
   \node(44) at (4.5,-4.2){\ten};
   
   \node(52) at (2.75,-5.1){\lvn};
 
   \node(61) at (1,-6){\twv};
   \node(63) at (3.05, -6){\thn};
   \node(64) at (4.5,-6){\frn};

    \draw[->] (11)--(14) node[above, midway, scale=.75] {$$} ;
    
    \draw[->] (22)--(23) node[above, midway, scale=.75] {$$} ;
    
    \draw[->] (42)--(43) node[above, midway, scale=.75] {$$} ;
  
    \draw[->] (61)--(63) node[above, midway, scale=.75] {$$} ;
    
    \draw[->] (63)--(64) node[above, midway, scale=.75] {$$} ;
     
    \draw[->] (11)--(31) node[right=1pt, midway, scale=.75] {$$} ;
    \draw[->] (31)--(61) node[right=1pt, midway, scale=.75] {$$} ;
 
    \draw[->] (4.5,-1.2)--(4.5, -2.8) node[right=1pt, midway, scale=.75] {$$} ;      
    \draw[->] (4.5,-3.2)--(4.5, -3.8) node[right=1pt, midway, scale=.75] {$$} ; 
    \draw[->] (44)--(64) node[right=1pt, midway, scale=.75] {$$} ;
   
     \draw[->] (11)--(22) ; 
                                  
     \draw[->] (31)--(42) ;
     \draw[->] (42)--(52) ;
     
     \draw[->] (22)--(33) ;
     \draw[->] (33)--(43) ;

     \draw[->] (23)--(14) ; 
     \draw[-] (23)--(3.89,-2.85) ;
     \draw[->] (4.05,-3.2)--(44) ;
     \draw[->] (33)--(44) ;

    \draw[->] (61)--(52) ;
    \draw[->] (43)--(64) ;
    \draw[->] (52)--(64) ;
     
    \node at (4.1,-2.53)[scale=.9]{\circled1} ; 
    \node at (3,-2.53)[scale=.9]{\circled2} ;
    \node at (2.1,-3.62)[scale=.9]{\circled3} ; 
    \node at (3,-4.7)[scale=.9]{\circled4} ;
    \node at (2.8,-5.6)[scale=.9]{\circled5} ;
 \epic
\]
\end{small}

The commutativity of the unlabeled subdiagram is easily checked.

The commutativity of \circled2 and \circled5 follows at once the abstract definition of the map \cite[(3.5.4.1)]{li},
a map whose concrete realization is the usual one---see, e.g., the two lines preceding \cite[3.1.9]{li}.

The commutativity of \circled3 is an instance of Lemma~\ref{mon_*}.

Finally, the commutativity of subdiagrams \circled1 and \circled4 results from the next lemma, which  expresses compatibility of (the counit map for) $\hom\>$-$\>\otimes$ adjunction with sheafification.

\begin{sublem} For any\/ $R$-complexes\/ $\SEE$ and\/ $\SF,$ the following natural diagram commutes.
\[
\def\1{$\SH{\<R}\<\Hr_S(\SEE,\SF\>)\Otimes{\<Y}\SH{\<R}\SEE$}
\def\2{$\CH_Y(\SH{\<R}\<\SEE,\SH{\<R}\<\SF\>)\Otimes{\<Y}\SH{\<R}\SEE$}
\def\3{$\SH{\<R}\<(\Hr_S(\SEE,\SF\>)\Otimes{\<\<R}\SEE)$}
\def\4{$\SH{\<R}\<\SF$}
  \bpic[xscale=5, yscale=1.65]

   \node(32) at (1,-3){\1};
   \node(33) at (2,-3){\2};
      
   \node(52) at (1,-4){\3};
   \node(53) at (2,-4){\4};
  
   \draw[->] (32)--(33) ;
   
   \draw[->] (52)--(53) ;
   
   \draw[->] (32)--(52) ;
   
   \draw[->] (33)--(53) ;

 \epic
\]

\end{sublem}

\begin{proof} (Sketch.)
It may be assumed that $\SEE$ is a direct limit of bounded-above flat complexes, and that
$\SF$ is K-injective. Then the diagram can be replaced by its non-derived counterpart, whose commutativity can be verified elementwise, via the definitions of the maps involved.
\end{proof}

This completes the proof of Proposition~\ref{concrete zeta}.
\end{proof}

\begin{subrem}
With assumptions and notation as in Proposition \ref{flat and hom}, that proposition and
~\ref{concrete zeta} show that the following natural diagram commutes:
\[
\def\1{$\SH{S}\<\R\<\<\Hom_S(\LL\varphi^*\SF,\ush\varphi\SG)$}
\def\2{$\SH{S}\<\ush\varphi\R\<\<\Hom_{\<R}(\SF,\SG)$}
\def\3{$\R\>\sHom_\sX(\LL f^*\<\<F\<,f^\flat G)$}
\def\4{$f^\flat\R\>\sHom_Y(F\<,G)$}
\def\5{$\bar{f\:\<}^{\!\<*}\R\>\sHom_{\>\oY}(\LL \phi^*\<\<F\<,\pt G\>)G)$}
\def\6{$\bar{f\:\<}^{\!\<*}\<\<\pt\R\>\sHom_Y(F\<,G\>)$}
  \bpic[xscale=4.75, yscale=1.65]

   \node(32) at (1,-3){\1};
   \node(33) at (2,-3){\2};
      
   \node(42) at (1,-4){\3};
   \node(43) at (2,-4){\4};
   
   \node(52) at (1,-5){\5};
   \node(53) at (2,-5){\6};
  
   \draw[->] (32)--(33) node[above=1, midway, scale=.75]{$\SH{\<R}\zeta_0$} ;
   
   \draw[->] (52)--(53) node[below=1, midway, scale=.75]{$\bar{f\:\<}^{\!\<*}\<\<\bar\zeta$} ;
   
   \draw[->] (32)--(42) node[left=1, midway, scale=.75]{$\simeq$} ;
   \draw[->] (42)--(52) node[left=1, midway, scale=.75]{$\simeq$} 
                                  node[right=1, midway, scale=.75]{\eqref{phi to f}};

   \draw[->] (33)--(43) node[right=1, midway, scale=.75]{$\simeq$} ;
   \draw[->] (53)--(43) node[right=1, midway, scale=.75]{$\simeq$} ;
 \epic
\]

Conversely, a more concrete proof---eschewing adjoints---of this commutativity would, together with~\ref{flat and hom}, give~\ref{concrete zeta}.

\end{subrem}

\begin{small}
\begin{subxrz} Prove Proposition~\ref{concrete zeta} along the following lines:

Show how sheafification respects conjugacy of functorial maps. 
Then show that $\zeta_0$ and $\zeta$ are right-conjugate, respectively, to the projection maps 
\begin{align*}
p_\varphi\colon\varphi_*\SEE\Otimes{R}\SF&\lto\varphi_*(\SEE\Otimes{S}\LL\varphi^*\SF\>),\\
p^{}_{\!f}\colon\R\fst E\Otimes{Y}F&\lto\R\fst(E\Otimes{\sX}\LL f^*\<\<F\>). 
\end{align*}
(Cf.~\cite[4.2.3(e)]{li}.)
Finally, apply Lemma~\ref{dualproj}.
\end{subxrz}
\end{small}
\end{cosa}

\begin{cosa}[Completion of proof of ~\ref{Kozpf}] \label{transci}
Recall that ~\ref{Kozpf} reduces\va{-1} via~\ref{lem:Kozpf} to the situation where
the closed immersions\va2 $X\xto{\,f\,} Y\xto{\,g\,}Z$ correspond to surjective ring homomorphisms $R\overset{\lift{.5},\varphi,\;}\twoheadrightarrow S \overset{\lift{.8},\xi,\ }\twoheadrightarrow T$ whose kernels $J$ and~$I$\va1 are generated by Koszul-regular sequences $\mathbf r=(r^{}_1,\dots,r_e)$ and~$\mathbf{\bar s}=(\bar s^{}_1,\dots,\bar s_d)$ respectively, and, if $s_i\in R$ is such that  
\mbox{$\bar s_i=\xi(s_i)\ (1\le i\le d)$} then
the $R$-sequence $(\mathbf{r,s})=(r^{}_1,\dots,r_e,s^{}_1,\dots, s_d)$ is 
Koszul-regular and it generates the kernel $L$ of $\xi\varphi$.
Thus \ref{Kozpf} results from its commutative\kf-algebra analog, Proposition~\ref{affine version} below. 

\pagebreak[3]
Proposition~\ref{affine version}. involves commutative\kf-algebra versions of some of the maps in Lemma~\ref{alt c}. 
Set $\mathbf{\bar s}\set (\bar s^{}_1,\dots,\bar s_d)$, let $K_{\<S}(\mathbf{\bar s})$ be the associated Koszul complex,
let $\vartheta^{}_{\mathbf{\bar s}}\colon T^d\!\iso\! I/\<I^{\>2\>}$ be the $T$-isomorphism sending 
the~$i$-th canonical basis element of~ $T^d$ to $(\bar s_i+I^2)\in I/I^2\ (1\le i\le d),$ 
and let~ 
$\Ush c{\mathbf{\>\>\bar s}}$~be the natural composite $\D(T)$-isomorphism
\begin{align*}
N_\xi\set\Hom_T\!\big(\<{\lift1.7,{\bigwedge},\lift2,{\<d},}_{\mkern-7.5mu T^{\mathstrut}}\<\<
\big(\<I/\<I^{\>2\>}\big)\<,\<T\big)[-d\>]
&\xto{\<\<\via\,\vartheta^{}_{\mathbf{\bar s}}\>}
\Hom_T\!\big(\<{\lift1.7,{\bigwedge},\lift2,{\<d},}_{\mkern-7.5mu T^{\mathstrut}}\<\<
\big(T^{\>d}\>\big)\<,\<T\big)[-d\>]\\
&\iso
\xi^*\<\<\big(H^d\<\<\Hom_S(K_{\<S}(\mathbf{\bar s}),S)\big)[-d\>]\\
&\iso
\xi^*\<\<\big(H^d\R\<\<\<\Hom_S(\xi_*T,S)\big)[-d\>]\\
&\iso 
\big(\xi^*\<\<H^d\xi_*\R\<\<\<\Hom_\xi(T,S)\big)[-d\>]\\
&\iso
\big(H^d\R\<\<\Hom_\xi(T,S)\big)[-d\>]\\
&\iso
\R\<\<\Hom_\xi(T,S)=:\ush\xi \<S.
\end{align*}	
(The last isomorphism holds since  \mbox{$H^i\R\<\<\Hom_\xi(T,S)\!=\!H^i\<\<\Hom_S(K_{\<S}(\mathbf{\bar s}),S)\!=\!0$}
if $i\ne d$.)
In view of the relevant details of the equivalence described in section~\ref{affine aspect},
Lemma~\ref{alt c} implies that $\Ush c{\mathbf{\>\>\bar s}}$ is obtained (up to canonical isomorphism) by applying the global section functor to $\>c_{\<\<f}^{\>\flat}\colon \omega_{\<\<f}\iso f^\flat\OY$ in (\ref{KosReg}.1). It follows (or can be checked directly) that $\Ush c{\mathbf{\>\>\bar s}}$
does not depend on the choice of the Koszul-regular generating sequence $\mathbf{\bar s}$ of $I$,
so that it can---and will---be denoted by $\Ush c\xi$. \va2

The $\D(S)$-isomorphism \va{-3}
\[
\Ush c\varphi \colon N_\varphi\set\Hom_S\<\<\big(\<{\lift1.7,{\bigwedge},\lift2,{\<e},}_{\mkern-7.5mu S^{\mathstrut}}\<\<\big(\<J/\<J^2\big)\<,S\big)[-e]
\iso\ush\varphi \<R
\]
and the $\D(T)$-isomorphism\va{-2}
\[
\Ush c{{\xi\varphi}}\colon N_{\xi\varphi}\set\Hom_T\<\<\big(\<{\lift1.7,{\bigwedge},\lift2.4,{\<d+e},}_{\mkern-29mu T^{\mathstrut}}
\mkern10mu\big(\<L/\<L^{\>2\>}\big)\<,\<T\big)[-d-\<e]\iso\ush{(\xi\varphi)}\<R
\]
are defined analogously.\va2

Note that $N_\xi$ (resp. \kern-3pt$N_\varphi\>$, $N_{\xi\varphi}$) is flat over $T$ (resp. \kern-3pt$S$, $T$).

\begin{subprop}\label{affine version} 
The next\/ $\D(T)$-diagram, with\/ $\Hr\set\Hom,$  commutes, and its sheafification is isomorphic to the \/$\D(X)$-diagram in~\textup{\ref{Kozpf}.} 
\[
\def\1{$\ush\xi\<\ush\varphi \<R$}
\def\2{$\ush{(\xi\varphi)}\<R$}
\def\3{$N_\xi\otimes_{\mkern.5mu T}\xi^*\<\<N_\varphi$}
\def\4{$N_{\xi\varphi}$}
\def\5{$\ush\xi \<S\Otimes{T}\LL\>\xi^*\<\ush\varphi \<R$}
 \bpic[xscale=4.75, yscale=1.5]

   \node(11) at (1,-3){\1} ;   
   \node(12) at (2,-3){\2} ;
   
   \node(101) at (1,-2){\5} ;
    
   \node(21) at (1,-1){\3} ;  
   \node(22) at (2,-1){\4} ;
 
    \draw[->] (11)--(12)  node[above=1, midway, scale=.85] {\kern8pt\Iso}
                                   node[below=1, midway, scale=.75] {\kern8pt\textup{\eqref{pi}}} ;   
    
    \draw[->] (21)--(22)  node[above=1, midway, scale=.75] {\kern-8pt$h$}
                                     node[below=1, midway, scale=.75] {\kern-8pt\eqref{exact'aff}\kf} ;
    
    \draw[<-] (11)--(101) node[left=1, midway, scale=.75] {$\chi^{}_0(S,\ush\varphi \<R)$} 
                                     node[right, midway, scale=.75]{\eqref{chi0}} ;
                                   
    \draw[<-](101)--(21)  node[left, midway, scale=.9] {$\Ush c\xi\Otimes{T}\LL\>\xi^*
    \<\<\Ush c\varphi$} ;

    \draw[<-] (12)--(22)  node[right, midway, scale=.9]{$\Ush c{{\xi\varphi}}$} ;
 
  \epic
\]
\end{subprop}

\begin{proof} 
As regards sheafification, there are evident natural isomorphisms
\begin{align*}
\SH{\>T}\<N_\xi
&\iso
\upcheck{\big({\lift1.7,{\bigwedge},\lift2.4,{\<d},}_{\mkern-10mu X^{\mathstrut}}\<f^*\<\big(\CI\<\</\CI^{\>2\>}\big)\big)\!}[-d\>]=\colon\omega_{\<\<f}\\
\SH{\>T}\<N_\varphi
&\iso
f^*\<\upcheck{\big({\lift1.7,{\bigwedge},\lift2,{\<e},}_{\mkern-7mu Y^{\mathstrut}}\>g^*\<\big(\CJ\<\</\<\<\CJ^{\>2\>}\big)\big)\!}[-e]=\colon\omega_{\<g}\\
\SH{\>T}\<N_{\xi\varphi}
&\iso
\upcheck{\big({\lift1.7,{\bigwedge},\lift2.4,{\<d+e},}_{\mkern-33mu X^{\mathstrut}}\mkern12mu (g\<f)^{\<*}\<\big(\CL/\<\CL^{\>2\>}\big)\big)\!}[-d-\<\<e]=:\omega_{g\<f}\>,
\end{align*}
via which the sheafification of the top row in \ref{affine version} is naturally isomorphic to the top row in  \ref{Kozpf}.

Lemma~\ref{alt c}, \emph{mutatis mutandis,} shows the sheafification of the right column in \ref{affine version} to be isomorphic to the right column in  \ref{Kozpf}.  The top arrows in the respective left columns can be treated similarly.

That $\chi^{}_0$ in \ref{affine version} sheafifies to the map in \ref{Kozpf} labeled \ref{flat and tensor} 
is given by Proposition~\ref{concrete chi}.

That the sheafification of the bottom row in \ref{affine version} is naturally isomorphic to the bottom row in  \ref{Kozpf} is given by \eqref{sheafify pf}, \emph{mutatis mutandis}.
\va2

Next, to show the commutativity of the diagram in~\ref{affine version}, it suffices to show the commutativity of its adjoint, as follows.\va2

Let $p_{\<\varphi}$ and $p_\xi$ be as in \eqref{projaff} \emph{ff.}\va1

For a sequence $\mathbf q$ in a ring $Q\<$, let 
$\Kb{Q};{\mathbf q};$ be the complex $\Hom_Q(K_{\<\<Q}(\mathbf q),Q)$.
Let
\(
\mathfrak t_{\>\>\mathbf q}\colon\Kb{Q};{(\mathbf{q})};\to 
K_{\!Q}^0(\mathbf{q})=\Hom_Q(Q,Q)=Q
 \)
be the natural map. 

One has the composite natural isomorphism (with $\tilde\gamma^{}_0$ an isomorphism because the complex $K_R(\mathbf r)$ is strictly perfect)
\begin{align*}
\mathfrak g_{\mathbf r,\mathbf s}\colon\Kb{R};{\mathbf r};\otimes_R\Kb{R};{\mathbf s};
&=
\Hom_R(K_R(\mathbf r),R)\otimes_R\Hom_R(K_{\<\<R}(\mathbf s),R)\\
&\underset{\lift1.3,\tilde\gamma^{}_0,}\iso
\Hom_R(K_{\<\<R}(\mathbf r),\Hom_R(K_{\<\<R}(\mathbf s),R))\\
&\iso
\Hom_R(K_{\<\<R}(\mathbf r)\<\<\otimes_R\< K_{\<\<R}(\mathbf s),R)
\iso
\Kb{R};{\mathbf{(r,s)}};.
\end{align*}

Let $d_{\>\mathbf{\bar s}}$ be the natural composite $\D(S)$-isomorphism\va{-1}
\[
\xi_*N_\xi \xto{\<\xi_*\Ush c{\xi}\>}\xi_*\ush\xi \<S
\iso\R\<\<\Hom_S(\xi_*T,S)\iso \Kb{S};{\mathbf{\bar s}};,
\]
and define analogously\va{-2}
\[
d_{\>\mathbf r}\colon
\varphi_*N_\varphi \iso \Kb{R};{\mathbf r};,\qquad
d_{\>\mathbf r\<, \>\mathbf s}\colon
\varphi_*\xi_*N_{\xi\varphi} \iso \Kb{R};{(\mathbf r,\mathbf s)};.
\]
\vskip2pt\stepcounter{equation}
Now expand the adjoint diagram naturally as in diagram (\theequation)\  below.
Diagram-chasing shows that proving commutativity of all the subdiagrams of this expanded
diagram gives the desired commutativity of the adjoint diagram itself. 

\pagebreak
For each of the unlabeled subdiagrams, commutativity is either obvious or straightforward to verify.

\begin{small}
\begin{figure}[t]
\[
 \def\1{$\Kb{R};{(\mathbf r, \mathbf s)};$} 
 \def\2{$\Kb{R};{\mathbf r};\otimes_R \Kb{R};{\mathbf s};$} 
 \def\3{$\ \varphi_{\<*}\ush\varphi\<\<R\<\otimes_R\<\Kb R;\mathbf s;\ \,$} 
 \def\4{$\varphi_{\<*}(\ush\varphi\<\<R\otimes_S\Kb S;\mathbf{\bar s};)$} 
 \def\5{$\!\varphi_{\<*}(S\otimes_{S}\<\ush\varphi\<\<R)$} 
 \def\7{$\varphi_{\<*}\ush\varphi\<\<R$} 
 \def\9{$R$} 
 \def\ten{$\varphi_{\<*}N_{\<\varphi}\otimes_R \Kb{R};{\mathbf s};$} 
 \def\lvn{$\varphi_{\<*}\ush\varphi \<R$} 
 \def\thn{$\varphi_{\<*}( N_{\<\varphi}\<\<\otimes_S\<\<\Kb{S};{\mathbf{\bar s}};)$} 
 \def\frn{$\ \varphi_{\<*}( \Kb{S};{\mathbf{\bar s}};\otimes_S\<N_{\<\varphi})$} 
 \def\ffn{$\varphi_{\<*}(\xi_*N_{\xi}\otimes_S\<N_{\<\varphi})$} 
 \def\sxn{$\varphi_{\<*}(\xi_*\ush\xi S\Otimes{S}\ush\varphi \<R)$} 
 \def\svn{$\varphi_{\<*}\xi_*(\<N_\xi\otimes_{\mkern.5mu T}\xi^*\<\<N_\varphi\<)$} 
 \def\egn{$\varphi_{\<*}\xi_*N_{\xi\varphi}$} 
 \def\ntn{$\varphi_{\<*}\xi_*(\ush\xi \<S\<\Otimes{T}\<\LL\>\xi^*\<\<\ush\varphi \<R)$} 
 \def\twy{$\ \varphi_{\<*}\xi_*\ush\xi\ush\varphi \<R$} 
 \def\twn{$\varphi_{\<*}\xi_*\ush{(\xi\varphi)}\<R$} 
 \def\twt{$\varphi_{\<*}(\ush\varphi\<\<R\otimes_{S}S)$} 
 \def\twth{$(R\otimes_{R}R)\ $} 
 \def\twfr{$(R\otimes_{R}\<\Kb{R};{\mathbf s};)$} 
 \def\twfv{$\varphi_{\<*}\ush\varphi\<\<R\otimes_{R}R$}
\bpic[xscale=2.45, yscale=1.5] 

   \node(11) at (1,-.8){\svn} ;  
   \node(15) at (5,-.8){\egn} ;  
  
   \node(22) at (1.9,-2){\frn} ;  
   \node(23) at (3,-1.4){\thn} ;
   \node(24) at (4.1,-2.35){\1} ;  
  
   \node(33) at (3,-2.75){\ten} ;

   \node(43) at (3,-4.285){\2} ;
   \node(44) at (4.1,-4.285){\9} ; 
   
   \node(52) at (1.9,-5){\ffn} ;
   \node(53) at (3.55,-5){\twfr} ;    
   \node(54) at (4.1,-5.93){\twth} ; 

   \node(63) at (3,-5.93){\3} ;
   \node(65) at (5,-5.93){\twn} ; 

   \node(73) at (3,-7.35){\4} ;

   \node(82) at (1.9,-6.64){\sxn} ;
   \node(84) at (4.1,-7.35){\twfv} ;
   
   \node(92) at (1.9,-8.84){\5} ;
   \node(93) at (3,-8.84){\twt} ;
   \node(94) at (4.1,-8.84){\7} ;
  
   \node(101) at (1,-10){\ntn} ;  
   \node(105) at (5,-10){\twy} ;

    \draw[->] (11)--(15) node[above, midway, scale=.75] {$\varphi_{\<*}\xi_*h$}
                                    node[below, midway, scale=.75] {\eqref{exact'aff}};
    
     \draw[->] (101)--(105) node[below, midway, scale=.75] {$\varphi_{\<*}\xi_*\chi^{}_0(S,\ush\varphi \<R)\qquad\quad$} ;

     \draw[double distance=2] (92)--(93) ; 
     
    \draw[->] (11)--(101)  node[left, midway, scale=.75]{$\varphi_*\xi_*(\Ush c{{\xi}}\<\otimes_T\<\Ush c{{\varphi^{}}})$} ;

    \draw[<-] (22)--(52) node[right, midway, scale=.85]
                {$\textup{\scriptsize via}_{}\,d_{\>\mathbf{\bar s}}$} ;
    \draw[->] (52)--(82) node[left, midway, scale=.85]{$\textup{\scriptsize via}_{}\<\<$} 
                                    node[right=-1, midway, scale=.85]{$\Ush c{\xi}\mkern4mu                                    
                                    \textup{\scriptsize and}\, \Ush c{\varphi}$} ;
                                    
    \draw[->] (23)--(33) node[left=-.5, midway, scale=.75]{$p^{-\<1}_{\mkern-1.5mu\varphi} $} ;
    \draw[<-] (43)--(33) node[left=-1, midway, scale=.85]
                                 {$\textup{\scriptsize via}_{}\,d_{\>\mathbf r}$} ;
    \draw[->] (43)--(63) ;
    \draw[->] (63)--(73) node[left=-.5, midway, scale=.75]{$p^{}_{\mkern-1.5mu\varphi} $} ;
    \draw[->] (24)--(44) node[left, midway, scale=.85]{$\mathfrak t_{\>\>\mathbf{r\<,\>s}}$} ;
    \draw[double distance=2] (54)--(44) ;
    \draw[->] (94)--(84) ;
    
    \draw[->] (15)--(65) node[right, midway, scale=.75]{$\varphi_*\xi_*\Ush c{{\xi\varphi^{}}}$} ;
    \draw[->] (105)--(65) node[right, midway, scale=.75]{\eqref{pi}};

     \draw[<-] (23)--(22) node[below=-1.85,midway, scale=.75]{$\mkern37mu\varphi_{\<*}s$}; 
     \draw[<-] (24)--(43) node[left=1,midway, scale=.85]{$\mathfrak g_{\mathbf{r\<,\>s}}\mkern2mu$} ;
     \draw[->] (63)--(53) ;
     \draw[->] (3.565,-5.22)--(54) node[left, midway, scale=.85]{$\textup{\scriptsize via}_{}\<\<$} 
                                    node[right=-1, midway, scale=.85]{$\mkern3mu\mathfrak t_{\>\>\mathbf{s}}$} ;
     \draw[->] (63)--(84) node[left, midway, scale=.85]{$\textup{\scriptsize via}_{}\<\<$} 
                                    node[right=-1, midway, scale=.85]{$\mkern3mu\mathfrak t_{\>\>\mathbf{s}}$} ;
     \draw[->] (84)--(54) ;
     \draw[->] (43)--(44) node[above,midway, scale=.85]
                        {$\!\mathfrak t_{\>\>\mathbf r}\<\<\otimes_{\<R}\<\<\mathfrak t_{\>\>\mathbf s}\>$} ;
     \draw[->] (65)--(44) ;
     \draw[->] (15)--(24) node[above=-2, midway, scale=.85]{$d_{\>\mathbf r\<,\>\mathbf s}         
                                      \mkern25mu$} ;
     \draw[<-] (4.23,-2.55)--(4.94,-5.72) ;
     \draw[->] (1.1,-1.01)--(1.79,-4.79) node[left=-.5, midway, scale=.75]{$\varphi_{\<*} p^{}_{\<\xi}$} ;
     \draw[->] (1.1,-9.78)--(1.79,-6.85) node[left=-.5, midway, scale=.75]{$\varphi_{\<*} p^{}_{\<\xi}\>$} ;
     \draw[->] (82)--(92) ;
     \draw[->] (93)--(84) node[below=-2.7, midway, scale=.75]
                                     {$\mkern28mu p^{-\<1}_{\mkern-1.5mu \varphi}$} ;
     \draw[->] (93)--(94) ;
     \draw[->] (4.89,-9.8)--(94) ;
     \draw[->] (73)--(93) node[left, midway, scale=.85]{$\textup{\scriptsize via}_{}\<\<$} 
                                    node[right=-1, midway, scale=.85]{$\mathfrak t_{\>\>\mathbf{\bar s}}$} ;     
    \node at (1.9,-1.42)[scale=.9]{\circled1} ;
    \node at (2.45,-5.43)[scale=.9]{\circled2} ;
    \node at (4.55,-6.64)[scale=.9]{\circled3} ;
    \node at (3,-9.52)[scale=.9]{\circled4} ;

 \epic
\]
\centerline{\bf\qquad(\theequation)}
\end{figure}
\end{small}

For the commutativity of \circled2, use the commutativity (resulting directly from 
the definition of $d_{\>\mathbf r}$) of
\[
\def\1{$\varphi_{\<*}( N_{\<\varphi}\otimes_S\<\Kb{S};{\mathbf{\bar s}};)$}
\def\2{$\varphi_{\<*} N_{\<\varphi}\otimes_R\Kb{R};{\mathbf{s}};$}
\def\3{$\Kb{R};{\mathbf r};\otimes_R \Kb{R};{\mathbf s};$}
\def\4{$\varphi_{\<*}(\ush\varphi\<\<R\otimes_S\<\Kb{S};{\mathbf{\bar s}};)$}
\def\5{$\varphi_{\<*}\ush\varphi\<\<R\otimes_R\Kb R;\mathbf s;,$}
  \bpic[xscale=4.5, yscale=1.35]

   \node(32) at (1,-3){\1} ;
   \node(33) at (2,-3){\2};

   \node(53) at (2,-4){\3};
  
   \node(62) at (1,-5){\4};
   \node(63) at (2, -5){\5} ;

   \draw[->] (33)--(32) node[above, midway, scale=.75]{$\ p_{\varphi}$} ;
   
   \draw[->] (63)--(62) node[below, midway, scale=.75]{$\ p_{\varphi}$} ;
   
   \draw[->] (32)--(62) node[left, midway, scale=.75]{$\via\>\Ush c{\varphi}$} ;
   
   \draw[->] (33)--(53) node[right=1, midway, scale=.75]{$\via\>d_{\>\mathbf r}$} ;
   \draw[->] (53)--(63) node[right=1, midway, scale=.75]{natural} ;

 \epic
\]
plus the definition of $d_{\>\mathbf{\bar s}}\>$, to reduce to noting the obvious commutativity of\va{-2} 
\[
\def\1{$\varphi_{\<*}( \Kb{S};{\mathbf{\bar s}};\otimes_S\<N_{\<\varphi})$}
\def\2{$\varphi_{\<*}(N_{\<\varphi}\otimes_S\<\Kb{S};{\mathbf{\bar s}};)$}
\def\3{$\varphi_{\<*}( \Kb{S};{\mathbf{\bar s}};\otimes_S\<\ush\varphi\<\<R)$}
\def\4{$\varphi_{\<*}(\ush\varphi\<\<R\otimes_S\<\Kb{S};{\mathbf{\bar s}};)$}
\def\5{$\varphi_{\<*}(S\otimes_S\<\ush\varphi\<\<R)$}
\def\6{$\varphi_{\<*}(\ush\varphi\<\<R\otimes_SS)$}
  \bpic[xscale=4.5, yscale=1.25]

   \node(32) at (1,-3){\1};
   \node(33) at (2,-3){\2};
      
   \node(52) at (1,-4){\3};
   \node(53) at (2,-4){\4};
  
   \node(62) at (1,-5){\5};
   \node(63) at (2, -5){\6};

   \draw[->] (33)--(32) node[above=1, midway, scale=.75]{\kern1pt natural} ;
   
   \draw[->] (53)--(52) node[above=1, midway, scale=.75]{\kern1pt natural} ;
   
\draw[double distance=2pt] (63)--(62) ;
   
   \draw[->] (32)--(52) node[left, midway, scale=.75]{$\via\>\Ush c{\varphi}$} ;
   \draw[->] (52)--(62) node[left, midway, scale=.75]{$\via\>\mathfrak t_{\>\>\mathbf{\bar s}}$} ;
   
   \draw[->] (33)--(53) node[right, midway, scale=.75]{$\via\>\Ush c{\varphi}$} ;
   \draw[->] (53)--(63) node[right, midway, scale=.75]{$\via\>\mathfrak t_{\>\>\mathbf{\bar s}}$} ;

 \epic
\]

The commutativity of \circled3 follows from the definition of the map \eqref{pi}.

The commutativity of \circled4 is given by Lemma~\ref{dualchi}, \emph{mutatis mutandis}..

It remains to verify the commutativity of \circled1, or equivalently, 
of the natural diagram \va{-2}
\[
 \def\1{$\Kb{R};{(\mathbf r, \mathbf s)};$} 
 \def\2{$\Kb{R};{\mathbf r};\otimes_R \Kb{R};{\mathbf s};$} 
 \def\ten{$\varphi_{\<*}N_{\<\varphi}\otimes_R \Kb{R};{\mathbf s};$} 
 \def\thn{$\varphi_{\<*}( N_{\<\varphi}\<\<\otimes_S\<\<\Kb{S};{\mathbf{\bar s}};)$} 
 \def\frn{$\ \varphi_{\<*}( \Kb{S};{\mathbf{\bar s}};\otimes_S\<N_{\<\varphi})$} 
 \def\ffn{$\varphi_{\<*}(\xi_*N_{\xi}\otimes_S\<N_{\<\varphi})$} 
 \def\svn{$\varphi_{\<*}\xi_*(\<N_\xi\otimes_{\mkern.5mu T}\xi^*\<\<N_\varphi\<)$} 
 \def\egn{$\varphi_{\<*}\xi_*N_{\xi\varphi}$} 
 \bpic[xscale=2.25, yscale=1.65] 

   \node(11) at (1,-1){\svn} ;  
   \node(15) at (3,-1){\egn} ;  
   \node(24) at (5,-1){\1} ;

   \node(43) at (5,-2){\2} ;
   \node(52) at (1,-2){\ffn} ;
      
   \node(22) at (1,-3){\frn} ;  
   \node(23) at (3,-3){\thn} ;
   \node(33) at (5,-3){\ten} ;
   
    \draw[<-] (11)--(15) node[above, midway, scale=.75] {$\varphi_{\<*}\xi_*h^{-\<1}$} ;
    \draw[<-] (23)--(33) node[below, midway, scale=.75]{$p_{\mkern-1.5mu\varphi} $} ;

    \draw[->] (22)--(52) node[left, midway, scale=.8]
                                {$\via d\>^{-\<1}_{\mathbf{\bar s}}$} ;
    \draw[->] (43)--(33) node[right, midway, scale=.8]
                                 {$\via d\>^{-\<1}_{\mathbf r}$} ;
 
     \draw[->] (23)--(22) node[below, midway, scale=.75]{$\varphi_*s $} ; 
     \draw[->] (24)--(43) node[right, midway, scale=.8]{$\mathfrak g^{-1}_{\mathbf{r\<,\>s}}\mkern3mu$} ;
     \draw[<-] (15)--(24) node[above, midway, scale=.8]
                                    {$\>d^{-\<1}_{\mathbf r\<,\>\mathbf s}$} ;
     \draw[->] (11)--(52) node[left=-.5, midway, scale=.75]{$\varphi_{\<*} p^{}_{\<\xi}$} ;
     
      \node at (3,-2)[scale=.9]{\circled1$'$} ; 
 \epic
\]
\vskip-2pt
This diagram is the canonical image in $\D(R)$ of an explicitly describable diagram in the category of $R$-complexes.
To see this, represent $d\>^{-\<1}_{\mathbf{\bar s}}$ 
as the image in~$\D(S)$ of a map of complexes, as follows (and analogously for 
$d\>^{-\<1}_{\mathbf{r}}$ and $d\>^{-\<1}_{\mathbf{r,s}}$):\va{-2}
\begin{sublem}\label{explicit d}
The map\/ $d\>^{-\<1}_{\mathbf{\bar s}}$ is the canonical image in\/~$\D(S)$ of the natural composite
map of\/~$S$-complexes
\begin{align*}
d\>'_{\mathbf{\bar s}}\colon \Kb{S};{\mathbf{\bar s}}; 
\lto 
\big(H^d\<\Kb{S};{\mathbf{\bar s}}; \big)[-d\>]
&\iso\!
\big(K^d_S(\mathbf{\bar s})\otimes_{\<S}\>\xi_*T\big)[-d\>]\\
&\iso\<
\xi_*\Hom_T\!\big(\lift1.7,{\bigwedge},\lift2.4,{\<d},_{\mkern-7mu T^{\mathstrut}}\<\<\big(T^{\>d}\>\big),T\big)[-d\>]\\
&\underset{\lift.8,\<\via\vartheta^{}_{\mathbf{\bar s}},\!}{\iso}\<
\xi_*\Hom_T\!\big(\lift1.7,{\bigwedge},\lift2.4,{\<d},_{\mkern-7mu T^{\mathstrut}}\<\<\big(I/I^{\>2\>}\big),T\big)[-d\>]
=
\xi_*N_\xi.
\end{align*}
\end{sublem}

\begin{proof} It suffices to prove equality of the adjoint maps, that is, that the border of the following natural diagram commutes:
\[
 \def\1{$\xi^*\<\<\Kb{S};{\mathbf{\bar s}};$} 
 \def\2{$\xi^*\<\<\big(H^d\<\Kb{S};{\mathbf{\bar s}}; \big)[-d\>]$} 
 \def\4{$\big(H^d\ush\xi \<S\big)[-d\>]$} 
\def\5{$\big(H^d\xi^*\<\<\Kb{S};{\mathbf{\bar s}};\big)[-d\>]$} 
 \def\6{$\ush\xi \<S$} 
 \def\8{$\xi^*\<\<\big(K^d_S(\mathbf{\bar s})\otimes_{\<S}\>\xi_*T\big)[-d\>]$}
 \def\9{$\xi^*\<\<\Hom_S\!\big(\lift1.7,{\bigwedge},\lift2.4,{\<d},_{\mkern-9mu S^{\mathstrut}}\<\<
               \big(I/I^{\>2\>}\big),S/I\big)[-d\>]$} 
\def\ten{$H^d\<\<\Hom_T(\xi^*\<\<\<K_{S}(\mathbf{\bar s}),T\>)[-d\>]$} 
 \def\lvn{$N_\xi$}
 \def\twv{$\Hom_T\!\big(\lift1.7,{\bigwedge},\lift2.4,{\<d},_{\mkern-7mu T^{\mathstrut}}\<\<\big(I/I^{\>2\>}\big),T\big)[-d\>]$} 
 \def\thn{$\Hom_T\!\big(\lift1.7,{\bigwedge},\lift2.4,{\<d},_{\mkern-7mu T^{\mathstrut}}\<\<\big(T^{\>d}\>\big),T\big)[-d\>]$} 
 \bpic[xscale=3, yscale=1.7] 

   \node(11) at (1,-1){\1} ;  
   \node(14) at (4,-1){\2} ;  
  
   \node(21) at (1,-1.95){\6} ;  
   \node(22) at (1.9,-1.95){\4} ;  
   \node(23) at (3.1,-1.95){\5} ;
  
   \node(34) at (4,-2.74){\8} ;

   \node(43) at (3.1,-3.53){\ten} ; 
   
   \node(51) at (1,-4.53){\lvn} ;  
   \node(52) at (2,-4.53){\twv} ;  
   \node(54) at (4,-4.53){\thn} ;

    \draw[->] (11)--(14) node[above, midway, scale=.75] {$$}
                                    node[above, midway, scale=.75] {};
    
     \draw[->] (21)--(22) ;

     \draw[double distance=2] (51)--(52) ; 
     \draw[<-] (52)--(54)  node[below, midway, scale=.75] {$\via \vartheta^{}_{\mathbf{\bar s}}$} ;

    \draw[->] (11)--(21) ;
    \draw[->] (21)--(51) node[left, midway, scale=.75] {$\xi_*\Ush c{\xi}^{-\<1}$} ;  
    
    \draw[->] (23)--(43) ; 
    
    \draw[->] (14)--(34) ;
    \draw[->] (34)--(54) ;  
    
     \draw[<-] (23)--(11) ;
     \draw[<-] (23)--(14) ; 
     \draw[->] (22)--(23) ;
     \draw[<-] (54)--(43) ;

    \node at (1.95,-3.24)[scale=.9]{\circled5} ;

 \epic
\]

Showing commutativity of subdiagram \circled5 is a minor variant of showing 
that $\Ush c\xi$ is identifiable with $\Gamma(X,c^{\>\flat}_{\<\<f})$
(see the third paragraph before~\ref{affine version}). Details are left to the reader.

The commutativity of the unlabeled subdiagrams is easy to check.

The desired conclusion results.
\end{proof}

In continuation of the proof of \ref{affine version}, to describe $d\>'_{\mathbf{\bar s}}$ more explicitly, the following abbreviations are helpful.\va1

With $r_{\!i}^L\set(r^{}_{\!i}+L^2)\in L/L^2$, and so on, one has the generators
 \begin{alignat*}{2}
\mathbf{r}^J&\set r^J_{\!\lift1,1,}\wedge\dots\wedge r^J_{\!e}\quad&&\textup{of\quad}
   \lift1.7,{\bigwedge},\lift2,{\<e},_{\mkern-8mu S^{\mathstrut}}\<\<\big(J/J^{\>2\>}\big),\\
\mathbf{\bar s}^I&\set \bar s^I_1\wedge\dots\wedge \bar s^I_d\quad&&\textup{of\quad}
   \lift1.7,{\bigwedge},\lift2.4,{\<d},_{\mkern-7mu T^{\mathstrut}}\<\<\big(I/I^{\>2\>}\big),\\
\mathbf{(r,s)}^L&\set r^L_{\!\lift1,1,}\wedge\dots\wedge r^L_{\!e}\wedge s^L_1\wedge\dots\wedge s^L_d 
\quad&&\textup{of\quad}
    \lift1.7,\bigwedge,\lift2.4,{\<d+e},_{\mkern-29mu T^{\mathstrut}}\mkern10mu\big(\<L/\<L^{\>2\>}\big).\end{alignat*}

With $(v_1,\dots,v_{d+e})$ the standard basis of $R^{d+e}\<=\<R^d\oplus R^e\<$, and $(w_1,\dots,w_d)$ the standard basis of $S^d\<$, one has the generators
\begin{alignat*}{2}
\mathbf{v}^{d+e}&\set v_1\wedge\dots\wedge v_{d+e}\quad&&\textup{of\quad}
   \lift1.7,{\bigwedge},\lift2,{\<{d+e}},_{\mkern-30mu R^{\mathstrut}}\mkern10mu\big(\<R^{\>d+e}\big),\\
\mathbf{v}^d&\set v_1\wedge\dots\wedge v_d \quad&&\textup{of\quad}
   \lift1.7,{\bigwedge},\lift2,{\<d},_{\mkern-8mu R^{\mathstrut}}\<\<\big(\<R^d\>\big),\\
\mathbf{v}^{d,e}&\set v_{d+1}\wedge\dots\wedge v_{d+e}\quad&&\textup{of\quad}
   \lift1.7,{\bigwedge},\lift2,{\<e},_{\mkern-8mu R^{\mathstrut}}\<\<\big(\<R^e\>\big),\\
\mathbf{w}^d&\set w_1\wedge\dots\wedge w_d\quad&&\textup{of\quad}
   \lift1.7,{\bigwedge},\lift2,{\<d},_{\mkern-8mu R^{\mathstrut}}\<\<\big(\<S^d\>\big).
\end{alignat*}

For a ring $Q$ and a generator $\mathbf g$  of a rank-one free $Q$-module $G$, denote
 by~$\mathbf 1_Q/\mathbf g$ the map in $\Hom_Q(G,Q)$ that takes $\mathbf g$ to the identity~$1_Q$
 of $Q\>$.\va2

One checks that in degree $d+e$, 
\[
\mathfrak g_{\mathbf r,\mathbf s}\big((1_R/\mathbf{v}^{d,e})\otimes_R (1_R/\mathbf{v}^d)\big)
=(-1)^{de}(1_R/\mathbf{v}^{d+e}).
\]

 \pagebreak[3]
Further, in degree $d$, 
\[
d\>'_{\mathbf{\bar s}}(1_S/\mathbf w^d)=1_T/\>\mathbf{\bar s}^I,
\]
and analogously,
\begin{alignat*}{2}
d\>'_{\mathbf{r}}(1_R/\mathbf{v}^{d,e})&=1_S/\mathbf{r}^J   &&\qquad\textup{(in degree $e$),}\\ 
d\>'_{\mathbf{r,s}}(1_R/\mathbf{v}^{d+e})&= 1_T/\mathbf{(r,s)}^L    &&\qquad\textup{(in degree $d+e$)}.
\end{alignat*}

Now one need only verify commutativity of the diagram of $R$-complexes represented by \circled1$'$  
with $d\>^{-\<1}$  replaced by $d\>'\<$; and for this, one need only look at what happens to  the generator 
$1_R/\mathbf{v}^{d+e}$ of 
the $R$-module $K_{\!R}^{d+e}((\mathbf r, \mathbf s))$.\va1

 Recalling \eqref{exact'aff}, one checks that moving around counterclockwise from 
 $K_{\!R}^{d+e}((\mathbf r, \mathbf s))$ to $\varphi_{\<*}(\xi_*N_\xi\otimes_S N_\varphi)^{d+e}$
 acts successively on $1_R/\mathbf{v}^{d+e}$ as:
 \begin{align*}
 1_R/\mathbf{v}^{d+e}
&\longmapsto
 1_T/\mathbf{(r,s)}^L\\
& \longmapsto
 (1_T/\mathbf{\bar s}^I)\otimes_T(1_T\otimes_S1_S/\mathbf{r^J})
 \longmapsto
 (1_T\</\>\mathbf{\bar s}^I)\otimes_S (1_S/ \mathbf{r}^J).
 \end{align*}

On the other hand, moving clockwise 
 acts successively on $1_R/\mathbf{v}^{d+e}$ as:
\begin{align*}
 1_R/\mathbf{v}^{d+e}
&\longmapsto
(-1)^{de}(1_R/\mathbf{v}^{d,e}\otimes_R(1_R/\mathbf{v}^d\>)\\
&\longmapsto
(-1)^{de}(1_S/\mathbf{r}^J)\otimes_R(1_R/\mathbf{v}^d\>)\\
&\longmapsto
(-1)^{de}(1_S/\mathbf{r}^J)\otimes_S(1_S/\mathbf{w}^d\>)\qquad\textup{(see \eqref{projaff} \emph{ff.})}\\
&\longmapsto
(1_S/\mathbf{w}^d\>)\otimes_S(1_S/\mathbf{r}^J) 
\longmapsto
  (1_T\</\>\mathbf{\bar s}^I)\otimes_S (1_S/ \mathbf{r}^J).
 \end{align*}

This completes the proof.
\end{proof}

\end{cosa}

\pagebreak


\begin{thebibliography}{AJL110}


%

\bibitem[AJL99]{AJL99}  Alonso Tarr\'io, L; Jerem\'ias L\'opez, A; Lipman, J:   Duality and flat base change on formal schemes.   \textit{Contemporary Math.}  \textbf{244}  (1999), 3--90.

\bibitem[AJL11]{AJL11} \bysame: Bivariance, Grothendieck duality and Hochschild homology I:
construction of a bivariant theory.    \textit{Asian J.\ Math.}  \textbf{15}  (2011), 451--498.

\bibitem[AJL14\kf]{AJL14} \bysame:
Bivariance, Grothendieck duality\- and Hochschild homology, II: The fundamental class of 
a flat scheme\kf-map. \textit{Advances in Mathematics} \textbf{257} (2014), 365--461.

%
%
%

\bibitem[BN93]{BN} B\"okstedt, M.; Neeman, A.:
Homotopy limits in triangulated categories.  \textit{Compo\-sitio Math.} \textbf{86} (1993), no.~2, 209--234.

\bibitem[B70]{Bo70} Bourbaki, N.: \textit{Alg\`ebre, Chapitres 1 \`a 3.}  Hermann, Paris, 1970.


\bibitem[B07]{Bo} \bysame: \textit{Alg\`ebre, Chapitre 10. Alg\`ebre homologique.} Springer-Verlag, Berlin-New York, 2007.

%

%
%

\bibitem[CD19]{CD}  Cisinski, D-C.; D\'eglise, F.:
Triangulated categories of mixed motives. \newline {\tt arXiv:0912.2110v4}.

\bibitem[Co00]{Co} Conrad, B.: \textit{Grothendieck Duality and Base Change.} Lecture Notes in Math.,  \textbf{1750}. Springer-Verlag, Berlin-New York, 2000.

\bibitem[De73]{De73} Deligne, P.: La formule de dualit\'e globale. \textit{Th\'eorie des topos et cohomologie \'etale des schemas $($SGA 4$)$ Tome 3,} 481--587, Lecture Notes in Math.,  \textbf{305}. Springer-Verlag, Berlin-New York, 1973.

%
%

\bibitem[Ga13]{Ga} Gaitsgory, D.:
Ind-coherent sheaves. \textit{Moscow Math. J.} \textbf{13} (2013), 399--528.

%
%

\bibitem[GrD61a]{EGA2} Grothendieck, A.; Dieudonn\'{e}, J. A.:
{\it El\'{e}ments  de  G\'{e}om\'{e}trie Alg\'{e}brique II, \'{E}tude
global \'elementaire de quelques classes de morphismes.}  Publications Math\'{e}matiques, {\bf 8},
Institut\- des Hautes \'{E}tudes  Scientifiques, Paris, 1961.

\bibitem[GrD61b]{EGA3} \bysame:
{\it El\'{e}ments  de  G\'{e}om\'{e}trie Alg\'{e}brique III, \'{E}tude
cohomologique des faisceaux coh\'erents.}  Publications Math\'{e}matiques, {\bf 11},
Institut\- des Hautes \'{E}tudes  Scientifiques, Paris, 1961.

%

\bibitem[GrD71]{EGA1} \bysame:
{\it El\'{e}ments de   G\'{e}om\'{e}trie Alg\'{e}brique I}. Grundlehren der
Mathematischen Wissenschaften  {\bf 166}. Sprin\-ger-Verlag, Berlin-New York, 1971.

\bibitem[H66]{RD} Hartshorne, R.: \textit{Residues and Duality.} Lecture
Notes in Math., \textbf{20}. Springer-Verlag, Berlin-New York, 1966.

\bibitem[Ha09]{Ha} Hashimoto, M.: Equivariant twisted Inverses. \textit{Foundations of Grothendieck duality for diagrams of schemes,} 261--475, Lecture Notes in Math., {\bf 1960}, Springer-Verlag, Berlin-New York, 2009.

\bibitem[HK90]{HK}
H\"ubl, R.; Kunz, E.:
Regular differential forms and duality for projective morphisms.
{\it J. reine Angew. Math.} {\bf 410} (1990), 84--108.

\bibitem[Ho22]{Ho}  H\"ormann, F.:
Derivator six-functor-formalisms--construction II. \newline {\tt arXiv:1902.03625v3}.

\bibitem[HS93]{HS}
H\"ubl, R.; Sastry, P.:
Regular differential forms and relative duality.
\emph{Amer. J. Math.} \textbf{115}  (1993), no.\ 4, 749--787.

\bibitem[Il71]{Il} Illusie, L.:
G\'en\'eralit\'es sur les conditions de finitude dans les cat\'egories d\'eriv\'ees, etc.
\textit{Th\'eorie des Intersections et Th\'eor\`eme de Riemann-Roch (SGA\;6).}
Lecture Notes in Math.\ \textbf{225},
Springer-\kern-.5pt Verlag, New York, 1971, 78--296.

%

\bibitem[KM71]{KM} Kelly, G.\,M.; Mac\,Lane,  S.: Coherence in closed categories. \emph{J. Pure Appl. Algebra} 
\textbf{1} (1971),  no. 1, 97--140.


%
%
%
%

\bibitem[L09a]{lil} Lipman, J.: Introduction to Grothendieck Duality.   (YMIS 09,  Sedano, Spain, 2009.)
{\tt www.math.purdue.edu/\~{}lipman.}

\bibitem[L09]{li} \bysame: Notes on derived categories and Grothendieck Duality. \textit{Foundations of Grothendieck duality for diagrams of schemes,} 1--259, Lecture Notes in Math., {\bf 1960}, Springer-Verlag, Berlin-New York, 2009.

\bibitem[LN07\kf]{LN07} \bysame; Neeman, A.:
Quasi-perfect scheme\kf-maps and boundedness of the twisted inverse image functor. 
\emph{Illinois J. Math.} \textbf{51} (2007), no.\ 1, 209--236.

\bibitem[LN17\kf]{LN17} \bysame; Neeman, A.:
On the fundamental class of an essentially smooth scheme\kf-map.
\emph{Algebraic Geometry} \textbf{5} (2018), 131--159.
         
\bibitem[LS92]{LS92} \bysame; Sastry, P.: Regular differentials and equidimensional scheme\kf-maps. 
\emph{J. Algebraic Geom.} \textbf{1} (1992), no.\ 1, 101--130. 

\bibitem[Lw72]{le} Lewis, G.: Coherence for a closed functor. \textit{Coherence in categories,}  148--195, Lecture Notes in Math., {\bf 281}, Springer-Verlag, Berlin-New York, 1972. 


\bibitem[LZ24\kf]{LZ}  Liu, Y.; Zheng, W.:
Enhanced six operations and base change theorem for higher artin stacks. {\tt arXiv:1211.5948v4}.

%

\bibitem[M98]{M} 
Mac\,Lane, S.: \textit{Categories for the Working Mathematician, Second Edition.}
Springer-Verlag,  Berlin-New York, 1998.

\bibitem[Nk05]{Nk05} Nayak, S.:
Pasting pseudofunctors. \textit{Variance and duality for Cousin complexes on formal schemes,} 195--271, Contemp. Math., \textbf{375}, Amer. Math. Soc., Providence, RI, 2005. 

\bibitem[Nk09]{Nk09} \bysame: 
Compactification for essentially finite type maps.
\emph{Adv.~Math.} \textbf{222} (2009),  527--546.

\bibitem[NkS19]{NS} \bysame; Sastry, P.: 
Grothendieck duality and transitivity for formal schemes. 
Revised version forthcoming in Fields Institute Monographs. (See also {\tt arXiv:1903.01779v3}.)

%

\bibitem[Nm96]{Nm96} Neeman, A.:
The Grothendieck duality theorem via Bousfield's techniques and
Brown representability.  
\emph{J.\ Amer.\ Math.\ Soc.} \textbf{9} (1996),   205--236.

\bibitem[Nm23\kf]{Nm23}  \bysame: 
An improvement on the base\kf-change theorem and 
the functor $f^!$. \emph{Bulletin\- of the Iranian Mathematical Society} \textbf{49:25} (2023),
163pp.

\bibitem[PS14]{PS} Ponto, K.; Shulman, M:
Traces in symmetric monoidal categories. \emph{Expo. Math}. \textbf{32} (2014), 248--273.


\bibitem[Sch18]{Sch}  Schn\"urer, O.\,M.:
Six operations on DG enhancements of derived categories of sheaves.
{\it Sel. Math. New Ser.} \textbf{24} (2018), 1805–1911. 
 
 \bibitem[SS20]{SS}  Sancho de Salas, F.; Torres Sancho, J.\,M.: 
Derived category of finite spaces and Grothendieck duality.
\emph{Mediterr. J. Math.} \textbf{17:80} (2020), 22pp.


%

\bibitem[St24]{Stacks} Stacks Project, {\tt http://stacks.math.columbia.edu}

%

\bibitem[TT90]{TT}
Thomason, R.\,W.\, Trobaugh, T.:
Higher algebraic K-theory of schemes and of derived categories.
\emph{The Grothendieck Festschrift} Vol.\,III, Progr. Math., no.\,88, 
Birkh\"auser, Boston, 1990, 247--435. 

\bibitem[V68]{V} Verdier, J.-L.:
Base change for twisted inverse image of coherent sheaves.
{\it Algebraic\- Geometry} (Bombay, 1968).  Oxford University Press,
London, 1969, 393--408.

\end{thebibliography}
\end{document}